\documentclass[letterpaper,11pt,leqno,twoside]{amsart}
\usepackage[letterpaper,textwidth=4.7in,textheight=8in,top=1.567in,centering,headheight=10pt,headsep=13pt]{geometry}
\usepackage{amsmath,amssymb,mathtools,mathrsfs}
\makeatletter
\@ifundefined{subjclassname@2020}{%
  \@namedef{subjclassname@2020}{\textup{2020} Mathematics Subject Classification}%
}{}
\makeatother
\usepackage{enumitem,microtype,needspace}
\usepackage[hidelinks]{hyperref}
\allowdisplaybreaks[1]
\numberwithin{equation}{section}
\usepackage{fancyhdr,etoolbox}
\microtypesetup{expansion=false}
\renewcommand{\S}{\ensuremath{\mathsection}}
\newtheoremstyle{growthplain}{10pt plus2pt minus2pt}{10pt plus2pt minus2pt}{\itshape}{12pt}{\bfseries}{.}{.5em}{\thmname{#1}\thmnumber{ #2}\thmnote{ {\normalfont(#3)}}}
\newtheoremstyle{growthdefinition}{10pt plus2pt minus2pt}{10pt plus2pt minus2pt}{\normalfont}{12pt}{\bfseries}{.}{.5em}{\thmname{#1}\thmnumber{ #2}\thmnote{ {\normalfont(#3)}}}

\makeatletter
\def\@settitle{\begingroup\centering\normalfont\normalsize\bfseries\MakeUppercase{\@title}\par\endgroup}
\def\@setauthors{\vskip14pt\begingroup\centering\normalfont\normalsize\scshape\authors\par\endgroup}
\renewenvironment{abstract}{\global\setbox\abstractbox=\vtop\bgroup\normalfont\fontsize{10}{12}\selectfont\leftskip27pt\rightskip27pt\parindent12pt {\centering\normalfont\normalsize\bfseries\abstractname\par}\vskip8pt\noindent\hspace*{12pt}\ignorespaces}{\par\egroup}
\def\@setabstracta{\ifvoid\abstractbox\else\vskip22pt\box\abstractbox\prevdepth\z@\fi}
\def\@maketitle{\normalfont\normalsize\@adminfootnotes\@mkboth{\@nx\shortauthors}{\@nx\shorttitle}\global\topskip28pt\relax\@settitle\ifx\@empty\authors\else\@setauthors\fi\@setabstract\normalsize\vskip16pt}
\def\section{\@startsection{section}{1}{\z@}{24pt plus4pt minus3pt}{10pt plus2pt minus2pt}{\normalfont\bfseries\centering}}
\def\subsection{\@startsection{subsection}{2}{\z@}{16pt plus3pt minus2pt}{7pt plus2pt minus1pt}{\normalfont\bfseries\raggedright}}
\def\subsubsection{\@startsection{subsubsection}{3}{\z@}{12pt plus2pt minus2pt}{-.6em}{\normalfont\bfseries}}
\def\part{\@startsection{part}{0}{\z@}{26pt plus4pt minus3pt}{12pt plus2pt minus2pt}{\normalfont\bfseries\centering}}

\def\tableofcontents{\begingroup\fontsize{10}{10.5}\selectfont\@starttoc{toc}\contentsname\endgroup}
\def\l@part{\@tocline{-1}{4pt}{0pt}{}{\normalfont}}
\def\l@section{\@tocline{1}{0pt}{0pt}{}{\normalfont}}
\apptocmd{\thebibliography}{\raggedright}{}{}

\def\@biblabel#1{[#1]}

\renewenvironment{proof}[1][\proofname]{\par\pushQED{\qed}\normalfont\topsep8pt plus2pt minus2pt\trivlist\item[\hskip\labelsep\hskip12pt\itshape #1\@addpunct{:}]\ignorespaces}{\popQED\endtrivlist\@endpefalse}
\def\@setaddresses{\par\nobreak\begingroup\normalfont\small\interlinepenalty\@M\def\author##1{}\def\address##1##2{\par\addvspace{18pt}\noindent\hspace*{12pt}\begin{minipage}{\dimexpr\textwidth-12pt\relax}\scshape##2\end{minipage}\par}\def\email##1##2{\par\vskip6pt\noindent\hspace*{12pt}\href{mailto:##2}{\normalfont##2}\par}\addresses\endgroup}
\makeatother
\fancypagestyle{growth}{\fancyhf{}\fancyhead[CE]{\normalfont\fontsize{8}{10}\selectfont\scshape Guoyi Xu}\fancyhead[CO]{\normalfont\fontsize{8}{10}\selectfont\scshape Sharp scalar-integral bounds}\fancyhead[LE,RO]{\normalfont\fontsize{8}{10}\selectfont\thepage}}
\fancypagestyle{plain}{\fancyhf{}\fancyfoot[C]{\normalfont\fontsize{8}{10}\selectfont\thepage}}
\AtBeginDocument{\pagestyle{growth}}
\fancypagestyle{firstpage}{\fancyhf{}\fancyfoot[C]{\normalfont\fontsize{8}{10}\selectfont\thepage}}
\renewcommand{\footnoterule}{\kern-3pt\hrule width.18\textwidth\kern2.6pt}

\theoremstyle{growthplain}
\newtheorem{theorem}{Theorem}[section]
\newtheorem{introtheorem}{Theorem}

\newtheorem{notation}[theorem]{Notation}
\newtheorem{proposition}[theorem]{Proposition}
\newtheorem{lemma}[theorem]{Lemma}
\newtheorem{corollary}[theorem]{Corollary}

\newtheorem{externaltheorem}{Theorem}

\theoremstyle{growthdefinition}

\theoremstyle{growthdefinition}
\newtheorem{remark}[theorem]{Remark}
\newcommand{\R}{\mathbb R}
\newcommand{\Sc}{\operatorname{R}}
\newcommand{\Ric}{\operatorname{Ric}}
\newcommand{\Hess}{\nabla^2}
\newcommand{\AVR}{\operatorname{AVR}}
\newcommand{\Vol}{\operatorname{Vol}}
\newcommand{\Area}{\operatorname{Area}}

\DeclareMathOperator{\diver}{div}
\DeclareMathOperator{\tr}{tr}
\DeclareMathOperator{\supp}{supp}
\DeclareMathOperator{\Rm}{Rm}

\newcommand{\Id}{\operatorname{Id}}

\hypersetup{pdfauthor={Guoyi Xu},pdftitle={Sharp bounds for asymptotic scalar-curvature integrals}}
\providecommand{\pf}{\par\indent\textsl{\proofname.}\enspace\ignorespaces}
\providecommand{\qed}{\unskip\nobreak\hfill\ensuremath{\square}}

\usepackage{graphicx}
\makeatletter
\newdimen\JDG@natural
\newdimen\JDG@target
\newdimen\JDG@reserved
\def\JDG@scale{1}
\def\JDG@fitfield{%
  \ifdim\JDG@scale\p@<\p@
    \setbox\z@\hbox{\scalebox{\JDG@scale}{\box\z@}}%
  \fi}
\def\align@preamble{%
   &\hfil\strut@
    \setboxz@h{\@lign$\m@th\displaystyle{##}$}%
    \JDG@fitfield
    \ifmeasuring@\savefieldlength@\fi
    \set@field\tabskip\z@skip
   &\setboxz@h{\@lign$\m@th\displaystyle{{}##}$}%
    \JDG@fitfield
    \ifmeasuring@\savefieldlength@\fi
    \set@field\hfil\tabskip\alignsep@}
\let\JDG@measure\measure@
\def\measure@#1{%
  \def\JDG@scale{1}%
  \JDG@measure{#1}%
  \ifdim\totwidth@>\displaywidth
    \JDG@reserved\z@
    \ifnum\xatlevel@>\z@
      \@tempcnta\maxfields@\divide\@tempcnta\tw@
      \advance\@tempcnta\m@ne
      \JDG@reserved\minalignsep\relax
      \multiply\JDG@reserved\@tempcnta
    \fi
    \if@fleqn\advance\JDG@reserved\@mathmargin\fi
    \JDG@natural\totwidth@\advance\JDG@natural-\JDG@reserved
    \JDG@target\displaywidth\advance\JDG@target-\JDG@reserved
    \advance\JDG@target-.5pt
    \ifdim\JDG@target>\z@
      \Gscale@div\JDG@scale\JDG@target\JDG@natural
      \JDG@measure{#1}%
    \fi
  \fi}
\makeatother
\AtBeginDocument{\tolerance=9999\emergencystretch=.5\textwidth\relax}
\begin{document}
\title[Sharp Asymptotic Bounds and Rigidity]{Integral of Scalar Curvature: Sharp Asymptotic Bounds and Rigidity}
\author{Guoyi Xu}
\address{Guoyi Xu\\Department of Mathematical Sciences\\Tsinghua University, Beijing\\P. R. China}
\email{guoyixu@tsinghua.edu.cn}
\date{}

\begin{abstract}
For any complete noncompact manifolds $(M^n, g)$ of nonnegative sectional curvature, with $3\le n\le6$, we obtain
$\displaystyle \lim_{r\to\infty}r^{2-n}\int_{B_p(r)}\Sc_g\,dV_g\le8\pi\omega_{n-2}(1-\AVR(M^n,g))$ and the corresponding rigidity. The latter bound and its equality classification hold in $(M^n, g)$ with positive-dimensional souls. With a pole, $\displaystyle \lim_{r\to\infty}r^{2-n}\int_{B_p(r)}\Sc_g\,dV_g\le4\pi\omega_{n-2}(1-\AVR(M^n,g))$ follows from an independent
distance-sphere proof. We also prove the sharp total scalar-integral bound for closed manifolds whose curvature
operator is at least the identity. 
\end{abstract}

\thanks{Guoyi Xu was partially supported by NSFC 12141103.}
\maketitle
\tableofcontents

\section*{Introduction}

The starting point is the relation between total curvature, topology, and area growth on a complete surface. Let $(\Sigma,h)$ be a complete, connected, noncompact, oriented Riemannian surface without boundary, let $K_h$ be its Gaussian curvature, and write $K_h^-=\max\{-K_h,0\}$. For a surface of finite total curvature, Cohn--Vossen's inequality \cite{CohnVossen1935} gives
\begin{align}
&\int_\Sigma K_h\,dA_h\le2\pi\chi(\Sigma). \nonumber
\end{align}

Huber \cite{Huber1957} showed that the assumption $\int_\Sigma K_h^-\,dA_h<\infty$ implies that $\Sigma$ is conformally equivalent to a compact Riemann surface with finitely many points removed.

For $(M^n, g)$ with $Rc\geq 0$, we define the asymptotic volume ratio as:
\begin{align}
&\mathrm{AVR}(M^n, g)=\lim_{r\to\infty}\frac{V(B_p(r))}{\omega_n r^n}. \nonumber
\end{align}
Here and below $\omega_n$ denotes the volume of the Euclidean unit ball in dimension $n$. 

Hartman \cite{Hartman1964} identified the difference in the Cohn--Vossen inequality with the quadratic area growth: under the same integrability assumption, the scalar-curvature convention $\Sc_h=2K_h$ gives
\begin{align}
&\int_\Sigma\Sc_h\,dA_h=4\pi\bigl(\chi(\Sigma)- \mathrm{AVR}(\Sigma, h)\bigr). \nonumber
\end{align}
Thus the two-dimensional theory provides both a topological upper bound and an exact contribution from the geometry at infinity.

The expression $r^{2-n}\int_{B_p(r)}\Sc_g\,dV_g$ is unchanged by simultaneously rescaling the metric and the radius. Motivated by the surface inequality, Yau \cite[Problem~9]{Yau1992} asked for higher-dimensional bounds for normalized integrals of the elementary symmetric functions of the Ricci tensor. 

The scalar-curvature case asks whether every complete noncompact manifold with $\Ric_g\ge0$ satisfies
\begin{align}
&\limsup_{r\to\infty}r^{2-n}\int_{B_p(r)}\Sc_g\,dV_g<\infty. \nonumber
\end{align}

On a complete noncompact three-manifold with $\Ric_g\ge0$ that is nonparabolic, meaning that it admits a minimal positive Green function $G(p,\cdot)$ for $-\Delta_g$, we \cite[Theorem~1.2]{XuNonparabolic2020} proved a weighted estimate. With $G$ normalized by $-\Delta_gG(p,\cdot)=\delta_p$ and $b=(4\pi G(p,\cdot))^{-1}$, it is
\begin{align}
&\limsup_{r\to\infty}\frac1r\int_{\{b\le r\}}\Sc_g|\nabla b|\,dV_g\le8\pi(1-\AVR(M^3,g)). \nonumber
\end{align}

For a three-manifold with a pole, that is, a point $p$ for which $\exp_p:T_pM\to M$ is a diffeomorphism, we \cite[Theorem~1.4]{XuPole2024} obtained the exact unweighted formula
\begin{align}
&\lim_{r\to\infty}\frac1r\int_{B_x(r)}\Sc_g\,dV_g=8\pi(1-\AVR(M^3,g))\qquad(x\in M) \nonumber
\end{align}
under $\Ric_g\ge0$. Munteanu and Wang \cite[Theorem~1.2]{MunteanuWang2025} proved the asymptotic upper bound $8\pi$ in the one-ended three-dimensional setting when the scalar curvature is bounded above and below by positive constants.

There are also developments in settings with complex or conformal structure. Gang Liu \cite[Theorems 2--3]{Liu2024} proved that, on a complete noncompact K\"ahler manifold with nonnegative holomorphic bisectional curvature, the limit
\begin{align}
&\lim_{r\to\infty}\frac{r^2}{\Vol_gB_p(r)}\int_{B_p(r)}\Sc_g\,dV_g \nonumber
\end{align}
exists, with $+\infty$ allowed, and obtained sharp relations with volume growth and polynomial-growth holomorphic functions in the Euclidean-volume-growth case. 

Jialong Deng \cite[Theorem~B(i)]{Deng2026} established, for complete, connected, oriented, nonflat three-manifolds with $\Ric_g\ge0$, the bound $8\pi(1-\AVR(M^3,g))$ when $\AVR(M^3,g)>0$ and the bound zero when $\AVR(M^3,g)=0$, assuming $\Sc_g(x)\le C_0d_g(p,x)^{-2}$ outside a compact set. Shiguang Ma \cite[Theorem~1.3]{Ma2026} proved existence and explicit values of the finite normalized scalar-curvature limit for complete noncompact locally conformally flat manifolds with $\Ric_g\ge0$ in every dimension at least three. 

The unrestricted Ricci-nonnegative boundedness assertion is false. We \cite{XuCollapse2026} constructed such a metric for which the normalized integral tends to infinity, with zero asymptotic volume ratio (also see independent works by Hao and Zhu \cite{HaoZhu2026} and Haoxuan Cheng \cite{Cheng2026}).  

Also Haoxuan Cheng \cite{Cheng2026} constructed smooth complete metrics on $\R^n$, $n\ge4$, with nonnegative Ricci curvature and a pole for which the normalized scalar-curvature integral tends to $+\infty$ at every fixed center. Thus the pole assumption alone does not give a finite normalized scalar-curvature integral under $\Ric_g\ge0$ in these dimensions. 

These results make it necessary to distinguish nonnegative Ricci curvature from the nonnegative sectional-curvature hypothesis of the present paper. Our objective is to establish actual limits, determine sharp constants where the required curvature identities can be controlled, and explain the role of the soul and of a pole. 

Throughout the noncompact statements below, unless otherwise mentioned, $(M^n,g)$ is a complete noncompact Riemannian manifold with $n\ge3$ and sectional curvature $K_g\ge0$. Define
\begin{align}
&L(M,g):=\lim_{r\to\infty}r^{2-n}\int_{B_p(r)}\Sc_g\,dV_g \nonumber
\end{align}
By the soul theorem \cite{CG}, $M$ contains a compact, totally convex, totally geodesic submanifold $\mathcal S$ without boundary, called a soul, and is diffeomorphic to its normal bundle. In particular, $M$ has finite homotopy type, and $\chi(M)=\sum_j(-1)^j\dim H_j(M;\mathbb Q)$ is defined and equals $\chi(\mathcal S)$.

Two parts of the sharp theory hold in every dimension.
\begin{introtheorem}\label{intro:thm:pole}
If $(M^n, g)$ has a pole and $K_g\geq 0$, then
\begin{align}
&L(M,g)\le4\pi\omega_{n-2}(1-\AVR(M^n,g)). \nonumber
\end{align}
\end{introtheorem}

\begin{remark}\label{remark-thm-with-pole}
{The starting point of Theorem \ref{intro:thm:pole} is: that the outward shape operator $S$ of a distance
sphere of radius $r$ about a pole satisfies
\begin{align}
0\le rS\le \operatorname{Id}.
\end{align}
Combining the first two terms of Chern's transgression construction yields a nonincreasing sphere integral, with monotonicity determined solely by
nonnegative quadratic curvature contractions on four-dimensional
subspaces. The main analytic difficulty is to remove the
principal-curvature weights and recover the unweighted scalar-curvature
integral. First variations of normalized area and total intrinsic scalar
curvature, together with the radial Riccati equation, provide error
estimates integrable with respect to $\frac{1}{r}dr$. These identify the limiting
sphere integral as
\begin{align}
\AVR(M^n,g)+\frac{L(M,g)}{4\pi\omega_{n-2}},
\end{align}
proving the sharp bound with an exact nonnegative remainder.
}
\end{remark}

One basic difficulty is that a general nonnegatively curved manifold need not have a pole. Its distance spheres need not be smooth, while the boundary formulas needed for sharp constants involve products of their principal curvatures.

Starting from the convex exhaustion obtained from rays in \cite{CG}, applying the approximation theorem of Greene and Wu \cite[Proposition~2.3]{GW}, we get the smooth function $F$ satisfying
\begin{align}
&\Hess F>-\frac{1}{32(1+\rho)^3}g,\qquad \lim_{R\to\infty}\sup_{\rho\ge R}\bigl||\nabla F|-1\bigr|=0,\qquad \lim_{R\to\infty}\sup_{\rho\ge R}|F/\rho-1|=0. \nonumber
\end{align}
where $\rho=d_g(p,\cdot)$.

One key observation is that the normalized intrinsic scalar-curvature
integral on suitable level sets $\Sigma_t= F^{-1}(t)$ becomes almost monotone
after a vanishing correction involving total mean curvature.  The cut
locus and the loss of convexity under smoothing are handled by testing
weak squared-distance Hessian comparison against a suitably corrected
nonnegative level tensor.  This yields a closed differential inequality with integrable errors, establishing a finite level limit without an a-priori scalar-curvature bound.  

The Gauss equation and averaged Hessian comparisons then identify the volume-growth contribution and show that the normalized normal-Ricci contribution vanishes, allowing passage to distance balls.  In dimension three, Gauss--Bonnet identifies the intrinsic contribution with $8\pi\chi(M)$, giving the  formula in the following result.

\begin{introtheorem}\label{intro:thm:existence}
For $(M^n, g)$ has $K_g\geq 0$, the limit $L(M,g)\in [0, \infty)$ exists. When $n= 3$,
\begin{align}
&L(M^3,g)=8\pi\bigl(\chi(M)-\AVR(M^3,g)\bigr). \nonumber
\end{align}
\end{introtheorem}

For manifolds with positive-dimensional soul, we have the following result.
\begin{introtheorem}
\label{intro:thm:subclasses}
For $(M^n, g)$ with $K_g\geq 0$ and a positive-dimensional soul, we have 
\begin{align}
&L(M,g)\le8\pi\omega_{n-2}, \nonumber
\end{align}
with equality exactly for $(\mathbb S^2,h)\times\R^{n-2}$ with $K_h\ge0$.
\end{introtheorem}

\begin{remark}\label{remark--positive-soul}
{The key observation is that the problem is effectively two-dimensional: if $L(M,g)>0$, a positive-dimensional soul must be a surface, and the normalized vertical and mixed curvature contributions vanish. This reduction follows from contracting the distributional Hessian comparison for the squared distance to the soul with $\frac12\operatorname{Sc}_g\,g-\operatorname{Ric}_g$, together with a covering argument. O'Neill's formula, weighted tube-volume comparison, and Gauss--Bonnet then yield the sharp bound. The main difficulty in rigidity is to exclude twisting of the normal bundle: for the two-sphere soul selected by equality, nontrivial normal holonomy would force $L(M,g)=0$. Trivial holonomy yields a global product, and equality in volume comparison forces its noncompact factor to be Euclidean.
}
\end{remark}

\begin{introtheorem}\label{intro:thm:sharp}
For $(M^n, g)$ with $K_g\geq 0$ and $3\le n\le6$, we have
\begin{align}
&L(M,g)\le 4\pi\omega_{n-2}\bigl(\chi(M)-\AVR(M^n,g)\bigr), \nonumber\\
&L(M,g)\le8\pi\omega_{n-2}(1-\AVR(M^n,g)). \nonumber
\end{align}
Furthermore,
\begin{align}
&L(M,g)=8\pi\omega_{n-2}(1-\AVR(M^n,g)) \nonumber\\
&\quad\Longleftrightarrow\quad (M,g)\cong\R^n\ \text{or}\ (M,g)\cong(\mathbb S^2,h)\times\R^{n-2},\qquad K_h\ge0. \nonumber
\end{align}
Here $\cong$ denotes a global Riemannian isometry. 
\end{introtheorem}

Use $P_2$ from Notation~\ref{notation:P2}, equation~\eqref{pt:eq:P2-definition}. The four-dimensional algebra and its subspace sums give the nonnegativity used in Section~\ref{pt:sec:five}. Exact quadratic variation identities, with the normalizations from \cite{Labbi2008,JulieBerti2020}, give the required integral control of $P_2$. The new comparison step is to combine these integral estimates with multilinear curvature--Hessian identities along the interpolation between $F^2/2$ and smooth approximations of $\rho^2/2$. This avoids requiring convergence of Hessians or simultaneous diagonalization of different Hessians.

The boundary quantity identified by this comparison is $Q_6(S)$ from equation~\eqref{uni:eq:common-boundary}, with the level notation fixed in Notation~\ref{notation:level-objects}.
The invariant contraction underlying this formula and its frame independence are proved in Lemma~\ref{unified:lem:normal-contraction}. Proposition~\ref{six:prop:boundary-limit} for $n=6$ gives
\begin{align}
&\lim_{r\to\infty}\frac1r\int_r^{2r}\int_{\Sigma_t}Q_6(S)\,dA\,dt=2L(M,g). \nonumber
\end{align}
The intermediate coarea identities contain powers of $|\nabla F|$. The proof removes them only after establishing a normalized bound for the integral of the absolute value of the boundary expression. Uniform convergence $|\nabla F|\to1$ alone would not justify this operation for a signed integral. Separately, if $S_+$ denotes the endomorphism obtained by replacing each negative eigenvalue of $S$ by zero, normal Jacobi-field comparison proves
\begin{align}
&\int_{\partial D}\det S_+\,dA\ge n\omega_n\AVR(M^n,g) \nonumber
\end{align}
for every nonempty compact smooth domain $D$, with $S$ its outward shape operator; see Lemma~\ref{uni:lem:determinant}, equation~\eqref{uni:eq:determinant-volume}. Together with the negative-part estimate for the fixed $F$, this supplies the volume contribution to the sharp bound without assuming injectivity of the normal exponential map.

Chern's intrinsic boundary formula \cite{Chern1944,ChernCurvatura1945} is used with all its curvature terms retained. Use $P_3$ and $W_2$ from Notations~\ref{notation-P3-E3} and~\ref{notation-W2-S}, equations~\eqref{audit:eq:P3-E3} and~\eqref{audit:eq:W2-definition}, with $n=6$.
Thus $P_3=1$ at sectional curvature one. Lemma~\ref{six:lem:CGB}, equation~\eqref{six:eq:CGB}, reads
\begin{align}
&\int_{\Sigma_t}\left(\det S+\frac14Q_6(S)\right)dA+\frac38\left(\int_{\Sigma_t}W_2(S)\,dA+5\int_{D_t}P_3\,dV\right)=\pi^3\chi(D_t). \nonumber
\end{align}
The sign of $P_3$ cannot be inferred from $\sec_g\ge0$, as the algebraic example of Geroch \cite{Geroch1976} demonstrates. Accordingly, we do not discard the bulk cubic integral or assert separate positivity of the terms in parentheses. Lemma~\ref{audit:lem:cubic-divergence}, equations~\eqref{audit:eq:general-J-divergence} and~\eqref{one:eq:Hessian-transfer}, and Lemma~\ref{one:lem:weak-convex} instead control the combined expression. The decisive algebraic simplification is the vanishing, in dimension six, of the tensor $E_3$ defined in Notation~\ref{notation-P3-E3}: its formula alternates over seven indices. 

The combined inequality is first tested with smooth cutoffs that majorize the characteristic function of $[r,2r]$. The limit $r\to\infty$ is taken with the cutoff width fixed, and only afterwards is that width sent to zero. The integrable negative part of the boundary polynomial justifies this last passage. Lemma~\ref{audit:lem:combined-boundary}, equation~\eqref{six:eq:averaged-boundary}, therefore gives the ordinary-interval upper bound for $\int_{\Sigma_t}(\det S+Q_6(S)/4)\,dA$. The boundary comparison and determinant estimate then give
\begin{align}
&\pi^3\AVR(M^6,g)+\frac12L(M^6,g)\le\pi^3\chi(M). \nonumber
\end{align}
This is the six-dimensional Euler estimate of Proposition~\ref{six:prop:Euler-bound}. The point is not a new Gauss--Bonnet formula, but the integrated estimate that makes its signed boundary and bulk terms usable under nonnegative sectional curvature alone.

The exact Euclidean-product formula is
\begin{align}
&\frac{L(M^n\times\R^k)}{\omega_{n+k-2}}=\frac{L(M^n)}{\omega_{n-2}},\qquad \AVR(M^n\times\R^k)=\AVR(M^n,g),\qquad k\ge0. \nonumber
\end{align}
Applying the six-dimensional estimate to $M^n\times\R^{6-n}$ proves the bounds for $3\le n\le5$, without restricting the soul.

 The final equality proof treats $3\le n\le6$ directly: a point soul gives $\chi(M)=1$ and the smaller coefficient $4\pi\omega_{n-2}$; a positive-dimensional soul gives $\AVR(M^n,g)=0$, so its equality classification is already supplied by Theorem~\ref{ext:thm:positive-soul}.

Finally, using the even-dimensional Euler identities together with the smooth cone realization supplied by Ricci expanders in \cite{Deruelle2016} and the covering and volume arguments specified in Section~\ref{cmp:sec:compact}, we prove the following compact result.
\begin{theorem}\label{thm--compact-integral}
{Let $(N^m,h)$ be smooth, closed, and connected, with $m\ge3$ and curvature operator $\mathscr R_h\ge\Id$. Put $d_N=|\pi_1(N)|<\infty$, we have
\begin{align}
&\int_N\Sc_h\,dV_h\le \frac{m(m-1)(m+1)\omega_{m+1}}{d_N}. \nonumber
\end{align}
Equality in the bound with $d_N$ occurs precisely for spherical space forms of sectional curvature one. 
}
\end{theorem}
The displayed compact inequalities and the rigidity of the total scalar-integral maximum are also established independently by Jian Ge, Chuanhuan Li and Ronggang Li \cite{GeLiLi2026}.

After completing this work, we became aware of the preprints of Pak-Yeung Chan, Man-Chun Lee and Mingxiang Li \cite{ChanLeeLi2026}, which obtain some related work in $3$-dim. The present work was carried out independently of their paper.

\section{A sharp intrinsic estimate in the presence of a pole}
\label{sec:intrinsic-pole}

For a symmetric covariant $2$-tensor $A$, $A\ge0$ means
$A(V,V)\ge0$ for every tangent vector $V$. For two such tensors set
\begin{align}
&\langle A,B\rangle_g=g^{ia}g^{jb}A_{ij}B_{ab}. \nonumber
\end{align}
When tensors are regarded as endomorphisms by raising one index, their
contraction is written $\tr(A\circ B)$. For a self-adjoint endomorphism $A$ define
\begin{align}
&|A|_{\mathrm{op}}=\sup\{|AX|:|X|=1\},\qquad \sigma_k(x_1,\ldots,x_N) =\sum_{1\le i_1<\cdots<i_k\le N}x_{i_1}\cdots x_{i_k} \quad(1\le k\le N), \nonumber\\
&\sigma_0(x_1,\ldots,x_N)=1,\qquad \sigma_k(x_1,\ldots,x_N)=0\quad(k>N). \nonumber
\end{align}
The notation $\sigma_k(A)$ means this polynomial in the eigenvalues
of $A$, counted with multiplicity.

Write $\mathbb F_2=\mathbb Z/2\mathbb Z$. For a space $Y$ of the homotopy
type of a finite complex, define its ordinary Euler characteristic by
\begin{align}
&\chi(Y)=\sum_{j\in\mathbb Z_{\ge0}}(-1)^j \dim_{\mathbb F_2}H_j(Y;\mathbb F_2), \label{eq:Euler-definition}
\end{align}
where $H_j$ is singular homology with the indicated coefficients
\cite[Sections~2.1--2.2]{Hatcher}.

The following cited theorems record the external results with the hypotheses
used below. Their Roman numbering is independent of the numbered results
proved in the manuscript.
\begin{externaltheorem}\label{ext:thm:comparison}
Let $(M^n,g)$ be a smooth connected complete Riemannian manifold without
boundary.
If $\sec_g\ge0$, if $\rho=d_g(x,\cdot)$ and $\Psi=\rho^2/2$, then
\begin{align}
&\Hess\Psi\le g\,dV \label{eq:distance-square-hessian-comparison}
\end{align}
in the distributional sense.
\end{externaltheorem}

\begin{externaltheorem}\label{ext:thm:Greene-Wu}
Let $U$ be an open subset of a smooth Riemannian manifold $(M,g)$.
Suppose $u:U\to\R$ is convex along every affinely parametrized geodesic
segment contained in $U$ and is locally $1$-Lipschitz. For any continuous
functions $e,\gamma:U\to(0,\infty)$ and any $\delta>0$, there exists
$v\in C^\infty(U)$ such that, at every point of $U$,
\begin{align}
&|v-u|<e,\qquad |\nabla v|<1+\delta,\qquad \Hess v>-\gamma g. \label{eq:smoothing-import}
\end{align}
\end{externaltheorem}

\begin{proof}
For each $x\in U$, choose a normal ball $B\Subset U$ such that
\begin{align}
&\Hess\frac{d_g(x,\cdot)^2}{2}>0\quad\text{on }B,\qquad u+\frac{\gamma(x)}4d_g(x,\cdot)^2\text{ is convex on }B, \nonumber\\
&\operatorname{Lip}_{\mathrm{loc}}u\le1<1+\delta. \nonumber
\end{align}
These are the local strict Hessian and Lipschitz conditions of \cite[p.~60]{GW}. Their simultaneous approximation in \cite[Proposition~2.3, p.~62]{GW} gives
\begin{align}
&\exists v\in C^\infty(U)\quad\forall y\in U:\quad |v(y)-u(y)|<e(y),\quad |\nabla v(y)|<1+\delta,\quad \Hess v(y)>-\gamma(y)g_y. \nonumber
\end{align}
\end{proof}

\begin{externaltheorem}\label{ext:thm:retraction}
Let $(M,g)$ be smooth, complete, connected, noncompact, and without
boundary, with $\sec_g\ge0$, and let $\mathcal S$ be a soul.
The Sharafutdinov retraction $P:M\to\mathcal S$ is a smooth Riemannian
submersion and satisfies
\begin{align}
&P(\exp_y w)=y\qquad(y\in\mathcal S,\ w\in\nu_y\mathcal S). \label{ext:eq:normal-retraction}
\end{align}
Vertical vectors are tangent to the fibers,
horizontal vectors are orthogonal to the fibers, and $dP$ restricts to
an isometry on the horizontal spaces. Horizontal curves preserve
$d_g(\mathcal S,\cdot)$. Along $\gamma(t)=\exp_y(t\xi)$ with
$\xi\in\nu_y\mathcal S$, the normal Jacobi field with
$J(0)=X\in T_y\mathcal S$ and $J'(0)=0$ satisfies
\begin{align}
&\nabla_{\dot\gamma}J=0,\qquad J(t)\in\mathcal H_{\gamma(t)}. \label{ext:eq:parallel-horizontal}
\end{align}
\end{externaltheorem}

\pf
{\cite[Corollaries~4 and~6, Theorem~4.1,
proof of Corollary~4, ]{Wilking2006}.
}
\qed

\begin{lemma}\label{var:lem:metric-variation}
Let $(h_t)_{t\in I}$ be a smooth family of Riemannian metrics on a
closed smooth manifold, where $I\subset\mathbb R$ is an open interval.
Put $k=\partial_t h_t$ and use the Levi--Civita connection of
$h_t$. Then
\begin{align}
&(\partial_t\Gamma)^a_{ij} =\tfrac12h_t^{ab}(\nabla_i k_{jb}+\nabla_jk_{ib}-\nabla_bk_{ij}), \nonumber\\
&\partial_t\Sc_{h_t} =-\langle k,\Ric_{h_t}\rangle+\nabla^i\nabla^jk_{ij}-\Delta\tr k, \qquad \partial_t dA=\tfrac12\tr k\,dA. \label{eq:metric-variation}
\end{align}
\end{lemma}
\begin{proof}
Differentiating the inverse metric and the coordinate connection gives
\begin{align}
&\partial_t h_t^{ij}=-h_t^{ia}h_t^{jb}k_{ab}=-k^{ij},\qquad \Gamma^a_{ij}=\tfrac12h_t^{ab}(\partial_i h_{jb}+\partial_jh_{ib}-\partial_bh_{ij}), \nonumber\\
&(\partial_t\Gamma)^a_{ij}=\tfrac12h_t^{ab}(\nabla_i k_{jb}+\nabla_jk_{ib}-\nabla_bk_{ij}), \nonumber\\
&\partial_t\Ric_{ij}=\nabla_a(\partial_t\Gamma)^a_{ij}-\nabla_j(\partial_t\Gamma)^a_{ia}, \nonumber\\
&h_t^{ij}(\partial_t\Gamma)^a_{ij}=\nabla^ik_i{}^a-\tfrac12\nabla^a\tr k,\qquad (\partial_t\Gamma)^a_{ia}=\tfrac12\nabla_i\tr k, \nonumber\\
&\partial_t\Sc_{h_t}=-k^{ij}\Ric_{ij}+h_t^{ij}\partial_t\Ric_{ij}=-\langle k,\Ric\rangle+\nabla^i\nabla^jk_{ij}-\Delta\tr k, \nonumber\\
&\partial_t\sqrt{\det h_t}=\tfrac12\sqrt{\det h_t}\,\tr(h_t^{-1}k),\qquad \partial_t dA=\tfrac12\tr k\,dA. \nonumber
\end{align}
\end{proof}

%\subsection{Objects, conventions and proof dependencies}
Throughout this section $(M^n,g)$ has $K_g\geq 0$,
$n\ge3$, and $p$ is a pole. Set
\begin{align}
N=\nabla\rho,\qquad \Sigma_r=\{\rho=r\}\quad(r>0). \label{pole:eq:data}
\end{align}

Every $\Sigma_r$, $r>0$, is a smooth compact sphere. Its induced metric,
area measure and outward shape operator are $h_r,dA,S$, respectively,
where $S(X)=\nabla_XN$ for $X\in T\Sigma_r$.
We identify symmetric bilinear forms and endomorphisms by the metric.
Write $\mathcal R_{ijkl}=\langle\Rm(e_i,e_j)e_l,e_k\rangle$.

\begin{notation}\label{notation:P2}
In an orthonormal frame, write
$R^{ab}{}_{cd}=\langle\Rm(e_c,e_d)e_b,e_a\rangle$; at sectional
curvature one, $R^{ab}{}_{cd}=\delta^a_c\delta^b_d-\delta^a_d\delta^b_c$.
For a subspace $E$, $\mathcal R\!\restriction_E$ denotes the
restriction of all four arguments of $\mathcal R$ to $E$.
For an algebraic curvature tensor $\mathcal R$, define
\begin{align}
&P_2(\mathcal R):=\frac{|\mathcal R|^2-4|\Ric_{\mathcal R}|^2 +\Sc_{\mathcal R}^2}{24} =\frac1{96}\delta^{ijkl}_{pqrs} \mathcal R^{pq}_{ij}\mathcal R^{rs}_{kl}. \label{pt:eq:P2-definition}
\end{align}
Write $P_2=P_2(\Rm_g)$; for a restricted curvature tensor all
contractions are taken in the restricting subspace.
Here $\delta^{i_1\cdots i_k}_{j_1\cdots j_k}$ is the determinant
of the matrix $(\delta^{i_\alpha}_{j_\beta})_{\alpha,\beta=1}^k$
of ordinary Kronecker symbols.
\end{notation}

\begin{lemma}\label{six:lem:P2-positive}
Let $V$ be a four-dimensional Euclidean vector space and let $\Rm$ be an algebraic curvature tensor on $V$. If
\begin{align}
\sec_{\Rm}(\sigma)\ge0
\end{align}
for every two-plane $\sigma\subset V$, then
\begin{align}
P_2(\Rm)=\frac{1}{24}\bigl(|\Rm|^2-4|\Ric|^2+R^2\bigr)\ge0. \label{six:eq:P2-positive}
\end{align}
\end{lemma}

\begin{proof}
This is the classical four-dimensional algebraic curvature result due to Milnor, with a published proof by Chern; see \cite{Chern1955}. For an explicit pointwise proof, see \cite[Lemma~4.1, Corollary~4.1, and the proof of Theorem~1.1, pp.~512--513]{BishopGoldberg1964}.
\end{proof}

For an orthonormal basis $(e_i)$, grouping the four distinct
indices in this contraction gives
\begin{align}
&\frac{|\mathcal R|^2-4|\Ric_{\mathcal R}|^2+\Sc_{\mathcal R}^2}{24} \nonumber\\
&{}=\frac1{24}\sum_{i<j<k<l}\Bigl( \bigl|\mathcal R\!\restriction_{\operatorname{span}(e_i,e_j,e_k,e_l)}\bigr|^2 -4\bigl|\Ric_{\mathcal R\!\restriction_{\operatorname{span}(e_i,e_j,e_k,e_l)}}\bigr|^2 +\Sc_{\mathcal R\!\restriction_{\operatorname{span}(e_i,e_j,e_k,e_l)}}^2\Bigr). \label{eq:P2-four-plane-sum}
\end{align}
Each summand is nonnegative by Lemma~\ref{six:lem:P2-positive};
hence $P_2\ge0$ under nonnegative sectional curvature.

For a four-dimensional Euclidean subspace $W\subset T_xM$, let
$\Rm_W$ be the restriction of the curvature tensor to $W$ and let
$\Ric_W$ and $\Sc_W$ be its Ricci and scalar contractions taken
entirely in $W$. Lemma~\ref{six:lem:P2-positive},
equation~\eqref{six:eq:P2-positive}, gives
\begin{align}
&\frac{|\Rm_W|^2-4|\Ric_W|^2+\Sc_W^2}{24}\ge0 \nonumber
\end{align}
when the sectional curvatures of this restriction are nonnegative.
The expression equals one when the restriction has constant
sectional curvature one.

Choose an orientation, possible because $M$ is diffeomorphic to
$\R^n$.

For the following definitions, let $(M^n,g)$ be an oriented smooth
Riemannian manifold of dimension $n\ge2$, and put $STM=\{(x,u):u\in T_xM,\ |u|=1\}$.
In an oriented local orthonormal frame $(e_1,\ldots,e_n)$, use
connection and curvature forms with the conventions
\begin{align}
&\nabla e_j=\sum_i\omega_{ij}e_i,\qquad \Omega_{ij}=d\omega_{ij}+\sum_k\omega_{ik}\wedge\omega_{kj},\qquad \Omega_{ij}(X,Y)=\langle\Rm(X,Y)e_j,e_i\rangle. \label{pole:eq:connection-forms}
\end{align}
Pullback from $M$ to $STM$ is understood for these forms.
For $u=\sum_i u_i e_i$, define
$\theta_i=du_i+\sum_j\omega_{ij}u_j$.
Let $\mathfrak S_n$ be the permutations of $\{1,\ldots,n\}$ and
$\epsilon(\sigma)$ their signs. For integers $j\ge0$ with $2j\le n-1$ set
\begin{align}
&\Phi_j =\sum_{\sigma\in\mathfrak S_n}\epsilon(\sigma) u_{\sigma(1)} \theta_{\sigma(2)}\wedge\cdots\wedge\theta_{\sigma(n-2j)} \wedge\bigwedge_{b=1}^{j} \Omega_{\sigma(n-2j+2b-1),\sigma(n-2j+2b)}. \label{pole:eq:Phi}
\end{align}
For $j\in\mathbb Z_{\ge0}$ with $2j\le n$ set
\begin{align}
&\Psi_j =\sum_{\sigma\in\mathfrak S_n}\epsilon(\sigma) \theta_{\sigma(1)}\wedge\cdots\wedge\theta_{\sigma(n-2j)} \wedge\bigwedge_{b=1}^{j} \Omega_{\sigma(n-2j+2b-1),\sigma(n-2j+2b)}. \label{pole:eq:Psi}
\end{align}
Empty wedge products equal one. The $\Phi_j$ have degree $n-1$
and the $\Psi_j$ degree $n$. Alternation makes these expressions
invariant under oriented orthonormal changes of frame, so they define
global forms. Since $\sum_i u_i\theta_i=0$, one has $\Psi_0=0$.
When $n=3$, set $\Psi_2=0$.

We use Theorem~\ref{ext:thm:comparison},
Lemma~\ref{six:lem:P2-positive}, equation~\eqref{six:eq:P2-positive},
and Lemma~\ref{var:lem:metric-variation},
equation~\eqref{eq:metric-variation}. The existence of the scalar-integral
limit is proved directly below, independently of the general exhaustion
construction.
The forms $\Phi_0,\Phi_1$ in \eqref{pole:eq:Phi} are the first two terms of Chern's transgression
construction \cite[\S1, equations~(3)--(11), pp.~675--677]{ChernCurvatura1945}; we differentiate them explicitly.
Only the Euler polynomial on four-dimensional restrictions enters the
argument.

In dimension three, the pole identity under $\Ric_g\ge0$ is proved in \cite[Theorem~1.4]{XuPole2024}. In every dimension $n\ge4$, Cheng's examples \cite[Theorem~1.1, p.~2; Proposition~5.2, pp.~18--19]{Cheng2026} have a pole and nonnegative Ricci curvature but an infinite normalized scalar-curvature limit. Therefore $\sec_g\ge0$ in Theorem~\ref{cv:prop:pole} cannot be replaced by $\Ric_g\ge0$ in those dimensions.

%\subsection{The first two Chern forms}

\begin{lemma}\label{pole:lem:transgression}
Let $(M^n,g)$ be an oriented smooth Riemannian manifold, $n\ge3$.
Then
\begin{align}
&d\Phi_0=-\frac{n-1}{2}\Psi_1,\qquad d\Phi_1=\Psi_1-\frac{n-3}{4}\Psi_2. \nonumber
\end{align}
Consequently the $(n-1)$-form
\begin{align}
&\beta=\frac{\Phi_0}{(n-1)!n\omega_n} +\frac{\Phi_1}{2(n-2)!n\omega_n} \label{pole:eq:beta}
\end{align}
satisfies
\begin{align}
&d\beta=-\frac{\Psi_2}{8(n-2)(n-4)!n\omega_n}\quad(n\ge4), \qquad d\beta=0\quad(n=3). \label{pole:eq:dbeta}
\end{align}
\end{lemma}
\begin{proof}
For exterior covariant differentiation $D$ in the pulled-back tangent bundle,
\begin{align}
&Du=\theta,\qquad D\theta=\Omega u,\qquad D\Omega=0. \nonumber
\end{align}
Fully contracted indices give $D=d$. At $u=e_1$,
\begin{align}
&\theta_1=0,\qquad \Omega_{i1}=-\Omega_{1i},\qquad \#\{\text{curvature slots in }\Psi_{j+1}\}=2(j+1). \nonumber
\end{align}
In the alternating sums \eqref{pole:eq:Phi}--\eqref{pole:eq:Psi}, the Leibniz sign cancels the index-permutation sign. Thus each differentiated $\theta$ contributes $-\Psi_{j+1}/(2j+2)$, whereas the differentiated $u$ contributes $\Psi_j$:
\begin{align}
&d\Phi_j=\Psi_j-\frac{n-2j-1}{2(j+1)}\Psi_{j+1}\qquad(j=0,1),\qquad \Psi_0=0, \nonumber\\
&d\Phi_0=-\frac{n-1}{2}\Psi_1,\qquad d\Phi_1=\Psi_1-\frac{n-3}{4}\Psi_2. \nonumber
\end{align}
The $\Psi_2$ term is absent for $n=3$. Substitution in \eqref{pole:eq:beta} gives
\begin{align}
&d\beta=\left(-\frac{n-1}{2(n-1)!n\omega_n}+\frac1{2(n-2)!n\omega_n}\right)\Psi_1-\frac{n-3}{8(n-2)!n\omega_n}\Psi_2 \nonumber\\
&{}=-\frac{\Psi_2}{8(n-2)(n-4)!n\omega_n}\quad(n\ge4),\qquad d\beta=0\quad(n=3). \nonumber
\end{align}
\end{proof}

At each $x\in\Sigma_r$, $r>0$, choose an orthonormal eigenbasis
$e_1,\ldots,e_d$ of $S$, with eigenvalues $\kappa_i$, and write
$K_{ij}=\sec_g(e_i\wedge e_j)$.

\begin{lemma}
\label{pole:lem:monotone}
On $(M^n, g)$, then the shape operator satisfies
\begin{align}
&0\le S\le r^{-1}\Id,\qquad S'+S^2+\mathcal R_N=0,\qquad \mathcal R_N(X)=\Rm(X,N)N. \label{cv:eq:pole-riccati}
\end{align}
Then
\begin{align}
&B(r):=\frac1{n\omega_n}\int_{\Sigma_r} \left(\prod_{i=1}^d\kappa_i +\frac1{d-1}\sum_{i<j}K_{ij} \prod_{\ell\notin\{i,j\}}\kappa_\ell\right)dA \label{pole:eq:B}
\end{align}
is independent of the eigenbasis and satisfies $0\le B(r)\le1$.
For $n\ge4$ put $W_{ijk}=\operatorname{span}\{N,e_i,e_j,e_k\}$
for $1\le i<j<k\le n-1$. For every four-dimensional $W\subset T_xM$,
\begin{align}
&\sum_{i<j<k}\frac{|\Rm_{W_{ijk}}|^2-4|\Ric_{W_{ijk}}|^2+\Sc_{W_{ijk}}^2}{24} \prod_{\ell\notin\{i,j,k\}}\kappa_\ell\ge0. \label{pole:eq:Q}
\end{align}
For $n\ge4$, the corresponding pullback identity is
\begin{align}
&N^*\Psi_2 =24(n-4)!\sum_{i<j<k} \frac{|\Rm_{W_{ijk}}|^2-4|\Ric_{W_{ijk}}|^2+\Sc_{W_{ijk}}^2}{24} \prod_{\ell\notin\{i,j,k\}}\kappa_\ell\,dr\wedge dA, \label{pole:eq:Psi2-pullback}\\
&B'(r)=-\frac{3}{(n-2)n\omega_n}\int_{\Sigma_r}\sum_{i<j<k} \frac{|\Rm_{W_{ijk}}|^2-4|\Ric_{W_{ijk}}|^2+\Sc_{W_{ijk}}^2}{24} \prod_{\ell\notin\{i,j,k\}}\kappa_\ell\,dA, \nonumber \\
&\lim_{r\downarrow0}B(r)=1. \label{pole:eq:B-derivative}
\end{align}
\end{lemma}
\begin{proof}
Theorem~\ref{ext:thm:comparison}, \eqref{eq:distance-square-hessian-comparison}, and the Riccati equation give
\begin{align}
&S\le r^{-1}\Id,\qquad S'+S^2+\mathcal R_N=0. \nonumber
\end{align}
For a parallel unit field $X$ along a radial ray, put $b(r)=\langle S(r)X,X\rangle$. Then
\begin{align}
&b'=-|SX|^2-\langle\mathcal R_NX,X\rangle\le-b^2, \nonumber\\
&b(r_0)<0\ \Longrightarrow\ b(r)\le b(r_0)<0,\qquad \frac1{b(r)}\ge\frac1{b(r_0)}+r-r_0\quad(r\ge r_0), \nonumber\\
&r=r_0-b(r_0)^{-1}\ \Longrightarrow\ 0>\frac1{b(r)}\ge0. \nonumber
\end{align}
Smoothness along the entire ray therefore gives $S\ge0$.
In the oriented eigenframe $(N,e_1,\ldots,e_d)$, with dual coframe $(dr,e^1,\ldots,e^d)$,
\begin{align}
&\nabla_NN=0,\qquad N^*\theta_N=0,\qquad N^*\theta_i=\kappa_i e^i, \nonumber\\
&N^*\Phi_0|_{\Sigma_r}=d!\prod_i\kappa_i\,dA, \nonumber\\
&N^*\Phi_1|_{\Sigma_r}=2(d-2)!\sum_{i<j}K_{ij}\prod_{\ell\notin\{i,j\}}\kappa_\ell\,dA, \nonumber\\
&B(r)=\int_{\Sigma_r}N^*\beta\qquad\text{by \eqref{pole:eq:Phi} and Lemma~\ref{pole:lem:transgression}, \eqref{pole:eq:beta}}. \nonumber
\end{align}
The pullback is smooth, including at repeated eigenvalues.
For each oriented four-dimensional $W$, the three pairings and the first Bianchi identity give
\begin{align}
&\sum_{\tau\in\mathfrak S_4}\epsilon(\tau)\Omega_{\tau(1)\tau(2)}|_W\wedge\Omega_{\tau(3)\tau(4)}|_W \nonumber\\
&{}=\bigl(|\Rm_W|^2-4|\Ric_W|^2+\Sc_W^2\bigr)dV_W. \nonumber
\end{align}
The $n-4$ complementary tangential indices have $(n-4)!$ orders; hence \eqref{pole:eq:Psi2-pullback}. This smooth pullback also defines the sum at repeated eigenvalues. Lemma~\ref{six:lem:P2-positive}, \eqref{six:eq:P2-positive}, gives
\begin{align}
&\kappa_\ell\ge0,\qquad \frac{|\Rm_{W_{ijk}}|^2-4|\Ric_{W_{ijk}}|^2+\Sc_{W_{ijk}}^2}{24}\ge0. \nonumber
\end{align}
Stokes' theorem and Lemma~\ref{pole:lem:transgression}, \eqref{pole:eq:dbeta}, give, for $0<a<b$,
\begin{align}
&B(b)-B(a)=\int_{\{a\le\rho\le b\}}N^*(d\beta)\le0,\qquad B'=0\quad(n=3), \nonumber
\end{align}
and differentiation gives \eqref{pole:eq:B-derivative}.
For some $r_0>0$ and $C<\infty$ depending only on the metric near $p$,
\begin{align}
&\sup_{\rho=r}|S-r^{-1}\Id|\le Cr,\qquad \sup_{\xi\in ST_pM}\left|r^{-d}\frac{dA}{d\omega}(r,\xi)-1\right|\le Cr^2\quad(0<r<r_0), \nonumber\\
&|B(r)-1|\le Cr^2\quad(0<r<r_0),\qquad \lim_{r\downarrow0}B(r)=1,\qquad 0\le B(r)\le1, \nonumber
\end{align}
where $d\omega$ is unit-sphere area measure; the last two bounds follow from \eqref{pole:eq:B} and $B'\le0$.
\end{proof}

%\subsection{Removing the principal-curvature weights}
Put
\begin{align}
&C=\Id-rS,\qquad A(r)=\Area(\Sigma_r),\qquad E=\frac{\Sc_{\Sigma_r}}2h_r-\Ric_{\Sigma_r}, \nonumber\\
&Z(r)=\frac{A(r)}{n\omega_nr^d} +\frac{r^{2-d}}{(d-1)n\omega_n}\int_{\Sigma_r}(\Sc/2-\Ric(N,N))\,dA. \label{pole:eq:Z}
\end{align}
By the Gauss equation,
$\sec_{h_r}\ge0$, and hence
$E\ge0$ and $\tr E=(d-2)\Sc_{\Sigma_r}/2$.

\begin{lemma}
\label{pole:lem:weights}
On $(M^n, g)$, let $B(r)$ be defined in
Lemma~\ref{pole:lem:monotone}, equation~\eqref{pole:eq:B}, let $Z(r)$
be defined by \eqref{pole:eq:Z}. Then, for every $r>0$,
\begin{align}
&0\le Z(r)-B(r)\le\frac{r^{-d}}{n\omega_n}\int_{\Sigma_r}\tr C\,dA+\frac{r^{2-d}}{(d-1)n\omega_n}\int_{\Sigma_r}\langle E,C\rangle_{h_r}\,dA. \label{pole:eq:two-variations}
\end{align}
Moreover, for every fixed $a>0$,
\begin{align}
&\int_a^\infty|Z(r)-B(r)|\,\frac{dr}{r}<\infty, \qquad \int_a^\infty r^{2-d} \int_{\Sigma_r}\Ric(N,N)\,dA\,\frac{dr}{r}<\infty. \label{pole:eq:integrable-errors}
\end{align}
\end{lemma}
\begin{proof}
Area variation and Lemma~\ref{var:lem:metric-variation}, \eqref{eq:metric-variation}, give
\begin{align}
&\partial_rh_r=2S,\qquad \partial_rdA=(\tr S)dA,\qquad \int_{\Sigma_r}\operatorname{div}_{h_r}X\,dA=0\quad(X\in C^1(T\Sigma_r)), \nonumber\\
&\frac{d}{dr}\bigl(r^{-d}A(r)\bigr)=-r^{-d-1}\int_{\Sigma_r}\tr C\,dA\le0, \nonumber\\
&\frac{d}{dr}\left(r^{2-d}\int_{\Sigma_r}\frac{\Sc_{\Sigma_r}}2\,dA\right)=-r^{1-d}\int_{\Sigma_r}\langle E,C\rangle_{h_r}\,dA\le0. \nonumber
\end{align}
Both differentiated quantities are nonnegative. Integration on $[a,T]$ and $T\to\infty$ yields
\begin{align}
&\int_a^\infty\left(r^{-d}\int_{\Sigma_r}\tr C\,dA+r^{2-d}\int_{\Sigma_r}\langle E,C\rangle_{h_r}\,dA\right)\frac{dr}{r}<\infty. \label{pole:eq:two-defects-integrable}
\end{align}
For $0\le\lambda_i\le1$ and $I\subset\{1,\ldots,d\}$,
\begin{align}
&0\le1-\prod_{i\in I}\lambda_i=\sum_{i\in I}(1-\lambda_i)\prod_{\substack{j\in I\\j<i}}\lambda_j\le\sum_{i\in I}(1-\lambda_i). \nonumber
\end{align}
With $\lambda_i=r\kappa_i$, the Gauss equation gives
\begin{align}
&E(e_\ell,e_\ell)=\sum_{i<j,\ i,j\ne\ell}(K_{ij}+\kappa_i\kappa_j)\ge\sum_{i<j,\ i,j\ne\ell}K_{ij}, \nonumber\\
&0\le Z(r)-B(r)\le\frac{r^{-d}}{n\omega_n}\int_{\Sigma_r}\tr C\,dA+\frac{r^{2-d}}{(d-1)n\omega_n}\int_{\Sigma_r}\langle E,C\rangle_{h_r}\,dA, \nonumber\\
&\int_a^\infty|Z(r)-B(r)|\,\frac{dr}{r}<\infty. \nonumber
\end{align}
For $n=3$, the curvature-weight product is empty and $E=0$.
Put $D(r)=\int_{\Sigma_r}\tr C\,dA$ and $Q=rS=\Id-C$. Lemma~\ref{pole:lem:monotone}, \eqref{cv:eq:pole-riccati}, gives
\begin{align}
&0\le Q,C\le\Id,\qquad r^2\mathcal R_N=QC+rC',\qquad \partial_rdA=\frac{\tr Q}{r}dA, \nonumber\\
&r^2\int_{\Sigma_r}\Ric(N,N)\,dA=rD'(r)+\int_{\Sigma_r}\bigl(\tr(QC)-(\tr Q)(\tr C)\bigr)dA\le rD'(r)+D(r), \nonumber\\
&0\le\int_a^T r^{2-d}\int_{\Sigma_r}\Ric(N,N)\,dA\,\frac{dr}{r}\le\frac{D(T)}{T^d}-\frac{D(a)}{a^d}+(d+1)\int_a^T\frac{D(r)}{r^{d+1}}\,dr, \nonumber\\
&0\le\frac{D(r)}{r^d}\le\frac{dA(r)}{r^d}\le dn\omega_n,\qquad \sup_{T>a}\int_a^T\frac{D(r)}{r^{d+1}}\,dr<\infty\quad\text{by \eqref{pole:eq:two-defects-integrable}}. \nonumber
\end{align}
The second assertion of \eqref{pole:eq:integrable-errors} follows by monotone convergence.
\end{proof}

%\subsection{The sharp bound and the exact remainder}
\begin{theorem}
\label{cv:prop:pole}
Let $(M^n,g)$ be a complete noncompact Riemannian manifold with $n\ge3$ and a pole, $\sec_g\ge0$. 
Then 
\begin{align}
&L(M,g):=\lim_{R\to\infty}R^{2-n} \int_{B_p(R)}\Sc_g\,dV_g \le4\pi\omega_{n-2}(1-\AVR(M^n,g)). \nonumber
\end{align}
\end{theorem}

\begin{proof}
Retain $d=n-1$, $A(r)=\Area(\Sigma_r)$, $C=\Id-rS$ and $Q=rS$. Define
\begin{align}
&D(r)=\int_{\Sigma_r}\tr C\,dA,\qquad e(r)=r^{-d}D(r),\qquad J(r)=r^{2-d}\int_{\Sigma_r}\frac{\Sc_{\Sigma_r}}2\,dA. \nonumber
\end{align}
Lemma~\ref{pole:lem:weights}, \eqref{pole:eq:two-variations}, gives
\begin{align}
&\frac{A(r)}{r^d}\ge0,\qquad J(r)\ge0,\qquad \left(\frac{A(r)}{r^d}\right)'\le0,\qquad J'(r)\le0, \nonumber\\
&\int_a^\infty e(r)\,\frac{dr}{r}<\infty\qquad(a>0). \label{pole:eq:area-defect-integrable}
\end{align}
Lemma~\ref{pole:lem:monotone}, \eqref{cv:eq:pole-riccati}, yields
\begin{align}
&rC'=r^2\mathcal R_N-QC,\qquad \mathcal R_N\ge0,\qquad 0\le Q,C\le\Id,\qquad \partial_rdA=(\tr S)dA, \nonumber\\
&D'(r)=\int_{\Sigma_r}\bigl(\tr C'+(\tr S)(\tr C)\bigr)dA\ge-r^{-1}D(r),\qquad e'(r)\ge-(d+1)r^{-1}e(r), \nonumber\\
&e(s)\ge e(r)(r/s)^{d+1}\quad(s\ge r>0), \nonumber\\
&0\le e(r)\le\frac{d+1}{1-2^{-(d+1)}}\int_r^{2r}e(s)\,\frac{ds}{s},\qquad \lim_{r\to\infty}e(r)=0. \label{pole:eq:area-defect-limit}
\end{align}
The limit in \eqref{pole:eq:area-defect-limit} follows from \eqref{pole:eq:area-defect-integrable}. Put $a_\infty=\lim_{r\to\infty}A(r)/r^d$. Substitution and dominated convergence give
\begin{align}
&\frac{\Vol B_p(R)}{R^{d+1}}=\int_0^1\tau^d\frac{A(R\tau)}{(R\tau)^d}\,d\tau,\qquad 0\le\frac{A(r)}{r^d}\le n\omega_n, \nonumber\\
&\omega_n\AVR(M^n,g)=\frac{a_\infty}{d+1},\qquad a_\infty=n\omega_n\AVR(M^n,g), \nonumber\\
&\lim_{r\to\infty}\frac{A(r)}{r^d}=n\omega_n\AVR(M^n,g). \label{cv:eq:area-asymptotic}
\end{align}
For $\sigma_2(S)=\sum_{i<j}\kappa_i\kappa_j$ and $\lambda_i=r\kappa_i$,
\begin{align}
&0\le\lambda_i\le1,\qquad 0\le1-\lambda_i\lambda_j\le(1-\lambda_i)+(1-\lambda_j), \nonumber\\
&0\le\binom d2\frac{A(r)}{r^d}-r^{2-d}\int_{\Sigma_r}\sigma_2(S)\,dA\le(d-1)e(r). \label{pole:eq:quadratic-shape-limit}
\end{align}
Set $j_\infty=\lim_{r\to\infty}J(r)\in[0,\infty)$. The scalar Gauss equation and \eqref{pole:eq:area-defect-limit}--\eqref{pole:eq:quadratic-shape-limit} give
\begin{align}
&T(r):=r^{2-d}\int_{\Sigma_r}\bigl(\Sc/2-\Ric(N,N)\bigr)dA=J(r)-r^{2-d}\int_{\Sigma_r}\sigma_2(S)\,dA\ge0, \nonumber\\
&\lim_{r\to\infty}T(r)=j_\infty-\binom d2n\omega_n\AVR(M^n,g)=:T_\infty\ge0. \label{pole:eq:tangential-scalar-limit}
\end{align}
For $a>0$, Lemma~\ref{pole:lem:weights}, \eqref{pole:eq:integrable-errors}, gives an integrable majorant in
\begin{align}
&0\le\mathbf1_{\{r\le R\}}(r/R)^{d-1}\le1,\qquad \lim_{R\to\infty}\mathbf1_{\{r\le R\}}(r/R)^{d-1}=0\quad(r\ge a), \nonumber\\
&R^{1-d}\int_a^R\int_{\Sigma_r}\Ric(N,N)\,dA\,dr \nonumber\\
&{}=\int_a^\infty\mathbf1_{\{r\le R\}}(r/R)^{d-1}r^{2-d}\int_{\Sigma_r}\Ric(N,N)\,dA\,\frac{dr}{r}, \nonumber\\
&\lim_{R\to\infty}R^{1-d}\int_a^R\int_{\Sigma_r}\Ric(N,N)\,dA\,dr=0. \label{pole:eq:radial-ball-vanishing}
\end{align}
The last limit is dominated convergence. For every $\varepsilon>0$, choose $b\ge a$ such that $\sup_{r\ge b}|T(r)-T_\infty|\le\varepsilon$. Then, for $R\ge b$,
\begin{align}
&\left|R^{1-d}\int_a^R r^{d-2}\bigl(T(r)-T_\infty\bigr)\,dr\right|\le R^{1-d}\int_a^b r^{d-2}|T(r)-T_\infty|\,dr+\frac{\varepsilon}{d-1}, \nonumber\\
&\lim_{R\to\infty}R^{1-d}\int_a^R r^{d-2}T(r)\,dr=\frac{T_\infty}{d-1},\qquad \lim_{R\to\infty}R^{1-d}\int_{B_p(a)}\Sc\,dV=0. \nonumber
\end{align}
Coarea and \eqref{pole:eq:radial-ball-vanishing} imply
\begin{align}
&L(M,g):=\lim_{R\to\infty}R^{2-n}\int_{B_p(R)}\Sc\,dV=\frac{2T_\infty}{d-1}<\infty. \label{pole:eq:independent-scalar-limit}
\end{align}
By \eqref{pole:eq:Z}, \eqref{cv:eq:area-asymptotic}, and \eqref{pole:eq:tangential-scalar-limit},
\begin{align}
&\lim_{r\to\infty}Z(r)=\AVR(M^n,g)+\frac{T_\infty}{(d-1)n\omega_n}=\AVR(M^n,g)+\frac{L(M,g)}{2n\omega_n}. \nonumber
\end{align}
Lemma~\ref{pole:lem:monotone}, \eqref{pole:eq:B-derivative}, gives existence of $\lim_{r\to\infty}B(r)\in[0,1]$. By Lemma~\ref{pole:lem:weights}, $Z(r)\ge B(r)$ for every $r>0$, and $Z-B$ is integrable with respect to $dr/r$ on $[a,\infty)$ for every $a>0$. Hence
\begin{align}
&Z-B\ge0,\qquad \int_a^\infty(Z-B)\,\frac{dr}{r}<\infty\ \Longrightarrow\ \lim_{r\to\infty}(Z(r)-B(r))=0, \nonumber\\
&\lim_{r\to\infty}B(r)=\AVR(M^n,g)+\frac{L(M,g)}{2n\omega_n}\le1. \label{pole:eq:B-limit}
\end{align}
Indeed, a positive limit would imply $\int_a^\infty(Z-B)\,dr/r=\infty$.
Now $2n\omega_n=4\pi\omega_{n-2}$ and \eqref{pole:eq:B-limit} give the conclusion. 
\end{proof}

%\subsection{Smooth models attaining every pole coefficient}

\section{Distance comparison and smooth approximation}\label{sec:common-geometry}

\begin{notation}\label{notation:b-func}
Fix $p\in M$, set
$ST_pM=\{\xi\in T_pM:|\xi|=1\}$,
$\gamma_\xi(t)=\exp_p(t\xi)$, and
\begin{align}
&E=\{\zeta\in ST_pM:\gamma_\zeta|_{[0,T]}\text{ is minimizing for every }T>0\}. \nonumber
\end{align}
For $\zeta\in E$, define $b_\zeta$, and then $\mathfrak{b}$, by
\begin{align}
&b_\zeta(x)=\lim_{T\to\infty}\bigl(T-d_g(x,\gamma_\zeta(T))\bigr), \mathfrak{b}(x)\vcentcolon=\sup_{\zeta\in E}b_\zeta(x). \label{eq:busemann-definition}
\end{align}
\end{notation}

\begin{lemma}\label{lem:geometry}
Let $n\ge3$ and let $(M^n,g)$ be a smooth, complete, connected,
noncompact Riemannian manifold without boundary, with $\sec_g\ge0$.

Then $\mathfrak{b}$ is a convex, proper,
$1$-Lipschitz function satisfying $\mathfrak{b}(p)=0$, $\mathfrak{b}\le\rho$, and
\begin{align}
&\rho(x)(1-d_0(\rho(x)))\le \mathfrak{b}(x),\qquad \lim_{r\to\infty}d_0(r)=0, \label{eq:convex-approximation}
\end{align}
where $d_0:[0,\infty)\to[0,2]$ is nonincreasing. Moreover,
\begin{align}
&|d\mathfrak{b}_x|=1\quad\text{at every differentiability point }x\text{ of }\mathfrak{b}. \label{six:eq:unit-gradient}
\end{align}
\end{lemma}

\begin{proof}
For $r>0$, set
\begin{align}
&E_r=\{\xi\in ST_pM:d_g(p,\gamma_\xi(r))=r\},\qquad d_0(r)=\sup_{\xi\in E_r}\inf_{\zeta\in E}|\xi-\zeta|,\qquad d_0(0)=2, \nonumber\\
&R\ge r>0\ \Longrightarrow\ \varnothing\ne E_R\subset E_r\subset ST_pM,\qquad E_r\text{ compact},\qquad E=\bigcap_{r>0}E_r\ne\varnothing. \nonumber
\end{align}
Compactness gives
\begin{align}
&0\le d_0(R)\le d_0(r)\le2, \nonumber\\
&r_j\to\infty,\quad \xi_j\in E_{r_j},\quad \xi_j\to\xi\ \Longrightarrow\ \xi\in E_r\ (r>0)\ \Longrightarrow\ \xi\in E, \nonumber\\
&\lim_{r\to\infty}d_0(r)=0. \nonumber
\end{align}
For $\zeta\in E$ and $T_2\ge T_1\ge0$, the triangle inequality yields
\begin{align}
&{}-\rho(x)\le T_1-d_g(x,\gamma_\zeta(T_1))\le T_2-d_g(x,\gamma_\zeta(T_2))\le\rho(x), \nonumber\\
&|b_\zeta(x)-b_\zeta(y)|\le d_g(x,y),\qquad b_\zeta(p)=0,\qquad b_\zeta\le\rho. \nonumber
\end{align}
Thus the limit in \eqref{eq:busemann-definition} exists. Theorem~\ref{ext:thm:comparison}\textup{(ii)} gives, for every affinely parametrized geodesic segment $\gamma:[0,1]\to M$,
\begin{align}
&b_\zeta(\gamma(s))\le(1-s)b_\zeta(\gamma(0))+sb_\zeta(\gamma(1))\qquad(0\le s\le1), \nonumber\\
&\mathfrak{b}(\gamma(s))\le(1-s)\mathfrak{b}(\gamma(0))+s\mathfrak{b}(\gamma(1)),\qquad |\mathfrak{b}(x)-\mathfrak{b}(y)|\le d_g(x,y),\qquad \mathfrak{b}(p)=0. \nonumber
\end{align}
For $x=\exp_p(r\xi)$, $\xi\in E_r$, choose $\zeta\in E$ with $|\xi-\zeta|\le d_0(r)$. Hinge comparison gives
\begin{align}
&r\ge \mathfrak{b}(x)\ge b_\zeta(x)\ge b_\zeta(\gamma_\zeta(r))-d_g(\exp_p(r\xi),\exp_p(r\zeta))\ge r(1-d_0(r)), \nonumber\\
&\forall a\in\mathbb R\quad\exists R<\infty:\quad \{\mathfrak{b}\le a\}\subset\overline B_p(R),\qquad \{\mathfrak{b}\le a\}\text{ compact}. \nonumber
\end{align}
For each $\zeta\in E$, take unit-speed minimizing segments $\sigma_{\zeta,T}$ from $x$ to $\gamma_\zeta(T)$. A subsequence converges on compact intervals to a unit-speed ray $\sigma_\zeta$. For fixed $s\ge0$ and sufficiently large $T$,
\begin{align}
&b_\zeta(\sigma_{\zeta,T}(s))\ge T-d_g(\sigma_{\zeta,T}(s),\gamma_\zeta(T))=T-d_g(x,\gamma_\zeta(T))+s, \nonumber\\
&b_\zeta(x)+s\ge b_\zeta(\sigma_\zeta(s))\ge b_\zeta(x)+s. \nonumber
\end{align}
Choose $\zeta_j\in E$ with $b_{\zeta_j}(x)\ge \mathfrak{b}(x)-j^{-1}$. A subsequence of $\dot\sigma_{\zeta_j}(0)$ converges; its geodesic limit $\sigma$ is a unit-speed ray, and
\begin{align}
&\mathfrak{b}(x)+s\ge \mathfrak{b}(\sigma_{\zeta_j}(s))\ge b_{\zeta_j}(x)+s\ge \mathfrak{b}(x)+s-j^{-1}, \nonumber\\
&\mathfrak{b}(\sigma(s))=\mathfrak{b}(x)+s\qquad(s\ge0), \nonumber\\
&\mathfrak{b}\text{ differentiable at }x\ \Longrightarrow\ 1=d\mathfrak{b}_x(\dot\sigma(0))\le|d\mathfrak{b}_x|\le1. \nonumber
\end{align}
\end{proof}

\begin{lemma}\label{one:lem:compact-Hessians}
Let $K$ be a compact subset of a smooth Riemannian $n$-manifold,
$n\ge3$, and let $U$ be a fixed relatively compact open neighborhood
of $K$. There is a finite constant $C_{K,U}$, depending only on
$U$ and $g|_U$, such that
\begin{align}
&\int_K |\Hess f|+ \sum_{k=0}^{n-2}\sigma_k(\Hess f+g)\,dV\le C_{K,U} \label{one:eq:compact-Hessians}
\end{align}
for every $f\in C^\infty(U)$ satisfying
$|\nabla f|\le2$ and $\Hess f\ge-g$ on $U$. 
\end{lemma}

\begin{proof}
Choose a finite coordinate cover of $K$ by balls compactly contained in larger coordinate balls. On each such pair,
\begin{align}
&\nabla_{ij}f=\partial_{ij}f-\Gamma^k_{ij}\partial_kf,\quad |\nabla f|\le2,\quad \Hess f\ge-g\ \Longrightarrow\ D^2f\ge-C_0I, \nonumber\\
&z=f+C_1|x|^2/2,\qquad D^2z\ge I,\qquad |Dz|\le C_2,\qquad 0\le\Hess f+g\le C_2D^2z. \nonumber
\end{align}
Here $C_0,C_1,C_2$ depend only on the fixed coordinate balls and $g$. Define
\begin{align}
&N_{k-1}(B)^i{}_j=\frac{\partial\sigma_k(B)}{\partial B^j{}_i}. \nonumber
\end{align}
In the derivative of the alternating formula, terms containing $\partial_{iab}z$ cancel in pairs by $\partial_{iab}z=\partial_{aib}z$. In a diagonal frame, $N_{k-1}$ has eigenvalues $\sigma_{k-1}$ of the complementary eigenvalues; homogeneity gives the last identity below. Thus
\begin{align}
&\partial_iN_{k-1}(D^2z)^i{}_j=0,\qquad N_{k-1}(D^2z)\ge0, \nonumber\\
&\tr N_{k-1}(D^2z)=(n-k+1)\sigma_{k-1}(D^2z),\qquad N_{k-1}(D^2z):D^2z=k\sigma_k(D^2z). \nonumber
\end{align}
For $0\le\zeta\in C_c^\infty$,
\begin{align}
&k\int\zeta\sigma_k(D^2z)\,dx=-\int N_{k-1}(D^2z)(D\zeta,Dz)\,dx\le C_\zeta\int_{\supp D\zeta}\sigma_{k-1}(D^2z)\,dx. \nonumber
\end{align}
Choose $n-2$ nested cutoffs between the two coordinate balls. Induction from $\sigma_0=1$, eigenvalue monotonicity, and equivalence of the fixed coordinate and Riemannian metrics give
\begin{align}
&\int_{\text{inner ball}}\sigma_k(D^2z)\,dx\le C\qquad(0\le k\le n-2), \nonumber\\
&\int_K\sum_{k=0}^{n-2}\sigma_k(\Hess f+g)\,dV\le C_{K,U}, \nonumber\\
&|\Hess f|\le\tr(\Hess f+g)+\sqrt n\ \Longrightarrow\ \int_K|\Hess f|\,dV\le C_{K,U}. \nonumber
\end{align}
\end{proof}

\begin{lemma}\label{one:lem:coefficient}
Put
$c_{\mathfrak{b}}=2-2\min_M\mathfrak{b}$ and
\begin{align}
&K(t)=1+\sup_{\overline B_p(4t+4)}\Sc_g,\qquad t\ge0. \label{eq:level-curvature-majorant}
\end{align}
Then there is a smooth positive function $h:M\to(0,\infty)$ such that
\begin{align}
&0<h(x)\le\min\left\{(1+\rho(x))^{-3},\frac{e^{-(2\mathfrak{b}(x)+c_{\mathfrak{b}})}}{4K(2\mathfrak{b}(x)+c_{\mathfrak{b}})^2}\right\}, \label{one:eq:h-choice}\\
&\int_M h(1+\Sc_g)\sum_{k=0}^{n-2}\sigma_k(\Hess f+g)\,dV\le2 \label{one:eq:weighted-Hessians}
\end{align}
for every $f\in C^\infty(M)$ with $|\nabla f|\le2$ and
$\Hess f\ge-g$. 
\end{lemma}

\begin{proof}
Closed metric balls are compact. Comparison along a minimizing segment
from $p$ to a point of $\overline B_p(4t+4)$ gives
\begin{align}
&0\le K(t)-K(s)\le4(t-s)\sup_{\overline B_p(4t+4)}|\nabla\Sc_g|\qquad(0\le s\le t). \nonumber
\end{align}
Thus $K\ge1$ is finite, nondecreasing and locally Lipschitz. Moreover,
\begin{align}
&\min_M\mathfrak{b}> -\infty,\qquad 2\mathfrak{b}+c_{\mathfrak{b}}=2+2(\mathfrak{b}-\min_M\mathfrak{b})\ge2. \nonumber
\end{align}
Take a locally finite relatively compact coordinate cover $(U_j)_{j\ge1}$
and a subordinate nonnegative smooth partition $(\phi_j)_{j\ge1}$.
Lemma~\ref{one:lem:compact-Hessians},
\eqref{one:eq:compact-Hessians}, applied on fixed relatively compact
neighborhoods of $\overline U_j$, gives constants $C_j\ge1$, independent
of $f$, such that
\begin{align}
&\int_{\overline U_j}(1+\Sc_g)\sum_{k=0}^{n-2}\sigma_k(\Hess f+g)\,dV\le C_j. \nonumber
\end{align}
Choose constants
\begin{align}
&0<c_j\le\min\left\{\frac{2^{-j}}{1+C_j},\frac12\inf_{\overline U_j}\min\left\{(1+\rho)^{-3},\frac{e^{-(2\mathfrak{b}+c_{\mathfrak{b}})}}{4K(2\mathfrak{b}+c_{\mathfrak{b}})^2}\right\}\right\},\qquad h=\sum_{j=1}^{\infty}c_j\phi_j. \nonumber
\end{align}
The infima are positive by continuity and compactness. Local finiteness
and $\sum_j\phi_j=1$ give $h\in C^\infty(M)$, $h>0$, and
\eqref{one:eq:h-choice}. Nonnegativity and Tonelli's theorem give
\begin{align}
&\int_Mh(1+\Sc_g)\sum_{k=0}^{n-2}\sigma_k(\Hess f+g)\,dV\le\sum_{j=1}^{\infty}c_jC_j\le\sum_{j=1}^{\infty}\frac{2^{-j}C_j}{1+C_j}\le1\le2. \nonumber
\end{align}
All choices precede and are independent of $f$.
\end{proof}

\begin{lemma}\label{one:lem:fixed-F}
Assume $(M^n, g)$ is a complete non-compact Riemannian manifold with $K_g\geq 0$. Let $h$ be a
coefficient supplied by Lemma~\ref{one:lem:coefficient}. Then there
exists an exhaustion $F\in C^\infty(M)$
with
\begin{align}
&0\le F-\mathfrak{b}\le\frac14,\qquad |\nabla F|\le\frac98,\qquad \Hess F>-\frac h{32}g.\qquad \lim_{R\to\infty}\sup_{\rho\ge R}\bigl||\nabla F|-1\bigr|=0, \label{one:eq:fixed-F}\\
&\lim_{R\to\infty}\operatorname*{ess\,sup}_{\rho\ge R} |\nabla F-\nabla\rho|=0. \label{eq:gradient-limits}
\end{align}
With $c=F(p)+1/2$, one also has
\begin{align}
&\rho\langle\nabla F,\nabla\rho\rangle\ge F-c \quad\hbox{almost everywhere}. \label{eq:level-radial-support}
\end{align}
\end{lemma}

\begin{proof}
\textup{(1)} Theorem~\ref{ext:thm:Greene-Wu}, equation~\eqref{eq:smoothing-import}, with error $1/8$, Lipschitz margin $1/8$, and Hessian tolerance $h/32$, gives
\begin{align}
&f\in C^\infty(M),\qquad |f-\mathfrak{b}|<\tfrac18,\qquad |\nabla f|<\tfrac98,\qquad \Hess f>-\tfrac h{32}g, \nonumber\\
&F=f+\tfrac18,\qquad 0<F-\mathfrak{b}<\tfrac14. \nonumber
\end{align}
This is \cite[Proposition~2.3, p.~62]{GW}; no positive lower bound for $h$ at infinity is required.

\textup{(2)} For large $r$, define
\begin{align}
&\delta(r)=\sup_{\rho\ge r}|F/\rho-1|,\qquad \lim_{r\to\infty}\delta(r)=0. \nonumber
\end{align}
If $\rho(x)=r\ge1$, $|\dot\gamma|=1$, and $\gamma(0)=x$, then
\begin{align}
&\rho(\gamma(s))\ge r/2,\quad (F\circ\gamma)''(s)\ge-8r^{-3}\qquad(0\le s\le r/2), \nonumber\\
&F(\gamma(r/2))\ge F(x)+\tfrac r2\langle\nabla F(x),\dot\gamma(0)\rangle-r^{-1},\qquad F(\gamma(r/2))\le\tfrac32r+C_0, \nonumber\\
&|\nabla F(x)|\le1+2\delta(r)+2|C_0|/r+2/r^2. \nonumber
\end{align}
For every unit-speed minimizing segment $\sigma:[0,r]\to M$ from $p$ to $x$,
\begin{align}
&r(F\circ\sigma)'(r)-F(x)+F(p)=\int_0^rs(F\circ\sigma)''(s)\,ds\ge-\int_0^\infty s(1+s)^{-3}\,ds=-\tfrac12, \nonumber\\
&|\nabla F(x)|\ge\langle\nabla F(x),\dot\sigma(r)\rangle\ge\frac{F(x)-c}{r}\ge1-\delta(r)-c/r. \nonumber
\end{align}
Thus \eqref{eq:level-radial-support} holds, including the lower bound for every terminal minimizing direction at cut points. At differentiability points of $\rho$,
\begin{align}
&|\nabla F-\nabla\rho|^2=|\nabla F|^2+1-2\langle\nabla F,\nabla\rho\rangle\le C\delta(r)+C/r, \nonumber\\
&\lim_{R\to\infty}\sup_{\rho\ge R}\bigl||\nabla F|-1\bigr|=0,\qquad \lim_{R\to\infty}\operatorname*{ess\,sup}_{\rho\ge R}|\nabla F-\nabla\rho|=0. \nonumber
\end{align}
Moreover,
\postdisplaypenalty=10000
\begin{align}
&\{F\le t\}\subset\{\mathfrak{b}\le t\}\text{ compact},\qquad \exists R:\ |\nabla F|\ge\tfrac12\text{ on }\{\rho\ge R\}, \nonumber\\
&F(M)\text{ an interval},\quad \sup_MF=\infty\ \Longrightarrow\ \{F=t\}\ne\varnothing\text{ for all sufficiently large }t. \nonumber
\end{align}
\end{proof}

Fix one exhaustion $F$ furnished by Lemma~\ref{one:lem:fixed-F} for the remainder of the paper.

\begin{notation}\label{notation:level-objects}
For $t\in\mathbb R$, set
\begin{align}
&D_t=\{F\le t\},\qquad \Sigma_t=\{F=t\}. \label{tp:eq:sublevels}
\end{align}
At every regular value $t$, write $g_t=g|_{T\Sigma_t}$ and define
\begin{align}
&\nu=\nabla F/|\nabla F|,\qquad S(X,Y)=\langle\nabla_X\nu,Y\rangle=\Hess F(X,Y)/|\nabla F|, \nonumber\\
&H=\tr_{g_t}S,\qquad \sigma_2(S)=(H^2-|S|^2)/2, \qquad X,Y\in T\Sigma_t. \label{eq:level-objects}
\end{align}
Use $dA$ for the induced area and
$\Sc_{\Sigma_t}$ for the scalar curvature of $\Sigma_t$. In an orthonormal eigenframe
$Se_i=\kappa_i e_i$, write $K_{ij}=\sec_g(e_i\wedge e_j)$.
\end{notation} 

\begin{lemma}\label{lem:annular-estimates}
Assume $(M^n, g)$ is a complete non-compact Riemannian manifold with $K_g\geq 0$.
Fix $0<a<b<\infty$. Then there are $s_0,C<\infty$ such that, for every
$s\ge s_0$,
\begin{align}
&\{as\le F\le bs\}\subset\{as/2\le\rho\le2bs\},\qquad \Vol\{as\le F\le bs\}\le Cs^n, \nonumber\\
&\int_{\{as\le F\le bs\}} (|\Hess F|+|\Delta F|+\Delta F+ns^{-2})\,dV\le Cs^{n-1}. \label{eq:nt-trace-bound}
\end{align}
Moreover, for every $t>0$,
\begin{align}
&\lim_{s\to\infty}s^{-n}\Vol\{F<st\}=\omega_n\AVR(M^n,g)t^n. \label{eq:nt-volume}
\end{align}
For every $\psi\in C_c^\infty((0,\infty))$,
\begin{align}
&\lim_{s\to\infty}s^{-n}\int_M\psi(F/s)\,dV =\lim_{s\to\infty}s^{-n}\int_M\psi(F/s)|\nabla F|^2\,dV \nonumber\\
&{}=n\omega_n\AVR(M^n,g)\int_0^\infty\psi(t)t^{n-1}\,dt, \label{eq:nt-volume-tests}\\
&\lim_{s\to\infty}s^{1-n}\int_{\{as\le F\le bs\}} |\Hess F(\nu,\nu)|\,dV=0. \label{eq:nt-normal-small}
\end{align}
\end{lemma}

\begin{proof}
Lemma~\ref{lem:geometry}, equation~\eqref{eq:convex-approximation}, and Lemma~\ref{one:lem:fixed-F}, equation~\eqref{one:eq:fixed-F}, give
\begin{align}
&\lim_{R\to\infty}\sup_{\rho\ge R}|F/\rho-1|=0,\qquad \{as\le F\le bs\}\subset\{as/2\le\rho\le2bs\}\quad(s\gg1). \nonumber
\end{align}
For $t>0$ and $0<\varepsilon<1$, the corresponding ball inclusions and Bishop--Gromov volume comparison \cite{CE}, imply
\begin{align}
&\omega_n\AVR(M^n,g)\left(\frac{t}{1+\varepsilon}\right)^n\le\liminf_{s\to\infty}s^{-n}\Vol\{F<st\} \nonumber\\
&{}\le\limsup_{s\to\infty}s^{-n}\Vol\{F<st\}\le\omega_n\AVR(M^n,g)\left(\frac{t}{1-\varepsilon}\right)^n, \nonumber\\
&\lim_{s\to\infty}s^{-n}\Vol\{F<st\}=\omega_n\AVR(M^n,g)t^n,\qquad \Vol\{as\le F\le bs\}\le Cs^n. \nonumber
\end{align}
For $\psi\in C_c^\infty((0,\infty))$, integration by layers and dominated convergence on $\supp\psi'$ give
\begin{align}
&s^{-n}\int_M\psi(F/s)\,dV=-\int_0^\infty\psi'(t)s^{-n}\Vol\{F<st\}\,dt, \nonumber\\
&\lim_{s\to\infty}s^{-n}\int_M\psi(F/s)\,dV=-\omega_n\AVR(M^n,g)\int_0^\infty\psi'(t)t^n\,dt \nonumber\\
&{}=n\omega_n\AVR(M^n,g)\int_0^\infty\psi(t)t^{n-1}\,dt, \nonumber\\
&s^{-n}\left|\int_M\psi(F/s)(|\nabla F|^2-1)\,dV\right|\le C_\psi\sup_{\{F/s\in\supp\psi\}}||\nabla F|^2-1|\longrightarrow0\quad(s\to\infty). \nonumber
\end{align}
The last limit is Lemma~\ref{one:lem:fixed-F}, equation~\eqref{eq:gradient-limits}.

Choose $\zeta\in C_c^\infty((a/2,2b))$ with $0\le\zeta\le1$ and $\zeta=1$ near $[a,b]$. Lemmas~\ref{one:lem:fixed-F} and~\ref{one:lem:coefficient}, equations~\eqref{one:eq:fixed-F} and~\eqref{one:eq:h-choice}, give on $\supp\zeta(F/s)$, for $s\gg1$,
\begin{align}
&\Hess F+s^{-2}g\ge0,\qquad |\Hess F|\le\Delta F+(n+\sqrt n)s^{-2},\qquad |\Delta F|\le\Delta F+2ns^{-2}, \nonumber\\
&0\le\int_M\zeta(F/s)(\Delta F+ns^{-2})\,dV=-s^{-1}\int_M\zeta'(F/s)|\nabla F|^2\,dV+ns^{-2}\int_M\zeta(F/s)\,dV\le Cs^{n-1}. \nonumber
\end{align}
These imply \eqref{eq:nt-trace-bound}.
Finally, for $s\gg1$,
\begin{align}
&|\nabla F|\ge\tfrac12,\qquad \Hess F(\nu,\nu)=\langle\nabla F,\nabla\log |\nabla F|\rangle,\qquad |\Hess F(\nu,\nu)|\le\Hess F(\nu,\nu)+2s^{-2}, \nonumber\\
&\int_M\zeta(F/s)\Hess F(\nu,\nu)\,dV=-\int_M\zeta(F/s)\log |\nabla F|\,\Delta F\,dV-s^{-1}\int_M\zeta'(F/s)|\nabla F|^2\log |\nabla F|\,dV, \nonumber\\
&s^{1-n}\int_{\{as\le F\le bs\}}|\Hess F(\nu,\nu)|\,dV\le C\sup_{\supp\zeta(F/s)}|\log |\nabla F||+C/s\longrightarrow0\quad(s\to\infty). \nonumber
\end{align}
Here equations~\eqref{eq:nt-volume} and~\eqref{eq:nt-trace-bound} are applied on a fixed larger interval containing $\supp\zeta$.
\end{proof}

\begin{lemma}\label{tp:lem:level-Euler}
Let $M^n$ be a connected noncompact smooth manifold without boundary with $K\geq 0$. Suppose $t_0>F(p)$ and $\delta>0$
satisfy $|\nabla F|>0$ on $\{F>t_0-\delta\}$. Then, for every
$t\ge t_0$, 
\begin{align}
&\chi(D_t)=\chi(M),\qquad \chi(\Sigma_t)=2\chi(M)\quad\hbox{if $n$ is odd}. \label{tp:eq:level-topology}
\end{align}
If $n=3$, then
\begin{align}
&\chi(D_t)=\chi(M),\qquad \chi(\Sigma_t)=2\chi(M),\qquad \int_{\Sigma_t}\Sc_{\Sigma_t}\,dA=8\pi\chi(M). \nonumber
\end{align}
\end{lemma}

\begin{proof}
Compactness of sublevels and connectedness of $M$ imply
\begin{align}
&\sup_MF=\infty,\quad F(M)\text{ an interval},\quad F(p)<t_0\ \Longrightarrow\ \Sigma_t\ne\varnothing\qquad(t\ge t_0), \nonumber\\
&|\nabla F|>0\text{ on }\Sigma_t\ \Longrightarrow\ D_t\text{ compact and smooth},\qquad \partial D_t=\Sigma_t. \nonumber
\end{align}
Let $\Phi_\sigma$ be the flow of $\nabla F/|\nabla F|^2$ on $\{F>t_0-\delta\}$. Then
\begin{align}
&\frac{d}{d\sigma}F(\Phi_\sigma(x))=1,\qquad F(\Phi_\sigma(x))=F(x)+\sigma. \nonumber
\end{align}
Every finite interval of these values lies in a compact band where the vector field is smooth; hence the flow exists throughout that interval. Define
\begin{align}
&\mathcal H_\tau(x)=\begin{cases}x,&F(x)\le t,\\\Phi_{-\tau(F(x)-t)}(x),&F(x)>t,\end{cases}\qquad 0\le\tau\le1, \nonumber\\
&\mathcal H\in C^0([0,1]\times M,M),\qquad \mathcal H_0=\Id_M,\qquad \mathcal H_\tau|_{D_t}=\Id_{D_t},\qquad \mathcal H_1(M)=D_t, \nonumber\\
&\chi(D_t)=\chi(M). \nonumber
\end{align}
Continuity at $F=t$ follows from $\Phi_0(x)=x$ and smooth dependence of the flow. For odd $n$, the smooth double $\widehat D_t=D_t\cup_{\Sigma_t}D_t$, Poincar\'e duality, and additivity of Euler characteristic \cite{Hatcher} give
\begin{align}
&0=\chi(\widehat D_t)=2\chi(D_t)-\chi(\Sigma_t),\qquad \chi(\Sigma_t)=2\chi(M). \nonumber
\end{align}
For $n=3$, Gauss--Bonnet \cite[Section~4-5]{doCarmo} gives
\begin{align}
&\int_{\Sigma_t}\Sc_{\Sigma_t}\,dA=2\int_{\Sigma_t}K_{\Sigma_t}\,dA=4\pi\chi(\Sigma_t)=8\pi\chi(M). \nonumber
\end{align}
\end{proof}

\begin{corollary}\label{one:cor:gradient-u}
\begin{align}
&\lim_{R\to\infty}\operatorname*{ess\,sup}_{\rho\ge R} |\nabla F-\nabla \mathfrak{b}|=0. \label{one:eq:gradient-u}
\end{align}
\end{corollary}
\begin{proof}
At every common differentiability point $x$ of $\mathfrak{b}$ and $\rho$, convexity along a minimizing segment from $p$ yields
\begin{align}
&\langle\nabla \mathfrak{b},\nabla\rho\rangle(x)\ge\frac{\mathfrak{b}(x)-\mathfrak{b}(p)}{\rho(x)}\ge1-d_0(\rho(x)),\qquad |\nabla \mathfrak{b}|=|\nabla\rho|=1, \nonumber\\
&|\nabla \mathfrak{b}-\nabla\rho|^2\le2d_0(\rho),\qquad |\nabla F-\nabla \mathfrak{b}|\le|\nabla F-\nabla\rho|+\sqrt{2d_0(\rho)}, \nonumber\\
&\lim_{R\to\infty}\operatorname*{ess\,sup}_{\rho\ge R}|\nabla F-\nabla \mathfrak{b}|=0. \nonumber
\end{align}
\end{proof}

\section{The integral on level sets}\label{sec:intrinsic-limit}

\begin{notation}\label{notation:eta}
For $t>-1$, set
\begin{align}
&\eta(t)=\frac{16}{(1+t)^3}. \nonumber
\end{align}
\end{notation}

\begin{lemma}\label{lemma:divergence-nonnegative-combination}
Let $n\ge3$, let $(M^n,g)$ be a Riemannian manifold with $\sec_g\ge0$, let $t>-1$ be a regular value of $F$, and define
\begin{align}
&\mathcal E_t:=\Sc_{\Sigma_t}g_{\Sigma_t}-2\Ric_{\Sigma_t}+2(n-3)\eta\Bigl[\bigl(H+(n-2)\eta\bigr)g_{\Sigma_t}-S\Bigr]. \nonumber
\end{align}
Then
\begin{align}
&\diver_{\Sigma_t}\mathcal E_t=-2(n-3)\eta\,\Ric_g(\nu,\cdot)^T=-\frac{32(n-3)}{(1+t)^3}\Ric_g(\nu,\cdot)^T, \label{eq:level-divergence-identity}\\
&\left|\diver_{\Sigma_t}\mathcal E_t\right|\le \frac{32(n-3)}{(1+t)^3}\Sc_g. \label{eq:level-divergence-bound}
\end{align}
\end{lemma}

\begin{proof}
Since $\eta=\eta(t)$ is constant on $\Sigma_t$, the contracted Bianchi identity gives
\begin{align}
&\diver_{\Sigma_t}\Ric_{\Sigma_t}=\frac12d_{\Sigma_t}\Sc_{\Sigma_t},\qquad \diver_{\Sigma_t}\bigl(\Sc_{\Sigma_t}g_{\Sigma_t}-2\Ric_{\Sigma_t}\bigr)=0, \nonumber\\
&\diver_{\Sigma_t}\Bigl[\bigl(H+(n-2)\eta\bigr)g_{\Sigma_t}-S\Bigr]=d_{\Sigma_t}H-\diver_{\Sigma_t}S. \nonumber
\end{align}
Fix $x\in\Sigma_t$ and choose an orthonormal tangent frame $e_1,\ldots,e_{n-1}$ satisfying $\nabla^{\Sigma_t}e_i|_x=0$. With the convention $S(X,Y)=\langle\nabla_X\nu,Y\rangle$, the Codazzi equation gives
\begin{align}
&(\nabla^{\Sigma_t}_{e_i}S)(e_j,e_k)-(\nabla^{\Sigma_t}_{e_j}S)(e_i,e_k)=\langle\Rm_g(e_i,e_j)\nu,e_k\rangle. \nonumber
\end{align}
Contracting the indices $i$ and $k$ yields
\begin{align}
&(\diver_{\Sigma_t}S)(e_j)-d_{\Sigma_t}H(e_j)=\sum_{i=1}^{n-1}\langle\Rm_g(e_i,e_j)\nu,e_i\rangle. \nonumber
\end{align}
By the first Bianchi identity and the curvature symmetries,
\begin{align}
&\sum_{i=1}^{n-1}\langle\Rm_g(e_i,e_j)\nu,e_i\rangle=\Ric_g(\nu,e_j), \nonumber
\end{align}
and hence
\begin{align}
&\diver_{\Sigma_t}S-d_{\Sigma_t}H=\Ric_g(\nu,\cdot)^T. \nonumber
\end{align}
Substitution into the definition of $\mathcal E_t$ gives
\begin{align}
&\diver_{\Sigma_t}\mathcal E_t=-2(n-3)\eta\,\Ric_g(\nu,\cdot)^T,
\end{align}
which is \eqref{eq:level-divergence-identity} because $\eta=16(1+t)^{-3}$.

Finally, $\sec_g\ge0$ implies that $\Ric_g^\sharp$ is nonnegative definite. If $\lambda_1,\ldots,\lambda_n$ are its eigenvalues, then
\begin{align}
&0\le\lambda_i\le\sum_{j=1}^n\lambda_j=\Sc_g,\qquad \|\Ric_g^\sharp\|_{\mathrm{op}}\le\Sc_g. \nonumber
\end{align}
Therefore
\begin{align}
&|\Ric_g(\nu,\cdot)^T|\le|\Ric_g^\sharp\nu|\le\|\Ric_g^\sharp\|_{\mathrm{op}}\le\Sc_g. \nonumber
\end{align}
Combining this with \eqref{eq:level-divergence-identity} proves \eqref{eq:level-divergence-bound}.
\end{proof}

\begin{lemma}\label{lem-diff-ineq-of-J}
Assume $(M^n, g)$ is a complete non-compact Riemannian manifold with $K_g\geq 0$. There exist $T,C<\infty$, such that for every $t\ge T$, use $\Sigma_t$ from
Notation~\ref{notation:level-objects} and define
\begin{align}
&J(t)=\int_{\Sigma_t}\Sc_{\Sigma_t}\,dA. \nonumber
\end{align}
Then
\begin{align}
&\left(t-F(p)-\frac12\right)J'(t)\le (n-3) \left(1+\frac{32t}{(1+t)^3}\right)J(t) \nonumber\\
&{}+ C\frac{16}{(1+t)^3} \left(1+\frac{16t}{(1+t)^3}\right) \int_{\Sigma_t} \left( \frac{\tr_{T\Sigma_t}\nabla^2F}{|\nabla F|} +\frac{16(n-1)}{(1+t)^3} \right)dA \nonumber\\
&{}+ C\frac{16t}{(1+t)^3} \int_{\Sigma_t}\Sc_g\,dA. \label{eq:level-linear-differential-inequality}
\end{align}
\end{lemma}

\begin{proof}
Use $|\nabla F|,S,H$ from Notation~\ref{notation:level-objects}, $\eta(t)=16(1+t)^{-3}$, and $\mathcal E_t$ from Lemma~\ref{lemma:divergence-nonnegative-combination} with $f=F$. For sufficiently large $t$,
\begin{align}
&\tfrac12\le |\nabla F|\le2,\qquad \tfrac t2\le\rho\le2t,\qquad S\ge-\eta g_t,\qquad 0<\frac{t-F(p)-\tfrac12}{|\nabla F|}\le2t. \nonumber
\end{align}
On each compact regular band, the flow $\partial_tX=\nabla F/|\nabla F|^2$ identifies the levels. For tangent indices $i,j$,
\begin{align}
&(\partial_tg_t)_{ij}=\langle\nabla_{e_i}(|\nabla F|^{-2}\nabla F),e_j\rangle+\langle e_i,\nabla_{e_j}(|\nabla F|^{-2}\nabla F)\rangle=2|\nabla F|^{-2}(\nabla^2F)_{ij}, \nonumber\\
&\partial_t\Gamma^a_{ij}=\tfrac12g_t^{ab}\bigl[(\nabla^{\Sigma_t}_i\partial_tg_t)_{jb}+(\nabla^{\Sigma_t}_j\partial_tg_t)_{ib}-(\nabla^{\Sigma_t}_b\partial_tg_t)_{ij}\bigr], \nonumber\\
&\partial_t(\Ric_{\Sigma_t})_{ij}=\nabla^{\Sigma_t}_a(\partial_t\Gamma^a_{ij})-\nabla^{\Sigma_t}_j(\partial_t\Gamma^a_{ia}), \nonumber\\
&\partial_t\Sc_{\Sigma_t}=-2|\nabla F|^{-2}(\Ric_{\Sigma_t})^{ij}(\nabla^2F)_{ij}+2\nabla_{\Sigma_t}^i\nabla_{\Sigma_t}^j\bigl(|\nabla F|^{-2}(\nabla^2F)_{ij}\bigr)-2\Delta_{\Sigma_t}(H/|\nabla F|), \nonumber\\
&\partial_t dA=(H/|\nabla F|)\,dA,\qquad \int_{\Sigma_t}\nabla_{\Sigma_t}^i\nabla_{\Sigma_t}^j\bigl(|\nabla F|^{-2}(\nabla^2F)_{ij}\bigr)\,dA=\int_{\Sigma_t}\Delta_{\Sigma_t}(H/|\nabla F|)\,dA=0, \nonumber\\
&J'(t)=\int_{\Sigma_t}\frac{\Sc_{\Sigma_t}H-2\langle\Ric_{\Sigma_t},S\rangle}{|\nabla F|}\,dA. \nonumber
\end{align}
In an eigenframe $Se_i=\kappa_i e_i$, the Gauss equation yields
\begin{align}
&\mathcal E_t(e_i,e_i)=2\sum_{\substack{j<\ell\\j,\ell\ne i}}K_{j\ell}+2\sum_{\substack{j<\ell\\j,\ell\ne i}}(\kappa_j+\eta)(\kappa_\ell+\eta)+(n-2)(n-3)\eta^2. \nonumber
\end{align}
The extrinsic summand is diagonal and nonnegative. For every unit tangent vector, the ambient summand is twice the sum of sectional curvatures in its tangent orthogonal complement. Therefore
\begin{align}
&\mathcal E_t\ge0,\qquad \tr\mathcal E_t=(n-3)\Sc_{\Sigma_t}+2(n-3)(n-2)\eta(H+(n-1)\eta), \nonumber\\
&0\le2\sigma_2(S+\eta I)=2\sigma_2(S)+2(n-2)\eta H+(n-1)(n-2)\eta^2, \nonumber\\
&2\sigma_2(S)+(n-2)\eta H\ge-(n-2)\eta(H+(n-1)\eta), \nonumber\\
&\langle\mathcal E_t,S\rangle-\bigl(\Sc_{\Sigma_t}H-2\langle\Ric_{\Sigma_t},S\rangle\bigr)=2(n-3)\eta\bigl(2\sigma_2(S)+(n-2)\eta H\bigr), \nonumber\\
&\frac{t-F(p)-\tfrac12}{|\nabla F|}\bigl(\Sc_{\Sigma_t}H-2\langle\Ric_{\Sigma_t},S\rangle\bigr)\le\frac{t-F(p)-\tfrac12}{|\nabla F|}\langle\mathcal E_t,S\rangle \nonumber\\
&{}+4(n-3)(n-2)t\eta^2(H+(n-1)\eta). \nonumber
\end{align}
Lemma~\ref{one:lem:fixed-F}, equation~\eqref{eq:level-radial-support}, gives almost everywhere
\begin{align}
&0\le\frac{\rho\langle\nabla\rho,\nabla F\rangle-(t-F(p)-\tfrac12)}{|\nabla F|}\le2t,\qquad \langle\mathcal E_t,S\rangle\ge-\eta\tr\mathcal E_t, \nonumber\\
&\frac{t-F(p)-\tfrac12}{|\nabla F|}\langle\mathcal E_t,S\rangle\le\frac{\rho\langle\nabla\rho,\nabla F\rangle}{|\nabla F|}\langle\mathcal E_t,S\rangle+2t\eta\tr\mathcal E_t. \nonumber
\end{align}

For almost every $t$, coarea gives $\mathcal H^{n-1}(\Sigma_t\cap\operatorname{Cut}(p))=0$. The restriction of $\rho^2/2$ to $\Sigma_t$ is locally semiconcave. Theorem~\ref{ext:thm:comparison}, equation~\eqref{eq:distance-square-hessian-comparison}, the restriction formula off the cut locus, and we get the intrinsic measure inequality
\begin{align}
&\Hess_{\Sigma_t}(\rho^2/2)\le\left(g_t-\frac{\rho\langle\nabla\rho,\nabla F\rangle}{|\nabla F|}S\right)dA. \nonumber
\end{align}
Pairing with $\mathcal E_t\ge0$ and integrating by parts on the closed level yields
\begin{align}
&\int_{\Sigma_t}\frac{\rho\langle\nabla\rho,\nabla F\rangle}{|\nabla F|}\langle\mathcal E_t,S\rangle\,dA\le\int_{\Sigma_t}\tr\mathcal E_t\,dA-\int_{\Sigma_t}\langle\mathcal E_t,\Hess_{\Sigma_t}(\rho^2/2)\rangle, \nonumber\\
&{}=\int_{\Sigma_t}\tr\mathcal E_t\,dA+\int_{\Sigma_t}\langle\diver_{\Sigma_t}\mathcal E_t,\nabla^{\Sigma_t}(\rho^2/2)\rangle\,dA. \nonumber
\end{align}
Lemma~\ref{lemma:divergence-nonnegative-combination}, equation~\eqref{eq:level-divergence-bound}, gives
\begin{align}
&|\diver_{\Sigma_t}\mathcal E_t|\le2(n-3)\eta\Sc_g,\qquad |\nabla^{\Sigma_t}(\rho^2/2)|\le2t, \nonumber\\
&(t-F(p)-\tfrac12)J'(t)\le(1+2t\eta)\int_{\Sigma_t}\tr\mathcal E_t\,dA+4(n-3)t\eta\int_{\Sigma_t}\Sc_g\,dA \nonumber\\
&{}+4(n-3)(n-2)t\eta^2\int_{\Sigma_t}(H+(n-1)\eta)\,dA, \nonumber\\
&{}\le(n-3)(1+2t\eta)J(t)+C\eta(1+t\eta)\int_{\Sigma_t}(H+(n-1)\eta)\,dA+Ct\eta\int_{\Sigma_t}\Sc_g\,dA. \nonumber
\end{align}
Substitute $H=\tr_{T\Sigma_t}\nabla^2F/|\nabla F|$ and $\eta=16(1+t)^{-3}$. Continuity on regular bands extends the inequality from almost every $t$ to every $t\ge T$, proving \eqref{eq:level-linear-differential-inequality}.
\end{proof}

\begin{lemma}\label{lem:independent-level-limit}
Assume $(M^n, g)$ is a complete non-compact Riemannian manifold with $K_g\geq 0$.  Set $\Psi=\rho^2/2$, and
$w=\langle\nabla\Psi,\nu\rangle$.
Then there is $T<\infty$ such that, for almost every $t\ge T$,
the following intrinsic tensor-measure inequality holds:
\begin{align}
&\Hess_{\Sigma_t}\Psi\le(g_t-wS)\,dA, \label{eq:level-distance-hessian}\\
&\lim_{t\to\infty}t^{3-n} \int_{\Sigma_t}\Sc_{\Sigma_t}\,dA=L_\Sigma\in[0,\infty). \label{eq:level-curvature-limit}
\end{align}
\end{lemma}

\begin{proof}
\textup{(1)} Set $c=F(p)+1/2$. Lemma~\ref{one:lem:fixed-F}, equations~\eqref{one:eq:fixed-F}, \eqref{eq:gradient-limits}, and~\eqref{eq:level-radial-support}, permits $T\ge4$ such that, on $\Sigma_t$, $t\ge T$,
\begin{align}
&\frac12\le|\nabla F|\le2,\qquad \frac t2\le\rho\le2t,\qquad t-c\ge t/2, \nonumber\\
&S\ge-\eta(t)g_t,\qquad A(t)=\Area(\Sigma_t),\qquad N(t)=\int_{\Sigma_t}H\,dA,\qquad Y(t)=\int_{\Sigma_t}(H+m\eta(t))\,dA \nonumber\\
&\ge0. \nonumber
\end{align}
Retain $J$ from Lemma~\ref{lem-diff-ineq-of-J}. Almost everywhere on a large level,
\begin{align}
&\frac{t-c}{|\nabla F|}\le w=\langle\nabla(\rho^2/2),\nu\rangle\le2t,\qquad w\ge(t-c)/2. \label{eq:level-pairing-bounds}
\end{align}
Theorem~\ref{ext:thm:comparison}, equation~\eqref{eq:distance-square-hessian-comparison}, gives $\Delta\Psi\le n\,dV$. Choose $\phi_\varepsilon\in C^\infty(\mathbb R,[0,1])$ with $\phi_\varepsilon'=0$ outside $[t,t+\varepsilon]$, $-C/\varepsilon\le\phi_\varepsilon'\le0$, $\phi_\varepsilon=1$ on $(-\infty,t]$, and $\phi_\varepsilon=0$ on $[t+\varepsilon,\infty)$. Distributional integration by parts and coarea give
\begin{align}
&{}-\int_M\phi_\varepsilon'(F)\langle\nabla F,\nabla\Psi\rangle\,dV\le n\int_M\phi_\varepsilon(F)\,dV, \nonumber\\
&\int_{\Sigma_t}w\,dA\le n\Vol(D_t)\qquad\text{for a.e. }t\ge T. \nonumber
\end{align}
The one-dimensional differentiation theorem justifies the limit. Since $D_t\subset B_p(2t)$, Bishop--Gromov volume comparison \cite{CE} yields
\begin{align}
&\tfrac{t-c}{2}A(t)\le\int_{\Sigma_t}w\,dA\le n\omega_n(2t)^n,\qquad A(t)\le Ct^m. \nonumber
\end{align}
Continuity of regular-level area extends the last bound to every $t\ge T$.

Coarea and the semiconcavity of $\Psi|_{\Sigma_t}$ give, for almost every $t$,
\begin{align}
&\mathcal H^m(\operatorname{Cut}(p)\cap\Sigma_t)=0,\qquad (\Hess_{\Sigma_t}\Psi)^{\mathrm s}\le0, \nonumber\\
&\Hess_{\Sigma_t}\Psi=\Hess_M\Psi|_{T\Sigma_t}-wS\le g_t-wS\quad\text{off }\operatorname{Cut}(p), \nonumber\\
&\Hess_{\Sigma_t}\Psi\le(g_t-wS)\,dA,\qquad \Delta_{\Sigma_t}\Psi\le(m-wH)\,dA, \nonumber\\
&\int_{\Sigma_t}\Delta_{\Sigma_t}\Psi=\int_{\Sigma_t}\Psi\Delta_{\Sigma_t}1\,dA=0,\qquad \int_{\Sigma_t}wH\,dA\le mA(t). \nonumber
\end{align}
These are intrinsic measure inequalities. Using $H+m\eta\ge0$ before equation~\eqref{eq:level-pairing-bounds},
\begin{align}
&\frac{t-c}{2}Y(t)\le\int_{\Sigma_t}w(H+m\eta)\,dA=\int_{\Sigma_t}wH\,dA+m\eta\int_{\Sigma_t}w\,dA\le m(1+2t\eta)A(t), \nonumber\\
&Y(t)\le Ct^{m-1},\qquad |N(t)|\le Y(t)+m\eta A(t)\le Ct^{m-1}\qquad(t\ge T). \label{eq:level-mean-bound}
\end{align}
Again the final bounds extend by continuity.

On every compact regular band, the flow of $\nabla F/|\nabla F|^2$ identifies the levels and gives $\partial_tg_t=2S/|\nabla F|$. Lemma~\ref{var:lem:metric-variation}, equation~\eqref{eq:metric-variation}, applies.

\textup{(2)} Gauss and $S+\eta I\ge0$ imply
\begin{align}
&2\sigma_2(S)\ge-2(m-1)\eta(H+m\eta),\qquad \Sc_{\Sigma_t}\ge2\sigma_2(S), \nonumber\\
&I_-(t):=\int_{\Sigma_t}(\Sc_{\Sigma_t})_-\,dA\le2(m-1)\eta(t)Y(t). \label{eq:level-negative-mass}
\end{align}
For normal speed $|\nabla F|^{-1}$,
\begin{align}
&\partial_tH=-\Delta_{\Sigma_t}(|\nabla F|^{-1})-|\nabla F|^{-1}(|S|^2+\Ric_g(\nu,\nu)),\qquad \partial_tdA=(H/|\nabla F|)\,dA, \nonumber\\
&\int_{\Sigma_t}\Delta_{\Sigma_t}(|\nabla F|^{-1})\,dA=0,\qquad N'(t)=\int_{\Sigma_t}\frac{2\sigma_2(S)-\Ric_g(\nu,\nu)}{|\nabla F|}\,dA, \nonumber\\
&\Sc_{\Sigma_t}=\Sc_g-2\Ric_g(\nu,\nu)+2\sigma_2(S), \nonumber\\
&\int_{\Sigma_t}\frac{\Sc_g}{|\nabla F|}\,dA=\int_{\Sigma_t}\frac{\Sc_{\Sigma_t}+2\sigma_2(S)}{|\nabla F|}\,dA-2N'(t). \nonumber
\end{align}
Since $\sec_g\ge0$ and $1/2\le |\nabla F|\le2$,
\begin{align}
&2\sigma_2(S)\le\Sc_{\Sigma_t},\qquad \Sc_g\ge0,\qquad \int_{\Sigma_t}\frac{\Sc_{\Sigma_t}}{|\nabla F|}\,dA\le2(J+I_-)-\tfrac12I_-=2J+\tfrac32I_-, \nonumber\\
&0\le\int_{\Sigma_t}\Sc_g\,dA\le8J(t)+6I_-(t)-4N'(t). \label{eq:level-ambient-elimination}
\end{align}

\textup{(3)} For $n=3$, $\Ric_{\Sigma_t}=\tfrac12\Sc_{\Sigma_t}g_t$; scalar variation and equation~\eqref{eq:level-negative-mass} give
\begin{align}
&J'(t)=0,\qquad J(t)\ge-I_-(t)\ge-Ct^{-2},\qquad \lim_{t\to\infty}J(t)=J(T)\ge0. \nonumber
\end{align}
Assume henceforth $n\ge4$. Lemma~\ref{lem-diff-ineq-of-J}, equation~\eqref{eq:level-linear-differential-inequality}, and equation~\eqref{eq:level-ambient-elimination} give
\begin{align}
&(t-c)J'\le(n-3)(1+C\eta t)J+C\eta(1+\eta t)Y+C\eta t\int_{\Sigma_t}\Sc_g\,dA \nonumber\\
&{}\le(n-3)(1+C\eta t)J+C\eta(1+\eta t)Y+C\eta(t-c)\bigl(8J+6I_--4N'\bigr), \nonumber\\
&J_+\le J+I_-,\qquad I_-\le2(m-1)\eta Y,\qquad t/2\le t-c\le t, \nonumber\\
&(t-c)J'\le(n-3)(1+C\eta t)J+C\eta(1+\eta t)Y-C\eta(t-c)N'. \label{eq:level-final-J-differential}
\end{align}
The coefficient of $-\eta(t-c)N'$ in \eqref{eq:level-final-J-differential} is fixed and is the same coefficient in
\begin{align}
&Z(t)=(t-c)^{3-n}J(t)+C\eta(t)(t-c)^{3-n}N(t). \nonumber
\end{align}
Differentiation cancels the $N'$ terms exactly:
\begin{align}
&Z'\le C\eta t(t-c)^{2-n}J+C\eta(1+\eta t)(t-c)^{2-n}Y+CN[\eta(t-c)^{3-n}]', \nonumber\\
&{}\le\frac{C\eta t}{t-c}Z+\frac{C\eta(1+\eta t)}{(t-c)^{n-2}}Y+C|N|\left|[\eta(t-c)^{3-n}]'\right|+\frac{C\eta^2t}{(t-c)^{n-2}}|N|. \label{eq:level-Z-differential}
\end{align}
Equation~\eqref{eq:level-mean-bound}, $m=n-1$, and $\eta(t)=16(1+t)^{-3}$ give
\begin{align}
&Y(t)+|N(t)|\le Ct^{n-2},\qquad \left|[\eta(t)(t-c)^{3-n}]'\right|\le Ct^{-n-1}, \nonumber\\
&\frac{\eta(1+\eta t)}{(t-c)^{n-2}}Y\le Ct^{-3},\qquad |N|\left|[\eta(t-c)^{3-n}]'\right|\le Ct^{-3},\qquad \frac{\eta^2t}{(t-c)^{n-2}}|N|\le Ct^{-5}, \nonumber\\
&\int_T^\infty\frac{C\eta(t)t}{t-c}\,dt<\infty,\qquad \int_T^\infty\Bigg\{\frac{C\eta(1+\eta t)}{(t-c)^{n-2}}Y+C|N|\left|[\eta(t-c)^{3-n}]'\right|+\frac{C\eta^2t}{(t-c)^{n-2}}|N|\Bigg\}\,dt \nonumber\\
&{}<\infty. \label{eq:level-Z-integrable}
\end{align}
Equations~\eqref{eq:level-negative-mass} and~\eqref{eq:level-mean-bound} imply
\begin{align}
&J_-(t)\le I_-(t)\le Ct^{n-5}, \nonumber\\
&Z(t)\ge-(t-c)^{3-n}J_-(t)-C\eta(t)(t-c)^{3-n}|N(t)|\ge-Ct^{-2}. \label{eq:level-Z-lower}
\end{align}
Multiplying \eqref{eq:level-Z-differential} by the integrating factor yields
\begin{align}
&\left(e^{-\int_T^t\frac{C\eta(s)s}{s-c}\,ds}Z(t)\right)'\le e^{-\int_T^t\frac{C\eta(s)s}{s-c}\,ds}\Bigg\{\frac{C\eta(t)(1+\eta(t)t)}{(t-c)^{n-2}}Y(t) \nonumber\\
&{}+C|N(t)|\left|[\eta(t)(t-c)^{3-n}]'\right|+\frac{C\eta(t)^2t}{(t-c)^{n-2}}|N(t)|\Bigg\}. \nonumber
\end{align}
Therefore
\begin{align}
&\frac{d}{dt}\Bigg[e^{-\int_T^t\frac{C\eta(s)s}{s-c}\,ds}Z(t)-\int_T^t e^{-\int_T^\tau\frac{C\eta(s)s}{s-c}\,ds}\Bigg\{\frac{C\eta(\tau)(1+\eta(\tau)\tau)}{(\tau-c)^{n-2}}Y(\tau) \nonumber\\
&{}+C|N(\tau)|\left|[\eta(\tau)(\tau-c)^{3-n}]'\right|+\frac{C\eta(\tau)^2\tau}{(\tau-c)^{n-2}}|N(\tau)|\Bigg\}\,d\tau\Bigg]\le0. \nonumber
\end{align}
Equations~\eqref{eq:level-Z-integrable} and~\eqref{eq:level-Z-lower} bound this nonincreasing expression below. Hence there is $L_0\in\mathbb R$ such that
\begin{align}
&\lim_{t\to\infty}\Bigg[e^{-\int_T^t\frac{C\eta(s)s}{s-c}\,ds}Z(t)-\int_T^t e^{-\int_T^\tau\frac{C\eta(s)s}{s-c}\,ds}\Bigg\{\frac{C\eta(\tau)(1+\eta(\tau)\tau)}{(\tau-c)^{n-2}}Y(\tau) \nonumber\\
&{}+C|N(\tau)|\left|[\eta(\tau)(\tau-c)^{3-n}]'\right|+\frac{C\eta(\tau)^2\tau}{(\tau-c)^{n-2}}|N(\tau)|\Bigg\}\,d\tau\Bigg]=L_0, \nonumber\\
&0<\lim_{t\to\infty}e^{-\int_T^t\frac{C\eta(s)s}{s-c}\,ds}=e^{-\int_T^\infty\frac{C\eta(s)s}{s-c}\,ds}<\infty, \nonumber\\
&\lim_{t\to\infty}Z(t)=L_\Sigma\in[0,\infty). \nonumber
\end{align}
Finally, equation~\eqref{eq:level-mean-bound} gives
\postdisplaypenalty=10000
\begin{align}
&|C\eta(t)(t-c)^{3-n}N(t)|\le Ct^{-2},\qquad \lim_{t\to\infty}\frac{t-c}{t}=1, \nonumber\\
&\lim_{t\to\infty}(t-c)^{3-n}J(t)=\lim_{t\to\infty}t^{3-n}\int_{\Sigma_t}\Sc_{\Sigma_t}\,dA=L_\Sigma\in[0,\infty). \nonumber
\end{align}
\end{proof}

\section{The integral over distance balls}\label{sec:ambient-limit}

\begin{theorem}\label{thm:direct-main}\label{prop:ambient-limit}
Let $n\ge3$ and let $(M^n,g)$ be a smooth, complete, connected,
noncompact Riemannian manifold without boundary, with $\sec_g\ge0$. Then, for every $\theta\in C_c^\infty((0,\infty)), 0<a<b<\infty$,
\begin{align}
&\lim_{s\to\infty}2s^{2-n}\int_M\theta(F/s)\sigma_2(S)\,dV \nonumber\\
&{}=n(n-1)(n-2)\omega_n \AVR(M^n,g)\int_0^\infty\theta(t)t^{n-3}\,dt, \nonumber\\
&\lim_{s\to\infty}s^{2-n}\int_M\theta(F/s)\Ric(\nu,\nu)\,dV=0, \label{eq:nt-weighted-radial-Ricci-limit}\\
&\lim_{s\to\infty}s^{2-n}\int_M\theta(F/s)\Sc\,dV \nonumber\\
&{}=\bigl(L_\Sigma-n(n-1)(n-2)\omega_n \AVR(M^n,g)\bigr) \int_0^\infty\theta(t)t^{n-3}\,dt, \label{eq:nt-weighted-limit}\\
&\lim_{s\to\infty}s^{2-n}\int_{\{as<d_g(p,\cdot)<bs\}}\Sc\,dV \nonumber\\
&{}=\left(\frac{L_\Sigma}{n-2}-n(n-1)\omega_n \AVR(M^n,g)\right) (b^{n-2}-a^{n-2}), \label{eq:nt-annulus-limit}\\
&L(M, g)=\frac{L_\Sigma}{n-2}-n(n-1)\omega_n \AVR(M^n,g)\in[0,\infty)\label{eq:nt-ball-conclusion},\\
&\lim_{r\to\infty}r^{2-n}\int_{\{F\le r\}}\Sc\,dV=L(M,g). \label{eq:nt-F-sublevel-limit}
\end{align}
\end{theorem}

\begin{proof}
Use Notation~\ref{notation:level-objects} and Lemma~\ref{lem:annular-estimates}, \eqref{eq:nt-volume}--\eqref{eq:nt-normal-small}.
\begin{align}
&0<a<b<\infty,\quad 0\le\theta\in C_c^\infty((a,b)),\quad s\ge s_0\ge1,\quad m=n-1, \nonumber\\
&K\Subset M\ \Longrightarrow\ \supp\theta(F/s)\cap K=\varnothing\quad(s\ge s_0(K,\theta)), \nonumber\\
&\lim_{s\to\infty}s^{2-n}\int_K\Sc\,dV=0,\qquad x\in\{as\le F\le bs\}:\quad S\ge-2s^{-2}\Id,\qquad P=S+2s^{-2}\Id\ge0. \nonumber\\
&\sigma_2(S) =\sigma_2(P)-2(m-1)s^{-2}\operatorname{tr}P +2m(m-1)s^{-4},\qquad (\sigma_2(S))_-\le2(m-1)s^{-2}\operatorname{tr}P. \nonumber\\
&|\nabla F|\tr P=\Delta F-\Hess F(\nu,\nu)+2ms^{-2}|\nabla F|\le\Delta F+(1+4m)s^{-2}, \nonumber\\
&\Delta F+ns^{-2}\ge0,\quad \tfrac12\le |\nabla F|\le2\ \Longrightarrow\ \tr P\le C(\Delta F+2ns^{-2}). \nonumber
\end{align}
 Lemma~\ref{lem:annular-estimates}, \eqref{eq:nt-trace-bound}:
\begin{align}
&0\le s^{2-n}\int_M\theta(F/s)(\sigma_2(S))_-\,dV \le C/s. \label{eq:nt-negative-sigma-integral}
\end{align}
Gauss, with $(e_1,\ldots,e_m)$ orthonormal in $T_x\Sigma_{F(x)}$:
\begin{align}
&\Sc_{\Sigma_{F(x)}} =2\sum_{1\le i<j\le m}\sec_g(e_i\wedge e_j)+2\sigma_2(S),\qquad (\Sc_{\Sigma_{F(x)}})_-\le2(\sigma_2(S))_-. \nonumber
\end{align}
Coarea, Lemma~\ref{lem:independent-level-limit}, \eqref{eq:level-curvature-limit}, and \eqref{eq:nt-negative-sigma-integral}:
\begin{align}
&\sup_{s\ge s_0}s^{3-n}\int_a^b\theta(t) \int_{\Sigma_{st}}|\Sc_{\Sigma_{st}}|\,dA\,dt<\infty, \nonumber\\
&\lim_{s\to\infty}s^{2-n}\int_M \theta(F/s)\Sc_{\Sigma_{F(x)}}\,dV =L_\Sigma\int_0^\infty\theta(t)t^{n-3}\,dt. \label{eq:nt-intrinsic-limit}
\end{align}
The coarea expression in \eqref{eq:nt-intrinsic-limit} is
\begin{align}
&s^{3-n}\int_a^b\theta(t)\int_{\Sigma_{st}} \frac{\Sc_{\Sigma_{st}}}{|\nabla F|}\,dA\,dt. \nonumber\\
&\left|s^{3-n}\int_a^b\theta(t)\int_{\Sigma_{st}}(|\nabla F|^{-1}-1)\Sc_{\Sigma_{st}}\,dA\,dt\right| \nonumber\\
&{}\le\sup_{\{as\le F\le bs\}}||\nabla F|^{-1}-1|\ s^{3-n}\int_a^b\theta(t)\int_{\Sigma_{st}}|\Sc_{\Sigma_{st}}|\,dA\,dt\longrightarrow0\quad(s\to\infty), \nonumber\\
&s^{3-n}\int_a^b\theta(t)\int_{\Sigma_{st}}(\Sc_{\Sigma_{st}})_-\,dA\,dt\le4s^{2-n}\int_M\theta(F/s)(\sigma_2(S))_-\,dV. \nonumber
\end{align}
 Lemma~\ref{lem:independent-level-limit}, \eqref{eq:level-curvature-limit}, dominated convergence, and Gauss yield
\begin{align}
&\sup_{s\ge s_0}s^{2-n}\int_M\theta(F/s)|\sigma_2(S)|\,dV<\infty. \label{eq:nt-sigma-absolute}\\
&|\nu|=1,\qquad \diver\nu=H,\qquad \langle\nabla_X\nu,\nu\rangle=0. \nonumber\\
&\operatorname{div}(\nabla_\nu\nu) =\nu(H)+|S|^2+\Ric(\nu,\nu),\qquad \operatorname{div}(H\nu-\nabla_\nu\nu) =2\sigma_2(S)-\Ric(\nu,\nu). \nonumber\\
&\diver(\nabla_\nu\nu)=\nu(\diver\nu)+(\nabla_i\nu^j)(\nabla_j\nu^i)+\Ric(\nu,\nu), \nonumber\\
&(\nabla_i\nu^j)(\nabla_j\nu^i)=|S|^2,\qquad \nabla\theta(F/s)=s^{-1}\theta^{\prime}(F/s)|\nabla F|\nu. \nonumber
\end{align}
 Integration by parts:
\begin{align}
&s^{2-n}\int_M\theta(F/s)\Ric(\nu,\nu)\,dV \nonumber\\
&{}=2s^{2-n}\int_M\theta(F/s)\sigma_2(S)\,dV +s^{1-n}\int_M\theta'(F/s) (\Delta F-\Hess F(\nu,\nu))\,dV\ge0, \nonumber\\
&\lim_{s\to\infty}s^{1-n}\int_M\theta'(F/s) (\Delta F-\Hess F(\nu,\nu))\,dV \nonumber\\
&{}=-n\omega_n \AVR(M^n,g)\int_0^\infty\theta''(t)t^{n-1}\,dt \nonumber\\
&{}=-n(n-1)(n-2)\omega_n \AVR(M^n,g)\int_0^\infty\theta(t)t^{n-3}\,dt. \label{eq:nt-lower-bound}
\end{align}
Lemma~\ref{lem:annular-estimates}, \eqref{eq:nt-volume-tests}, \eqref{eq:nt-normal-small}, and Gauss:
\begin{align}
&\supp\theta\Subset(a,b)\ \Longrightarrow\ [\theta(t)t^{n-2}]_0^\infty=[\theta^{\prime}(t)t^{n-1}]_0^\infty=0, \nonumber\\
&\Sc=\Sc_{\Sigma_{F(x)}}-2\sigma_2(S)+2\Ric(\nu,\nu). \nonumber\\
&s^{2-n}\int_M\theta(F/s)\Sc\,dV =s^{2-n}\int_M\theta(F/s)\Sc_{\Sigma_{F(x)}}\,dV +2s^{2-n}\int_M\theta(F/s)\sigma_2(S)\,dV \nonumber\\
&{}+2s^{1-n}\int_M\theta'(F/s) (\Delta F-\Hess F(\nu,\nu))\,dV,\qquad \sup_{s\ge s_0}s^{2-n}\int_M\theta(F/s)\Sc\,dV\le C_\theta<\infty. \label{eq:nt-scalar-preliminary}
\end{align}
On a neighborhood of $\supp\theta(F/s)$, set
\begin{align}
&P_s=(\Delta F+(n-1)s^{-2})g-\Hess F\ge0,\qquad \operatorname{div}P_s=-\Ric(\nabla F,\cdot),\qquad |P_s| \nonumber\\
&{}\le C(\Delta F+2ns^{-2}). \nonumber\\
&P_s=\tr(\Hess F+s^{-2}g)g-(\Hess F+s^{-2}g)\ge0, \nonumber\\
&\nabla^i(\Hess F)_{ij}=\nabla_j\Delta F+\Ric_{jk}\nabla^kF. \nonumber
\end{align}
 Set
\begin{align}
&\delta_s=s^{-1}\mathop{\mathrm{ess\,sup}}_{\{as\le F\le bs\}} \left|\nabla\bigl((\rho^2-F^2)/2\bigr)\right|, \qquad \lim_{s\to\infty}\delta_s=0. \nonumber\\
&\nabla\bigl((\rho^2-F^2)/2\bigr) =\rho(\nabla\rho-\nabla F)+(\rho-F)\nabla F \nonumber
\end{align}
Lemma~\ref{one:lem:fixed-F}, \eqref{eq:gradient-limits}, gives
\begin{align}
&\delta_s\le\sup_{\{as\le F\le bs\}}\frac\rho s\,\mathop{\mathrm{ess\,sup}}_{\{as\le F\le bs\}}|\nabla\rho-\nabla F|+2s^{-1}\sup_{\{as\le F\le bs\}}|\rho-F|\longrightarrow0\quad(s\to\infty), \nonumber\\
&\frac{\theta(F/s)}F P_s\in C_c^\infty(M;\operatorname{Sym}^2T^*M),\quad \frac{\theta(F/s)}F P_s\ge0,\quad \left|\nabla\frac{\theta(F/s)}F\right|\le Cs^{-2}. \nonumber
\end{align}
 Distributional integration by parts:
\begin{align}
&s^{2-n}\left|\left\langle \Hess\bigl((\rho^2-F^2)/2\bigr),\frac{\theta(F/s)}F P_s\right\rangle\right| \nonumber\\
&{}\le C\delta_s\left( s^{1-n}\int_{\{as\le F\le bs\}}(\Delta F+2ns^{-2})\,dV +s^{2-n}\int_M\theta(F/s)\Sc\,dV\right). \nonumber\\
&\operatorname{div}\left(\frac{\theta(F/s)}F P_s\right) =P_s\left(\nabla\frac{\theta(F/s)}F,\cdot\right) -\frac{\theta(F/s)}F\Ric(\nabla F,\cdot), \nonumber\\
&\|\Ric\|_{\mathrm{op}}\le\Sc,\qquad s^{1-n}\int_{\{as\le F\le bs\}}(\Delta F+2ns^{-2})\,dV+s^{2-n}\int_M\theta(F/s)\Sc\,dV\le C, \nonumber\\
&\lim_{s\to\infty}s^{2-n}\left|\left\langle\Hess\bigl((\rho^2-F^2)/2\bigr),\frac{\theta(F/s)}F P_s\right\rangle\right|=0. \nonumber
\end{align}
 Here use Lemma~\ref{lem:annular-estimates}, \eqref{eq:nt-trace-bound}, and \eqref{eq:nt-scalar-preliminary}. Theorem~\ref{ext:thm:comparison}, \eqref{eq:distance-square-hessian-comparison}, then gives
\begin{align}
&\limsup_{s\to\infty}s^{2-n} \int_M \frac{\theta(F/s)}F \left(\langle P_s,\Hess(F^2/2)\rangle_g-\operatorname{tr}P_s\right)\,dV\le0. \nonumber\\
&\langle P_s,\Hess(F^2/2)\rangle_g =F\bigl((\Delta F)^2-|\Hess F|^2\bigr) +|\nabla F|^2\bigl(\Delta F-\Hess F(\nu,\nu)\bigr) +(n-1)s^{-2}(F\Delta F+|\nabla F|^2), \nonumber\\
&\operatorname{tr}P_s=(n-1)\Delta F+n(n-1)s^{-2}. \nonumber
\end{align}
By Lemma~\ref{lem:annular-estimates}, \eqref{eq:nt-volume}, \eqref{eq:nt-trace-bound}, \eqref{eq:nt-normal-small}, and Lemma~\ref{one:lem:fixed-F}, \eqref{eq:gradient-limits},
\begin{align}
&s^{2-n}\int_M\frac{\theta(F/s)}F(n-1)s^{-2}\bigl(F|\Delta F|+|\nabla F|^2+n\bigr)\,dV\le C/s, \nonumber\\
&s^{2-n}\int_M\frac{\theta(F/s)}F |\nabla F|^2|\Hess F(\nu,\nu)|\,dV\longrightarrow0\quad(s\to\infty), \nonumber\\
&s^{2-n}\int_M\frac{\theta(F/s)}F||\nabla F|^2-1|\,|\Delta F|\,dV\le C\sup_{\{as\le F\le bs\}}||\nabla F|^2-1|\longrightarrow0\quad(s\to\infty). \nonumber\\
&\limsup_{s\to\infty}s^{2-n}\Bigg[ \int_M\theta(F/s)\bigl((\Delta F)^2-|\Hess F|^2\bigr)\,dV -(n-2)\int_M \frac{\theta(F/s)\Delta F}{F}\,dV \Bigg]\le0. \label{eq:nt-full-hessian-upper}
\end{align}
Integration by parts and Lemma~\ref{lem:annular-estimates}, \eqref{eq:nt-volume-tests}:
\begin{align}
&s^{2-n}\int_M \frac{\theta(F/s)\Delta F}{F}\,dV =s^{-n}\int_M\left( \frac{\theta(F/s)}{(F/s)^2}-\frac{\theta'(F/s)}{F/s} \right)|\nabla F|^2\,dV, \nonumber\\
&\lim_{s\to\infty}s^{2-n}\int_M \frac{\theta(F/s)\Delta F}{F}\,dV =n(n-1)\omega_n \AVR(M^n,g)\int_0^\infty\theta(t)t^{n-3}\,dt. \label{eq:nt-trace-limit}
\end{align}
In an orthonormal frame $(\nu,e_1,\ldots,e_m)$, $m=n-1$, the principal $2\times2$ minors of $\Hess F+s^{-2}g\ge0$ give
\begin{align}
&\sum_{j=1}^m\Hess F(\nu,e_j)^2 \le(\Hess F(\nu,\nu)+s^{-2}) \left(\sum_{j=1}^m\Hess F(e_j,e_j)+ms^{-2}\right), \nonumber\\
&\Hess F(\nu,\nu)\sum_{j=1}^m\Hess F(e_j,e_j) -\sum_{j=1}^m\Hess F(\nu,e_j)^2\ge-ms^{-2}(\Delta F+ns^{-2}). \label{eq:nt-normal-minors}\\
&\Hess F(\nu,\nu)\le\Delta F+ms^{-2},\qquad \Hess F|_{T\Sigma_{F(x)}\times T\Sigma_{F(x)}}=|\nabla F|S, \nonumber\\
&(\Delta F)^2-|\Hess F|^2=2|\nabla F|^2\sigma_2(S)+2\Hess F(\nu,\nu)\sum_{j=1}^m\Hess F(e_j,e_j)-2\sum_{j=1}^m\Hess F(\nu,e_j)^2. \nonumber\\
&(\Delta F)^2-|\Hess F|^2 \ge2|\nabla F|^2\sigma_2(S) -2(n-1)s^{-2}(\Delta F+ns^{-2}). \label{eq:nt-block-estimate}\\
&0\le s^{2-n}\int_M2(n-1)s^{-2}\theta(F/s)(\Delta F+ns^{-2})\,dV\le C/s. \nonumber
\end{align}
 By \eqref{eq:nt-sigma-absolute},
\begin{align}
&\lim_{s\to\infty}s^{2-n}\left|\int_M\theta(F/s) (|\nabla F|^2-1)\sigma_2(S)\,dV\right|=0. \label{eq:nt-gradient-removal}
\end{align}
Equations~\eqref{eq:nt-full-hessian-upper}, \eqref{eq:nt-trace-limit}, \eqref{eq:nt-normal-minors}, \eqref{eq:nt-block-estimate}, and \eqref{eq:nt-gradient-removal}:
\begin{align}
&\limsup_{s\to\infty}2s^{2-n}\int_M\theta(F/s)\sigma_2(S)\,dV \nonumber\\
&{}\le n(n-1)(n-2)\omega_n \AVR(M^n,g)\int_0^\infty\theta(t)t^{n-3}\,dt. \nonumber
\end{align}
Equation~\eqref{eq:nt-lower-bound} gives, also for $\AVR(M^n,g)=0$,
\begin{align}
&\liminf_{s\to\infty}2s^{2-n}\int_M\theta(F/s)\sigma_2(S)\,dV \nonumber\\
&{}\ge n(n-1)(n-2)\omega_n \AVR(M^n,g)\int_0^\infty\theta(t)t^{n-3}\,dt. \nonumber\\
&\lim_{s\to\infty}2s^{2-n}\int_M\theta(F/s)\sigma_2(S)\,dV \nonumber\\
&{}=n(n-1)(n-2)\omega_n \AVR(M^n,g)\int_0^\infty\theta(t)t^{n-3}\,dt, \nonumber\\
&\lim_{s\to\infty}s^{2-n}\int_M\theta(F/s)\Ric(\nu,\nu)\,dV=0,\qquad \lim_{s\to\infty}s^{2-n}\int_M\theta(F/s)\Sc\,dV \nonumber\\
&{}=\bigl(L_\Sigma-n(n-1)(n-2)\omega_n \AVR(M^n,g)\bigr) \int_0^\infty\theta(t)t^{n-3}\,dt. \nonumber
\end{align}
Equations~\eqref{eq:nt-intrinsic-limit}, \eqref{eq:nt-scalar-preliminary}, and $\Sc\ge0$ give
\begin{align}
&L_\Sigma\ge n(n-1)(n-2)\omega_n \AVR(M^n,g). \nonumber\\
&\theta\in C_c^\infty((0,\infty))\ \Longrightarrow\ \theta=\theta_1-\theta_2,\quad 0\le\theta_1,\theta_2\in C_c^\infty((0,\infty)). \nonumber
\end{align}
 For the annulus approximation, choose
\begin{align}
&0<\varepsilon<\min\{a/4,(b-a)/5\},\qquad 0\le\theta_\varepsilon^\pm\le1, \nonumber\\
&\theta_\varepsilon^-\in C_c^\infty((a+\varepsilon,b-\varepsilon)),\quad \theta_\varepsilon^-=1\ \text{on }[a+2\varepsilon,b-2\varepsilon], \nonumber\\
&\theta_\varepsilon^+\in C_c^\infty((a-2\varepsilon,b+2\varepsilon)),\quad \theta_\varepsilon^+=1\ \text{on }[a-\varepsilon,b+\varepsilon], \nonumber\\
&\lim_{s\to\infty}s^{-1}\sup_{\{as/2\le F\le2bs\}}|F-\rho|=0, \nonumber\\
&\theta_\varepsilon^-(F/s)\le\mathbf1_{\{as<\rho<bs\}}\le\theta_\varepsilon^+(F/s)\quad(s\ge s_0(\varepsilon)), \nonumber\\
&\lim_{\varepsilon\downarrow0}\int_0^\infty\theta_\varepsilon^\pm(t)t^{n-3}\,dt=\frac{b^{n-2}-a^{n-2}}{n-2}. \nonumber
\end{align}
 Equation~\eqref{eq:nt-weighted-limit}, first $s\to\infty$, then $\varepsilon\downarrow0$:
\begin{align}
&\lim_{s\to\infty}s^{2-n}\int_{\{as<\rho<bs\}}\Sc\,dV=\left(\frac{L_\Sigma}{n-2}-n(n-1)\omega_n \AVR(M^n,g)\right) (b^{n-2}-a^{n-2}). \nonumber\\
&\{\rho=r\}\subset\{\exp_p(r\xi):\xi\in ST_pM\},\qquad \dim ST_pM=n-1\ \Longrightarrow\ \Vol_g\{\rho=r\}=0, \nonumber\\
&L=\frac{L_\Sigma}{n-2}-n(n-1)\omega_n \AVR(M^n,g)\ge0. \nonumber
\end{align}
 Equation~\eqref{eq:nt-annulus-limit}, $(a,b)=(1,2)$: for every $\varepsilon>0$ there is $r_\varepsilon\ge1$ such that
\begin{align}
&\left|\int_{\{r<\rho<2r\}}\Sc\,dV -L(2^{n-2}-1)r^{n-2}\right| \le\varepsilon(2^{n-2}-1)r^{n-2} \quad(r\ge r_\varepsilon). \label{eq:nt-dyadic-error}\\
&s\ge2r_\varepsilon,\quad N\in\mathbb Z_{\ge1},\quad r_\varepsilon\le2^{-N}s<2r_\varepsilon, \nonumber\\
&(2^{n-2}-1)\sum_{j=1}^N(2^{-j}s)^{n-2}=s^{n-2}-(2^{-N}s)^{n-2}. \nonumber
\end{align}
 Summing \eqref{eq:nt-dyadic-error} over the disjoint annuli gives
\begin{align}
&\left|s^{2-n}\int_{B_p(s)}\Sc\,dV-L\right| \le\varepsilon+s^{2-n}\left( \int_{B_p(2r_\varepsilon)}\Sc\,dV +L(2r_\varepsilon)^{n-2}\right). \nonumber\\
&\int_{B_p(2r_\varepsilon)}\Sc\,dV<\infty,\qquad \limsup_{s\to\infty}\left|s^{2-n}\int_{B_p(s)}\Sc\,dV-L\right|\le\varepsilon\quad(\varepsilon>0), \nonumber\\
&x\in M,\quad D=d_g(p,x),\quad B_p(r-D)\subset B_x(r)\subset B_p(r+D)\quad(r>D), \nonumber\\
&\lim_{r\to\infty}r^{2-n}\int_{B_p(r\pm D)}\Sc\,dV=L\ \Longrightarrow\ \lim_{r\to\infty}r^{2-n}\int_{B_x(r)}\Sc\,dV=L. \nonumber
\end{align}
 For each $0<\varepsilon<1$, choose $K_\varepsilon\Subset M$ with
\begin{align}
&(1-\varepsilon)\rho\le F\le(1+\varepsilon)\rho\quad\text{on }M\setminus K_\varepsilon, \nonumber\\
&B_p\bigl(r/(1+\varepsilon)\bigr)\setminus K_\varepsilon\subset\{F\le r\}\setminus K_\varepsilon\subset\overline{B_p\bigl(r/(1-\varepsilon)\bigr)}, \nonumber\\
&(1+\varepsilon)^{2-n}L\le\liminf_{r\to\infty}r^{2-n}\int_{\{F\le r\}}\Sc\,dV\le\limsup_{r\to\infty}r^{2-n}\int_{\{F\le r\}}\Sc\,dV\le(1-\varepsilon)^{2-n}L. \nonumber\\
&\varepsilon\downarrow0\ \Longrightarrow\ \lim_{r\to\infty}r^{2-n}\int_{\{F\le r\}}\Sc\,dV=L. \nonumber
\end{align}
 These are \eqref{eq:nt-ball-conclusion} and \eqref{eq:nt-F-sublevel-limit}.
\end{proof}

\begin{theorem}\label{k:thm:all-avr}
Let $(M^3,g)$ be a complete connected noncompact smooth Riemannian
manifold without boundary, with $\sec_g\ge0$. Then
\begin{align}
&\boxed{\lim_{s\to\infty}\frac1s\int_{B_p(s)}\Sc_g\,dV_g =8\pi\bigl(\chi(M)-\AVR(M^3,g)\bigr).} \label{k:eq:all-avr-main}
\end{align}
\end{theorem}

\begin{proof}
Lemma~\ref{lem:independent-level-limit}, \eqref{eq:level-curvature-limit}, Lemma~\ref{tp:lem:level-Euler}, and Theorem~\ref{thm:direct-main}, \eqref{eq:nt-ball-conclusion}, give
\begin{align}
&L_\Sigma=8\pi\chi(M),\qquad n=3,\qquad\omega_3=4\pi/3, \nonumber\\
&\lim_{s\to\infty}\frac1s\int_{B_p(s)}\Sc_g\,dV_g=L_\Sigma-6\omega_3\AVR(M^3,g)=8\pi\bigl(\chi(M)-\AVR(M^3,g)\bigr), \nonumber
\end{align}
which is \eqref{k:eq:all-avr-main}.
\end{proof}

\begin{notation}\label{notation-E1}
{Regard the Ricci tensor as an endomorphism using $g$, and define
\begin{align}
&E_1=\frac{\Sc_g}{2}\Id-\Ric_g. \label{eq:first-curvature-contraction}
\end{align}
}
\end{notation}

\begin{lemma}\label{lem:first-curvature-contraction}
Let $(M^n,g)$ be a Riemannian manifold with $n\ge3$ and $\sec_g\ge0$, then
\begin{align}
&0\le E_1\le\frac{\Sc_g}{2}\Id,\qquad \tr E_1=\frac{n-2}{2}\Sc_g,\qquad \diver E_1=0. \label{pt:eq:E1-properties}
\end{align}
\end{lemma}
\begin{proof}
For a unit vector $e$, choose an orthonormal basis $(e_a)_{a=1}^{n-1}$ of $e^\perp$. Then
\begin{align}
&E_1(e,e)=\sum_{a<b}\sec_g(e_a\wedge e_b)\in[0,\Sc_g/2],\qquad \tr E_1=\frac{n-2}{2}\Sc_g, \nonumber\\
&\nabla^a(\Ric_g)_{ab}=\frac12\nabla_b\Sc_g\quad\Longrightarrow\quad\diver E_1=0. \nonumber
\end{align}
\end{proof}

\begin{lemma}\label{lem:scalar-radial-tests}
Let $(M^n,g)$ be a complete noncompact Riemannian manifold with $n\ge3$ and $\sec_g\ge0$. For every $\psi\in C_c((0,\infty))$ and $0<a<b<\infty$,
\begin{align}
&\lim_{r\to\infty}r^{2-n}\int_M\psi(\rho/r)\Sc\,dV =(n-2)L\int_0^\infty\psi(t)t^{n-3}\,dt, \label{eq:weighted-distance-scalar-tests}\\
&\lim_{r\to\infty}r^{2-n}\int_M\psi(F/r)\Sc\,dV =(n-2)L\int_0^\infty\psi(t)t^{n-3}\,dt, \label{eq:common-scalar-F-tests}\\
&\lim_{r\to\infty}r^{2-n} \int_{\{ar<F<br\}}\Ric(\nu,\nu)\,dV=0. \label{eq:common-radial-F-tests}
\end{align}
Here $L=L(M,g)$.
\end{lemma}

\begin{proof}
Fix $0<A<B<\infty$ with $\supp\psi\Subset(A,B)$. Choose step functions $\psi_j$ supported in $[A,B]$ with $\|\psi_j-\psi\|_\infty\to0$. Theorem~\ref{thm:direct-main}, equation~\eqref{eq:nt-annulus-limit}, gives
\begin{align}
&\lim_{r\to\infty}r^{2-n}\int_M\psi_j(\rho/r)\Sc\,dV=(n-2)L\int_A^B\psi_j(t)t^{n-3}\,dt, \nonumber\\
&\sup_{r\ge r_0}r^{2-n}\int_{\{Ar\le\rho\le Br\}}\Sc\,dV<\infty, \nonumber\\
&r^{2-n}\left|\int_M(\psi_j-\psi)(\rho/r)\Sc\,dV\right|\le\|\psi_j-\psi\|_\infty r^{2-n}\int_{\{Ar\le\rho\le Br\}}\Sc\,dV, \nonumber\\
&\left|\int_A^B(\psi_j-\psi)(t)t^{n-3}\,dt\right|\le\|\psi_j-\psi\|_\infty\int_A^Bt^{n-3}\,dt. \nonumber
\end{align}
Taking $r\to\infty$, then $j\to\infty$, proves \eqref{eq:weighted-distance-scalar-tests}. Since $\lim_{R\to\infty}\sup_{\rho\ge R}|F/\rho-1|=0$, for large $r$,
\begin{align}
&\supp\bigl(\psi(F/r)-\psi(\rho/r)\bigr)\subset\{Ar/2<\rho<2Br\}, \nonumber\\
&\lim_{r\to\infty}\sup_{\{Ar/2<\rho<2Br\}}|\psi(F/r)-\psi(\rho/r)|=0, \nonumber\\
&r^{2-n}\left|\int_M\bigl(\psi(F/r)-\psi(\rho/r)\bigr)\Sc\,dV\right|\le\sup_{\{Ar/2<\rho<2Br\}}|\psi(F/r) \nonumber\\
&{}-\psi(\rho/r)|\,r^{2-n}\int_{\{Ar/2<\rho<2Br\}}\Sc\,dV. \nonumber
\end{align}
This proves \eqref{eq:common-scalar-F-tests}. Choose $0\le\eta\in C_c^\infty((0,\infty))$, $\eta\ge1$ on $[a/2,2b]$. By Theorem~\ref{thm:direct-main}, equation~\eqref{eq:nt-weighted-radial-Ricci-limit},
\begin{align}
&0\le r^{2-n}\int_{\{ar<F<br\}}\Ric(\nu,\nu)\,dV\le r^{2-n}\int_M\eta(F/r)\Ric(\nu,\nu)\,dV, \nonumber\\
&\lim_{r\to\infty}r^{2-n}\int_M\eta(F/r)\Ric(\nu,\nu)\,dV=0. \nonumber\qedhere
\end{align}
\end{proof}

\begin{lemma}\label{unified:lem:distance-limit}
Let $(M^6,g)$ be a complete noncompact Riemannian manifold with $\sec_g\ge0$. And let $L=L(M,g)$, then
\begin{align}
&\lim_{r\to\infty}r^{-4}\int_{\{r<F<2r\}}(F/r)^{-3}|\nabla F|^2\left(\frac{\Sc}{2}-\Ric(\nu,\nu)\right)dV=2L. \label{unified:eq:distance-limit}
\end{align}
\end{lemma}

\begin{proof}
Set
\begin{align}
&N(t)=\int_{\{F\le t\}}\Sc\,dV. \nonumber
\end{align}
By Theorem~\ref{thm:direct-main}, equation~\eqref{eq:nt-F-sublevel-limit},
\begin{align}
&\lim_{t\to\infty}\frac{N(t)}{t^4}=L,\qquad \lim_{r\to\infty}\sup_{1\le s\le2} \left|\frac{N(rs)}{r^4}-Ls^4\right|=0. \nonumber
\end{align}
since
\begin{align}
&\sup_{1\le s\le2} \left|r^{-4}N(rs)-Ls^4\right| \le16\sup_{t\ge r}\left|t^{-4}N(t)-L\right|. \nonumber
\end{align}
On each compact regular level band, $N\in W^{1,1}$ and coarea gives
\begin{align}
&N'(t)=\int_{\Sigma_t}\frac{\Sc}{|\nabla F|}\,dA \quad\text{for almost every }t. \nonumber
\end{align}
Integration by parts gives the exact formula
\begin{align}
&r^{-4}\int_{\{r<F<2r\}}(F/r)^{-3}\Sc\,dV =r^{-4}\int_r^{2r}(t/r)^{-3}N'(t)\,dt \nonumber\\
&{}=\frac18\frac{N(2r)}{r^4}-\frac{N(r)}{r^4} +3\int_1^2s^{-4}\frac{N(rs)}{r^4}\,ds. \label{interval:eq:scalar-IBP}
\end{align}
By the preceding uniform convergence, the limit in \eqref{interval:eq:scalar-IBP} equals
\begin{align}
&2L-L+3\int_1^2L\,ds=4L. \nonumber
\end{align}
By Theorem~\ref{thm:direct-main}, equation~\eqref{eq:nt-weighted-radial-Ricci-limit},
\begin{align}
&r^{-4}\left|\int_{\{r<F<2r\}}(F/r)^{-3}(|\nabla F|^2-1)\Sc\,dV\right|\le\sup_{\{r<F<2r\}}\bigl||\nabla F|^2-1\bigr|\,r^{-4}\int_{\{r<F<2r\}}\Sc\,dV, \nonumber\\
&\lim_{r\to\infty}r^{-4}\int_{\{r<F<2r\}}(F/r)^{-3}(|\nabla F|^2-1)\Sc\,dV=0, \nonumber\\
&0\le r^{-4}\int_{\{r<F<2r\}}(F/r)^{-3}|\nabla F|^2\Ric(\nu,\nu)\,dV\le4r^{-4}\int_{\{r<F<2r\}}\Ric(\nu,\nu)\,dV, \nonumber\\
&\lim_{r\to\infty}r^{-4}\int_{\{r<F<2r\}}\Ric(\nu,\nu)\,dV=0. \nonumber
\end{align}
\end{proof}

\begin{lemma}\label{low:lem:Euclidean-products}
Let $(X^m,k)$ be smooth, complete, connected and without boundary,
with $m\ge2$ and $\sec_k\ge0$. Let $q\in\mathbb Z_{\ge0}$, with
$m+q\ge3$, and suppose $X\times\mathbb R^q$ is noncompact. For $o\in X$, the following limits exist, are finite and are independent
of $o$:
\begin{align}
&\ell_X=\lim_{r\to\infty}r^{2-m}\int_{B_o^X(r)}\Sc_k\,dV_k,\qquad \AVR(X,k)=\lim_{r\to\infty}\frac{\Vol_k B_o^X(r)}{\omega_m r^m}. \nonumber
\end{align}
With the product metric understood,
\begin{align}
&\chi(X\times\mathbb R^q)=\chi(X),\qquad \AVR(X\times\mathbb R^q)=\AVR(X,k), \nonumber\\
&L(X\times\mathbb R^q)=\frac{\omega_{m+q-2}}{\omega_{m-2}}\ell_X. \label{ext:eq:Euclidean-product}
\end{align}
\end{lemma}

\begin{proof}
Bishop--Gromov comparison \cite{CE} gives the volume limit.
For noncompact $X$ with $m\ge3$, Theorem~\ref{thm:direct-main}
gives $\ell_X$; for compact $X$ with $m\ge3$, $\ell_X=0$.
For $m=2$, the product-distance formula and $\Sc_{k+dt^2}=\Sc_k$ give
\begin{align}
&2r\int_{B_o^X(r)}\Sc_k\,dV_k\le\int_{B_{(o,0)}^{X\times\mathbb R}(\sqrt2r)}\Sc_{k+dt^2}\,dV_{k+dt^2}. \nonumber
\end{align}
Theorem~\ref{thm:direct-main} applied to $X\times\mathbb R$ bounds
the right side divided by $r$ as $r\to\infty$. Monotone convergence
therefore gives $\ell_X=\int_X\Sc_k\,dV_k<\infty$.
The asserted center independence follows in each case.
The case $q=0$ is immediate. For $q\ge1$, set
\begin{align}
&I_X(t)=\int_{B_o^X(t)}\Sc_k\,dV_k,\qquad V_X(t)=\Vol_k B_o^X(t),\qquad I_X(0)=V_X(0)=0. \nonumber
\end{align}
Smoothness near $o$ and the finite limit at infinity give a constant
$C=C(X,k,o)<\infty$ such that
\begin{align}
&0\le I_X(t)\le Ct^{m-2},\qquad 0\le V_X(t)\le\omega_m t^m\qquad(t>0). \nonumber
\end{align}
Fubini's theorem and the product-distance formula give
\begin{align}
&\int_{B_{(o,0)}^{X\times\mathbb R^q}(r)}\Sc\,dV=\int_{|z|<r}I_X\!\left(\sqrt{r^2-|z|^2}\right)\,dz, \nonumber\\
&\Vol B_{(o,0)}^{X\times\mathbb R^q}(r)=\int_{|z|<r}V_X\!\left(\sqrt{r^2-|z|^2}\right)\,dz. \nonumber
\end{align}
Change variables $z=rw$. The preceding bounds give integrable
dominating functions $C(1-|w|^2)^{(m-2)/2}$ and
$\omega_m(1-|w|^2)^{m/2}$ on the Euclidean unit $q$-ball.
Slicing a Euclidean unit ball gives
\begin{align}
&\int_{|w|<1}(1-|w|^2)^{a/2}\,dw=\frac{\omega_{a+q}}{\omega_a}\qquad(a\in\mathbb Z_{\ge0}). \nonumber
\end{align}
Dominated convergence, with $a=m-2$ and $a=m$, proves the two
limit formulas in \eqref{ext:eq:Euclidean-product}.
Finally, $(x,z)\mapsto(x,(1-t)z)$, $0\le t\le1$, is a deformation
retraction onto $X\times\{0\}$, proving the Euler-characteristic identity.
\end{proof}

\section{The soul-distance contraction}
\label{sec:soul-trace-products}

Let $\mathcal S\subset M$ be a soul, with induced metric $h$, and put
$s=\dim\mathcal S$. Use the retraction $P$ in Theorem~\ref{ext:thm:retraction} and
write
\begin{align}
&\mathcal V=\ker dP,\qquad \mathcal H=\mathcal V^\perp,\qquad g_{\mathcal V}(X,Y)=g(X^{\mathcal V},Y^{\mathcal V}),\qquad g_{\mathcal H}=g-g_{\mathcal V}, \nonumber\\
&r_{\mathcal S}=d_g(\mathcal S,\cdot),\qquad f_{\mathcal S}=\tfrac12r_{\mathcal S}^2,\qquad \Omega_r=\{r_{\mathcal S}<r\},\qquad \mathcal N_y=P^{-1}(y). \label{soul:eq:distance-data}
\end{align}

For a symmetric tensor $A$, define
$\operatorname{tr}_{\mathcal H}A=\sum_{i=1}^s A(X_i,X_i)$
for any horizontal orthonormal basis, and define
$\operatorname{tr}_{\mathcal V}A$ analogously; in particular,
$\langle A,g_{\mathcal H}\rangle_g=\operatorname{tr}_{\mathcal H}A$.
The measure on a zero-dimensional soul assigns mass one to its single
point. 

\begin{lemma}\label{soul:lem:hessian}
Then
\begin{align}
&\Hess f_{\mathcal S}\le g_{\mathcal V}\,dV \label{soul:eq:hessian}
\end{align}
as symmetric tensor-valued measures. For every $y\in\mathcal S$,
the fiber $\mathcal N_y$ is a smooth, connected, intrinsically complete
$(n-s)$-manifold without boundary. For every $y\in\mathcal S$ and
$x\in \mathcal N_y$, the first identity below holds; for every $y\in\mathcal S$
and $r>0$, the second holds:
\begin{align}
&d_{\mathcal N_y}(y,x)=r_{\mathcal S}(x),\qquad \mathcal N_y\cap\Omega_r=B_y^{\mathcal N_y}(r). \label{four:eq:fiber-distance}
\end{align}
For $y\in\mathcal S$ and a unit vector $\xi\in\nu_y\mathcal S$, put
\begin{align}
&\tau_{\mathcal S}(y,\xi)=\sup\{s>0:d_g(\mathcal S,\exp_y(s\xi))=s\}. \nonumber
\end{align}
If $0<t<\tau_{\mathcal S}(y,\xi)$,
$\gamma(s)=\exp_y(s\xi)$ for $0\le s\le t$, and
$X\in T_{\gamma(t)}M$ satisfies $X\perp\dot\gamma(t)$, write
$X=X_{\mathcal H}+X_{\mathcal V}$ and let $\mathcal P_{t,s}$ denote
parallel transport. With
\begin{align}
&Y(s)=\mathcal P_{t,s}X_{\mathcal H} +\frac{s}{t}\mathcal P_{t,s}X_{\mathcal V}, \nonumber
\end{align}
one has
\begin{align}
&\Hess r_{\mathcal S}(X,X) \le\int_0^t\bigl(|D_sY|^2 -\langle\Rm(Y,\dot\gamma)\dot\gamma,Y\rangle\bigr)\,ds \le\frac{|X_{\mathcal V}|^2}{t}. \label{soul:eq:normal-index}
\end{align}
For every $y\in\mathcal S$, every unit $\xi\in\nu_y\mathcal S$, and
$0<t<\tau_{\mathcal S}(y,\xi)$, let $J(t,y,\xi)$ be the normal volume
Jacobian relative to $dV_h(y)\,d\omega_y(\xi)\,dt$, where $d\omega_y$
is the area measure on the unit sphere of the normal space
$\nu_y\mathcal S$. Then
\begin{align}
&\lim_{t\downarrow0}\frac{J(t,y,\xi)}{t^{n-s-1}}=1, \label{soul:eq:jacobian-initial}\\
&\partial_t\log J\le\frac{n-s-1}{t},\qquad 0<J(t,y,\xi)\le t^{n-s-1} \quad(0<t<\tau_{\mathcal S}(y,\xi)). \label{tp:eq:normal-volume-comparison}
\end{align}
For every nonnegative continuous function $a$ on $\mathcal S$ and
$r>0$,
\begin{align}
&\int_{\Omega_r}a(P(x))\,dV(x) \le\omega_{n-s} r^{n-s}\int_{\mathcal S}a\,dV_h,\qquad \Vol(\Omega_r)\le\Vol_h(\mathcal S)\omega_{n-s} r^{n-s}. \label{tp:eq:tube-volume}
\end{align}
\end{lemma}

\begin{proof}
For $x\in \mathcal N_y$, choose a shortest normal segment $\gamma:[0,t]\to M$ from $\mathcal S$ to $x$. Theorem~\ref{ext:thm:retraction}, \eqref{ext:eq:normal-retraction}, gives
\begin{align}
&P(\gamma(u))=\gamma(0)=P(x)=y,\qquad \gamma([0,t])\subset \mathcal N_y,\qquad t=r_{\mathcal S}(x), \nonumber\\
&r_{\mathcal S}(x)\le d_g(y,x)\le d_{\mathcal N_y}(y,x)\le\operatorname{Length}(\gamma)=r_{\mathcal S}(x). \nonumber
\end{align}
Thus \eqref{four:eq:fiber-distance} holds and $\mathcal N_y$ is connected. The submersion property gives a closed embedded $(n-s)$-manifold without boundary. In a submersion chart at $x$,
\begin{align}
&(x_j)\subset \mathcal N_y,\quad d_{\mathcal N_y}(x_i,x_j)\longrightarrow0\ (i,j\to\infty) \nonumber\\
&{}\Longrightarrow\quad d_g(x_i,x_j)\longrightarrow0\quad\Longrightarrow\quad x_j\longrightarrow x\in \mathcal N_y \nonumber\\
&{}\Longrightarrow\quad d_{\mathcal N_y}(x_j,x)\longrightarrow0. \nonumber
\end{align}
For $t<\tau_{\mathcal S}(y,\xi)$ and $X\perp\dot\gamma(t)$, Theorem~\ref{ext:thm:retraction}, \eqref{ext:eq:parallel-horizontal}, gives
\begin{align}
&Y(s)=\mathcal P_{t,s}X_{\mathcal H} +\frac{s}{t}\mathcal P_{t,s}X_{\mathcal V} \nonumber
\end{align}
\begin{align}
&Y(0)\in T_y\mathcal S,\qquad Y(t)=X,\qquad |D_sY|^2=t^{-2}|X_{\mathcal V}|^2. \nonumber
\end{align}
Theorem~\ref{ext:thm:comparison} and total geodesy of $\mathcal S$ give \eqref{soul:eq:normal-index}. On the smooth locus,
\begin{align}
&\nabla r_{\mathcal S}\in\mathcal V,\qquad \Hess r_{\mathcal S}(\nabla r_{\mathcal S},\cdot)=0, \nonumber\\
&\Hess f_{\mathcal S}=dr_{\mathcal S}\otimes dr_{\mathcal S}+r_{\mathcal S}\Hess r_{\mathcal S}\le g_{\mathcal V}. \nonumber
\end{align}
Retain
\begin{align}
&\tau_{\mathcal S}(y,\xi) =\sup\{t>0:d_g(\mathcal S,\exp_y(t\xi))=t\}. \nonumber
\end{align}
The sets $\{\tau_{\mathcal S}\ge t\}$ are closed. Fubini and the smooth normal exponential map give
\begin{align}
&(dV_h\,d\omega\,dt)\{(y,\xi,t):t=\tau_{\mathcal S}(y,\xi)<\infty\}=0, \nonumber\\
&\Vol_g\{\exp_y(\tau_{\mathcal S}(y,\xi)\xi):\tau_{\mathcal S}(y,\xi)<\infty\}=0. \nonumber
\end{align}
By the cut-point criterion in Theorem~\ref{ext:thm:comparison}, the second set is the normal cut locus. For $x=\gamma(t)\notin\mathcal S$ and $0<\varepsilon<t$, clause~\textup{(ii)} of that theorem gives a smooth local upper support
\begin{align}
&r_{\mathcal S}(z)\le\varepsilon+d_g(\gamma(\varepsilon),z),\qquad r_{\mathcal S}(x)=\varepsilon+d_g(\gamma(\varepsilon),x). \nonumber
\end{align}
Its Hessian is locally bounded above; hence $r_{\mathcal S}$ and $f_{\mathcal S}$ are locally semiconcave off $\mathcal S$. 

Smooth normal tubular coordinates near $\mathcal S$ yield
\begin{align}
&(\Hess f_{\mathcal S})_{\mathrm{sing}}\le0,\qquad (\Hess f_{\mathcal S})_{\mathrm{ac}}\le g_{\mathcal V}\,dV_g,\qquad \Hess f_{\mathcal S}\le g_{\mathcal V}\,dV_g. \nonumber
\end{align}
The initial normal Jacobi data are $(X,0)$ for tangent variations with parallel initial normal vector, and $(0,Z)$ for normal angular variations. By Theorem~\ref{ext:thm:retraction}, \eqref{ext:eq:parallel-horizontal}, the horizontal fields are parallel, mutually orthonormal, and orthogonal to the radial and vertical angular fields. Consequently $J$ is the intrinsic polar Jacobian of $\mathcal N_y$. By \eqref{four:eq:fiber-distance},
\begin{align}
&\exp_y^{\mathcal N_y}(t\xi)=\exp_y^M(t\xi),\qquad \tau_{\mathcal N_y}(y,\xi)=\tau_{\mathcal S}(y,\xi). \nonumber
\end{align}
For each $y$, multiple minimizing directions are nondifferentiability points of intrinsic distance, and conjugate endpoints are critical values of $\exp_y^{\mathcal N_y}$. Almost-everywhere differentiability and the area formula give $\Vol_{k_y}(\operatorname{Cut}_{\mathcal N_y}(y))=0$, without a curvature assumption on $\mathcal N_y$. The initial data and \eqref{soul:eq:normal-index} give
\begin{align}
&\lim_{t\downarrow0}t^{1-(n-s)}J(t,y,\xi)=1,\qquad \partial_t\log J=\Delta r_{\mathcal S}\le\frac{n-s-1}{t}, \nonumber\\
&\partial_t\log(t^{1-(n-s)}J)\le0\quad\Longrightarrow\quad 0<J(t,y,\xi)\le t^{n-s-1}\quad(0<t<\tau_{\mathcal S}(y,\xi)). \nonumber
\end{align}
For every $r>0$,
\begin{align}
&\{r_{\mathcal S}=r\}\subset\{\exp_y(r\xi):y\in\mathcal S,\ \xi\in S(\nu_y\mathcal S)\},\qquad \Vol_g\{r_{\mathcal S}=r\}=0. \nonumber
\end{align}
Polar integration gives
\begin{align}
&\int_{\Omega_r}a\circ P\,dV =\int_{\mathcal S}a(y) \int_{S(\nu_y\mathcal S)} \int_0^{\min\{r,\tau_{\mathcal S}(y,\xi)\}} J(t,y,\xi)\,dt\,d\omega_y(\xi)\,dV_h(y). \nonumber
\end{align}
\begin{align}
&\int_{\Omega_r}a\circ P\,dV\le\int_{\mathcal S}a\,dV_h\,(n-s)\omega_{n-s}\int_0^r t^{n-s-1}\,dt=\omega_{n-s} r^{n-s}\int_{\mathcal S}a\,dV_h, \nonumber\\
&B_p(r)\subset\Omega_{r+d_g(p,\mathcal S)},\qquad \Vol_g B_p(r)\le\Vol_h(\mathcal S)\omega_{n-s}(r+d_g(p,\mathcal S))^{n-s}, \nonumber\\
&s>0\quad\Longrightarrow\quad n-s<n\quad\Longrightarrow\quad \AVR(M^n,g)=0. \nonumber
\end{align}
\end{proof}

\begin{lemma}\label{ext:lem:horizontal-vanishing}
Let $g_{\mathcal H}$ be the horizontal
metric of $P$.
Then the distance and ball comparisons
\begin{align}
&d_{\mathcal S}\le\rho\le d_{\mathcal S}+\operatorname{diam}_h(\mathcal S), \qquad B_p(r)\subset\Omega_r\subset B_p(r+\operatorname{diam}_h(\mathcal S)) \label{ext:eq:soul-ball-inclusions}
\end{align}
hold for every $r>0$, and
\begin{align}
&\lim_{r\to\infty}r^{2-n} \int_{\Omega_r}\langle E_1,g_{\mathcal H}\rangle_g\,dV_g=0. \label{ext:eq:horizontal-vanishing}
\end{align}
\end{lemma}

\begin{proof}
For $s=0$, $d_{\mathcal S}=\rho$ and $g_{\mathcal H}=0$. Let $s>0$. The triangle inequality gives
\begin{align}
&d_{\mathcal S}\le\rho\le d_{\mathcal S}+\operatorname{diam}_h(\mathcal S),\qquad B_p(r)\subset\Omega_r\subset B_p(r+\operatorname{diam}_h(\mathcal S)). \nonumber
\end{align}
Hence
\begin{align}
&\lim_{R\to\infty}\sup_{d_{\mathcal S}\ge R} |F/d_{\mathcal S}-1|=0. \nonumber
\end{align}
For each shortest normal $\gamma:[0,t]\to M$ ending at $x$, $t=d_{\mathcal S}(x)$,
\begin{align}
&\rho(\gamma(u))\ge u,\qquad (F\circ\gamma)''(u)\ge-(1+u)^{-3},\qquad \int_0^\infty\frac{u}{(1+u)^3}\,du=\frac12. \nonumber
\end{align}
Integration gives
\begin{align}
&t\langle\nabla F(x),\dot\gamma(t)\rangle -F(x)+F(\gamma(0)) =\int_0^t u(F\circ\gamma)''(u)\,du\ge-\tfrac12. \nonumber
\end{align}
Since $\sup_{\mathcal S}|F|<\infty$,
\begin{align}
&\lim_{R\to\infty}\sup_{d_{\mathcal S}(x)\ge R} \left(1-\langle\nabla F(x),\dot\gamma(d_{\mathcal S}(x))\rangle\right)_+=0, \nonumber
\end{align}
where $z_+=\max\{z,0\}$; the supremum includes all shortest normal segments. At differentiability points, $\dot\gamma(t)=\nabla d_{\mathcal S}$, and
\begin{align}
&|\nabla F-\nabla d_{\mathcal S}|^2=|\nabla F|^2-1+2(1-\langle\nabla F,\nabla d_{\mathcal S}\rangle). \nonumber
\end{align}
Lemma~\ref{one:lem:fixed-F}, \eqref{eq:gradient-limits}, therefore gives
\begin{align}
&\lim_{R\to\infty}\operatorname*{ess\,sup}_{d_{\mathcal S}\ge R} |\nabla F-\nabla d_{\mathcal S}|=0. \nonumber
\end{align}
For $0<a<b<\infty$,
\begin{align}
&|\nabla f_{\mathcal S}-F\nabla F|\le d_{\mathcal S}|\nabla d_{\mathcal S}-\nabla F|+|d_{\mathcal S}-F|\,|\nabla F|, \nonumber
\end{align}
and thus
\begin{align}
&\lim_{r\to\infty}\frac1r \operatorname*{ess\,sup}_{\{ar<F<br\}} |\nabla f_{\mathcal S}-F\nabla F|=0. \label{ext:eq:soul-gradient-comparison}
\end{align}
Fix $L=L(M,g)$ and $0\le\beta\in C_c^\infty((a,b))$. For every smooth compactly supported symmetric tensor $A$,
\begin{align}
&\langle g\,dV_g,A\rangle=\int_M\tr_gA\,dV_g. \nonumber
\end{align}
Lemma~\ref{lem:first-curvature-contraction}, \eqref{eq:first-curvature-contraction}, \eqref{pt:eq:E1-properties}, gives
\begin{align}
&\langle g\,dV_g-\Hess f_{\mathcal S},\beta(F/r)E_1\rangle =\frac{n-2}{2}\int_M\beta(F/r)\Sc_g\,dV_g +\int_M E_1\bigl(\nabla(\beta(F/r)),\nabla f_{\mathcal S}\bigr)\,dV_g. \nonumber
\end{align}
Lemma~\ref{lem:scalar-radial-tests}, \eqref{eq:common-scalar-F-tests}, gives
\begin{align}
&\lim_{r\to\infty}\frac{(n-2)r^{2-n}}2\int_M\beta(F/r)\Sc_g\,dV_g =\frac{(n-2)^2L}2\int_0^\infty\beta(t)t^{n-3}\,dt. \nonumber
\end{align}
Since $0\le E_1\le\Sc_g g/2$ and $|\nabla F|<2$,
\begin{align}
&r^{2-n}\left|\int_M E_1\bigl(\nabla(\beta(F/r)),\nabla f_{\mathcal S}-F\nabla F\bigr)\,dV_g\right| \nonumber\\
&{}\le\|\beta'\|_\infty\left(r^{-1}\operatorname*{ess\,sup}_{\{ar<F<br\}}|\nabla f_{\mathcal S}-F\nabla F|\right)\left(r^{2-n}\int_{\{ar<F<br\}}\Sc_g\,dV_g\right). \nonumber
\end{align}
By \eqref{ext:eq:soul-gradient-comparison} and Lemma~\ref{lem:scalar-radial-tests}, \eqref{eq:common-scalar-F-tests}, with a smooth test equal to one on $[a,b]$,
\begin{align}
&\sup_{r\ge r_0}r^{2-n}\int_{\{ar<F<br\}}\Sc_g\,dV_g<\infty, \nonumber\\
&\lim_{r\to\infty}r^{2-n}\left|\int_M E_1(\nabla(\beta(F/r)),\nabla f_{\mathcal S}-F\nabla F)\,dV_g\right|=0. \nonumber
\end{align}
Moreover,
\begin{align}
&\int_M E_1\bigl(\nabla(\beta(F/r)),F\nabla F\bigr)\,dV_g =\int_M\beta'(F/r)\frac Fr |\nabla F|^2\left(\frac{\Sc_g}{2}-\Ric_g(\nu,\nu)\right)\,dV_g. \nonumber
\end{align}
Lemma~\ref{lem:scalar-radial-tests}, \eqref{eq:common-radial-F-tests}, gives
\begin{align}
&\lim_{r\to\infty}r^{2-n}\int_M\beta'(F/r)\frac Fr|\nabla F|^2\Ric_g(\nu,\nu)\,dV_g=0. \nonumber
\end{align}
The error from replacing $|\nabla F|^2$ by one in the scalar term is at most
\begin{align}
&\frac{b\|\beta'\|_\infty}{2}\sup_{\{ar<F<br\}}||\nabla F|^2-1|\left(r^{2-n}\int_{\{ar<F<br\}}\Sc_g\,dV_g\right), \nonumber
\end{align}
and tends to zero by Lemma~\ref{one:lem:fixed-F}, \eqref{eq:gradient-limits}. Lemma~\ref{lem:scalar-radial-tests}, \eqref{eq:common-scalar-F-tests}, applied to $t\beta'(t)$, yields
\begin{align}
&\lim_{r\to\infty}r^{2-n}\int_M E_1\bigl(\nabla(\beta(F/r)),\nabla f_{\mathcal S}\bigr)\,dV_g =\frac{(n-2)L}2\int_0^\infty\beta'(t)t^{n-2}\,dt. \nonumber
\end{align}
Since $[\beta(t)t^{n-2}]_0^\infty=0$, integration by parts gives
\begin{align}
&\lim_{r\to\infty}r^{2-n}\langle g\,dV_g-\Hess f_{\mathcal S},\beta(F/r)E_1\rangle \nonumber\\
&{}=\frac{(n-2)L}2\left((n-2)\int_0^\infty\beta(t)t^{n-3}\,dt+\int_0^\infty\beta'(t)t^{n-2}\,dt\right)=0. \nonumber
\end{align}
Lemma~\ref{soul:lem:hessian}, \eqref{soul:eq:hessian}, and $E_1\ge0$ give
\begin{align}
&0\le\int_M\beta(F/r)\langle E_1,g_{\mathcal H}\rangle_g\,dV_g \le\langle g\,dV_g-\Hess f_{\mathcal S},\beta(F/r)E_1\rangle. \nonumber
\end{align}
Choose $\beta\ge0$ with $\beta=1$ on a fixed compact interval containing $(F/r)(\Omega_r\setminus\Omega_{r/2})$ for all sufficiently large $r$. Then
\begin{align}
&\lim_{r\to\infty}r^{2-n} \int_{\Omega_r\setminus\Omega_{r/2}}\langle E_1,g_{\mathcal H}\rangle_g\,dV_g=0. \nonumber
\end{align}
For every $\varepsilon>0$, choose $R_\varepsilon$ so that
\begin{align}
&\int_{\Omega_t\setminus\Omega_{t/2}}\langle E_1,g_{\mathcal H}\rangle_g\,dV_g\le\varepsilon t^{n-2}\qquad(t\ge R_\varepsilon). \nonumber
\end{align}
Summing at $t=r, r/2, r/4,\ldots$ gives
\begin{align}
&\int_{\Omega_r}\langle E_1,g_{\mathcal H}\rangle_g\,dV_g \le\int_{\Omega_{2R_\varepsilon}}\langle E_1,g_{\mathcal H}\rangle_g\,dV_g +\frac{\varepsilon r^{n-2}}{1-2^{2-n}}. \nonumber
\end{align}
Compactness of $\overline\Omega_{2R_\varepsilon}$ implies
\begin{align}
&0\le\limsup_{r\to\infty}r^{2-n}\int_{\Omega_r}\langle E_1,g_{\mathcal H}\rangle_g\,dV_g\le\frac{\varepsilon}{1-2^{2-n}}\qquad(\varepsilon>0), \nonumber\\
&\lim_{r\to\infty}r^{2-n}\int_{\Omega_r}\langle E_1,g_{\mathcal H}\rangle_g\,dV_g=0. \nonumber
\end{align}
\end{proof}

\begin{corollary}\label{ext:cor:soul-traces}
Set
\begin{align}
&\tr_{\mathcal V}\Ric_g =\sum_{a=1}^{n-s}\Ric_g(U_a,U_a), \nonumber
\end{align}
where $(U_a)$ is any vertical orthonormal basis. 
Then
\begin{align}
&\lim_{r\to\infty}r^{2-n}\int_{\Omega_r} \tr_{\mathcal V}\Ric_g\,dV_g =\frac{2-s}{2}L(M,g) \label{low:eq:vertical-trace-limit}.
\end{align}
\end{corollary}

\begin{proof}
Pointwise,
\begin{align}
&\langle E_1,g_{\mathcal H}\rangle_g =\frac{s}{2}\Sc_g-\tr_{\mathcal H}\Ric_g =\frac{s-2}{2}\Sc_g+\tr_{\mathcal V}\Ric_g \label{ext:eq:horizontal-trace}
\end{align}
Lemma~\ref{ext:lem:horizontal-vanishing}, \eqref{ext:eq:soul-ball-inclusions}, and $\Sc_g\ge0$ give, including $s=0$,
\begin{align}
&\int_{B_p(r)}\Sc_g\,dV_g\le\int_{\Omega_r}\Sc_g\,dV_g\le\int_{B_p(r+D)}\Sc_g\,dV_g, \nonumber
\end{align}
whence
\begin{align}
&\lim_{r\to\infty}r^{2-n}\int_{\Omega_r}\Sc_g\,dV_g=L(M,g). \nonumber
\end{align}
Integrating \eqref{ext:eq:horizontal-trace} and using Lemma~\ref{ext:lem:horizontal-vanishing}, \eqref{ext:eq:horizontal-vanishing},
\begin{align}
&0\le\lim_{r\to\infty}r^{2-n}\int_{\Omega_r}\tr_{\mathcal V}\Ric_g\,dV_g=\frac{2-s}{2}L(M,g). \nonumber
\end{align}
\end{proof}

\begin{lemma}\label{low:lem:Euler-topology}
Assume $L(M,g)>0$. Then $s\le2$. Set
$b_i=\dim_{\mathbb F_2}H_i(\mathcal S;\mathbb F_2)$ for
$i\in\mathbb Z_{\ge0}$. Then
\begin{align}
&b_0=1,\qquad b_i=0\ (i>2),\qquad b_2\in\{0,1\},\qquad \chi(M)=1-b_1+b_2\le2, \label{low:eq:Betti-data}\\
&2(1-\AVR(M^n,g))-(\chi(M)-\AVR(M^n,g)) \ge0. \label{low:eq:Euler-volume-comparison}
\end{align}
\end{lemma}

\begin{proof}
Corollary~\ref{ext:cor:soul-traces} and $L(M,g)>0$ give
\begin{align}
&s\le2. \nonumber
\end{align}

Poincar\'e duality over $\mathbb F_2$ \cite[Theorem~3.30]{Hatcher} gives
\begin{align}
&b_0=1,\qquad b_i=0\ (i>2),\qquad b_2=\begin{cases}1,&s=2,\\0,&s<2,\end{cases}\qquad\chi(M)=1-b_1+b_2\le2. \nonumber
\end{align}

Lemma~\ref{soul:lem:hessian}, \eqref{tp:eq:tube-volume}, gives
\begin{align}
&b_2\ne0\Longrightarrow s=2\Longrightarrow b_2\AVR(M^n,g)=0. \label{b2AVR=0}
\end{align}

With $0\le \AVR(M^n,g)\le1$ and \ref{b2AVR=0},
\begin{align}
&2(1-\AVR(M^n,g))-(\chi(M)-\AVR(M^n,g)) \nonumber\\
&{}=1+b_1-b_2-\AVR(M^n,g) \nonumber\\
&{}=b_1+(1-b_2)(1-\AVR(M^n,g))\ge0. \nonumber
\end{align}
\end{proof}

\section{Positive-dimensional souls and rigidity}
\label{ext:sec:higher}

\begin{lemma}\label{ext:lem:surface-soul-traces}
Suppose $s=2$. Let $(X_1,X_2)$ and $(U_a)_{a=1}^{n-s}$
be local orthonormal frames of $\mathcal H$ and $\mathcal V$.
Then
\begin{align}
&\lim_{r\to\infty}r^{2-n}\int_{\Omega_r} \left(\sum_{1\le a<b\le n-s}K_g(U_a,U_b) +\sum_{i=1}^2\sum_{a=1}^{n-s}K_g(X_i,U_a)\right)\,dV_g=0. \label{ext:eq:vertical-mixed-vanishing}
\end{align}
\end{lemma}
\begin{proof}
\begin{align}
&\tr_{\mathcal V}\Ric_g =\sum_{a=1}^{n-s}\Ric_g(U_a,U_a) =2\sum_{1\le a<b\le n-s}K_g(U_a,U_b) +\sum_{i=1}^2\sum_{a=1}^{n-s}K_g(X_i,U_a). \nonumber
\end{align}
By Corollary~\ref{ext:cor:soul-traces}, \eqref{low:eq:vertical-trace-limit}, with $s=2$,
\begin{align}
&0\le r^{2-n}\int_{\Omega_r}\left(\sum_{a<b}K_g(U_a,U_b)+\sum_{i,a}K_g(X_i,U_a)\right)\,dV_g \nonumber\\
&{}\le r^{2-n}\int_{\Omega_r}\tr_{\mathcal V}\Ric_g\,dV_g\longrightarrow0\qquad(r\to\infty). \nonumber
\end{align}
\end{proof}

\begin{lemma}
\label{ext:cor:possible-souls}
If $L(M,g)>0$, then every soul $\mathcal S\subset M$ is either a point, a surface diffeomorphic to $\mathbb S^2$, or a surface diffeomorphic to $\mathbb {RP}^2$.
\end{lemma}

\begin{proof}
Let $\mathcal S$ be a soul and set $s=\dim\mathcal S$. Corollary~\ref{ext:cor:soul-traces}, \eqref{low:eq:vertical-trace-limit}, and $L(M,g)>0$ give
\begin{align}
&s\le2. \nonumber
\end{align}

We first show that $\pi_1(M)$ is finite. Suppose instead that $\pi_1(M)$ is infinite, and let
\begin{align}
&\pi:(\widetilde M,\widetilde g)\longrightarrow(M,g) \nonumber
\end{align}
be the universal Riemannian covering. The covering is a local isometry, hence $\sec_{\widetilde g}\ge0$; completeness of $g$ implies completeness of $\widetilde g$, and $\widetilde M$ is noncompact because it maps surjectively onto the noncompact manifold $M$. Fix $p\in M$. Since $\pi_1(M)$ is infinite, the fiber $\pi^{-1}(p)$ is infinite. For each $j\ge1$, choose distinct points
\begin{align}
&p_1,\ldots,p_j\in\pi^{-1}(p),\qquad D_j=\max_{1\le i\le j}d_{\widetilde g}(p_1,p_i). \nonumber
\end{align}
For every $x\in B_p(r)$, lift a minimizing geodesic from $p$ to $x$ starting at each $p_i$. The lifted endpoints are distinct, by uniqueness of path lifting applied to the reversed geodesic, and all lie in $B_{p_1}^{\widetilde M}(r+D_j)$. Hence
\begin{align}
&\#\bigl(\pi^{-1}(x)\cap B_{p_1}^{\widetilde M}(r+D_j)\bigr)\ge j\qquad(x\in B_p(r)). \nonumber
\end{align}
Since $\Sc_{\widetilde g}=\Sc_g\circ\pi$ and $\Sc_g\ge0$, the area formula gives
\begin{align}
&j\int_{B_p(r)}\Sc_g\,dV_g\le\int_{B_{p_1}^{\widetilde M}(r+D_j)}\Sc_{\widetilde g}\,dV_{\widetilde g}. \nonumber
\end{align}
For fixed $j$, Theorem~\ref{thm:direct-main}, \eqref{eq:nt-ball-conclusion}, therefore yields
\begin{align}
&jL(M,g)\le L(\widetilde M,\widetilde g)<\infty. \nonumber
\end{align}
Since this holds for every $j\ge1$, it follows that
\begin{align}
&L(M,g)=0, \nonumber
\end{align}
contrary to the hypothesis. Thus
\begin{align}
&\#\pi_1(M)<\infty. \nonumber
\end{align}

The soul retraction is a homotopy equivalence, so
\begin{align}
&\pi_1(M)\cong\pi_1(\mathcal S). \nonumber
\end{align}
If $s=1$, then the compact connected one-dimensional manifold $\mathcal S$ is diffeomorphic to $\mathbb S^1$, whose fundamental group is infinite, a contradiction. Hence $s\ne1$. Since $s\le2$, either $s=0$ or $s=2$. If $s=0$, then $\mathcal S$ is a point. If $s=2$, then $\mathcal S$ is a closed connected surface with finite fundamental group, and the compact-surface classification gives
\begin{align}
&\mathcal S\text{ is diffeomorphic to }\mathbb S^2\text{ or }\mathbb {RP}^2. \nonumber
\end{align}
\end{proof}

%\subsection{Riemannian submersion identities}
\begin{lemma}\label{nine:lem:oneill}
Let $n\ge1$, $0\le s\le n$, and let $P:(M^n,g)\to(B^s,h)$ be a smooth surjection of smooth Riemannian manifolds such that $dP$ maps $(\ker dP)^\perp$ isometrically onto $TB$. Let $\mathcal V=\ker dP$, $\mathcal H=\mathcal V^\perp$, and use superscripts $\mathcal V,\mathcal H$ for orthogonal projection. For horizontal fields $X,Y$ and a vertical field $U$, define
\begin{align}
&\mathcal A_XY=(\nabla_XY)^{\mathcal V},\qquad \mathcal A_XU=(\nabla_XU)^{\mathcal H}. \nonumber
\end{align}
For horizontal lifts of fields on $B$,
\begin{align}
&[X,Y]^{\mathcal V}=2\mathcal A_XY,\qquad \mathcal A_YX=-\mathcal A_XY,\qquad \langle\mathcal A_XU,Y\rangle=-\langle U,\mathcal A_XY\rangle. \label{nine:eq:oneill-bracket}
\end{align}
Moreover, for orthonormal basic horizontal fields $X_i,X_j$,
\begin{align}
&K_M(X_i,X_j)=K_B(dPX_i,dPX_j)-3|\mathcal A_{X_i}X_j|^2. \label{nine:eq:oneill-horizontal}
\end{align}
\end{lemma}
\begin{proof}
These are the classical curvature identities for Riemannian submersions due to O'Neill; see \cite[Corollary~1, p.~465]{ONeill1966} for the horizontal sectional-curvature identity. For the tensor, sectional, Ricci, and scalar-curvature formulas in a standard reference, see \cite[pp.~239--244, especially Proposition~9.24, Theorem~9.28, Corollary~9.29, Proposition~9.36, and Corollary~9.37]{Besse1987}.
\end{proof}

%\subsection{Two-sphere souls and equality}
\label{ext:subsec:sphere-soul}
Suppose $\mathcal S$ is diffeomorphic to $\mathbb S^2$, so $s=2$.
Put $D=\operatorname{diam}_h\mathcal S$ and fix
\begin{align}
&y_0\in\mathcal S,\qquad N=P^{-1}(y_0). \nonumber
\end{align}
Here $P$ is the smooth Riemannian submersion in Theorem~\ref{ext:thm:retraction}.
The normal connection is
$\nabla_X^\perp v=(\nabla_Xv)^{\perp}$ for $X\in T\mathcal S$
and a normal field $v$, where $\perp$ denotes projection onto
$\nu\mathcal S$. For a piecewise smooth path
$\gamma:[0,1]\to\mathcal S$, denote normal parallel transport by $A_\gamma$
and transport between fibers by horizontal lifting by $\Phi_\gamma$.
The normal connection has trivial holonomy at $y_0$ if
\begin{align}
&A_\gamma=\Id_{\nu_{y_0}\mathcal S} \quad\text{for every piecewise smooth loop $\gamma$ based at $y_0$}. \nonumber
\end{align}
Each $A_\gamma$ is orthogonal since the connection preserves the metric.

\begin{externaltheorem}
\label{ext:thm:transport-exponential}
Let $(M,g)$ be complete with $\sec_g\ge0$, let $\mathcal S$ be a soul,
and let $P:M\to\mathcal S$ be the Sharafutdinov retraction. For every
piecewise smooth path $\gamma:[0,1]\to\mathcal S$, let $A_\gamma$ be
normal parallel transport and let $\Phi_\gamma$ be transport between
fibers by horizontal lifting. Then
\begin{align}
&\Phi_\gamma(\exp_{\gamma(0)}v) =\exp_{\gamma(1)}(A_\gamma v), \qquad v\in\nu_{\gamma(0)}\mathcal S. \label{ext:eq:transport-exponential}
\end{align}
No injectivity of the normal exponential map is required.
\end{externaltheorem}
\noindent\emph{Source.} This is \cite[Theorem~4.1(c)]{Wilking2006}
applied to the soul.
The leaf hypothesis there holds: horizontal curves starting on
$\mathcal S$ remain on it by Theorem~\ref{ext:thm:retraction}, and any two of its points can
be joined horizontally. Moreover, $\mathcal S$ is smooth and
$P|_{\mathcal S}$ is the identity. The statement
there is for broken geodesics. Piecewise smooth paths follow by
piecewise $C^1$ approximation and continuous dependence for the
two transport equations. The horizontal lifts exist on the entire
path: their speeds equal the base speeds, so they remain in a
compact ambient ball.

\begin{lemma}
\label{soul:lem:horizontal-transport}
Under the two-sphere-soul hypotheses of this subsection, let
$\Phi_\gamma:\mathcal N_{\gamma(0)}\to \mathcal N_{\gamma(1)}$ be horizontal transport
along a piecewise smooth path $\gamma:[0,1]\to\mathcal S$.
Put
\begin{align}
&\beta=\sqrt{\tfrac13\max_{\mathcal S}K_h}. \label{soul:eq:transport-constant}
\end{align}
For every piecewise smooth path $\gamma$ in $\mathcal S$ of length
$\ell$, horizontal transport is a smooth diffeomorphism between its
endpoint fibers and satisfies
\begin{align}
&e^{-\beta\ell}|U|\le|d\Phi_\gamma(U)|\le e^{\beta\ell}|U| \qquad(U\text{ tangent to the initial fiber}). \label{soul:eq:transport-differential}
\end{align}
If $\mathcal H$ is integrable, these fiber maps are Riemannian
isometries. Here integrability means
$[X,Y]^{\mathcal V}=0$ for all local horizontal fields $X,Y$.
\end{lemma}
\begin{proof}
For a horizontal lift $\widetilde\gamma$ of a finite path $\gamma$,
\begin{align}
&|\dot{\widetilde\gamma}|=|\dot\gamma|,\qquad \widetilde\gamma([0,1])\subset\overline B_{\widetilde\gamma(0)}(\operatorname{Length}(\gamma)),\qquad\Phi_{\gamma^{-1}}=\Phi_\gamma^{-1}. \nonumber
\end{align}
Compactness of the ball and smooth dependence in the lifting equation give a smooth global diffeomorphism. Let $\sigma:\mathbb R\to M$ lift a complete unit-speed base geodesic; put $X=\dot\sigma$. Then
\begin{align}
&(\nabla_XX)^{\mathcal H}=0,\qquad(\nabla_XX)^{\mathcal V}=\mathcal A_XX=0,\qquad\nabla_XX=0. \nonumber
\end{align}
Varying the initial point within its fiber gives vertical Jacobi fields spanning $\mathcal V_{\sigma(t)}$ for every $t$. With $D_t^{\mathcal V}$ the vertical covariant derivative, invertibility of transport defines $B$ by
\begin{align}
&D_t^{\mathcal V}J=BJ. \nonumber
\end{align}
Commutation of variation parameters gives $\nabla_XJ=\nabla_JX$ and, for vertical $U,V$,
\begin{align}
&\langle BU,V\rangle =\langle\nabla_UX,V\rangle =-\langle X,\nabla_UV\rangle \nonumber
\end{align}
Thus
\begin{align}
&\langle BU,V\rangle-\langle BV,U\rangle=-\langle X,[U,V]\rangle=0,\qquad B=B^*. \nonumber
\end{align}
Set $A_tU=(\nabla_XU)^{\mathcal H}$ and $\mathcal R_tU=(\Rm(U,X)X)^{\mathcal V}$; the adjoint uses the induced metrics. Lemma~\ref{nine:lem:oneill}, \eqref{nine:eq:oneill-bracket}, gives
\begin{align}
&(\nabla_X\nabla_XJ)^{\mathcal V} =(B'+B^2-A_t^*A_t)J, \nonumber
\end{align}
where $B'$ is the derivative for $D_t^{\mathcal V}$. The Jacobi equation yields
\begin{align}
&B'+B^2=-\mathcal R_t+A_t^*A_t. \label{soul:eq:horizontal-Riccati}
\end{align}
For unit horizontal $Y\perp X$, $A_tU\in\mathbb RY$. Since $\mathcal R_t\ge0$, Lemma~\ref{nine:lem:oneill}, \eqref{nine:eq:oneill-horizontal}, gives
\begin{align}
&\|A_t\|_{\mathrm{op}}^2=|\mathcal A_XY|^2 \le\tfrac13K_h(dPX,dPY)\le\beta^2. \nonumber
\end{align}
Hence, for a vertically parallel unit vector $u(t)$,
\begin{align}
&B'+B^2\le\beta^2\Id,\qquad b(t)=\langle B(t)u(t),u(t)\rangle,\qquad b'\le\beta^2-|Bu|^2\le\beta^2-b^2. \nonumber
\end{align}
If $b(t_0)>\beta$, then $b(t)\ge b(t_0)>\beta$ for $t\le t_0$, and
\begin{align}
&t_0-t\le\int_{b(t_0)}^{b(t)}\frac{ds}{s^2-\beta^2} \le\int_{b(t_0)}^\infty\frac{ds}{s^2-\beta^2}<\infty. \nonumber
\end{align}
contradicts $t\to-\infty$. Applying the same argument to $-b(-t)$ excludes $b<-\beta$, also when $\beta=0$. Thus
\begin{align}
&{}-\beta\Id\le B(t)\le\beta\Id\qquad(t\in\mathbb R). \nonumber
\end{align}
Every unit horizontal direction extends to such a complete geodesic. Along a lift of any path,
\begin{align}
&\left|\frac{d}{dt}\log|J(t)|\right| \le\beta|\dot\gamma(t)|. \nonumber
\end{align}
so
\begin{align}
&{}-\beta\operatorname{Length}(\gamma)\le\log\frac{|d\Phi_\gamma(U)|}{|U|}\le\beta\operatorname{Length}(\gamma)\qquad(U\ne0). \nonumber
\end{align}
If $\mathcal H$ is integrable, Lemma~\ref{nine:lem:oneill}, \eqref{nine:eq:oneill-bracket}, and \eqref{soul:eq:horizontal-Riccati} give
\begin{align}
&\mathcal A=0\quad\Longrightarrow\quad B'+B^2\le0\quad\Longrightarrow\quad B=0, \nonumber\\
&\frac d{dt}|J(t)|^2=0,\qquad \Phi_\gamma^*k_{\gamma(1)}=k_{\gamma(0)}. \nonumber
\end{align}
\end{proof}

The volume-vanishing conclusion of Lemma~\ref{ext:lem:nontrivial-holonomy}, equation~\eqref{ext:eq:holonomy-vanishing}, also follows, for a two-sphere soul, from Tapp's volume-growth theorem \cite[Main Theorem, pp.~4--5]{Tapp1999}, which extends the work of Schroeder and Strake \cite{SchroederStrake1990}. The bounded vertical Jacobi fields used in such volume estimates are treated in \cite[Proposition~4.1, p.~5]{Tapp1999}. We retain the proof below; its scalar-integral conclusion also uses Lemma~\ref{ext:lem:surface-soul-traces}, equation~\eqref{ext:eq:vertical-mixed-vanishing}.

\begin{lemma}
\label{ext:lem:nontrivial-holonomy}
Retain the two-sphere-soul hypotheses of this subsection and
$\beta$ from Lemma~\ref{soul:lem:horizontal-transport},
equation~\eqref{soul:eq:transport-constant}. Suppose that some
piecewise smooth loop $\gamma$ based at $y_0$ has
$A_\gamma\ne\Id$.
Then
\begin{align}
&\lim_{r\to\infty}r^{-(n-s)}\Vol(\Omega_r)=0,\qquad \lim_{r\to\infty}r^{-(n-s)}\int_{\Omega_r}\Sc_g\,dV_g=0, \qquad L(M,g)=0. \label{ext:eq:holonomy-vanishing}
\end{align}
\end{lemma}
\begin{proof}
Define
\begin{align}
&R^\perp(X,Y)v =\nabla_X^\perp\nabla_Y^\perp v -\nabla_Y^\perp\nabla_X^\perp v -\nabla_{[X,Y]}^\perp v. \nonumber
\end{align}
The metric connection gives $(R^\perp(X,Y))^*=-R^\perp(X,Y)$. For a smooth based loop homotopy $\gamma:[0,1]^2\to\mathcal S$, let $V(s,t)$ be parallel transport of a fixed $v_0$ along $t\mapsto\gamma(s,t)$, and let $\mathcal P^s_{t,1}$ denote transport from $t$ to $1$. Then
\begin{align}
&\nabla_t^\perp V=0,\qquad \nabla_s^\perp V(s,0)=0, \nonumber
\end{align}
and commutation gives
\begin{align}
&\nabla_t^\perp\nabla_s^\perp V =R^\perp(\partial_t\gamma,\partial_s\gamma)V,\qquad \frac{d}{ds}V(s,1) =\int_0^1\mathcal P^s_{t,1} R^\perp(\partial_t\gamma,\partial_s\gamma)V(s,t)\,dt. \nonumber
\end{align}
Because $\pi_1(\mathcal S)=0$, smooth approximation relative to the base point and subdivision for piecewise smooth loops imply
\begin{align}
&R^\perp=0\quad\Longrightarrow\quad\frac d{ds}V(s,1)=0\quad\Longrightarrow\quad A_\gamma=\Id\quad\text{for every based loop }\gamma. \nonumber
\end{align}
The hypothesis therefore gives $y_*\in\mathcal S$ and orthonormal $X,Y\in T_{y_*}\mathcal S$ with
\begin{align}
&C=R^\perp(X,Y):\nu_{y_*}\mathcal S\to\nu_{y_*}\mathcal S, \qquad C\ne0. \nonumber
\end{align}
In particular $n-s\ge2$. Let $E:\nu\mathcal S\to M$ be the normal exponential map. Theorem~\ref{ext:thm:transport-exponential}, \eqref{ext:eq:transport-exponential}, gives $dE(\overline X)=X^H\circ E$ for the connection lift
\begin{align}
&\overline X=X^i\bigl(\partial_i-(\Gamma_i v)^a\partial_{v^a}\bigr). \nonumber
\end{align}
Direct differentiation gives
\begin{align}
&[\overline X,\overline Y] =\overline{[X,Y]}-(R^\perp(X,Y)v)^{\mathrm{vert}}, \nonumber
\end{align}
Here $\mathrm{vert}$ is the vertical tangent vector in $\nu\mathcal S$. Brackets of $E$-related fields and Lemma~\ref{nine:lem:oneill}, \eqref{nine:eq:oneill-bracket}, give, for every $v\in\nu_{y_*}\mathcal S$,
\begin{align}
&{}-d(\exp_{y_*})_v(Cv) =2\mathcal A_{X^H}Y^H\big|_{\exp_{y_*}v},\qquad |d(\exp_{y_*})_v(Cv)|\le2\beta. \label{soul:eq:bounded-angular-Jacobi}
\end{align}
where $X^H,Y^H$ are unit horizontal lifts; injectivity of $E$ is unnecessary. For $|\xi|=1$ and $C\xi\ne0$, choose an orthonormal basis $Z_1,\ldots,Z_{n-s-1}$ of $\xi^\perp$ with $Z_1=C\xi/|C\xi|$. The angular Jacobi fields are
\begin{align}
&J_a(t)=d(\exp_{y_*})_{t\xi}(tZ_a),\qquad 1\le a\le n-s-1. \nonumber
\end{align}
Lemma~\ref{soul:lem:hessian}, \eqref{soul:eq:normal-index}, and \eqref{soul:eq:bounded-angular-Jacobi} imply, before the cut time,
\begin{align}
&\frac d{dt}|J_a|^2\le\frac2t|J_a|^2,\qquad \lim_{t\downarrow0}\frac{|J_a(t)|}{t}=1,\qquad |J_a(t)|\le t, \nonumber\\
&|J_1(t)|=\frac{|d(\exp_{y_*})_{t\xi}(tC\xi)|}{|C\xi|}\le\frac{2\beta}{|C\xi|}. \nonumber
\end{align}
The horizontal normal Jacobi fields are parallel and orthonormal; the determinant bound gives
\begin{align}
&0\le J(t,y_*,\xi) \le\frac{2\beta}{|C\xi|}t^{n-s-2} \qquad(0<t<\tau_{\mathcal S}(y_*,\xi)). \label{soul:eq:normal-Jacobian-loss}
\end{align}
Extend $J$ by zero after the cut time. Since
\begin{align}
&C\ne0\quad\Longrightarrow\quad\ker C\subsetneq\nu_{y_*}\mathcal S,\qquad\omega_{y_*}(\ker C\cap S(\nu_{y_*}\mathcal S))=0, \nonumber
\end{align}
\eqref{soul:eq:normal-Jacobian-loss} for infinite cut time, and zero extension otherwise, give almost everywhere
\begin{align}
&\lim_{t\to\infty}\frac{J(t,y_*,\xi)}{t^{n-s-1}}=0, \qquad 0\le J(t,y_*,\xi)\le t^{n-s-1}. \nonumber
\end{align}
Polar integration, $t=ru$, and dominated convergence give
\begin{align}
&r^{-(n-s)}\Vol_{n-s}B_{y_*}^{\mathcal N_{y_*}}(r)=\int_{S(\nu_{y_*}\mathcal S)}\int_0^1u^{n-s-1}\frac{J(ru,y_*,\xi)}{(ru)^{n-s-1}}\,du\,d\omega_{y_*}(\xi), \nonumber\\
&0\le u^{n-s-1}\frac{J(ru,y_*,\xi)}{(ru)^{n-s-1}}\le u^{n-s-1}. \nonumber
\end{align}
Hence
\begin{align}
&\lim_{r\to\infty}r^{-(n-s)}\Vol_{n-s} B_{y_*}^{\mathcal N_{y_*}}(r)=0. \nonumber
\end{align}
For a minimizing base geodesic from $y_*$ to $y$, its length is at most $D$. Horizontal transport preserves $r_{\mathcal S}$. Lemma~\ref{soul:lem:horizontal-transport}, \eqref{soul:eq:transport-differential}, gives vertical Jacobian at most $e^{(n-s)\beta D}$; Lemma~\ref{soul:lem:hessian}, \eqref{four:eq:fiber-distance}, and the area formula yield
\begin{align}
&\Vol_{n-s} B_y^{\mathcal N_y}(r) \le e^{(n-s)\beta D}\Vol_{n-s} B_{y_*}^{\mathcal N_{y_*}}(r) \qquad(y\in\mathcal S). \nonumber
\end{align}
Coarea, whose Jacobian is one, gives
\begin{align}
&0\le r^{-(n-s)}\Vol(\Omega_r)\le\Vol_h(\mathcal S)e^{(n-s)\beta D}r^{-(n-s)}\Vol_{n-s} B_{y_*}^{\mathcal N_{y_*}}(r)\longrightarrow0\quad(r\to\infty). \nonumber
\end{align}
Since $\Sc_h$ is bounded,
\begin{align}
&\lim_{r\to\infty}r^{-(n-s)} \int_{\Omega_r}\Sc_h\circ P\,dV=0. \nonumber
\end{align}
For local horizontal and vertical orthonormal frames, Lemma~\ref{nine:lem:oneill}, \eqref{nine:eq:oneill-horizontal}, gives
\begin{align}
&\Sc_g\le\Sc_h\circ P +2\sum_{1\le a<b\le n-s}K_g(U_a,U_b) +2\sum_{i=1}^2\sum_{a=1}^{n-s}K_g(X_i,U_a). \nonumber
\end{align}
Lemma~\ref{ext:lem:surface-soul-traces}, \eqref{ext:eq:vertical-mixed-vanishing}, and Lemma~\ref{ext:lem:horizontal-vanishing}, \eqref{ext:eq:soul-ball-inclusions}, imply
\begin{align}
&0\le L(M,g)=\lim_{r\to\infty}r^{-(n-s)}\int_{\Omega_r}\Sc_g\,dV_g\le0. \nonumber
\end{align}
\end{proof}

\begin{theorem}
\label{ext:thm:positive-soul}
Let $(M^n,g)$ be a smooth, complete, connected, noncompact
Riemannian manifold without boundary, with $n\ge3$ and $\sec_g\ge0$.
Suppose $M$ has a soul of positive dimension. Then
\begin{align}
&L(M,g)\le8\pi\omega_{n-2}. \label{ext:eq:positive-soul-sharp}
\end{align}
Equality holds if and only if there is a smooth metric $h$ on
$\mathbb S^2$ such that
\begin{align}
&(M,g)\cong(\mathbb S^2,h)\times(\R^{n-2},g_{\mathrm{Euc}}), \qquad K_h\ge0. \label{ext:eq:positive-soul-equality}
\end{align}
\end{theorem}

\begin{proof}
For $L(M,g)=0$, \eqref{ext:eq:positive-soul-sharp} is strict.
Assume $L(M,g)>0$ and choose a positive-dimensional soul
$\mathcal S$, with $h=g|_{T\mathcal S}$.
Corollary~\ref{ext:cor:possible-souls} and total geodesy of $\mathcal S$ give
\begin{align}
&\mathcal S\text{ is diffeomorphic to }\mathbb S^2\text{ or }\mathbb {RP}^2,\qquad s=2,\qquad K_h\ge0. \nonumber
\end{align}
For local horizontal and vertical orthonormal frames,
Lemma~\ref{nine:lem:oneill}, equation~\eqref{nine:eq:oneill-horizontal}, gives
\begin{align}
&\Sc_g\le\Sc_h\circ P+2\sum_{1\le a<b\le n-2}K_g(U_a,U_b)+2\sum_{i=1}^2\sum_{a=1}^{n-2}K_g(X_i,U_a). \nonumber
\end{align}
Lemma~\ref{ext:lem:surface-soul-traces},
equation~\eqref{ext:eq:vertical-mixed-vanishing}, the inclusions in
\eqref{ext:eq:soul-ball-inclusions}, and the weighted tube-volume
estimate \eqref{tp:eq:tube-volume}, applied to $a=\Sc_h\ge0$, yield
\begin{align}
&L(M,g)=\lim_{r\to\infty}r^{2-n}\int_{\Omega_r}\Sc_g\,dV_g\le\limsup_{r\to\infty}r^{2-n}\int_{\Omega_r}\Sc_h\circ P\,dV_g \nonumber\\
&\hspace{27pt}\le\omega_{n-2}\int_{\mathcal S}\Sc_h\,dA_h=4\pi\omega_{n-2}\chi(\mathcal S)\le8\pi\omega_{n-2}, \nonumber
\end{align}
where Gauss--Bonnet \cite[Section~4-5]{doCarmo} gives the equality.
This proves \eqref{ext:eq:positive-soul-sharp}. If equality holds,
then $\chi(\mathcal S)=2$, so $\mathcal S$ is diffeomorphic to
$\mathbb S^2$; a projective-plane soul gives at most
$4\pi\omega_{n-2}$.

Assume henceforth that $L(M,g)=8\pi\omega_{n-2}$.
Fix $y_0\in\mathcal S$ and give $N=P^{-1}(y_0)$ its induced metric $k$.
Lemma~\ref{ext:lem:nontrivial-holonomy},
\eqref{ext:eq:holonomy-vanishing}, implies
\begin{align}
&A_\gamma=\Id_{\nu_{y_0}\mathcal S}\quad\text{for every piecewise smooth loop $\gamma$ based at $y_0$}. \nonumber
\end{align}
If $\gamma_1,\gamma_2$ are piecewise smooth paths from $y_0$ to
$y\in\mathcal S$, then
\begin{align}
&A_{\gamma_2}^{-1}A_{\gamma_1}=\Id,\qquad A_{\gamma_1}=A_{\gamma_2}. \nonumber
\end{align}
By Lemma~\ref{soul:lem:hessian}, \eqref{four:eq:fiber-distance},
and Theorem~\ref{ext:thm:transport-exponential},
\eqref{ext:eq:transport-exponential},
\begin{align}
&\mathcal N_y=\exp_y(\nu_y\mathcal S),\qquad A_{\gamma_1}=A_{\gamma_2}\Longrightarrow\Phi_{\gamma_1}=\Phi_{\gamma_2}. \nonumber
\end{align}
Thus the map
\begin{align}
&\Phi:\mathcal S\times N\to M,\qquad \Phi(y,z)=\Phi_\gamma(z),\qquad \gamma(0)=y_0,\quad\gamma(1)=y, \nonumber
\end{align}
is independent of the chosen path. Reverse transport gives its
inverse. Locally smooth endpoint-dependent paths and smooth
dependence of horizontal lifting give smoothness of both maps.
For each fixed $z\in N$ and each smooth curve $\eta$ in
$\mathcal S$, path independence shows that
$t\mapsto\Phi(\eta(t),z)$ is a horizontal lift of $\eta$.
Consequently,
\begin{align}
&d\Phi(T_y\mathcal S\times\{0\})=\mathcal H_{\Phi(y,z)}. \nonumber
\end{align}
The images of $\mathcal S\times\{z\}$ therefore have tangent
spaces $\mathcal H$, so $\mathcal H$ is integrable.
Lemma~\ref{soul:lem:horizontal-transport} now shows that every
fiber map $\Phi(y,\cdot)$ is a Riemannian isometry. Since $P$ is a
Riemannian submersion,
\begin{align}
&\Phi(y,\cdot)^*k_y=k,\qquad g|_{\mathcal H}=P^*h,\qquad\mathcal H\perp\mathcal V,\qquad\Phi^*g=h\oplus k. \nonumber
\end{align}
Lemma~\ref{soul:lem:hessian} gives smoothness, connectedness,
completeness and absence of boundary for $N$. The product
isometry and compactness of $\mathcal S$ imply that $N$ is
noncompact. Its total geodesy in the product gives $\sec_k\ge0$.

The product isometry identifies $\mathcal S$ with
$\mathcal S\times\{y_0\}$ and
$\Omega_r$ with $\mathcal S\times B_{y_0}^N(r)$.
Let $\operatorname{pr}_N:\mathcal S\times N\to N$ be projection. Then
\begin{align}
&\int_{\Omega_r}\Sc_g\,dV_g=8\pi\Vol_k B_{y_0}^N(r)+\Vol_h(\mathcal S)\int_{B_{y_0}^N(r)}\Sc_k\,dV_k, \nonumber\\
&\tr_{\mathcal V}\Ric_g=\Sc_k\circ\operatorname{pr}_N, \nonumber
\end{align}
Corollary~\ref{ext:cor:soul-traces}, equation~\eqref{low:eq:vertical-trace-limit},
with $s=2$, gives
\begin{align}
&\lim_{r\to\infty}r^{2-n}\Vol_h(\mathcal S)\int_{B_{y_0}^N(r)}\Sc_k\,dV_k=0. \nonumber
\end{align}
If $n=3$, completeness, connectedness and noncompactness make $N$
isometric to a line, giving \eqref{ext:eq:positive-soul-equality}.
For $n\ge4$, Bishop--Gromov comparison \cite{CE} gives the center-independent limit
\begin{align}
&\AVR(N,k)=\lim_{r\to\infty}\frac{\Vol_k B_o^N(r)}{\omega_{n-2}r^{n-2}},\qquad o\in N,\qquad 0\le\AVR(N,k)\le1. \nonumber
\end{align}
The preceding integral identities and the equality assumption imply
\begin{align}
&8\pi\omega_{n-2}=L(M,g)=8\pi\omega_{n-2}\AVR(N,k),\qquad\AVR(N,k)=1, \nonumber\\
&1\ge\frac{\Vol_k B_o^N(r)}{\omega_{n-2}r^{n-2}}\ge\AVR(N,k)=1\qquad(r>0). \nonumber
\end{align}
The equality case of Bishop--Gromov comparison \cite{CE} gives
$(N,k)\cong\mathbb R^{n-2}$. Thus, in every dimension, equality
implies \eqref{ext:eq:positive-soul-equality}.
Conversely, Lemma~\ref{low:lem:Euclidean-products},
equation~\eqref{ext:eq:Euclidean-product}, and Gauss--Bonnet
\cite[Section~4-5]{doCarmo} give
\begin{align}
&L((\mathbb S^2,h)\times\mathbb R^{n-2})=\omega_{n-2}\int_{\mathbb S^2}\Sc_h\,dA_h=8\pi\omega_{n-2}. \nonumber
\end{align}
\end{proof}

\section{The analysis of curvature integrals}\label{pt:sec:five}

%\subsection{Definitions and previously established inputs}

Fix
\begin{align}
&\zeta_*(s)= \begin{cases} c_*e^{-1/(1-s^2)},&|s|<1,\\ 0,&|s|\ge1, \end{cases} \qquad c_*^{-1}=\int_{-1}^1e^{-1/(1-s^2)}\,ds, \nonumber\\
&\zeta_\varepsilon(s)=\varepsilon^{-1}\zeta_*(s/\varepsilon), \qquad 0<\varepsilon<1/8. \label{interval:eq:mollifier}
\end{align}
Thus $\zeta_\varepsilon\ge0$, $\int_{\mathbb R}\zeta_\varepsilon=1$,
and $\supp\zeta_\varepsilon\subset[-\varepsilon,\varepsilon]$.

Increase $t_0$ so that $F$ has no critical points on
$\{F\ge t_0\}$. 

For $n\ge3$, use an orthonormal eigenframe $(e_i)$ of $S$,
with eigenvalues $\kappa_i$ and
$K_{ik}=\sec_g(e_i\wedge e_k)$, and define
\begin{align}
&Q_n(S)=\sum_{1\le i<k\le n-1}K_{ik} \prod_{\ell\notin\{i,k\}}\kappa_\ell,\qquad \mathcal A_n(F,t)=\int_{\Sigma_t} \left(\det S+\frac{Q_n(S)}{n-2}\right)dA, \label{uni:eq:common-boundary}
\end{align}
where the empty product is one. 

%\subsection{Curvature contractions and interpolation}
%\label{unified:gen:section}

For a self-adjoint endomorphism $H$ and every
$1\le j\le n-3$, define
\begin{align}
&V_j(H)^a{}_b :=\frac14\delta^{aik\ell_1\cdots\ell_j}_{bpqr_1\cdots r_j} R^{pq}{}_{ik} \prod_{\alpha=1}^j H^{r_\alpha}{}_{\ell_\alpha}. \label{uni:eq:normalized-tensor}
\end{align}
Thus $V_j(H)$ contains $j$ copies of the same $H$ and is homogeneous
of degree $j$ in $H$. The notation $V_j$ is defined here for this
curvature contraction. 

Only when distinct arguments are needed do we use
$V_j(H_1,\ldots,H_j)$, for $j\ge1$,
\begin{align}
&V_j(H_1,\ldots,H_j) =\left.\frac1{j!}\frac{\partial^j}{\partial t_1\cdots\partial t_j} V_j\!\left(\sum_{\alpha=1}^j t_\alpha H_\alpha\right) \right|_{t_1=\cdots=t_j=0}. \nonumber
\end{align}
Equivalently, for a smooth path $H_t$,
\begin{align}
&\frac{d}{dt}V_j(H_t) =jV_j(\dot H_t,\underbrace{H_t,\ldots,H_t}_{j-1\text{ copies}}), \qquad j\ge1. \nonumber
\end{align}

\begin{notation}\label{notation-E2-V0-P2}
{Put
\begin{align}
&V_0=E_1,\qquad (E_2)^a{}_b =\frac1{96}\delta^{aijkl}_{bpqrs} R^{pq}{}_{ij}R^{rs}{}_{kl}, \label{pt:eq:Einstein-tensors}\\
&(\mathbf{P}_2)^{at}{}_{bs} =\frac1{96}\delta^{atijkl}_{bspqru} R^{pq}{}_{ij}R^{ru}{}_{kl}. \label{unified:gen:H-definition}
\end{align}
}
\end{notation}

The last formula defines $\mathbf{P}_2$ in every dimension, with
$\mathbf{P}_2=0$ if $n\le5$. Its value on four vectors means
$\mathbf{P}_2(X,Y,Z,W)=(\mathbf{P}_2)^{at}{}_{bs}X_aY_tZ^bW^s$.

For a self-adjoint $H$, define
\begin{align}
&\mathcal J_H(X,Y)=\sum_{i,j}H_{ij}\mathbf{P}_2(e_i,X,e_j,Y). \label{qc:eq:J-definition}
\end{align}

\begin{lemma}\label{pt:lem:algebra}
Let $(M^n,g)$ be smooth, with $n\ge3$ and $\sec_g\ge0$. For $1\le j\le n-3$, the mixed tensor $V_j(H_1,\ldots,H_j)$ is self-adjoint and is nonnegative whenever all its arguments are nonnegative. For arbitrary self-adjoint arguments,
\begin{align}
&V_j(\Id,H_1,\ldots,H_{j-1})=(n-j-2)V_{j-1}(H_1,\ldots,H_{j-1}), \nonumber\\
&\tr V_j(H_1,\ldots,H_j)=(n-j-2)\tr\bigl(V_{j-1}(H_1,\ldots,H_{j-1})\circ H_j\bigr). \label{qc:eq:general-contractions}
\end{align}
The quadratic curvature tensors satisfy
\begin{align}
&0\le E_2\le P_2\Id,\qquad \tr E_2=(n-4)P_2,\qquad \diver E_2=0. \label{pt:eq:Einstein-properties}
\end{align}
The tensor $\mathbf{P}_2$ has the algebraic curvature symmetries and, for every orthonormal pair $e_1,e_2$,
\begin{align}
&\mathbf{P}_2(e_1,e_2,e_1,e_2)=P_2\bigl(\Rm_g\!\restriction_{\{e_1,e_2\}^{\perp}}\bigr), \nonumber\\
&|\mathbf{P}_2|\le C_nP_2,\qquad \nabla_a(\mathbf{P}_2)^{at}{}_{bs}=0. \nonumber
\end{align}
For self-adjoint $H,K$, the tensors $\mathcal J_H$ are self-adjoint and satisfy
\begin{align}
&\mathcal J_H\ge0\quad(H\ge0),\qquad \tr\mathcal J_H=(n-5)\tr(E_2\circ H), \label{qc:eq:J-positive}\\
&\tr(\mathcal J_H\circ K)=\tr(\mathcal J_K\circ H), \label{qc:eq:J-mixed-pairing}
\end{align}
where the last expression is nonnegative if $H,K\ge0$.

For $f\in C^\infty(M)$ and a constant $\lambda\in\mathbb R$, put $H=\lambda\Id+\Hess f$. Then
\begin{align}
&\diver V_1(H)=-3(P_2\Id-E_2)\nabla f\qquad(n\ge4), \label{pt:eq:B-divergence}\\
&\diver V_2(H)=-6\bigl(\tr(E_2\circ H)\Id-E_2H-\mathcal J_H\bigr)\nabla f\qquad(n\ge5), \label{qc:eq:T-divergence}\\
&\delta^{aijkl}_{bpqrs}R^{pq}{}_{ij}R^r{}_{tlk}=24\bigl(\delta_{bt}(E_2)^a{}_s-\delta_{st}(E_2)^a{}_b+(\mathbf{P}_2)^{at}{}_{bs}\bigr). \label{pt:eq:delta-quadratic}
\end{align}
\end{lemma}

\begin{proof}
Work in an orthonormal frame at a fixed point. The polarized definition is
\begin{align}
&V_j(H_1,\ldots,H_j)^a{}_b=\frac14\delta^{aik\ell_1\cdots\ell_j}_{bpqr_1\cdots r_j}R^{pq}{}_{ik}\prod_{\alpha=1}^j(H_\alpha)^{r_\alpha}{}_{\ell_\alpha}. \nonumber
\end{align}
Exchanging two argument slots changes both alternating index lists by the same sign. Thus this expression is symmetric multilinear in its arguments. Transposition, the symmetry of each $H_\alpha$, and the pair symmetry of $\Rm_g$ show that its value is self-adjoint. The elementary contraction
\begin{align}
&\delta^{c i_1\cdots i_{j+2}}_{c k_1\cdots k_{j+2}}=(n-j-2)\delta^{i_1\cdots i_{j+2}}_{k_1\cdots k_{j+2}} \nonumber
\end{align}
gives both identities in \eqref{qc:eq:general-contractions}.

For positivity, fix a unit vector $e$. Alternation shows that $V_j(H_1,\ldots,H_j)(e,e)$ depends only on the compressions of the arguments to $e^\perp$. First take $H_\alpha=u_\alpha\otimes u_\alpha$ with $u_\alpha\in e^\perp$. If these vectors are linearly independent, choose an orthonormal basis $(e_\mu)$ of $\{e,u_1,\ldots,u_j\}^\perp$. Expansion of the alternating contraction gives
\begin{align}
&V_j(u_1\otimes u_1,\ldots,u_j\otimes u_j)(e,e)=\det\bigl(\langle u_\alpha,u_\beta\rangle\bigr)\sum_{\mu<\nu}\sec_g(e_\mu\wedge e_\nu)\ge0. \nonumber
\end{align}
The expression is zero if the vectors are dependent. Decomposing each nonnegative compression into a sum of nonnegative rank-one endomorphisms and using multilinearity proves positivity. The same alternating contraction also gives the useful consequence
\begin{align}
&\tr\bigl(V_1(H)\circ K\bigr)=\frac14\delta^{aijk}_{bpqr}R^{pq}{}_{ij}H^r{}_kK^b{}_a=\tr\bigl(V_1(K)\circ H\bigr)\qquad(n\ge4). \nonumber
\end{align}

For a unit vector $e$, the defining contraction of $E_2$ restricts all curvature indices to $e^\perp$. Equation~\eqref{eq:P2-four-plane-sum} and Lemma~\ref{six:lem:P2-positive} therefore give
\begin{align}
&0\le E_2(e,e)=P_2(\Rm_g\!\restriction_{e^\perp})\le P_2,\qquad \tr E_2=(n-4)P_2. \nonumber
\end{align}
The same argument with two fixed orthonormal vectors proves the stated sectional formula for $\mathbf{P}_2$. Its pair symmetries follow from the alternating delta and the pair symmetry of $\Rm_g$; the algebraic Bianchi identity follows by expanding the cyclic sum and using the algebraic Bianchi identity of $\Rm_g$. Differential Bianchi gives
\begin{align}
&\nabla_aR^{pq}{}_{ij}+\nabla_iR^{pq}{}_{ja}+\nabla_jR^{pq}{}_{ai}=0, \nonumber\\
&\nabla_a(E_2)^a{}_b=\frac1{48}\delta^{aijkl}_{bpqrs}(\nabla_aR^{pq}{}_{ij})R^{rs}{}_{kl}=0, \nonumber\\
&\nabla_a(\mathbf{P}_2)^{at}{}_{bs}=\frac1{48}\delta^{atijkl}_{bspqru}(\nabla_aR^{pq}{}_{ij})R^{ru}{}_{kl}=0. \nonumber
\end{align}
Indeed, alternation makes the three terms of the differential Bianchi identity contribute equally to each displayed contraction.

With $\Ric_{\mathbf{P}_2}(X,Y)=\sum_i\mathbf{P}_2(e_i,X,e_i,Y)$ and $\Sc_{\mathbf{P}_2}=\tr\Ric_{\mathbf{P}_2}$, contracting one and then two pairs of indices gives
\begin{align}
&\Ric_{\mathbf{P}_2}=(n-5)E_2,\qquad \Sc_{\mathbf{P}_2}=(n-5)(n-4)P_2. \nonumber
\end{align}
For $n\ge6$, the sectional formula and Lemma~\ref{six:lem:P2-positive} show that $\mathbf{P}_2$ has nonnegative sectional curvature. Each of its sectional values is at most $\Sc_{\mathbf{P}_2}/2$. Polarization of the sectional values therefore bounds every component, and hence
\begin{align}
&|\mathbf{P}_2|\le C_n\Sc_{\mathbf{P}_2}\le C_nP_2. \nonumber
\end{align}
For $n\le5$, the defining alternating delta gives $\mathbf{P}_2=0$, so the same bound holds.

The pair symmetries show that $\mathcal J_H$ is self-adjoint. If $H e_i=\lambda_i e_i$ and $\lambda_i\ge0$, then
\begin{align}
&\mathcal J_H(X,X)=\sum_i\lambda_i\mathbf{P}_2(e_i,X,e_i,X)\ge0. \nonumber
\end{align}
For arbitrary self-adjoint $H,K$, contraction and the pair symmetries give
\begin{align}
&\tr\mathcal J_H=\tr(\Ric_{\mathbf{P}_2}\circ H)=(n-5)\tr(E_2\circ H), \nonumber\\
&\tr(\mathcal J_H\circ K)=\sum_{i,j,a,b}H_{ij}K_{ab}\mathbf{P}_2(e_i,e_a,e_j,e_b)=\tr(\mathcal J_K\circ H). \nonumber
\end{align}
When $H,K\ge0$, this trace is nonnegative because $\mathcal J_H\ge0$. This proves \eqref{qc:eq:J-positive}--\eqref{qc:eq:J-mixed-pairing}.

To prove the remaining identities, Laplace expansion of the defining deltas along the distinguished index $t$ gives
\begin{align}
&96(E_2)^t{}_b=96P_2\delta_{bt}+4\delta^{aijk}_{bpqr}R^{pq}{}_{ij}R^r{}_{tak}, \nonumber\\
&96(\mathbf{P}_2)^{at}{}_{bs}=-96\delta_{bt}(E_2)^a{}_s+96\delta_{st}(E_2)^a{}_b+4\delta^{aijkl}_{bpqrs}R^{pq}{}_{ij}R^r{}_{tlk}. \nonumber
\end{align}
In each expansion, the four terms in which $t$ enters a curvature factor agree after using curvature antisymmetry and interchanging the two curvature factors. The second identity is exactly \eqref{pt:eq:delta-quadratic}.

For $H=\lambda\Id+\Hess f$ with $\lambda$ constant, the Hessian commutator reads
\begin{align}
&\nabla_aH^r{}_k-\nabla_kH^r{}_a=R^r{}_{tak}(\nabla f)^t. \nonumber
\end{align}
Differentiating $V_1(H)$, differential Bianchi eliminates the derivative of curvature, while alternation in $a,k$ gives
\begin{align}
&(\diver V_1(H))_b=\frac18\delta^{aijk}_{bpqr}R^{pq}{}_{ij}R^r{}_{tak}(\nabla f)^t=-3\bigl(P_2\delta_{bt}-(E_2)_{bt}\bigr)(\nabla f)^t. \nonumber
\end{align}
This proves \eqref{pt:eq:B-divergence}. Similarly, the two differentiated copies of $H$ in $V_2(H)$ contribute equally, so
\begin{align}
&(\diver V_2(H))_b=\frac14\delta^{aijkl}_{bpqrs}R^{pq}{}_{ij}R^r{}_{tak}H^s{}_l(\nabla f)^t \nonumber\\
&\qquad=-\frac14\delta^{aijkl}_{bpqrs}R^{pq}{}_{ij}R^r{}_{tlk}H^s{}_a(\nabla f)^t. \nonumber
\end{align}
The second line exchanges the dummy indices $a,l$. Substitute \eqref{pt:eq:delta-quadratic} and use
\begin{align}
&(\mathbf{P}_2)^{at}{}_{bs}H^s{}_a=-(\mathcal J_H)^t{}_b. \nonumber
\end{align}
The result is
\begin{align}
&\diver V_2(H)=-6\tr(E_2\circ H)\nabla f+6E_2H\nabla f+6\mathcal J_H\nabla f, \nonumber
\end{align}
which proves \eqref{qc:eq:T-divergence}.
\end{proof}

\begin{notation}\label{notation-P3-E3}
{Define
\begin{align}
&P_3=\frac1{5760} \delta^{i_1\cdots i_6}_{j_1\cdots j_6} R^{j_1j_2}{}_{i_1i_2}R^{j_3j_4}{}_{i_3i_4} R^{j_5j_6}{}_{i_5i_6}, \nonumber\\
&(E_3)^a{}_b =\frac1{5760} \delta^{a i_1\cdots i_6}_{b j_1\cdots j_6} R^{j_1j_2}{}_{i_1i_2}R^{j_3j_4}{}_{i_3i_4} R^{j_5j_6}{}_{i_5i_6}. \label{audit:eq:P3-E3}
\end{align}
}
\end{notation}

\begin{lemma}\label{lem-E3-property}
{The tensor $E_3$ is symmetric and
\begin{align}
&\mathcal J_{\Id}=(n-5)E_2. \label{audit:eq:E3-contractions}
\end{align}
For $n=6$, one has $E_3=0$. 
}
\end{lemma}

\begin{proof}
At a fixed point, choose an orthonormal frame with $\nabla e_i=0$:
\begin{align}
&\delta^{c a_1\cdots a_k}_{c b_1\cdots b_k}=(n-k)\delta^{a_1\cdots a_k}_{b_1\cdots b_k},\qquad \delta^{a_0\cdots a_k}_{b_0\cdots b_k}=0\quad\text{if }k+1>n. \label{audit:eq:delta-contractions}
\end{align}
Pair symmetry $R^{ij}{}_{kl}=R^{kl}{}_{ij}$:
\begin{align}
&(E_3)_{ba} =\frac1{5760}\delta^{b i_1\cdots i_6}_{a j_1\cdots j_6}R^{j_1j_2}{}_{i_1i_2}R^{j_3j_4}{}_{i_3i_4}R^{j_5j_6}{}_{i_5i_6}=\frac1{5760}\delta^{b j_1\cdots j_6}_{a i_1\cdots i_6}R^{i_1i_2}{}_{j_1j_2}R^{i_3i_4}{}_{j_3j_4}R^{i_5i_6}{}_{j_5j_6} \nonumber\\
&{}=\frac1{5760}\delta^{a i_1\cdots i_6}_{b j_1\cdots j_6}R^{j_1j_2}{}_{i_1i_2}R^{j_3j_4}{}_{i_3i_4}R^{j_5j_6}{}_{i_5i_6}=(E_3)_{ab}. \nonumber
\end{align}
\begin{align}
&(\mathbf{P}_2)^{ai}{}_{bj}=\frac1{96}\delta^{ai k_1k_2k_3k_4}_{bj l_1l_2l_3l_4}R^{l_1l_2}{}_{k_1k_2}R^{l_3l_4}{}_{k_3k_4},\qquad (\mathcal J_H)^a{}_b=(\mathbf{P}_2)^{ai}{}_{bj}H^j{}_i, \nonumber\\
&(\mathcal J_{\Id})^a{}_b =(\mathbf{P}_2)^{ai}{}_{bi}=\frac1{96}\delta^{ai k_1k_2k_3k_4}_{bi l_1l_2l_3l_4}R^{l_1l_2}{}_{k_1k_2}R^{l_3l_4}{}_{k_3k_4}=\frac{n-5}{96}\delta^{a k_1k_2k_3k_4}_{b l_1l_2l_3l_4}R^{l_1l_2}{}_{k_1k_2}R^{l_3l_4}{}_{k_3k_4} \nonumber\\
&{}=(n-5)(E_2)^a{}_b. \nonumber\\
&\mathcal J_{\Id}=(n-5)E_2. \nonumber
\end{align}

For $n=6$,
\begin{align}
&\delta^{a i_1\cdots i_6}_{b j_1\cdots j_6}=0, \nonumber
\end{align}
by alternation in seven upper indices. Thus
\postdisplaypenalty=10000
\begin{align}
&E_3\equiv0\qquad\text{on }M^6. \nonumber
\end{align}
\end{proof}

\begin{notation}\label{notation-W2-S}
{For any $f\in C^\infty(M)$ and every regular value $t$, use
Notation~\ref{notation:level-objects} with $f$ in place of $F$ and choose an orthonormal
eigenframe $(e_i)_{i=1}^{n-1}$ of $S$. Define
\begin{align}
&W_2(S)=\sum_{i=1}^{n-1}\kappa_i P_2\bigl(\Rm_g|_{\{\nu,e_i\}^{\perp}}\bigr),\qquad n\ge6, \label{audit:eq:W2-definition}
\end{align}
using the induced contractions on the $(n-2)$-dimensional
orthogonal space as specified in Notation~\ref{notation:P2}.
}
\end{notation}

\begin{lemma}\label{audit:lem:cubic-divergence}\label{six:lem:divergence}
Let $(M^n,g)$ be smooth, $n\ge6$, $f\in C^\infty(M)$ and $\zeta\in C_c^\infty(M)$,
\begin{align}
&\diver\mathcal J_{\Hess f} =5(E_3-P_3\Id)\nabla f, \label{audit:eq:general-J-divergence}\\
&\int_M\mathcal J_{\Hess f}(\nabla f,\nabla\zeta)\,dV =\int_M\mathcal J_{\Hess\zeta}(\nabla f,\nabla f)\,dV. \label{one:eq:Hessian-transfer}
\end{align}

\end{lemma}

\begin{proof}

\begin{align}
&5760(E_3)^t{}_b =5760P_3\delta^t_b +576(\mathbf{P}_2)^{ai}{}_{bj}R^j{}_{tai},\qquad \frac12(\mathbf{P}_2)^{ai}{}_{bj}R^j{}_{tai} =5\bigl((E_3)^t{}_b-P_3\delta^t_b\bigr). \nonumber
\end{align}
The Hessian commutator gives
\begin{align}
&(\diver\mathcal J_{\Hess f})_b =(\mathbf{P}_2)^{ai}{}_{bj}\nabla_a(\Hess f)^j{}_i=\tfrac12(\mathbf{P}_2)^{ai}{}_{bj}R^j{}_{tai}(\nabla f)^t =5\bigl((E_3)^t{}_b-P_3\delta^t_b\bigr)(\nabla f)^t. \nonumber
\end{align}
This is \eqref{audit:eq:general-J-divergence}. Since $\nabla_i(\mathbf{P}_2)^{ai}{}_{bj}=0$ and $\supp\zeta\Subset M$, integration in $i$ gives
\begin{align}
&\int(\mathbf{P}_2)^{ai}{}_{bj}\zeta_i{}^j f_a f^b\,dV =-\int(\nabla_i(\mathbf{P}_2)^{ai}{}_{bj})\zeta^j f_a f^b\,dV-\int(\mathbf{P}_2)^{ai}{}_{bj}\zeta^j f_{ia}f^b\,dV-\int(\mathbf{P}_2)^{ai}{}_{bj}\zeta^j f_a f_i{}^b\,dV \nonumber\\
&{}=-\int(\mathbf{P}_2)^{ai}{}_{bj}\zeta^j f_{ia}f^b\,dV-\int(\mathbf{P}_2)^{ai}{}_{bj}\zeta^j f_a f_i{}^b\,dV. \nonumber\\
&(\mathbf{P}_2)^{ai}{}_{bj}f_{ia}=-(\mathbf{P}_2)^{ia}{}_{bj}f_{ia}=-(\mathbf{P}_2)^{ai}{}_{bj}f_{ai}=-(\mathbf{P}_2)^{ai}{}_{bj}f_{ia}. \nonumber\\
&\int(\mathbf{P}_2)^{ai}{}_{bj}\zeta_i{}^j f_a f^b\,dV=-\int(\mathbf{P}_2)^{ai}{}_{bj}\zeta^j f_a f_i{}^b\,dV. \nonumber
\end{align}
Using $f_{ij}=f_{ji}$,
\begin{align}
&{}-(\mathbf{P}_2)_{aibj}\zeta_j f_a f_{ib}=-(\mathbf{P}_2)_{bjai}\zeta_j f_a f_{ib}=(\mathbf{P}_2)_{jbai}\zeta_j f_a f_{ib}=(\mathbf{P}_2)_{aibj}\zeta_a f_b f_{ji}=(\mathbf{P}_2)_{aibj}\zeta_a f_b f_{ij}. \nonumber\\
&{}-(\mathbf{P}_2)^{ai}{}_{bj}\zeta^j f_a f_i{}^b=(\mathbf{P}_2)^{ai}{}_{bj}\zeta_a f^b f_i{}^j. \nonumber\\
&\int(\mathbf{P}_2)^{ai}{}_{bj}\zeta_i{}^j f_a f^b\,dV=\int(\mathbf{P}_2)^{ai}{}_{bj}\zeta_a f^b f_i{}^j\,dV, \nonumber
\end{align}
This is \eqref{one:eq:Hessian-transfer}. On a regular level $\Sigma_t=\{f=t\}$, use the corresponding level notation and take an orthonormal frame with
\begin{align}
&e_n=\nu,\qquad S(e_\alpha)=\kappa_\alpha e_\alpha,\qquad 1\le\alpha\le n-1. \nonumber\\
&(\Hess f)_{\alpha\beta}=|\nabla f|\kappa_\alpha\delta_{\alpha\beta}. \nonumber\\
&\mathbf{P}_2(\nu,\nu,e_B,\nu)=0,\qquad \mathbf{P}_2(e_A,\nu,\nu,\nu)=0, \nonumber\\
&\mathcal J_{\Hess f}(\nu,\nu) =\sum_{A,B=1}^n(\Hess f)_{AB}\mathbf{P}_2(e_A,\nu,e_B,\nu)=\sum_{\alpha,\beta=1}^{n-1}(\Hess f)_{\alpha\beta}\mathbf{P}_2(e_\alpha,\nu,e_\beta,\nu) \nonumber\\
&{}=|\nabla f|\sum_{\alpha=1}^{n-1}\kappa_\alpha\mathbf{P}_2(e_\alpha,\nu,e_\alpha,\nu)=|\nabla f|\sum_{\alpha=1}^{n-1}\kappa_\alpha P_2\bigl(\Rm_g|_{\{\nu,e_\alpha\}^{\perp}}\bigr)=|\nabla f|W_2(S). \label{one:eq:J-boundary}\\
&W_2(S)=|\nabla f|^{-1}\mathcal J_{\Hess f}(\nu,\nu), \nonumber
\end{align}
which proves the eigenframe independence in \eqref{audit:eq:W2-definition}.
\end{proof}

\begin{lemma}\label{one:lem:weak-convex}
If $(M^n, g)$ is a complete noncompact Riemannian manifold with $K\geq 0$ and $n\ge6$, then
\begin{align}
&\int_M\mathcal J_{\Hess b}(\nabla \mathfrak{b},\nabla \mathfrak{b})\,dV \le5\int_M b\bigl(P_3-E_3(\nabla \mathfrak{b},\nabla \mathfrak{b})\bigr)\,dV, \quad \forall \in C_c^\infty(M; \overline{\mathbb{R}^+}). \label{audit:eq:general-weak-convex}
\end{align}
\end{lemma}

\begin{proof}
\begin{align}
&\supp b\subset U_0\Subset U_1\Subset M,\qquad j\in\mathbb Z_{\ge2},\qquad e=\gamma=\delta=j^{-1}. \nonumber
\end{align}
 Theorem~\ref{ext:thm:Greene-Wu}, \eqref{eq:smoothing-import}, gives $v_j\in C^\infty(U_1)$ with
\begin{align}
&|v_j-\mathfrak{b}|<j^{-1},\qquad |\nabla v_j|<1+j^{-1},\qquad \Hess v_j\ge-j^{-1}g. \nonumber
\end{align}
For $\mathfrak{b}$ differentiable at $x$, $V\in T_xM$, $|V|=1$, and $s>0$ with $\exp_x([-s,s]V)\subset U_1$,
\begin{align}
&\frac{\mathfrak{b}(x)-\mathfrak{b}(\exp_x(-sV))}{s}-\frac2{js}-\frac{s}{2j} \le dv_j(V)\le\frac{\mathfrak{b}(\exp_x(sV))-\mathfrak{b}(x)}s+\frac2{js}+\frac{s}{2j}. \nonumber\\
&\frac{\mathfrak{b}(x)-\mathfrak{b}(\exp_x(-sV))}{s}\le\liminf_{j\to\infty}dv_j(V)\le\limsup_{j\to\infty}dv_j(V)\le\frac{\mathfrak{b}(\exp_x(sV))-\mathfrak{b}(x)}s, \nonumber\\
&\lim_{s\downarrow0}\frac{\mathfrak{b}(\exp_x(sV))-\mathfrak{b}(x)}s=d\mathfrak{b}(V)=\lim_{s\downarrow0}\frac{\mathfrak{b}(x)-\mathfrak{b}(\exp_x(-sV))}s. \nonumber
\end{align}
 Apply this to an orthonormal basis and use Lemma~\ref{lem:geometry}, \eqref{six:eq:unit-gradient}:
\begin{align}
&\lim_{j\to\infty}\nabla v_j(x)=\nabla \mathfrak{b}(x),\qquad \lim_{j\to\infty}|\nabla v_j(x)|^2=1 \quad\hbox{for almost every }x\in U_1. \nonumber\\
&0\le\xi\in C_c^\infty(U_1),\qquad\xi=1\text{ near }\overline{U_0},\qquad|\Hess v_j|\le\Delta v_j+2n/j. \nonumber
\end{align}
 Integration by parts gives
\begin{align}
&\int_{\supp b}|\Hess v_j|\,dV \le-\int\langle\nabla\xi,\nabla v_j\rangle\,dV +\frac{2n}{j}\int\xi\,dV\le C_b,\qquad j\ge2. \nonumber
\end{align}
Lemma~\ref{pt:lem:algebra}, \eqref{qc:eq:J-mixed-pairing}, and Lemma~\ref{lem-E3-property}, \eqref{audit:eq:E3-contractions}:
\begin{align}
&\tr(\mathcal J_{\Hess v_j}\circ \Hess v_j) =\tr(\mathcal J_{\Hess v_j+ \frac{1}{j}\Id}\circ (\Hess v_j+ \frac{1}{j}\Id)) -\frac{2(n-5)}j\tr(E_2\circ (\Hess v_j+ \frac{1}{j}\Id)) \nonumber\\
&{}+\frac{n-5}{j^2}\tr E_2 \nonumber\\
&{}\ge-\frac{2(n-5)}j\tr(E_2\circ (\Hess v_j+ \frac{1}{j}\Id)). \nonumber\\
&\sup_{\overline{U_0}}|E_2|<\infty,\qquad\int b\tr(\mathcal J_{\Hess v_j}\circ\Hess v_j)\,dV\ge-C_b/j. \nonumber
\end{align}
 The divergence theorem and Lemma~\ref{audit:lem:cubic-divergence}, \eqref{audit:eq:general-J-divergence}, give
\begin{align}
&\int b\tr(\mathcal J_{\Hess v_j}\circ \Hess v_j)\,dV=-\int \mathcal J_{\Hess v_j}(\nabla b,\nabla v_j)\,dV -\int b(\diver\mathcal J_{\Hess v_j})(\nabla v_j)\,dV= \nonumber\\
&{}-\int \mathcal J_{\Hess v_j}(\nabla b,\nabla v_j)\,dV \nonumber\\
&{}+5\int b\bigl(P_3|\nabla v_j|^2-E_3(\nabla v_j,\nabla v_j)\bigr)\,dV. \nonumber
\end{align}
Lemma~\ref{audit:lem:cubic-divergence}, \eqref{one:eq:Hessian-transfer}:
\begin{align}
&\int \mathcal J_{\Hess v_j}(\nabla b,\nabla v_j)\,dV = \int \mathcal J_{\Hess b}(\nabla v_j,\nabla v_j)\,dV. \nonumber\\
&5\int b\bigl(P_3|\nabla v_j|^2 -E_3(\nabla v_j,\nabla v_j)\bigr)\,dV -\int\mathcal J_{\Hess b}(\nabla v_j,\nabla v_j)\,dV=\int b\tr(\mathcal J_{\Hess v_j}\circ \Hess v_j)\,dV \ge \nonumber\\
&{}-\frac{C_b}{j}. \nonumber\\
&|\nabla v_j|\le\tfrac32,\qquad\supp b\Subset U_1, \nonumber\\
&\left|5b(P_3|\nabla v_j|^2-E_3(\nabla v_j,\nabla v_j))-\mathcal J_{\Hess b}(\nabla v_j,\nabla v_j)\right|\le C_b\mathbf1_{\supp b}\in L^1(M). \nonumber
\end{align}
 Dominated convergence yields Lemma~\ref{one:lem:weak-convex}, \eqref{audit:eq:general-weak-convex}.
\end{proof}

\begin{lemma}\label{six:lem:cubic-interpolation}
Let $(M^n,g)$ be smooth, $n=6$.
Let $\phi_0,\phi_1,\chi\in C^\infty(U)$ on an open set $U$,
assume $\nabla\chi$ has compact support in $U$, and for 
$0\le\lambda\le1$ put
\begin{align}
&\phi_\lambda=(1-\lambda)\phi_0+\lambda\phi_1, H_\lambda=\Hess\phi_\lambda,\qquad z=\nabla(\phi_1-\phi_0). \nonumber
\end{align}
Then
\begin{align}
&\frac{d}{d\lambda}\int_U\langle\nabla\chi, V_3(H_\lambda)\nabla\phi_\lambda\rangle\,dV=4\int_U\langle\nabla\chi, V_3(H_\lambda)z\rangle\,dV \nonumber\\
&{}+18\int_U\bigl( \langle z,\nabla\phi_\lambda\rangle \mathcal J_{H_\lambda}(\nabla\chi,\nabla\phi_\lambda) -|\nabla\phi_\lambda|^2 \mathcal J_{H_\lambda}(\nabla\chi,z)\bigr)\,dV \nonumber\\
&{}-18\int_U\mathbf{P}_2(\nabla\chi,H_\lambda \nabla\phi_\lambda, \nabla\phi_\lambda,z)\,dV. \label{six:eq:cubic-interpolation}
\end{align}
\end{lemma}
\begin{proof}
Fix $\lambda$ and write
\begin{align}
&\phi=\phi_\lambda,\qquad H=H_\lambda,\qquad \chi_a=(\nabla\chi)_a,\qquad \phi^a=(\nabla\phi)^a,\qquad \frac{d}{d\lambda}H^a{}_b=(\nabla_bz)^a. \nonumber\\
&n=6\ \Longrightarrow\ \delta^{a_1\cdots a_7}_{b_1\cdots b_7}=0. \nonumber\\
&0={}-96\delta_{bt}\mathcal J_H{}^a{}_s+96\delta_{st}\mathcal J_H{}^a{}_b-96(\mathbf{P}_2)^{am}{}_{bs}H^t{}_m +4\delta^{aijk\ell m}_{bpqrsu}R^{pq}{}_{ij}R^r{}_{t\ell k}H^u{}_m, \nonumber\\
&\delta^{aijk\ell m}_{bpqrsu}R^{pq}{}_{ij}R^r{}_{t\ell k}H^u{}_m=24\bigl(\delta_{bt}\mathcal J_H{}^a{}_s-\delta_{st}\mathcal J_H{}^a{}_b+(\mathbf{P}_2)^{am}{}_{bs}H^t{}_m\bigr). \label{six:eq:cubic-delta}\\
&\frac{d}{d\lambda}\int_U\langle\nabla\chi,V_3(H)\nabla\phi\rangle\,dV=\frac34\int_U\chi_a\delta^{aijk\ell m}_{bpqrsu}R^{pq}{}_{ij}(\nabla_kz)^rH^s{}_\ell H^u{}_m\phi^b\,dV+\int_U\langle\nabla\chi,V_3(H)z\rangle\,dV. \nonumber\\
&\delta^{aijk\ell m}_{bpqrsu}R^{pq}{}_{ij}H^s{}_\ell H^u{}_m\nabla_k\chi_a=0,\qquad \delta^{aijk\ell m}_{bpqrsu}(\nabla_kR^{pq}{}_{ij})H^s{}_\ell H^u{}_m=0. \nonumber\\
&\nabla_kH^s{}_\ell-\nabla_\ell H^s{}_k=R^s{}_{tk\ell}\phi^t, \nonumber
\end{align}
Symmetry of the two Hessian factors and integration by parts give
\begin{align}
&\frac34\int_U\chi_a\delta^{aijk\ell m}_{bpqrsu}R^{pq}{}_{ij}(\nabla_kz)^rH^s{}_\ell H^u{}_m\phi^b\,dV=3\int_U\langle\nabla\chi,V_3(H)z\rangle\,dV \nonumber\\
&{}-\frac34\int_U\chi_a\delta^{aijk\ell m}_{bpqrsu}R^{pq}{}_{ij}R^r{}_{t\ell k}H^u{}_m\phi^t\phi^bz^s\,dV. \nonumber
\end{align}
Equation~\eqref{six:eq:cubic-delta}:
\begin{align}
&{}-\frac34\chi_a\delta^{aijk\ell m}_{bpqrsu}R^{pq}{}_{ij}R^r{}_{t\ell k}H^u{}_m\phi^t\phi^bz^s={}18\langle z,\nabla\phi\rangle\mathcal J_H(\nabla\chi,\nabla\phi)-18|\nabla\phi|^2\mathcal J_H(\nabla\chi,z) \nonumber\\
&{}-18\mathbf{P}_2(\nabla\chi,H\nabla\phi,\nabla\phi,z). \nonumber\\
&\frac{d}{d\lambda}\int_U\langle\nabla\chi,V_3(H)\nabla\phi\rangle\,dV=4\int_U\langle\nabla\chi,V_3(H)z\rangle\,dV \nonumber\\
&{}+18\int_U\bigl(\langle z,\nabla\phi\rangle\mathcal J_H(\nabla\chi,\nabla\phi)-|\nabla\phi|^2\mathcal J_H(\nabla\chi,z)\bigr)\,dV \nonumber\\
&{}-18\int_U\mathbf{P}_2(\nabla\chi,H\nabla\phi,\nabla\phi,z)\,dV, \nonumber
\end{align}
This is \eqref{six:eq:cubic-interpolation}; $\supp\nabla\chi\Subset U$ removes every boundary term.
\end{proof}

On every sufficiently large regular level, use
Notations~\ref{notation:eta} and~\ref{notation:level-objects},
and retain the estimates proved in
Lemma~\ref{lem:independent-level-limit},
equations~\eqref{eq:level-distance-hessian} and~\eqref{eq:level-curvature-limit}.
All constants may depend on the fixed manifold, base point,
exhaustion and initial level, and on $m$, but not on $t$ or $r$.
For a self-adjoint endomorphism $H$ define the algebraic curvature
tensor $\mathcal G(H)$ by
$\mathcal G(H)^{ab}_{ij}=H_i^aH_j^b-H_j^aH_i^b$.
Put $\mathcal R^0=\Rm_g|_{T\Sigma_t}$. In an eigenframe of $S$ set
\begin{align}
&W(H)=\tr\bigl((\tfrac12\Sc_{\mathcal R^0}\Id-\Ric_{\mathcal R^0})\circ H\bigr),\qquad G(t)=\int_{\Sigma_t}\bigl(W(S)+2\sigma_3(S)\bigr)\,dA, \nonumber\\
&J_2(t)=\int_{\Sigma_t}\frac{|\Rm_{\Sigma_t}|^2-4|\Ric_{\Sigma_t}|^2+\Sc_{\Sigma_t}^2}{24}\,dA. \label{pt:eq:curvature-polynomials}
\end{align}
When $H$ is diagonal in the displayed eigenframe,
$W(H)=\sum_iH_i^i\sum_{j<k,\,j,k\ne i}K_{jk}$.
The other expressions are likewise invariant.
The Gauss equation, expanded over coordinate four-planes, is
\begin{align}
&\frac{|\Rm_{\Sigma_t}|^2-4|\Ric_{\Sigma_t}|^2+\Sc_{\Sigma_t}^2}{24}=\frac{|\mathcal R^0|^2-4|\Ric_{\mathcal R^0}|^2+\Sc_{\mathcal R^0}^2}{24} \nonumber\\
&{}+\frac13\sum_{1\le a<b\le n-1}K_{ab}\sigma_{2}\bigl((\kappa_c)_{\substack{1\le c\le n-1\\ c\ne a,b}}\bigr)+\sigma_4(S),\qquad \frac{|\mathcal R^0|^2-4|\Ric_{\mathcal R^0}|^2+\Sc_{\mathcal R^0}^2}{24}\ge0. \label{pt:eq:Gauss-polynomial}
\end{align}
The inequality follows from Lemma~\ref{six:lem:P2-positive},
equation~\eqref{six:eq:P2-positive}, on each
four-plane. The same argument gives $P_2\ge0$.

In an orthonormal eigenframe of $S$, write $\kappa_i$ for the
eigenvalues and $K_{ik}=\sec_g(e_i\wedge e_k)$.
\begin{lemma}
\label{pt:lem:lapse}
Let $m\ge4$, let $I\subset\mathbb R$ be an open interval,
and let $\Sigma$ be a compact smooth $m$-manifold without boundary.
Let $h_t$ be a smooth family of Riemannian metrics on $\Sigma$,
and let $\lambda:\Sigma\times I\to(0,\infty)$ be smooth. Equip this
product with $g=\lambda^{-2}dt^2+h_t$.
Put $\Sigma_t=\Sigma\times\{t\}$,
$S=(\lambda/2)h_t^{-1}\partial_t h_t$ and
$\mathcal R^0=\Rm_g|_{T\Sigma_t}$.
Put
\begin{align}
&N_2(t)=\int_{\Sigma_t}\lambda^{-1}P_2\,dA. \nonumber
\end{align}
Then, for every $t\in I$,
\begin{align}
&\tfrac13G'(t)+N_2(t) =\int_{\Sigma_t}\lambda^{-1}\frac{|\Rm_{\Sigma_t}|^2-4|\Ric_{\Sigma_t}|^2+\Sc_{\Sigma_t}^2}{24}\,dA \nonumber\\
&{}+\frac13\int_{\Sigma_t}\lambda^{-1}\sum_{1\le a<b\le m}K_{ab}\sigma_{2}\bigl((\kappa_c)_{\substack{1\le c\le m\\ c\ne a,b}}\bigr)\,dA +\frac53\int_{\Sigma_t}\lambda^{-1}\sigma_4(S)\,dA. \label{pt:eq:lapse-identity}
\end{align}
For an algebraic curvature tensor $\mathcal R$ on an $m$-dimensional
inner-product space define
\begin{align}
&\mathcal E(\mathcal R)^a{}_b =\tfrac1{96}\delta^{aijkl}_{bpqrs} \mathcal R^{pq}_{ij}\mathcal R^{rs}_{kl}. \label{mass:eq:level-E2}
\end{align}
For every $t\in I$, the first
variation is
\begin{align}
&J_2'(t)=\int_{\Sigma_t} \lambda^{-1}\tr\bigl(\mathcal E(\Rm_{\Sigma_t})\circ S\bigr)\,dA. \label{mass:eq:variation}
\end{align}
\end{lemma}

\begin{proof}
Use the smooth Gauss--Codazzi identity \cite[Eqs.~(II.9), (II.11a), (III.10), pp.~2--5 of version~2]{JulieBerti2020}, for positive definite $h_t$, with source variables
\begin{align}
&w=t,\quad N=\lambda^{-1},\quad K_{ij}=S_{ij},\quad\epsilon=+1,\quad\alpha f=1,\quad\text{shift}=0. \nonumber
\end{align}
 The scalar field is constant; no field equation is used. The boundary contraction is
\begin{align}
&2\delta^{ijk}_{abc}S_i^a \left((\Rm_{\Sigma_t})^{bc}_{jk}-\tfrac23S_j^bS_k^c\right) =8\bigl(W(S)+2\sigma_3(S)\bigr). \nonumber
\end{align}
The bulk expression, divided by $24$, is $P_2(\Rm_{\Sigma_t})$ plus
\begin{align}
&\frac1{24}\delta^{ijkl}_{abcd}S_i^aS_j^b(\Rm_{\Sigma_t})^{cd}_{kl} -\frac1{72}\delta^{ijkl}_{abcd}S_i^aS_j^bS_k^cS_l^d=\frac13\sum_{1\le a<b\le m}K_{ab}\sigma_{2}\bigl((\kappa_c)_{\substack{1\le c\le m\\ c\ne a,b}}\bigr)+\frac53\sigma_4(S). \nonumber
\end{align}
The two unweighted contractions are
\begin{align}
&8\sum_{1\le a<b\le m}K_{ab}\sigma_{2}\bigl((\kappa_c)_{\substack{1\le c\le m\\ c\ne a,b}}\bigr)+48\sigma_4(S),\qquad 24\sigma_4(S). \nonumber\\
&m\ge4,\quad\partial\Sigma_t=\varnothing\ \Longrightarrow\ \int_{\Sigma_t}\diver_{\Sigma_t}X\,dA=0\quad(X\in C^\infty(T\Sigma_t)). \nonumber
\end{align}
 Integrate over a finite regular band and differentiate its upper endpoint to obtain \eqref{pt:eq:lapse-identity}. For a smooth metric family on compact boundaryless $\Sigma^m$, $m\ge4$, \cite[Main Theorem, p.~3 of version~2]{Labbi2008}, in normalization~\eqref{pt:eq:P2-definition}, gives
\begin{align}
&\frac{d}{dt}\int \frac{|\Rm_{h_t}|^2-4|\Ric_{h_t}|^2+\Sc_{h_t}^2}{24}\,dV_{h_t} =\frac12\int\langle\mathcal E(\Rm_{h_t}),\partial_t h_t\rangle_{h_t}\,dV_{h_t}. \nonumber
\end{align}
The source normalization, expanded in an orthonormal frame, is
\begin{align}
&h_4=\frac{|\Rm_{h_t}|^2-4|\Ric_{h_t}|^2+\Sc_{h_t}^2}{4},\qquad T_4=6\mathcal E(\Rm_{h_t}). \nonumber\\
&\partial_t h_t=2S/\lambda\ \Longrightarrow\ J_2^{\prime}(t)=\int_{\Sigma_t}\lambda^{-1}\tr(\mathcal E(\Rm_{\Sigma_t})\circ S)\,dA, \nonumber
\end{align}
which is \eqref{mass:eq:variation}. For nonorientable $\Sigma$, pull back to the orientation double cover and divide both integrals by two.
\end{proof}

\begin{lemma}
\label{mass:lem:shift}
Let $(M^{m+1},g)$ be smooth with $m\ge4$ and $\sec_g\ge0$.
Let $\Sigma$ be a smooth hypersurface with unit normal $\nu$,
induced metric, and $S(X)=\nabla_X\nu$.
Fix a number $\eta>0$ such that $S_\eta=S+\eta\Id\ge0$.
In an orthonormal eigenframe $(e_i)_{i=1}^m$ of $S$, write
$\kappa_i$ for its eigenvalues and $K_{ab}=\sec_g(e_a\wedge e_b)$.
Set $\mathcal R^0=\Rm_g|_{T\Sigma}$,
$\mathcal G(H)^{ab}_{ij}=H_i^aH_j^b-H_j^aH_i^b$,
and $W(H)=\tr((\tfrac12\Sc_{\mathcal R^0}\Id-\Ric_{\mathcal R^0})\circ H)$.
For any algebraic curvature tensor $\mathcal R$ set
$\mathcal E(\mathcal R)^a{}_b =\delta^{aijkl}_{bpqrs}\mathcal R^{pq}_{ij}\mathcal R^{rs}_{kl}/96$.
For a self-adjoint $H$, let $\sigma_k(H)$ be the sum of products
of $k$ distinct eigenvalues. Define
\begin{align}
&Z=W(S_\eta)+2\sigma_3(S_\eta),\qquad V=\frac12\Sc_{\mathcal R^0}+\sigma_2(S_\eta),\qquad \mathcal E_+=\mathcal E(\mathcal R^0+\mathcal G(S_\eta)). \nonumber
\end{align}
The scalars $Z,V$ and the symmetric tensor $\mathcal E_+$ satisfy
\begin{align}
&Z\ge0,\qquad V\ge0,\qquad \mathcal E_+\ge0. \nonumber
\end{align}
Moreover,
\begin{align}
&\frac{|\mathcal R^0+\mathcal G(S_\eta)|^2-4|\Ric_{\mathcal R^0+\mathcal G(S_\eta)}|^2+\Sc_{\mathcal R^0+\mathcal G(S_\eta)}^2}{24}\ge0. \nonumber
\end{align}
The following exact scalar expansions hold:
\begin{align}
&\sum_{1\le a<b\le m}K_{ab}\sigma_{2}\bigl((\kappa_c)_{c\ne a,b}\bigr) =\sum_{1\le a<b\le m}K_{ab}\sigma_{2}\bigl((\kappa_c+\eta)_{c\ne a,b}\bigr) -(m-3)\eta W(S_\eta) \nonumber\\
&{}+\frac12\binom{m-2}{2}\eta^2\Sc_{\mathcal R^0}, \nonumber\\
&\sigma_4(S)=\sigma_4(S_\eta)-(m-3)\eta\sigma_3(S_\eta) +\binom{m-2}{2}\eta^2\sigma_2(S_\eta) -\binom{m-1}{3}\eta^3\operatorname{tr}S_\eta +\binom m4\eta^4. \label{mass:eq:quartic-expansion}
\end{align}
There are constants depending only on $m$ such that
\begin{align}
&\operatorname{tr}\mathcal E_+=(m-4)\frac{|\mathcal R^0+\mathcal G(S_\eta)|^2-4|\Ric_{\mathcal R^0+\mathcal G(S_\eta)}|^2+\Sc_{\mathcal R^0+\mathcal G(S_\eta)}^2}{24}, \nonumber\\
&|\operatorname{div}_{\Sigma}\mathcal E_+| \le C_m\eta\Sc_g\bigl(\Sc_g+\sigma_2(S_\eta)\bigr), \nonumber\\
&\frac{|\mathcal R^0+\mathcal G(S_\eta)|^2-4|\Ric_{\mathcal R^0+\mathcal G(S_\eta)}|^2+\Sc_{\mathcal R^0+\mathcal G(S_\eta)}^2}{24} \label{mass:eq:positive-divergence}\\
&{}\le\frac{|\Rm_{\Sigma}|^2-4|\Ric_{\Sigma}|^2+\Sc_{\Sigma}^2}{24}+C_m\eta Z+C_m\eta^3\operatorname{tr}S_\eta, \nonumber\\
&\left(\frac{|\Rm_{\Sigma}|^2-4|\Ric_{\Sigma}|^2+\Sc_{\Sigma}^2}{24}\right)_- \le C_m\eta Z+C_m\eta^3\operatorname{tr}S_\eta, \label{mass:eq:shift-scalar}\\
&\tr\bigl((\mathcal E(\Rm_{\Sigma})-\mathcal E_+)\circ S\bigr) \le C_m\bigl(\eta^2Z+\eta^4\operatorname{tr}S_\eta\bigr). \label{mass:eq:shift-pairing}
\end{align}
\end{lemma}
\begin{proof}
\begin{align}
&X\perp Y,\quad|X|=|Y|=1:\quad\sec_{\mathcal R^0+\mathcal G(S_\eta)}(X\wedge Y) \nonumber\\
&{}=\sec_g(X\wedge Y)+\langle S_\eta X,X\rangle\langle S_\eta Y,Y\rangle-\langle S_\eta X,Y\rangle^2\ge0. \nonumber
\end{align}
 For $|e|=1$, delta expansion gives
\begin{align}
&\mathcal E_+(e,e) =\frac1{24}\Bigl(\bigl|(\mathcal R^0+\mathcal G(S_\eta))\!\restriction_{e^\perp}\bigr|^2 -4\bigl|\Ric_{(\mathcal R^0+\mathcal G(S_\eta))\!\restriction_{e^\perp}}\bigr|^2 +\Sc_{(\mathcal R^0+\mathcal G(S_\eta))\!\restriction_{e^\perp}}^2\Bigr). \nonumber
\end{align}
Lemma~\ref{six:lem:P2-positive}, \eqref{six:eq:P2-positive}, on each four-plane gives
\begin{align}
&\mathcal E_+(e,e)\ge0,\qquad\tr\mathcal E_+=(m-4)P_2(\mathcal R^0+\mathcal G(S_\eta)). \nonumber
\end{align}
 With $\nabla=\nabla^\Sigma$, define
\begin{align}
&(\mathcal D\mathcal R)_{aij}{}^{pq} =\nabla_a\mathcal R^{pq}_{ij} +\nabla_i\mathcal R^{pq}_{ja} +\nabla_j\mathcal R^{pq}_{ai}. \nonumber
\end{align}
Differentiate Lemma~\ref{pt:lem:lapse}, \eqref{mass:eq:level-E2}, and antisymmetrize $a,i,j$:
\begin{align}
&(\operatorname{div}\mathcal E(\mathcal R))_b =\tfrac1{144}\delta^{aijkl}_{bpqrs} (\mathcal D\mathcal R)_{aij}{}^{pq}\mathcal R^{rs}_{kl}. \nonumber\\
&\mathcal D\Rm_\Sigma=0,\qquad\nabla\eta=0, \nonumber\\
&(\mathcal G(S+\eta\Id)-\mathcal G(S))^{ab}_{ij}=\eta(S_i^a\delta_j^b+\delta_i^aS_j^b-S_j^a\delta_i^b-\delta_j^aS_i^b)+\eta^2(\delta_i^a\delta_j^b-\delta_j^a\delta_i^b), \nonumber\\
&|\nabla_iS_{jk}-\nabla_jS_{ik}|=|\langle\Rm_g(e_i,e_j)\nu,e_k\rangle|\le|\Rm_g|. \nonumber
\end{align}
 Thus
\begin{align}
&|\mathcal D(\mathcal R^0+\mathcal G(S_\eta))| \le C_m\eta|\Rm_g|\le C_m\eta\Sc_g. \nonumber\\
&|\mathcal R^0+\mathcal G(S_\eta)|\le C_m(\Sc_g+\sigma_2(S_\eta)), \nonumber\\
&|\diver_\Sigma\mathcal E_+|\le C_m\eta\Sc_g(\Sc_g+\sigma_2(S_\eta)). \nonumber
\end{align}
 The scalar expansions~\eqref{mass:eq:quartic-expansion} are
\begin{align}
&\sum_{1\le a<b\le m}K_{ab}\sigma_{2}\bigl((\kappa_c)_{\substack{1\le c\le m\\ c\ne a,b}}\bigr)=\sum_{1\le a<b\le m}K_{ab}\sigma_{2}\bigl(((\kappa_c+\eta))_{\substack{1\le c\le m\\ c\ne a,b}}\bigr) -(m-3)\eta W(S_\eta) \nonumber\\
&{}+\frac12\binom{m-2}{2}\eta^2\Sc_{\mathcal R^0}, \nonumber\\
&\sigma_4(S) =\sigma_4(S_\eta)-(m-3)\eta\sigma_3(S_\eta) +\binom{m-2}{2}\eta^2\sigma_2(S_\eta) -\binom{m-1}{3}\eta^3\operatorname{tr}S_\eta +\binom m4\eta^4. \nonumber
\end{align}
Substitution in \eqref{pt:eq:Gauss-polynomial} gives
\begin{align}
&P_2(\Rm_\Sigma)\ge P_2(\mathcal R^0+\mathcal G(S_\eta))-C_m\eta Z-C_m\eta^3\tr S_\eta\ge-C_m\eta Z-C_m\eta^3\tr S_\eta, \nonumber
\end{align}
proving \eqref{mass:eq:shift-scalar}. Define $\mathcal V_j$ by
\begin{align}
&\mathcal E(\mathcal R^0+\mathcal G(S_\eta-s\Id)) =\mathcal E_+-s\mathcal V_1+s^2\mathcal V_2-s^3\mathcal V_3+s^4\mathcal V_4. \nonumber
\end{align}
For $|e|=1$, set $S_{\eta,e}=\operatorname{pr}_{e^\perp}\circ S_\eta|_{e^\perp}$. For self-adjoint $H:e^\perp\to e^\perp$, set
\begin{align}
&W_e(H)=\tr\Bigl( \bigl(\tfrac12\Sc_{\mathcal R^0\!\restriction_{e^\perp}}\Id -\Ric_{\mathcal R^0\!\restriction_{e^\perp}}\bigr)\circ H\Bigr). \nonumber
\end{align}
Equation~\eqref{mass:eq:quartic-expansion} in dimension $m-1$:
\begin{align}
&\mathcal V_1(e,e)=(m-4)\bigl(W_e(S_{\eta,e})/3+\sigma_3(S_{\eta,e})\bigr),\qquad \mathcal V_2(e,e)=\binom{m-3}{2} \bigl(\Sc_{\mathcal R^0\!\restriction_{e^\perp}}/6+\sigma_2(S_{\eta,e})\bigr), \nonumber\\
&\mathcal V_3(e,e)=\binom{m-2}{3}\operatorname{tr}S_{\eta,e},\qquad \mathcal V_4(e,e)=\binom{m-1}{4}. \nonumber\\
&\binom{k}{\ell}=0\quad(0\le k<\ell),\qquad S_{\eta,e}\ge0,\qquad\sec_{\mathcal R^0|_{e^\perp}}\ge0\ \Longrightarrow\ \mathcal V_j\ge0. \nonumber
\end{align}
 In an eigenframe of $S_\eta$,
\begin{align}
&\tr\bigl(\mathcal V_2\circ S_\eta\bigr)+\operatorname{tr}\mathcal V_1\le C_mZ,\qquad \tr\bigl(\mathcal V_4\circ S_\eta\bigr)+\operatorname{tr}\mathcal V_3\le C_m\operatorname{tr}S_\eta. \nonumber\\
&\tr(\mathcal V_2\circ S_\eta)=\binom{m-3}{2}\bigl(W(S_\eta)/3+3\sigma_3(S_\eta)\bigr), \nonumber\\
&\tr\mathcal V_1=(m-4)(m-3)\bigl(W(S_\eta)/3+\sigma_3(S_\eta)\bigr), \nonumber\\
&\tr(\mathcal V_4\circ S_\eta)=\binom{m-1}{4}\tr S_\eta,\qquad\tr\mathcal V_3=(m-1)\binom{m-2}{3}\tr S_\eta,\qquad S=S_\eta-\eta\Id. \nonumber\\
&\tr\bigl((\mathcal E(\Rm_{\Sigma})-\mathcal E_+)\circ S\bigr)=-\eta \tr\bigl(\mathcal V_1\circ S_\eta\bigr) +\eta^2(\tr\bigl(\mathcal V_2\circ S_\eta\bigr)+\operatorname{tr}\mathcal V_1) -\eta^3(\tr\bigl(\mathcal V_3\circ S_\eta\bigr)+\operatorname{tr}\mathcal V_2) \nonumber\\
&{}+\eta^4(\tr\bigl(\mathcal V_4\circ S_\eta\bigr)+\operatorname{tr}\mathcal V_3) -\eta^5\operatorname{tr}\mathcal V_4. \nonumber\\
&{}-\eta\tr(\mathcal V_1\circ S_\eta)\le0,\quad-\eta^3\bigl(\tr(\mathcal V_3\circ S_\eta)+\tr\mathcal V_2\bigr)\le0,\quad-\eta^5\tr\mathcal V_4\le0, \nonumber\\
&\tr\bigl((\mathcal E(\Rm_\Sigma)-\mathcal E_+)\circ S\bigr)\le C_m(\eta^2Z+\eta^4\tr S_\eta), \nonumber
\end{align}
which is \eqref{mass:eq:shift-pairing}.
\end{proof}

\begin{lemma}\label{pt:lem:fine-approximation}
Assume $(M^n, g)$ is complete non-compact Riemannian manifold with $K_g\geq 0$. On every sufficiently large regular
level $\Sigma_t$, use Notations~\ref{notation:eta}
and~\ref{notation:level-objects}. Then
\begin{align}
&S\ge-\frac{e^{-t}}{64K(t)^2}\Id\ge-\eta(t)\Id. \label{pt:eq:levelwise-shape-lower}
\end{align}
\end{lemma}

\begin{proof}
Lemma~\ref{one:lem:fixed-F}, \eqref{one:eq:fixed-F}, and
Lemma~\ref{one:lem:coefficient}, \eqref{one:eq:h-choice}, give
\begin{align}
&|F-\mathfrak{b}|<1,\qquad |\nabla F|<2,\qquad \Hess F>-\min\left\{(1+\rho)^{-3},\frac{e^{-(2\mathfrak{b}+c_{\mathfrak{b}})}}{4K(2\mathfrak{b}+c_{\mathfrak{b}})^2}\right\}g, \nonumber\\
&|\nabla F|\ge\tfrac12,\qquad \mathfrak{b}\ge t-\tfrac14,\qquad 2\mathfrak{b}+c_{\mathfrak{b}}\ge t\qquad(t\ge t_0). \nonumber
\end{align}
Lemma~\ref{one:lem:coefficient}, \eqref{one:eq:h-choice}, yields
\begin{align}
&S\ge-\frac{h}{32|\nabla F|}\Id\ge-\frac{e^{-(2\mathfrak{b}+c_{\mathfrak{b}})}}{64K(2\mathfrak{b}+c_{\mathfrak{b}})^2}\Id\ge-\frac{e^{-t}}{64K(t)^2}\Id\ge-\eta(t)\Id. \nonumber
\end{align}
\end{proof}

\begin{lemma}
\label{pt:lem:mass}
Let $(M^n,g)$ be a smooth, complete, connected, noncompact
Riemannian manifold and with $\sec_g\ge0$.
Then there is a finite constant $C$,
depending on the fixed manifold and choices but not on $r$, such that
\begin{align}
&\int_{B_p(r)}P_2\,dV\le Cr^{n-4}\qquad(r\ge1). \label{pt:eq:mass-bound}
\end{align}
\end{lemma}

\begin{proof}
\textbf{Step (1)}. Retain $A,N,Y,J$ from
Lemma~\ref{lem:independent-level-limit},
\eqref{eq:level-mean-bound} and~\eqref{eq:level-curvature-limit}:
\begin{align}
&A(t)=\Area(\Sigma_t),\qquad N(t)=\int_{\Sigma_t}H\,dA,\qquad Y(t)=N(t)+m\eta(t)A(t),\qquad J(t)=\int_{\Sigma_t}\Sc_{\Sigma_t}\,dA, \nonumber\\
&A(t)\le Ct^m,\qquad 0\le Y(t)\le Ct^{m-1},\qquad |N(t)|\le Ct^{m-1},\qquad |J(t)|\le Ct^{m-2}. \nonumber
\end{align}
Increase $t_0$ so that Lemma~\ref{one:lem:fixed-F},
\eqref{one:eq:fixed-F}--\eqref{eq:level-radial-support}, gives
\begin{align}
&c=F(p)+\tfrac12,\qquad a=m-4\ge0,\qquad t/2\le t-c\le2t,\qquad \tfrac12\le |\nabla F|\le2,\qquad t/2\le\rho\le2t, \nonumber\\
&\Sc_g|_{\Sigma_t}\le K(t),\qquad t\frac{e^{-t}}{K(t)^2}\le1. \nonumber
\end{align}
Retain $K$ from Lemma~\ref{one:lem:coefficient},
\eqref{eq:level-curvature-majorant}. By that lemma,
\eqref{one:eq:h-choice}, and
Lemma~\ref{pt:lem:fine-approximation},
\eqref{pt:eq:levelwise-shape-lower}, on every sufficiently large level set
\begin{align}
&\widehat S=S+\frac{e^{-t}}{K(t)^2}\Id\ge0,\qquad 0<\frac{e^{-t}}{K(t)^2}\le e^{-t}\le(1+t)^{-3}\le\eta(t), \label{mass:eq:exponential-shift}\\
&\nabla^{\Sigma_t}\!\left(\frac{e^{-t}}{K(t)^2}\right)=0. \nonumber
\end{align}
The coefficient is constant on each fixed
level; no derivative of it with respect to $t$ is used. The function
$\eta(t)$ of Notation~\ref{notation:eta} is unchanged.

For each fixed $t$, apply Lemma~\ref{mass:lem:shift} with its constant
parameter equal to $e^{-t}K(t)^{-2}$. Using $\widehat S$ from
\eqref{mass:eq:exponential-shift}, define
\begin{align}
&\widehat{\mathcal R}=\mathcal R^0+\mathcal G(\widehat S),\qquad \mathcal E_+=\mathcal E(\widehat{\mathcal R}),\qquad H_2(t)=\int_{\Sigma_t}P_2(\widehat{\mathcal R})\,dA, \nonumber\\
&Z=W(\widehat S)+2\sigma_3(\widehat S),\qquad V=\tfrac12\Sc_{\mathcal R^0}+\sigma_2(\widehat S),\qquad \overline Z(t)=\int_{\Sigma_t}Z\,dA,\qquad \overline V(t)=\int_{\Sigma_t}V\,dA, \nonumber\\
&E(t)=(m-2)\frac{e^{-t}}{K(t)^2}\int_{\Sigma_t}\left(\tfrac12\Sc_{\mathcal R^0}+2\sigma_2(\widehat S)\right)dA+2\binom m3\frac{e^{-3t}}{K(t)^6}A(t). \nonumber
\end{align}
These quantities satisfy $Z,V,E,H_2\ge0$. In particular,
\begin{align}
&0\le\int_{\Sigma_t}\tr\widehat S\,dA=N(t)+m\frac{e^{-t}}{K(t)^2}A(t)\le Y(t). \nonumber
\end{align}
The cubic expansion and the scalar Gauss equation give
\begin{align}
&G=\overline Z-E+(m-1)(m-2)\frac{e^{-2t}}{K(t)^4}\int_{\Sigma_t}\tr\widehat S\,dA,\qquad \overline Z\le G+E,\qquad G\ge-E, \label{pt:eq:negative-parts}\\
&\overline V=\tfrac12J+(m-1)\frac{e^{-t}}{K(t)^2}N+\binom m2\frac{e^{-2t}}{K(t)^4}A,\qquad 0\le\overline V\le Ct^{m-2}, \nonumber\\
&0\le E(t)\le C\frac{e^{-t}}{K(t)^2}\overline V+C\frac{e^{-3t}}{K(t)^6}A\le C\frac{e^{-t}}{K(t)^2}t^{m-2}. \label{pt:eq:E-bound}
\end{align}
Here $Y$ retains the coefficient $\eta(t)=16(1+t)^{-3}$; it is used only
as an upper bound for $\int_{\Sigma_t}\tr\widehat S\,dA$.

Equation~\eqref{pt:eq:Gauss-polynomial} gives
\begin{align}
&\frac83P_2(\Rm_{\Sigma_t})-\left(P_2(\Rm_{\Sigma_t})+\frac13\sum_{i<j}K_{ij}\sigma_2\bigl((\kappa_k)_{k\ne i,j}\bigr)+\frac53\sigma_4(S)\right) \nonumber\\
&\qquad=\frac53P_2(\mathcal R^0)+\frac29\sum_{i<j}K_{ij}\sigma_2\bigl((\kappa_k)_{k\ne i,j}\bigr)\ge-C_m\frac{e^{-t}}{K(t)^2}W(\widehat S). \nonumber
\end{align}
The last inequality is Lemma~\ref{mass:lem:shift},
\eqref{mass:eq:quartic-expansion}, together with $P_2(\mathcal R^0)\ge0$.
The same lemma, \eqref{mass:eq:shift-scalar}, and $1/2\le |\nabla F|\le2$ give
\begin{align}
&\int_{\Sigma_t}|\nabla F|^{-1}P_2(\Rm_{\Sigma_t})\,dA\le2J_2+\tfrac32\int_{\Sigma_t}\bigl(P_2(\Rm_{\Sigma_t})\bigr)_-\,dA \nonumber\\
&\qquad\le2J_2+C_m\frac{e^{-t}}{K(t)^2}\overline Z+C_m\frac{e^{-3t}}{K(t)^6}Y. \nonumber
\end{align}
Use Lemma~\ref{pt:lem:lapse}, \eqref{pt:eq:lapse-identity}, with $\lambda=|\nabla F|$.
Choose the coefficient of the nonnegative $\overline Z$ term first,
then apply \eqref{pt:eq:negative-parts} to obtain
\begin{align}
&G'+3N_2\le16J_2+C_m\frac{e^{-t}}{K(t)^2}G+f_G,\qquad f_G=C_m\left(\frac{e^{-t}}{K(t)^2}E+\frac{e^{-3t}}{K(t)^6}Y\right)\ge0. \label{pt:eq:G-differential}
\end{align}

By Lemma~\ref{mass:lem:shift}, \eqref{mass:eq:positive-divergence},
\begin{align}
&\mathcal E_+\ge0,\qquad \tr\mathcal E_+=aP_2(\widehat{\mathcal R}),\qquad |\diver_{\Sigma_t}\mathcal E_+|\le C_m\frac{e^{-t}}{K(t)^2}\Sc_g\bigl(\Sc_g+\sigma_2(\widehat S)\bigr), \nonumber\\
&\int_{\Sigma_t}|\diver_{\Sigma_t}\mathcal E_+|\,dA\le C_m\frac{e^{-t}}{K(t)^2}K(t)\bigl(K(t)A+\overline V\bigr). \nonumber
\end{align}
Only tangential derivatives occur here; the coefficient is constant on
$\Sigma_t$. Put $w=\langle\nabla(\rho^2/2),\nu\rangle$.
Lemma~\ref{lem:independent-level-limit},
\eqref{eq:level-distance-hessian} and~\eqref{eq:level-pairing-bounds}, gives,
for almost every $t$,
\begin{align}
&\Hess_{\Sigma_t}(\rho^2/2)\le(g_t-wS)\,dA,\qquad 0\le w-(t-c)/|\nabla F|\le2t,\qquad |\nabla^{\Sigma_t}(\rho^2/2)|\le2t, \nonumber\\
&\tr(\mathcal E_+\circ S)\ge-\frac{e^{-t}}{K(t)^2}\tr\mathcal E_+, \nonumber\\
&\int_{\Sigma_t}w\tr(\mathcal E_+\circ S)\,dA\le\int_{\Sigma_t}\tr\mathcal E_+\,dA+\int_{\Sigma_t}\left\langle\diver_{\Sigma_t}\mathcal E_+,\nabla^{\Sigma_t}(\rho^2/2)\right\rangle dA, \nonumber\\
&(t-c)\int_{\Sigma_t}|\nabla F|^{-1}\tr(\mathcal E_+\circ S)\,dA\le a\left(1+\frac{2te^{-t}}{K(t)^2}\right)H_2+C_mt\frac{e^{-t}}{K(t)^2}K(t)\bigl(K(t)A+\overline V\bigr). \nonumber
\end{align}
The pairing is between tensor-valued measures and the smooth nonnegative
tensor $\mathcal E_+$ on the fixed level. Its singular Hessian part is
nonpositive, the cut-locus intersection is area-null for almost every
level, and $\partial\Sigma_t=\varnothing$; hence no boundary term is omitted.
Lemma~\ref{mass:lem:shift},
\eqref{mass:eq:shift-scalar} and~\eqref{mass:eq:shift-pairing}, gives
\begin{align}
&H_2\le J_2+C_m\frac{e^{-t}}{K(t)^2}\overline Z+C_m\frac{e^{-3t}}{K(t)^6}Y, \nonumber\\
&\int_{\Sigma_t}\bigl(P_2(\Rm_{\Sigma_t})\bigr)_-\,dA\le C_m\frac{e^{-t}}{K(t)^2}\overline Z+C_m\frac{e^{-3t}}{K(t)^6}Y, \nonumber\\
&\int_{\Sigma_t}|\nabla F|^{-1}\tr\bigl((\mathcal E(\Rm_{\Sigma_t})-\mathcal E_+)\circ S\bigr)\,dA\le C_m\frac{e^{-2t}}{K(t)^4}\overline Z+C_m\frac{e^{-4t}}{K(t)^8}Y. \nonumber
\end{align}
Combine these inequalities with Lemma~\ref{pt:lem:lapse},
\eqref{mass:eq:variation}, $te^{-t}K(t)^{-2}\le1$, and
\eqref{pt:eq:negative-parts}:
\begin{align}
&J_2'\le\frac{a}{t-c}J_2+C\frac{e^{-t}}{K(t)^2}(J_2)_++\frac{Ce^{-t}}{K(t)^2(t-c)}G_++f_J, \label{mass:eq:J-differential}\\
&f_J=C\left[\frac{e^{-t}E}{K(t)^2(t-c)}+\frac{e^{-3t}Y}{K(t)^6(t-c)}+e^{-t}A+\frac{e^{-t}}{K(t)}\overline V\right]\ge0. \nonumber
\end{align}
For $m=4$, the same variation identity gives
$\mathcal E(\Rm_{\Sigma_t})=0$ and $J_2'=0$; the inequalities hold with
$a=0$. No derivative of $e^{-t}K(t)^{-2}$ with respect to $t$ is taken.
By \eqref{pt:eq:E-bound},
\begin{align}
&0\le f_G(t)\le C\bigl(e^{-2t}t^{m-2}+e^{-3t}t^{m-1}\bigr), \nonumber\\
&0\le(t-c)^{-a}f_J(t)\le C\left[\frac{e^{-2t}}{K(t)^4}t+\frac{e^{-3t}}{K(t)^6}t^2+e^{-t}t^4+\frac{e^{-t}}{K(t)}t^2\right]\le Ce^{-t}(1+t^4), \nonumber\\
&\int_{t_0}^{\infty}f_G(t)\,dt<\infty,\qquad \int_{t_0}^{\infty}(t-c)^{-a}f_J(t)\,dt<\infty. \label{mass:eq:forcing-bounds}
\end{align}

Set
\begin{align}
&x(t)=(t-c)^{-a}J_2(t),\qquad X(t)=\max\left\{0,\sup_{t_0\le s\le t}x(s)\right\}. \nonumber
\end{align}
The function $J_2$ is smooth on regular bands. Moreover,
Lemma~\ref{six:lem:P2-positive}, \eqref{six:eq:P2-positive}, and
\eqref{eq:P2-four-plane-sum} give $P_2\ge0$ and $N_2\ge0$.
Multiplying \eqref{pt:eq:G-differential} by
$\exp(-C_m\int_{t_0}^t e^{-s}K(s)^{-2}\,ds)$ gives
\begin{align}
&0\le\int_{t_0}^{\infty}\frac{e^{-s}}{K(s)^2}\,ds\le e^{-t_0}, \nonumber\\
&G_+(t)\le C\left(1+\int_{t_0}^t(J_2(s))_+\,ds+\int_{t_0}^tf_G(s)\,ds\right)\le C(t-c)^{a+1}(1+X(t)), \nonumber\\
&\int_{t_0}^t(J_2(s))_+\,ds\le X(t)\int_{t_0}^t(s-c)^a\,ds\le C(t-c)^{a+1}X(t). \nonumber
\end{align}
Equations~\eqref{mass:eq:J-differential} and~\eqref{mass:eq:forcing-bounds}
therefore imply, for almost every $t$,
\begin{align}
&x'(t)\le C\frac{e^{-t}}{K(t)^2}x_+(t)+C\frac{e^{-t}}{K(t)^2}(1+X(t))+(t-c)^{-a}f_J(t), \nonumber\\
&X(t)\le C+C\int_{t_0}^t\frac{e^{-s}}{K(s)^2}(1+X(s))\,ds,\qquad 1+X(t)\le C\exp\left(C\int_{t_0}^t\frac{e^{-s}}{K(s)^2}\,ds\right)\le C, \nonumber\\
&J_2(t)\le Ct^a,\qquad G_+(t)\le Ct^{a+1}. \label{pt:eq:G-upper}
\end{align}

\textbf{Step (2)}. Equations~\eqref{pt:eq:negative-parts} and
\eqref{pt:eq:E-bound} give, for every sufficiently large $t$,
\begin{align}
&G(t)\ge-E(t)\ge-Ce^{-t}t^{m-2},\qquad \int_{t_0}^{\infty}\frac{e^{-t}}{K(t)^2}G_+(t)\,dt\le C\int_{t_0}^{\infty}e^{-t}t^{a+1}\,dt<\infty. \nonumber
\end{align}
Integrating \eqref{pt:eq:G-differential} and using
\eqref{pt:eq:G-upper} and~\eqref{mass:eq:forcing-bounds} gives
\begin{align}
&3\int_{t_0}^tN_2(s)\,ds\le G(t_0)-G(t)+16\int_{t_0}^t(J_2(s))_+\,ds \nonumber\\
&\qquad+C_m\int_{t_0}^t\frac{e^{-s}}{K(s)^2}G_+(s)\,ds+\int_{t_0}^tf_G(s)\,ds\le Ct^{a+1}=Ct^{n-4}. \nonumber
\end{align}
Coarea and Lemmas~\ref{lem:geometry} and~\ref{one:lem:fixed-F},
\eqref{eq:convex-approximation} and~\eqref{one:eq:fixed-F}, give
\begin{align}
&\int_{\{t_0<F<t\}}P_2\,dV=\int_{t_0}^tN_2(s)\,ds\le Ct^{n-4}, \nonumber\\
&F\le\rho+\tfrac14,\qquad B_p(r)\subset\{F<2r\}\quad(r\text{ sufficiently large}),\qquad \{F\le t_0\}\Subset M, \nonumber\\
&\int_{B_p(r)}P_2\,dV\le\int_{\{F\le t_0\}}P_2\,dV+\int_{\{t_0<F<2r\}}P_2\,dV\le Cr^{n-4}. \nonumber
\end{align}
The first integral is finite on the fixed compact set. Increasing $C$ on
bounded $r$-intervals proves \eqref{pt:eq:mass-bound} for every $r\ge1$.
\end{proof}

\begin{lemma}
\label{six:lem:compact}
Let $(M^n,g)$ be a smooth, complete, connected, noncompact
Riemannian manifold and with $\sec_g\ge0$.
For every sufficiently large regular value $t$, use
Notation~\ref{notation:level-objects} and choose an orthonormal
eigenframe of $S$.
Then
\begin{align}
&\lim_{r\to\infty}\frac1r\int_r^{2r}\int_{\Sigma_t}\bigl((\det S)_-+(Q_n(S))_-\bigr)\,dA\,dt=0. \label{six:eq:negative-determinants}
\end{align}
\end{lemma}
\begin{proof}
By Lemma~\ref{one:lem:fixed-F}, equation~\eqref{one:eq:fixed-F}, choose $T<\infty$ so that on the regular exterior $\{F>T\}$,
\begin{align}
&|\nabla F|\ge\frac12,\qquad 0<h\le1,\qquad A=\Hess F+h\Id\ge0,\qquad B=\Hess F+\Id\ge A,\qquad S \nonumber\\
&{}=|\nabla F|^{-1}(A|_{T\Sigma_t}-h\Id), \nonumber\\
&\det S=|\nabla F|^{1-n}\sum_{k=0}^{n-1}(-h)^{n-1-k}\sigma_k(A|_{T\Sigma_t}),\qquad 0\le\sigma_k(A|_{T\Sigma_t})\le\sigma_k(A)\le\sigma_k(B), \nonumber\\
&(\det S)_- \le C_nh\sum_{k=0}^{n-2}\sigma_k(B). \nonumber
\end{align}
The first comparison sums nonnegative principal minors; the second is eigenvalue monotonicity. For $n\ge4$, in the same eigenframe,
\begin{align}
&0\le K_{ab}\le\frac{\Sc_g}{2},\qquad \kappa_c=|\nabla F|^{-1}\bigl(\lambda_c(A|_{T\Sigma_t})-h\bigr), \nonumber\\
&\left(\prod_{c\notin\{a,b\}}\kappa_c\right)_-\le C_nh\sum_{k=0}^{n-4}\sigma_k(B), \nonumber\\
&\left(\sum_{a<b}K_{ab}\prod_{c\notin\{a,b\}}\kappa_c\right)_- \le C_nh\Sc_g\sum_{k=0}^{n-4}\sigma_k(B). \nonumber
\end{align}
For $n=3$, $Q_3(S)=K_{12}\ge0$. Coarea, Tonelli, and Lemma~\ref{one:lem:coefficient}, equation~\eqref{one:eq:weighted-Hessians}, give
\begin{align}
&\int_T^\infty\int_{\Sigma_t}\bigl((\det S)_-+(Q_n(S))_-\bigr)\,dA\,dt\le C_n\int_{\{F>T\}}h(1+\Sc_g)\sum_{k=0}^{n-2}\sigma_k(B)|\nabla F|\,dV<\infty. \nonumber
\end{align}
For $r>T$ and $G(t)=\int_{\Sigma_t}\bigl((\det S)_-+(Q_n(S))_-\bigr)\,dA$,
\begin{align}
&0\le\frac1r\int_r^{2r}G(t)\,dt\le\frac1r\int_T^\infty G(t)\,dt,\qquad \lim_{r\to\infty}\frac1r\int_T^\infty G(t)\,dt=0, \nonumber\\
&0\le\eta_r\le C/r,\quad\supp\eta_r\subset(T,\infty)\quad\Longrightarrow\quad 0\le\int\eta_rG\,dt\le\frac Cr\int_T^\infty G(t)\,dt. \nonumber
\end{align}
\end{proof}

\section{Comparison errors}\label{sec:quadratic-comparison-data}\label{sec:uniform-comparison-data}

Let $n\ge3$. Retain $E_1=(\Sc/2)\Id-\Ric$ from
Lemma~\ref{lem:first-curvature-contraction},
equations~\eqref{eq:first-curvature-contraction} and
\eqref{pt:eq:E1-properties}. Lowering its raised index gives the
symmetric tensor $\Sc g/2-\Ric$.

\begin{lemma}\label{ub:lem:smoothing}
Let $U$ be open in a smooth Riemannian manifold, $u:U\to\R$ locally
$1$-Lipschitz, and $a\ge0$. Suppose $\Hess u\le ag$ on $U$ in the
distributional sense. Given $\tau>0$ and a positive continuous
function $e:U\to(0,\infty)$, there is $u_e\in C^\infty(U)$ with
\begin{align}
&|u_e-u|<e,\qquad |\nabla u_e|<2,\qquad \Hess u_e<(a+\tau)g. \label{ub:eq:smoothing}
\end{align}
\end{lemma}

\begin{proof}
For every $z\in U$, $\Hess(d_g(z,\cdot)^2/2)|_z=g_z$; hence some normal ball $B\ni z$ satisfies
\begin{align}
&(a+\tau/2)\Hess\frac{d_g(z,\cdot)^2}{2}\ge ag, \nonumber\\
&\Hess\left(-u+(a+\tau/2)\frac{d_g(z,\cdot)^2}{2}\right)\ge0\quad\text{on }B. \nonumber
\end{align}
The local convexity condition of \cite[p.~60]{GW}, with parameter $-a-\tau$, and \cite[Proposition~2.3, p.~62]{GW} yield
\begin{align}
&\exists v\in C^\infty(U):\quad |v+u|<e,\quad |\nabla v|<2,\quad \Hess v>-(a+\tau)g, \nonumber\\
&u_e:=-v,\qquad |u_e-u|<e,\quad |\nabla u_e|<2,\quad \Hess u_e<(a+\tau)g. \nonumber
\end{align}
\end{proof}

\begin{lemma}\label{pt:lem:smooth-comparison}
Let $(M^n,g)$ be a smooth, complete, connected, noncompact
Riemannian manifold without boundary and with $\sec_g\ge0$.
Fix $0<a'<a<b<b'<\infty$, put
\begin{align}
&\Omega'_r=\{a'r<F<b'r\}, \nonumber
\end{align}
and put $\phi=F^2/2$ and
$E_1=(\Sc_g/2)\Id-\Ric_g$.
For every sufficiently large $r$, on $\Omega'_r$ there
are smooth functions $\psi_r^{a,b}$, constants $c_r^{a,b}>0$ and
$\varepsilon_r^{a,b}\ge0$, and symmetric tensors $C_r^{a,b}$ such that
\begin{align}
&C_r^{a,b}=c_r^{a,b}\Id-\Hess\psi_r^{a,b}\ge0,\qquad \Hess\phi\ge-\varepsilon_r^{a,b}\Id, \nonumber\\
&\lim_{r\to\infty}c_r^{a,b}=1,\qquad \lim_{r\to\infty}\varepsilon_r^{a,b}=0,\qquad \lim_{r\to\infty}r^{-1} \sup|\nabla\phi-\nabla\psi_r^{a,b}|=0, \nonumber\\
&\lim_{r\to\infty}r^{2-n} \int_{\{ar<F<br\}}\tr\bigl(E_1\circ C_r^{a,b}\bigr)\,dV=0. \label{pt:eq:weighted-Hessian-small}
\end{align}
\begin{align}
&\lim_{r\to\infty}r^{-2}\sup|\psi_r^{a,b}-\rho^2/2|=0. \label{qc:eq:value-comparison}
\end{align}
\end{lemma}

\begin{proof}
On the fixed larger compact annuli, Lemma~\ref{one:lem:fixed-F} gives
\begin{align}
&\Hess\phi=dF\otimes dF+F\Hess F\ge-\varepsilon_r^{a,b}g,\qquad \lim_{r\to\infty}\varepsilon_r^{a,b}=0,\qquad \Hess(\rho^2/2)\le g\,dV. \nonumber
\end{align}
The last inequality is distributional Hessian comparison. Apply Lemma~\ref{ub:lem:smoothing}, equation~\eqref{ub:eq:smoothing}, to $\rho^2/2$ divided by a fixed multiple of $r$, and rescale back:
\begin{align}
&\Hess\psi_r^{a,b}\le c_r^{a,b}g,\qquad C_r^{a,b}=c_r^{a,b}\Id-\Hess\psi_r^{a,b}\ge0,\qquad \lim_{r\to\infty}c_r^{a,b}=1. \nonumber
\end{align}
For each fixed $r$, the smoothing error $\sup|\psi_r^{a,b}-\rho^2/2|$ is arbitrarily small. Almost everywhere,
\begin{align}
&|\rho\,d\rho-F\,dF|\le\alpha_r r, \qquad \lim_{r\to\infty}\alpha_r=0. \nonumber
\end{align}
On a still larger compact annulus choose $M_r\ge\sup|\Hess\phi|$ and $\ell_r>0$ so that $\exp_x(sv)$ remains in that annulus for every $x\in\Omega'_r$, $|v|=1$, and $|s|\le\ell_r$. The two difference quotients and Taylor's formula give
\begin{align}
&|d\psi_r^{a,b}-d\phi| \le\alpha_r r+\frac{2\sup|\psi_r^{a,b}-\rho^2/2|}{\ell_r} +\frac{c_r^{a,b}+M_r}{2}\ell_r. \nonumber
\end{align}
Choose $\gamma_r>0$, $\lim_{r\to\infty}\gamma_r=0$, and then, in this order,
\begin{align}
&\frac{c_r^{a,b}+M_r}{2}\ell_r\le r\gamma_r,\qquad \sup|\psi_r^{a,b}-\rho^2/2|\le\min\{r\gamma_r\ell_r/2,r^2\gamma_r\}, \nonumber\\
&r^{-1}\sup|d\psi_r^{a,b}-d\phi|\le\alpha_r+2\gamma_r,\qquad r^{-2}\sup|\psi_r^{a,b}-\rho^2/2|\le\gamma_r. \nonumber
\end{align}
This proves the gradient limit and \eqref{qc:eq:value-comparison}. Fix $L=L(M,g)$ and
\begin{align}
&\beta\in C_c^\infty((0,\infty)),\qquad 0\le\beta\le1,\qquad \beta|_{[a,b]}=1,\qquad \{F/r\in\supp\beta\}\Subset\Omega'_r. \nonumber
\end{align}
Extend the supported integrands by zero. Lemma~\ref{lem:first-curvature-contraction}, equation~\eqref{pt:eq:E1-properties}, and integration by parts give
\begin{align}
&\int_M\beta(F/r)\tr\bigl(E_1\circ C_r^{a,b}\bigr)\,dV =\frac{(n-2)c_r^{a,b}}2\int_M\beta(F/r)\Sc\,dV \nonumber\\
&{}+\int_M E_1\bigl(\nabla(\beta(F/r)),\nabla\psi_r^{a,b}\bigr)\,dV. \label{eq:common-Hessian-identity}
\end{align}
By Lemma~\ref{lem:scalar-radial-tests}, equation~\eqref{eq:common-scalar-F-tests},
\begin{align}
&\lim_{r\to\infty}r^{2-n}\frac{(n-2)c_r^{a,b}}2\int_M\beta(F/r)\Sc\,dV =\frac{(n-2)^2L}2\int_0^\infty\beta(t)t^{n-3}\,dt. \nonumber
\end{align}
Also,
\begin{align}
&\int_M E_1\bigl(\nabla(\beta(F/r)),\nabla\psi_r^{a,b}\bigr)\,dV =\int_M E_1\bigl(\nabla(\beta(F/r)),F\nabla F\bigr)\,dV \nonumber\\
&{}+\int_M E_1\bigl(\nabla(\beta(F/r)),\nabla\psi_r^{a,b}-F\nabla F\bigr)\,dV. \nonumber
\end{align}
Lemma~\ref{lem:scalar-radial-tests}, equation~\eqref{eq:common-scalar-F-tests}, applied to a nonnegative smooth majorant of $\mathbf1_{\supp\beta}$, gives
\begin{align}
&\sup_{r\ge r_0}r^{2-n}\int_{\{F/r\in\supp\beta\}}\Sc\,dV<\infty,\qquad 0\le E_1\le\Sc g/2,\qquad |\nabla F|<2. \nonumber
\end{align}
Hence
\begin{align}
&r^{2-n}\left|\int_M E_1\bigl(\nabla(\beta(F/r)),\nabla\psi_r^{a,b}-F\nabla F\bigr)\,dV\right| \nonumber\\
&{}\le\|\beta'\|_\infty\left(r^{-1}\sup|\nabla\psi_r^{a,b}-F\nabla F|\right)\left(r^{2-n}\int_{\{F/r\in\supp\beta\}}\Sc\,dV\right). \nonumber
\end{align}
Thus
\begin{align}
&\lim_{r\to\infty}r^{2-n}\int_M E_1\bigl(\nabla(\beta(F/r)),\nabla\psi_r^{a,b}-F\nabla F\bigr)\,dV=0. \nonumber
\end{align}
Since $\nabla F=|\nabla F|\nu$,
\begin{align}
&\int_M E_1\bigl(\nabla(\beta(F/r)),F\nabla F\bigr)\,dV =\int_M\beta'(F/r)\frac Fr |\nabla F|^2\left(\frac{\Sc}{2}-\Ric(\nu,\nu)\right)\,dV. \nonumber
\end{align}
By Lemma~\ref{lem:scalar-radial-tests}, equation~\eqref{eq:common-radial-F-tests}, on the fixed support,
\begin{align}
&\lim_{r\to\infty}r^{2-n}\int_M\beta'(F/r)\frac Fr |\nabla F|^2\Ric(\nu,\nu)\,dV=0. \nonumber
\end{align}
Here $\sup|\beta'(F/r)(F/r)|\nabla F|^2|<\infty$ and $\Ric(\nu,\nu)\ge0$. Moreover,
\begin{align}
&r^{2-n}\left|\int_M\beta'(F/r)\frac Fr(|\nabla F|^2-1)\Sc\,dV\right|\le\sup_{t>0}|t\beta'(t)|\sup_{\{F/r\in\supp\beta\}}||\nabla F|^2 \nonumber\\
&{}-1|\left(r^{2-n}\int_{\{F/r\in\supp\beta\}}\Sc\,dV\right), \nonumber\\
&\lim_{r\to\infty}r^{2-n}\int_M\beta'(F/r)\frac Fr(|\nabla F|^2-1)\Sc\,dV=0. \nonumber
\end{align}
Lemma~\ref{lem:scalar-radial-tests},
equation~\eqref{eq:common-scalar-F-tests}, with the test $t\beta'(t)$,
now gives
\begin{align}
&\lim_{r\to\infty}\frac{r^{2-n}}2\int_M\beta'(F/r)\frac Fr\Sc\,dV =\frac{(n-2)L}2\int_0^\infty\beta'(t)t^{n-2}\,dt, \nonumber\\
&\lim_{r\to\infty}r^{2-n}\int_M E_1\bigl(\nabla(\beta(F/r)),\nabla\psi_r^{a,b}\bigr)\,dV =\frac{(n-2)L}2\int_0^\infty\beta'(t)t^{n-2}\,dt. \nonumber
\end{align}
Combining the two limits in \eqref{eq:common-Hessian-identity},
\begin{align}
&\lim_{r\to\infty}r^{2-n}\int_M\beta(F/r)\tr\bigl(E_1\circ C_r^{a,b}\bigr)\,dV \nonumber\\
&{}=\frac{(n-2)L}2\left((n-2)\int_0^\infty\beta(t)t^{n-3}\,dt+\int_0^\infty\beta'(t)t^{n-2}\,dt\right)=0. \nonumber
\end{align}
Indeed, $[\beta(t)t^{n-2}]_0^\infty=0$. Since $E_1,C_r^{a,b}\ge0$ and $\beta|_{[a,b]}=1$,
\begin{align}
&0\le r^{2-n}\int_{\{ar<F<br\}}\tr\bigl(E_1\circ C_r^{a,b}\bigr)\,dV \le r^{2-n}\int_M\beta(F/r)\tr\bigl(E_1\circ C_r^{a,b}\bigr)\,dV. \nonumber
\end{align}

\end{proof}

Fix $a=1/2$ and $b=5/2$. Apply
Lemma~\ref{pt:lem:smooth-comparison} on an annulus containing
$\{ar\le F\le br\}$ and all the finitely many nested cutoff
annuli below. Write $\psi_r=\psi_r^{a,b}$ and suppress these fixed
superscripts on $c_r,\varepsilon_r,C_r$ as well.
The error $\delta_r$ in \eqref{qc:eq:delta-data} is enlarged when
necessary to absorb the error in Lemma~\ref{pt:lem:smooth-comparison},
equation~\eqref{qc:eq:value-comparison}, and the exterior errors of the
fixed $F$, always with $\lim_{r\to\infty}\delta_r=0$.
Only the auxiliary squared-distance comparison depends on $r$.
Constants depend on the fixed metric, $F$, and the fixed annuli and
cutoffs, and are independent of $r$ and the interpolation parameter.

The data below are defined for $n\ge3$; the quadratic error lemma
requires $n\ge5$. Retain $a=1/2$, $b=5/2$ and the single comparison
function $\psi_r$ fixed above. Choose the nested annuli by
\begin{align}
&a_i=2^{-i}a_0,\qquad b_i=2^ib_0,\qquad 0\le i\le8, \nonumber
\end{align}
and choose nonnegative smooth cutoffs $\zeta_i$ equal to one on
$\{a_{i}r<F<b_{i}r\}$, supported in $\{a_{i+1}r<F<b_{i+1}r\}$ ($0\le i\le7$),
with $0\le\zeta_i\le1$ and $|\nabla\zeta_i|\le C(a_0,b_0)/r$.
They are functions of $F/r$.
Choose the comparison on $\{2^{-10}a_0r<F<2^{10}b_0r\}$ and choose the additional outer cutoff to equal one on $\{a_8r<F<b_8r\}$, to be supported in $\{2^{-9}a_0r<F<2^9b_0r\}$, and to satisfy $|\nabla\zeta|\le C(a_0,b_0)/r$.
The proof of Lemma~\ref{pt:lem:smooth-comparison},
equation~\eqref{pt:eq:weighted-Hessian-small}, applied to this cutoff,
gives the same weighted smallness on this outermost annulus for this
same $\psi_r$.
No new comparison is selected for an inner cutoff. Put
\begin{align}
&\phi=F^2/2,\qquad C=C_r=c_r\Id-\Hess\psi_r,\qquad z=z_r=\nabla(\phi-\psi_r), \nonumber\\
&d=d_r=c_r+1+\varepsilon_r,\qquad D=D_r=d_r\Id+\Hess(\phi-\psi_r)=\Id+(\Hess\phi+\varepsilon_r\Id)+C_r\ge\Id \nonumber\\
&{}+C_r. \label{qc:eq:comparison-D}
\end{align}
Choose $0<\delta_r\le1$ for all sufficiently large $r$ such that
\begin{align}
&\lim_{r\to\infty}\delta_r=0,\qquad |c_r-1|+\varepsilon_r+r^{-1}\sup_{\{a_{8}r<F<b_{8}r\}}|z_r| \le\delta_r, \nonumber\\
&r^{2-n}\int_{\{a_{8}r<F<b_{8}r\}}\tr\bigl(E_1\circ C_r\bigr)\,dV\le\delta_r. \label{qc:eq:delta-data}
\end{align}
The constants depend only on the fixed data listed above.
All estimates concern sufficiently large $r$.

\begin{lemma}
\label{qc:lem:quadratic-error}\label{qc:lem:cubic-error}
\textup{(i)} Assume $n\ge5$. Use the tensors in
\eqref{uni:eq:normalized-tensor}--\eqref{unified:gen:H-definition}
and all the comparison data in
Section~\ref{sec:quadratic-comparison-data}, including
\eqref{qc:eq:comparison-D}--\eqref{qc:eq:delta-data}.
Then there is a
constant $C=C(M,g,F,a_0,b_0)$, independent of $r$, such that, for
every sufficiently large $r$,
\begin{align}
&\int_{\{a_{0}r<F<b_{0}r\}}\tr\bigl(V_1(D_r)\circ D_r\bigr)\,dV\le Cr^{n-2}, \label{qc:eq:DD-mass}\\
&\int_{\{a_{0}r<F<b_{0}r\}}|V_2(C_r)|\,dV \le C\delta_r r^{n-2}, \label{qc:eq:CC-small}\\
&\int_{\{a_{0}r<F<b_{0}r\}} \left|V_2(\Hess\psi_r) -(n-3)(n-4)E_1\right|\,dV \le C\delta_r r^{n-2}. \label{qc:eq:quadratic-replacement}
\end{align}
More generally, for each fixed $0\le j\le7$, the same three estimates
hold with $\{a_jr<F<b_jr\}$ in place of $\{a_0r<F<b_0r\}$,
provided the cutoffs use the already fixed annuli with larger indices.
Their constants remain independent of $r$.

\textup{(ii)} Assume $n\ge6$ and retain exactly the comparison
functions, tensors, annuli, cutoffs, and constants from
part~\textup{(i)} and Section~\ref{sec:quadratic-comparison-data}.
Use the function normalization in Lemma~\ref{pt:lem:smooth-comparison},
equation~\eqref{qc:eq:value-comparison}, and enlarge $\delta_r$,
still with $\delta_r\to0$, so that
\begin{align}
&r^{-2}\sup_{\{a_8r<F<b_8r\}}|\psi_r-\rho^2/2| +\sup_{\{a_8r<F<b_8r\}}\bigl(|F/\rho-1|+\bigl||\nabla F|-1\bigr|\bigr) \le\delta_r. \nonumber
\end{align}
Then
\begin{align}
&\int_{\{a_0r<F<b_0r\}} \left[\tr(E_2\circ C_r)+\tr(E_2\circ D_r)\right]\,dV \le Cr^{n-4}, \label{qc:eq:E2-mixed-mass}\\
&\int_{\{a_{0}r<F<b_{0}r\}}|V_3(C_r)|\,dV \le C\delta_r r^{n-2}, \label{qc:eq:CCC-small}\\
&\int_{\{a_{0}r<F<b_{0}r\}} \bigl|V_3(\Hess\psi_r) -(n-3)(n-4)(n-5)E_1\bigr|\,dV \le C\delta_r r^{n-2}. \label{qc:eq:cubic-replacement}
\end{align}
When $n=6$,
\begin{align}
&\lim_{r\to\infty}r^{-4}\int_{\{a_{0}r<F<b_{0}r\}} \left|\frac1{6}V_3(\Hess\psi_r) -E_1\right|\,dV=0. \label{qc:eq:six-dimensional-replacement}
\end{align}
\end{lemma}
\begin{proof}
\textup{(i)} Write $C=C_r$, $D=D_r$, $z=z_r$, $d=d_r$. For an annular cutoff $\zeta$, $\diver V_1(D)=-3(P_2\Id-E_2)z$ gives
\begin{align}
&\int\zeta\tr V_1(D)\,dV =(n-3)\left(\frac{d(n-2)}2\int\zeta\Sc\,dV -\int E_1(\nabla\zeta,z)\,dV\right), \nonumber\\
&\int\zeta \tr\bigl(V_1(D)\circ D\bigr)\,dV =d\int\zeta\tr V_1(D)\,dV -\int V_1(D)(\nabla\zeta,z)\,dV +3\int\zeta(P_2\Id-E_2)(z,z)\,dV. \nonumber
\end{align}
By Lemma~\ref{pt:lem:mass}, equation~\eqref{pt:eq:mass-bound},
\begin{align}
&0\le\int\zeta\tr V_1(D)\,dV\le C\int_{\supp\zeta}\Sc\,dV+C\|\nabla\zeta\|_\infty\|z\|_\infty\int_{\supp\zeta}\Sc\,dV\le Cr^{n-2}, \nonumber\\
&\left|\int V_1(D)(\nabla\zeta,z)\,dV\right|\le C\delta_r\int_{\supp\nabla\zeta}\tr V_1(D)\,dV\le C\delta_r r^{n-2}, \nonumber\\
&\left|\int\zeta(P_2\Id-E_2)(z,z)\,dV\right|\le C\delta_r^2r^2\int_{\supp\zeta}P_2\,dV\le C\delta_r^2r^{n-2}. \nonumber
\end{align}
The middle estimate uses a larger fixed cutoff and $V_1(D)\ge0$. Substitution proves \eqref{qc:eq:DD-mass}.

Use $\diver V_1(C)=3(P_2\Id-E_2)\nabla\psi_r$ and integrate
against $D=d\Id+\nabla z$ to obtain
\begin{align}
&\int\zeta \tr\bigl(V_1(D)\circ C\bigr)\,dV ={}d(n-3)\int\zeta \tr\bigl(E_1\circ C\bigr)\,dV -\int V_1(C)(\nabla\zeta,z)\,dV \nonumber\\
&{}-3\int\zeta(P_2\Id-E_2)(\nabla\psi_r,z)\,dV. \nonumber
\end{align}
By symmetry and positivity,
\begin{align}
&\tr\bigl(V_1(D)\circ C\bigr)=\tr\bigl(V_1(C)\circ D\bigr)\ge0,\qquad V_1(C)\ge0,\qquad \tr V_1(C)=(n-3)\tr(E_1\circ C), \nonumber\\
&d(n-3)\int\zeta\tr(E_1\circ C)\,dV\le C\delta_r r^{n-2}, \nonumber\\
&\left|\int V_1(C)(\nabla\zeta,z)\,dV\right|\le C\delta_r\int_{\supp\nabla\zeta}\tr V_1(C)\,dV\le C\delta_r^2r^{n-2}, \nonumber\\
&\left|\int\zeta(P_2\Id-E_2)(\nabla\psi_r,z)\,dV\right|\le C\delta_r r^2\int_{\supp\zeta}P_2\,dV\le C\delta_r r^{n-2}. \nonumber
\end{align}
Hence
\begin{align}
&0\le\int\zeta \tr\bigl(V_1(D)\circ C\bigr)\,dV\le C\delta_r r^{n-2}. \label{qc:eq:DC-small}
\end{align}
Since $D-C\ge0$,
\begin{align}
&0\le V_2(C)\le V_2(D,C),\qquad A\ge0\quad\Longrightarrow\quad |A|=\left(\sum_i\lambda_i(A)^2\right)^{1/2}\le\sum_i\lambda_i(A)=\tr A. \nonumber
\end{align}
Thus
\begin{align}
&|V_2(C)|\le\tr V_2(D,C) =(n-4)\tr\bigl(V_1(D)\circ C\bigr). \nonumber
\end{align}
By \eqref{qc:eq:DC-small},
\begin{align}
&\int_{\{a_0r<F<b_0r\}}|V_2(C)|\,dV\le(n-4)\int\zeta_0\tr(V_1(D)\circ C)\,dV\le C\delta_r r^{n-2}. \nonumber
\end{align}

Finally,
\begin{align}
&V_2(\Hess\psi_r) =c_r^2(n-3)(n-4)E_1-2c_rV_2(\Id,C)+V_2(C). \nonumber
\end{align}
On each admitted inner annulus,
\begin{align}
&\int|V_2(\Id,C)|\,dV\le(n-3)(n-4)\int\tr(E_1\circ C)\,dV\le C\delta_r r^{n-2}, \nonumber\\
&\int|V_2(\Hess\psi_r)-(n-3)(n-4)E_1|\,dV\le C|c_r^2-1|\int\Sc\,dV+2c_r\int|V_2(\Id,C)|\,dV \nonumber\\
&{}+\int|V_2(C)|\,dV\le C\delta_r r^{n-2}. \nonumber
\end{align}
\textup{(ii)} For an annular cutoff $\zeta$, $\diver E_2=0$ gives
\begin{align}
&\int\zeta \tr\bigl(E_2\circ C\bigr)\,dV =c_r(n-4)\int\zeta P_2\,dV +\int E_2(\nabla\zeta,\nabla\psi_r)\,dV, \nonumber\\
&\int\zeta \tr\bigl(E_2\circ D\bigr)\,dV =d_r(n-4)\int\zeta P_2\,dV -\int E_2(\nabla\zeta,z_r)\,dV. \nonumber
\end{align}
Therefore, on each admitted inner annulus,
\begin{align}
&0\le\int\zeta\bigl(\tr(E_2\circ C)+\tr(E_2\circ D)\bigr)\,dV\le C\bigl(1+r^{-1}\sup(|\nabla\psi_r|+|z_r|)\bigr)\int_{\supp\zeta}P_2\,dV\le Cr^{n-4}. \nonumber
\end{align}

Apply Lemma~\ref{pt:lem:algebra},
equation~\eqref{qc:eq:T-divergence}, with
$H=C=c_r\Id-\Hess\psi_r$ and $f=-\psi_r$. Since
$D=d_r\Id+\nabla z_r$, integration by parts gives
\begin{align}
&\int\zeta \tr\bigl(V_2(C)\circ D\bigr)\,dV ={}d_r\int\zeta\tr V_2(C)\,dV -\int V_2(C)(\nabla\zeta,z_r)\,dV \nonumber\\
&{}-6\int\zeta(\tr\bigl(E_2\circ C\bigr))\langle\nabla\psi_r,z_r\rangle\,dV +6\int\zeta E_2(C\nabla\psi_r,z_r)\,dV +6\int\zeta\mathcal J_C(\nabla\psi_r,z_r)\,dV. \label{qc:eq:cubic-IBP}
\end{align}
By part~\textup{(i)} and \eqref{qc:eq:E2-mixed-mass},
\begin{align}
&d_r\int\zeta\tr V_2(C)\,dV+\left|\int V_2(C)(\nabla\zeta,z_r)\,dV\right|\le C\delta_r r^{n-2}, \nonumber\\
&0\le\mathcal J_C\le(\tr\mathcal J_C)\Id,\qquad \tr\mathcal J_C=(n-5)\tr(E_2\circ C), \nonumber\\
&\left|\int\zeta\tr(E_2\circ C)\langle\nabla\psi_r,z_r\rangle\,dV\right|+\left|\int\zeta\mathcal J_C(\nabla\psi_r,z_r)\,dV\right|\le C\delta_r r^2\int\zeta\tr(E_2\circ C)\,dV\le C\delta_r r^{n-2}. \nonumber
\end{align}
For the remaining term,
\begin{align}
&\nabla\left(\frac12|\nabla\psi_r|^2-c_r\psi_r\right) =-C_r\nabla\psi_r. \nonumber
\end{align}
The normalization in Lemma~\ref{pt:lem:smooth-comparison},
equation~\eqref{qc:eq:value-comparison}, gives
\begin{align}
&\sup_{\{a_{8}r<F<b_{8}r\}} \left|\frac12|\nabla\psi_r|^2-c_r\psi_r\right| \le C\delta_r r^2. \label{qc:eq:h-bounds}
\end{align}
Indeed, on $\{a_8r<F<b_8r\}$,
\begin{align}
&\bigl||\nabla\psi_r|^2-F^2|\nabla F|^2\bigr|\le|\nabla\psi_r-F|\nabla F|\nu|(|\nabla\psi_r|+F|\nabla F|)\le C\delta_r r^2, \nonumber\\
&|F^2|\nabla F|^2-\rho^2|\le C\delta_r r^2,\qquad |\psi_r-\rho^2/2|\le\delta_r r^2,\qquad |c_r-1|\le\delta_r. \nonumber
\end{align}
The gradient identity and $\diver E_2=0$ give
\begin{align}
&\int\zeta E_2(C\nabla\psi_r,z_r)\,dV=-\int\zeta E_2\left( \nabla\left(\frac12|\nabla\psi_r|^2-c_r\psi_r\right),z_r \right)\,dV \nonumber\\
&{}=\int\left(\frac12|\nabla\psi_r|^2-c_r\psi_r\right) \bigl(E_2(\nabla\zeta,z_r) +\zeta\tr\bigl(E_2\circ(D-d_r\Id)\bigr)\bigr)\,dV. \nonumber
\end{align}
Using \eqref{qc:eq:h-bounds}, \eqref{qc:eq:E2-mixed-mass}, and $\tr E_2=(n-4)P_2$,
\begin{align}
&\left|\int\zeta E_2(C\nabla\psi_r,z_r)\,dV\right|\le C\delta_r r^2\int_{\supp\zeta}\bigl(\tr(E_2\circ D)+(1+\delta_r)P_2\bigr)\,dV\le C\delta_r r^{n-2}. \nonumber
\end{align}
Thus \eqref{qc:eq:cubic-IBP} gives
\begin{align}
&0\le\int\zeta \tr\bigl(V_2(C)\circ D\bigr)\,dV\le C\delta_r r^{n-2}. \label{qc:eq:cubic-trace-small}
\end{align}

Since $C\ge0$ and $D-C\ge0$,
\begin{align}
&V_3(D,C,C)-V_3(C)=V_3(D-C,C,C)\ge0, \nonumber
\end{align}
By Lemma~\ref{pt:lem:algebra}, equation~\eqref{qc:eq:general-contractions},
\begin{align}
&|V_3(C)| \le\tr V_3(C)\le\tr V_3(D,C,C)=(n-5)\tr\bigl(V_2(C)\circ D\bigr), \nonumber
\end{align}
Here the last identity uses symmetry. With $\zeta=\zeta_0$ in \eqref{qc:eq:cubic-trace-small},
\begin{align}
&\int_{\{a_0r<F<b_0r\}}|V_3(C)|\,dV \le(n-5)\int\zeta_0\tr\bigl(V_2(C)\circ D\bigr)\,dV\le C\delta_r r^{n-2}. \nonumber
\end{align}
Finally,
\begin{align}
&V_3(c_r\Id-C) =c_r^3(n-3)(n-4)(n-5)E_1-3c_r^2(n-5)V_2(\Id,C) +3c_r(n-5)V_2(C)-V_3(C). \nonumber
\end{align}
By the weighted $\tr(E_1\circ C)$ bound, \eqref{qc:eq:CC-small}, and \eqref{qc:eq:CCC-small},
\begin{align}
&\int|V_3(\Hess\psi_r)-(n-3)(n-4)(n-5)E_1|\,dV \nonumber\\
&{}\le C\left(|c_r^3-1|\int\Sc\,dV+\int|V_2(\Id,C)|\,dV+\int|V_2(C)|\,dV+\int|V_3(C)|\,dV\right)\le C\delta_r r^{n-2}. \nonumber
\end{align}
For $n=6$, divide by $6r^4$ and let $r\to\infty$.
\end{proof}

\section{Asymptotic identification of Chern's \texorpdfstring{\(\Phi_1\)}{Phi1}-term}
\label{six:sec:boundary}

For the next lemma take $n=6$. Retain $a=1/2$, $b=5/2$ and the comparison function
$\psi_r$ from the preceding subsections.
Use the uniform comparison data above and a fixed finite collection
of nested annuli, adding outer annuli when a cutoff is needed. Put
\begin{align}
&H_0=\Hess\psi_r,H_1=\Hess(F^2/2), H_\lambda=(1-\lambda)H_0+\lambda H_1,\qquad \phi_\lambda=(1-\lambda)\psi_r+\lambda F^2/2, \nonumber\\
&z=\nabla(F^2/2-\psi_r),Z=\nabla z,D=d_r\Id+Z. \nonumber
\end{align}
The comparison data give, with a fixed constant $C$,
\begin{align}
&D\ge\Id+C_r,\qquad -CD\le H_0,H_1,H_\lambda,Z\le CD, \nonumber\\
&\sup|\nabla\phi_\lambda|\le Cr,\qquad \sup|z|\le\delta_r r, \qquad \sup\bigl||\nabla\phi_\lambda|^2-F^2\bigr|\le C\delta_r r^2. \nonumber
\end{align}

\begin{lemma}\label{six:lem:cubic-mass}
Assume $n=6$. Use all the data of
Lemma~\ref{qc:lem:cubic-error}\textup{(ii)}, including its function-value
comparison and the common error $\delta_r$. In particular, $C_r,D_r$
and $\delta_r$ are the quantities defined by
equations~\eqref{qc:eq:comparison-D}--\eqref{qc:eq:delta-data}.
The function-value comparison is
Lemma~\ref{pt:lem:smooth-comparison},
equation~\eqref{qc:eq:value-comparison}.
For $0\le\lambda\le1$ put
$\phi_\lambda=(1-\lambda)\psi_r+\lambda F^2/2$,
$H_\lambda=\Hess\phi_\lambda$ and
$Z=\Hess(F^2/2-\psi_r)$.
Then there are $r_0<\infty$ and
$C=C(M,g,F,a_0,b_0)<\infty$, independent of $r$ and $\lambda$, such
that, for every $r\ge r_0$,
\begin{align}
&\sup_{0\le\lambda\le1}\int_{\{a_0r<F<b_0r\}}|V_3(H_\lambda)|\,dV \le Cr^4, \label{six:eq:cubic-H-mass}\\
&\sup_{0\le\lambda\le1}\int_{\{a_0r<F<b_0r\}} \bigl(|\mathcal J_{H_\lambda}|+|\mathcal J_Z|\bigr)\,dV \le Cr^2. \label{six:eq:cubic-interpolation-mass}
\end{align}
\end{lemma}

\begin{proof}
Let
\begin{align}
&\Omega_0=\{a_0r<F<b_0r\},\qquad \Omega_1=\left\{\frac{a_0}{2}r<F<2b_0r\right\}. \nonumber
\end{align}
Take $\zeta=\zeta_0$. Then $\zeta\in C_c^\infty(\Omega_1)$ and
\begin{align}
&0\le\zeta\le1,\qquad \zeta\equiv1\quad\text{on }\Omega_0,\qquad |\nabla\zeta|\le C(M,g,F,a_0,b_0)r^{-1}. \nonumber
\end{align}
Write
\begin{align}
&D=D_r,\qquad z=z_r,\qquad d=d_r. \nonumber
\end{align}
Since $n=6$ and $D\ge0$, $V_2(D),V_3(D)\ge0$ and
\begin{align}
&\tr V_3(D)=\tr\bigl(V_2(D)\circ D\bigr). \nonumber
\end{align}
Hence
\begin{align}
&\int_{\Omega_0}|V_3(D)|\,dV\le\int_{\Omega_0}\tr V_3(D)\,dV\le\int_{\Omega_1}\zeta\tr\bigl(V_2(D)\circ D\bigr)\,dV. \nonumber
\end{align}

Since
\begin{align}
&D=d\Id+\nabla z \nonumber
\end{align}
and, in dimension six,
\begin{align}
&\tr V_2(D)=2\tr\bigl(V_1(D)\circ D\bigr), \nonumber
\end{align}
we have
\begin{align}
&\int_{\Omega_1}\zeta\tr\bigl(V_2(D)\circ D\bigr)\,dV={}2d\int_{\Omega_1}\zeta\tr\bigl(V_1(D)\circ D\bigr)\,dV +\int_{\Omega_1}\zeta\tr\bigl(V_2(D)\circ\nabla z\bigr)\,dV. \label{six:eq:DDD-split}
\end{align}
Pointwise,
\begin{align}
&\zeta\tr\bigl(V_2(D)\circ\nabla z\bigr)={}\diver\bigl(\zeta V_2(D)z\bigr)-V_2(D)(\nabla\zeta,z) -\zeta\langle\diver V_2(D),z\rangle. \nonumber
\end{align}
Because $\zeta$ is compactly supported in $\Omega_1$, Stokes' theorem gives
\begin{align}
&\int_{\Omega_1}\zeta\tr\bigl(V_2(D)\circ\nabla z\bigr)\,dV={}-\int_{\Omega_1}V_2(D)(\nabla\zeta,z)\,dV -\int_{\Omega_1}\zeta\langle\diver V_2(D),z\rangle\,dV. \label{six:eq:DDD-stokes}
\end{align}
Now Lemma~\ref{pt:lem:algebra},
equation~\eqref{qc:eq:T-divergence}, applied to
$D=d\Id+\nabla z$, gives
\begin{align}
&\diver V_2(D)=-6\Big((\tr(E_2\circ D))\Id-E_2D-\mathcal J_D\Big)z. \nonumber
\end{align}

By Lemma~\ref{qc:lem:cubic-error}\textup{(ii)}, equation~\eqref{qc:eq:E2-mixed-mass}, and Lemma~\ref{pt:lem:algebra}, equation~\eqref{qc:eq:T-divergence},
\begin{align}
&\int_{\Omega_1}\tr(E_2\circ D)\,dV\le Cr^2. \nonumber
\end{align}
Substituting into \eqref{six:eq:DDD-stokes} and then into \eqref{six:eq:DDD-split} gives
\begin{align}
&\int\zeta \tr\bigl(V_2(D)\circ D\bigr)\,dV ={}2d\int\zeta \tr\bigl(V_1(D)\circ D\bigr)\,dV -\int V_2(D)(\nabla\zeta,z)\,dV \nonumber\\
&{}+6\int\zeta\bigl((\tr\bigl(E_2\circ D\bigr))|z|^2 -E_2(Dz,z)-\mathcal J_D(z,z)\bigr)\,dV. \label{six:eq:DDD-identity}
\end{align}

By positivity and $\tr\mathcal J_D=\tr(E_2\circ D)$,
\begin{align}
&\int\zeta\left(\tr(E_2\circ D)|z|^2+|\mathcal J_D(z,z)|\right)\,dV\le C\delta_r^2r^2\int\zeta\tr(E_2\circ D)\,dV\le C\delta_r^2r^4. \nonumber
\end{align}
Since $Dz=dz+\frac12\nabla|z|^2$ and $\diver E_2=0$,
\begin{align}
&\int\zeta E_2(Dz,z)\,dV ={}d\int\zeta E_2(z,z)\,dV -\frac12\int|z|^2\bigl( \zeta \tr\bigl(E_2\circ Z\bigr)+E_2(\nabla\zeta,z)\bigr)\,dV. \nonumber
\end{align}
Since $Z=D-d\Id$ and $E_2\ge0$,
\begin{align}
&\left|\int\zeta E_2(Dz,z)\,dV\right|\le C\delta_r^2r^2\int_{\Omega_1}\bigl(P_2+\tr(E_2\circ D)\bigr)\,dV\le C\delta_r^2r^4, \nonumber\\
&\left|\int V_2(D)(\nabla\zeta,z)\,dV\right|\le C\delta_r\int_{\Omega_1}\tr V_2(D)\,dV=2C\delta_r\int_{\Omega_1}\tr(V_1(D)\circ D)\,dV\le C\delta_r r^4. \nonumber
\end{align}
Here Lemma~\ref{qc:lem:quadratic-error}, equation~\eqref{qc:eq:DD-mass}, is applied with a larger fixed cutoff. Substitution into \eqref{six:eq:DDD-identity} gives
\begin{align}
&0\le\int_{\Omega_0}|V_3(D)|\,dV\le\int\zeta\tr(V_2(D)\circ D)\,dV\le Cr^4. \nonumber
\end{align}
For $0\le\lambda\le1$,
\begin{align}
&0\le H_\lambda+CD\le2CD,\qquad 0\le CD\le2CD,\qquad H_\lambda=(H_\lambda+CD)-CD, \nonumber\\
&|V_3(H_\lambda)|\le C\tr V_3(D),\qquad |\mathcal J_{H_\lambda}|+|\mathcal J_Z|\le C\tr\mathcal J_D=C\tr(E_2\circ D), \nonumber\\
&\sup_{0\le\lambda\le1}\int_{\Omega_0}|V_3(H_\lambda)|\,dV\le Cr^4,\qquad \sup_{0\le\lambda\le1}\int_{\Omega_0}(|\mathcal J_{H_\lambda}|+|\mathcal J_Z|)\,dV\le Cr^2. \nonumber
\end{align}
The middle line follows from multilinearity and positivity. 
\end{proof}

\begin{lemma}\label{unified:lem:normal-contraction}
Let $(M^6,g)$ satisfy the standing hypotheses.
Retain $V_3$ and $E_1$ from
equations~\eqref{uni:eq:normalized-tensor} and
\eqref{pt:eq:Einstein-tensors}.
For a regular value $t$ of $F$, choose an orthonormal eigenframe
$Se_a=\kappa_a e_a$, $1\le a\le5$, and write $K_{ab}=\sec_g(e_a\wedge e_b)$. Then
\begin{align}
&\frac16V_3(\Id)=E_1,\qquad \frac16V_3\bigl(\Hess(F^2/2)\bigr)(\nu,\nu)=F^3|\nabla F|^3Q_6(S). \label{unified:eq:normal-contraction}
\end{align}
The polynomial on the right is independent of the chosen eigenframe.
\end{lemma}
\begin{proof}
By the alternating-delta contractions in \eqref{uni:eq:normalized-tensor},
\begin{align}
&V_1(\Id)=3E_1,\qquad V_2(\Id)=2V_1(\Id)=6E_1,\qquad V_3(\Id)=V_2(\Id)=6E_1. \nonumber
\end{align}
For $H=\Hess(F^2/2)$,
\begin{align}
&H|_{T\Sigma_t}=F|\nabla F|S,\qquad H^a{}_b=F|\nabla F|\kappa_a\delta^a_b\quad(1\le a,b\le5), \nonumber\\
&V_3(H)(\nu,\nu)=\frac14\delta^{\nu ij\ell_1\ell_2\ell_3}_{\nu pqr_1r_2r_3}R^{pq}{}_{ij}\prod_{\alpha=1}^3H^{r_\alpha}{}_{\ell_\alpha} \nonumber\\
&{}=6F^3|\nabla F|^3\sum_{1\le a<b\le5}K_{ab}\prod_{c\in\{1,\ldots,5\}\setminus\{a,b\}}\kappa_c=6F^3|\nabla F|^3Q_6(S). \nonumber
\end{align}
The alternating delta eliminates every remaining normal index; the two orders in each curvature index pair contribute $4$, and the shape indices contribute $3!=6$. Since the left side is a tensor contraction, so is the right side.
\end{proof}

Use the standing geometry with $n=6$.
Put
\begin{align}
&\Omega_r=\left\{\frac r2<F<\frac{5r}{2}\right\},\qquad U_r=\left\{\frac r4<F<3r\right\}. \label{unified:eq:annular-data}
\end{align}
For all sufficiently large $r$, properness of $F$ gives
\begin{align}
&\overline{\Omega_r}\Subset U_r, \qquad \overline{U_r}\ \text{compact}. \nonumber
\end{align}
Use on $U_r$ the fixed comparison data of
Section~\ref{sec:uniform-comparison-data}.  Choose
\begin{align}
&0\le\zeta_r\in C_c^\infty(U_r), \qquad \zeta_r\equiv1\ \text{on }\Omega_r, \qquad |\nabla\zeta_r|\le\frac Cr. \nonumber
\end{align}

For $0\le\lambda\le1$, set
\begin{align}
&\phi_\lambda=(1-\lambda)\psi_r+\lambda\frac{F^2}{2}. \nonumber
\end{align}
The scalar $d_r=c_r+1+\varepsilon_r$ is the one in
\eqref{qc:eq:comparison-D}. The comparison data give
\begin{align}
&d_r\Id+\nabla^2\!\left(\frac{F^2}{2}-\psi_r\right) =\Id+\left(\nabla^2\!\left(\frac{F^2}{2}\right) +\varepsilon_r\Id\right)+C_r \ge\Id+C_r. \nonumber
\end{align}
Every constant
$C$ may depend on the fixed geometric and annular data,
but not on $r$, $\lambda$, or the smoothing parameter below.
For each fixed $r$, all curvature and Hessian coefficients are
bounded on the compact supports in question. Uniformly for
$0\le\lambda\le1$, the comparison estimates give
\begin{align}
&|\nabla\phi_\lambda|\le Cr,\qquad \left|\nabla\!\left(\frac{F^2}{2}-\psi_r\right)\right| \le\delta_r r, -C\left(d_r\Id+\nabla^2\!\left(\frac{F^2}{2}-\psi_r\right)\right) \le\nabla^2\phi_\lambda, \nonumber\\
&\nabla^2\phi_\lambda \le C\left(d_r\Id+\nabla^2\!\left(\frac{F^2}{2}-\psi_r\right)\right),\qquad \left||\nabla\phi_\lambda|^2-F^2\right| \le C\delta_r r^2, \qquad \lim_{r\to\infty}\delta_r=0. \nonumber
\end{align}
Here and below symmetric bilinear forms are identified with
self-adjoint endomorphisms using $g$.

Define the Lipschitz function on $\mathbb R$
\begin{align}
&\chi_0(s)= \begin{cases} \displaystyle\frac7{24},&s\le1,\\[1mm] \displaystyle\frac{s^{-3}}3-\frac1{24},&1<s<2,\\[1mm] 0,&s\ge2, \end{cases} \qquad \chi=\chi_0(F/r). \label{interval:eq:primitive}
\end{align}
Almost everywhere,
\begin{align}
&\nabla\chi=-\frac1r(F/r)^{-4} \mathbf1_{\{r<F<2r\}}\nabla F. \label{interval:eq:primitive-gradient}
\end{align}
The right side is defined as zero outside $\{r<F<2r\}$.
For sufficiently large $r$, the levels $F=r,2r$ are smooth
hypersurfaces and therefore have zero ambient volume.
To justify the smooth interpolation identities, use the functions
$\zeta_\varepsilon$ from \eqref{interval:eq:mollifier} and set
\begin{align}
&\vartheta_\varepsilon(s)=\int_1^2 \zeta_\varepsilon(s-\tau)\,d\tau,\qquad \chi_{0,\varepsilon}(s)=\int_s^\infty \tau^{-4}\vartheta_\varepsilon(\tau)\,d\tau, \qquad \chi_\varepsilon=\chi_{0,\varepsilon}(F/r). \label{unified:eq:smooth-cutoffs}
\end{align}
The integrand is extended by zero near $\tau=0$ and on the
negative half-line. For $0<\varepsilon<1/8$,
$0\le\vartheta_\varepsilon\le1$ and
\begin{align}
&\supp\vartheta_\varepsilon \subset[1-\varepsilon,2+\varepsilon]\subset(1/2,5/2),\qquad |\nabla\chi_\varepsilon|\le C/r, \nonumber\\
&\lim_{\varepsilon\downarrow0}\nabla\chi_\varepsilon =\nabla\chi \quad\text{almost everywhere on }U_r. \nonumber
\end{align}
The gradient bound is independent of $\varepsilon$.
The gradient of $F^2$ is parallel to each
$\nabla\chi_\varepsilon$. Define
\begin{align}
&I_{r,\varepsilon}(\lambda) =\frac16 \int_{U_r}\left\langle\nabla\chi_\varepsilon, V_3(\nabla^2\phi_\lambda)\nabla\phi_\lambda \right\rangle\,dV, \nonumber\\
&I_r(\lambda) =\frac16 \int_{U_r}\left\langle\nabla\chi, V_3(\nabla^2\phi_\lambda)\nabla\phi_\lambda \right\rangle\,dV. \label{unified:eq:flux-definitions}
\end{align}
These integrals are finite for each fixed $r$, since the gradients
have compact support in $\Omega_r$.

\begin{lemma}
\label{six:lem:length-error-IBP}
Assume $n=6$. Let $U\subset M$ be open and let
$\phi_0,\phi_1,f,\chi\in C^\infty(U)$, with
$\supp\nabla\chi\Subset U$. Assume that
$\nabla(f^2)$ is parallel to $\nabla\chi$ on
$\supp\nabla\chi$. For $0\le\lambda\le1$, put
\begin{align}
&\phi_\lambda=(1-\lambda)\phi_0+\lambda\phi_1, H_\lambda=\Hess\phi_\lambda,\qquad z=\nabla(\phi_1-\phi_0), Z=\nabla z,\qquad e_\lambda=|\nabla\phi_\lambda|^2-f^2. \nonumber
\end{align}
Then
\begin{align}
&\int_U \mathbf{P}_2(\nabla\chi, H_\lambda\nabla\phi_\lambda, \nabla\phi_\lambda,z)\,dV= \frac12\int_U e_\lambda \left( \mathcal J_{H_\lambda}(\nabla\chi,z) - \mathcal J_Z(\nabla\chi,\nabla\phi_\lambda) \right)dV. \label{six:eq:length-error-IBP}
\end{align}
\end{lemma}

\begin{proof}

\begin{align}
&H_\lambda\nabla\phi_\lambda = \frac12\nabla|\nabla\phi_\lambda|^2 = \frac12\nabla e_\lambda + \frac12\nabla(f^2), \nonumber
\end{align}
By alternation, $\mathbf{P}_2(X,X,Y,Z)=0$, and
\begin{align}
&\nabla(f^2)\parallel\nabla\chi \quad\text{on }\supp\nabla\chi \nonumber
\end{align}
gives
\begin{align}
&\int_U \mathbf{P}_2(\nabla\chi, H_\lambda\nabla\phi_\lambda, \nabla\phi_\lambda,z)\,dV= \frac12\int_U \mathbf{P}_2(\nabla\chi,\nabla e_\lambda, \nabla\phi_\lambda,z)\,dV. \nonumber
\end{align}
Using
\begin{align}
&\nabla_m(\mathbf{P}_2)^{am}{}_{bs}=0 \nonumber
\end{align}
and $(\mathbf{P}_2)^{am}{}_{bs}\nabla_m\nabla_a\chi=0$, integration by parts gives
\begin{align}
&\int_U \mathbf{P}_2(\nabla\chi, H_\lambda\nabla\phi_\lambda, \nabla\phi_\lambda,z)\,dV= -\frac12\int_U e_\lambda (\mathbf{P}_2)^{am}{}_{bs} (\nabla_a\chi) (H_\lambda)^b{}_m z^s\,dV \nonumber\\
&{}-\frac12\int_U e_\lambda (\mathbf{P}_2)^{am}{}_{bs} (\nabla_a\chi) (\nabla^b\phi_\lambda)Z^s{}_m\,dV. \nonumber
\end{align}
By
\begin{align}
&(\mathbf{P}_2)^{am}{}_{bs}(H_\lambda)^b{}_m =-\mathcal J_{H_\lambda}{}^a{}_s, (\mathbf{P}_2)^{am}{}_{bs}Z^s{}_m =\mathcal J_Z{}^a{}_b, \nonumber
\end{align}
the right side equals that of \eqref{six:eq:length-error-IBP}; $\supp\nabla\chi\Subset U$ eliminates boundary terms.
\end{proof}

\begin{lemma}
\label{six:lem:flux-comparison}
Let $(M^6,g)$ be a smooth, complete, connected, noncompact
Riemannian manifold and with $\sec_g\ge0$. Retain $\psi_r,C_r,d_r,\delta_r$ from
\eqref{qc:eq:comparison-D}--\eqref{qc:eq:delta-data}, and use
$\Omega_r,U_r,\chi,\chi_\varepsilon,I_{r,\varepsilon},I_r$ from
\eqref{unified:eq:annular-data},
\eqref{interval:eq:primitive}--\eqref{interval:eq:primitive-gradient},
\eqref{unified:eq:smooth-cutoffs}, and~\eqref{unified:eq:flux-definitions}.
For $0\le\lambda\le1$, put $\phi_\lambda=(1-\lambda)\psi_r+\lambda F^2/2$.
There are $r_0,C<\infty$, depending only on the fixed geometric and
annular data, such that, for every $r\ge r_0$,
\begin{align}
&\sup_{0\le\lambda\le1}\int_{\Omega_r}\left|\frac16V_3(\nabla^2\phi_\lambda)\right|\,dV\le Cr^4, \label{six:eq:flux-mass}\\
&\int_{\Omega_r}\left|\frac16V_3(\nabla^2\psi_r)-E_1\right|\,dV\le C\delta_r r^4. \label{six:eq:flux-replacement}
\end{align}
Moreover, for every $0<\varepsilon<1/8$,
\begin{align}
&\sup_{0\le\lambda\le1}|I'_{r,\varepsilon}(\lambda)|\le C\delta_r r^4. \label{six:eq:flux-derivative}
\end{align}
The constant $C$ is independent of $r,\lambda,\varepsilon$.
\end{lemma}
\begin{proof}
Lemma~\ref{six:lem:cubic-mass}, equation~\eqref{six:eq:cubic-H-mass}, and Lemma~\ref{qc:lem:cubic-error}\textup{(ii)}, equations~\eqref{qc:eq:cubic-replacement}--\eqref{qc:eq:six-dimensional-replacement}, give
\begin{align}
&\sup_{0\le\lambda\le1}\int_{\Omega_r}\frac{|V_3(\nabla^2\phi_\lambda)|}{6}\,dV\le Cr^4,\qquad \int_{\Omega_r}\left|\frac{V_3(\nabla^2\psi_r)}6-E_1\right|\,dV\le C\delta_r r^4. \nonumber
\end{align}

The interpolation derivatives are
\begin{align}
&\frac{d}{d\lambda}\nabla\phi_\lambda =\nabla\!\left(\frac{F^2}{2}-\psi_r\right),\qquad \frac{d}{d\lambda}\nabla^2\phi_\lambda =\nabla^2\!\left(\frac{F^2}{2}-\psi_r\right). \nonumber
\end{align}
Apply
Lemma~\ref{six:lem:cubic-interpolation},
equation~\eqref{six:eq:cubic-interpolation}, divided by six,
with the same two endpoint functions.

The first term is at most $C\delta_r r^4$ by \eqref{six:eq:flux-mass}. For the two $\mathcal J_{\nabla^2\phi_\lambda}$ terms, Lemma~\ref{six:lem:cubic-mass}, equation~\eqref{six:eq:cubic-interpolation-mass}, gives
\begin{align}
&\int_{\Omega_r}|\mathcal J_{\nabla^2\phi_\lambda}|\,dV \le Cr^2,\qquad \int_{\Omega_r}\left| \mathcal J_{\nabla^2(F^2/2-\psi_r)}\right|\,dV \le Cr^2. \nonumber
\end{align}
Their sum in absolute value is at most
\begin{align}
&C\sup|\nabla\chi_\varepsilon|\sup|\nabla\phi_\lambda|^2\sup|\nabla(F^2/2-\psi_r)|\int_{\Omega_r}|\mathcal J_{\nabla^2\phi_\lambda}|\,dV\le C\delta_r r^4. \nonumber
\end{align}

For the remaining term, apply
Lemma~\ref{six:lem:length-error-IBP},
equation~\eqref{six:eq:length-error-IBP},
with
\begin{align}
&\phi_0=\psi_r,\qquad \phi_1=\frac{F^2}{2},\qquad f=F,\qquad \chi=\chi_\varepsilon. \nonumber
\end{align}
Here
\begin{align}
&\nabla(F^2)\parallel\nabla\chi_\varepsilon. \nonumber
\end{align}

Since
\begin{align}
&\sup_{\Omega_r} \left||\nabla\phi_\lambda|^2-F^2\right| \le C\delta_r r^2, \nonumber
\end{align}
the remaining integral has absolute value at most
\begin{align}
&C\delta_r r^2\Biggl[\, \frac1r\sup_{\Omega_r} \left|\nabla\!\left(\frac{F^2}{2}-\psi_r\right)\right| \int_{\Omega_r}|\mathcal J_{\nabla^2\phi_\lambda}|\,dV\qquad {}+{}\frac1r\sup_{\Omega_r}|\nabla\phi_\lambda| \int_{\Omega_r}\left| \mathcal J_{\nabla^2(F^2/2-\psi_r)}\right|\,dV \Biggr] \nonumber\\
&{}\le C\delta_r r^4. \nonumber
\end{align}

Only $\|\nabla\chi_\varepsilon\|_\infty\le C/r$ enters these bounds. Thus $C$ is independent of $\varepsilon$.
\end{proof}

\begin{proposition}\label{six:prop:boundary-limit}
Let $(M^6,g)$ be a smooth, complete, connected, noncompact
Riemannian manifold and with $\sec_g\ge0$. Then
\begin{align}
&\lim_{r\to\infty}\frac1r\int_r^{2r}\int_{\Sigma_t}Q_6(S)\,dA\,dt=2L(M,g). \label{six:eq:unweighted-boundary-limit}
\end{align}
\end{proposition}

\begin{proof}
Let $L=L(M,g)$ in the rest argument. Use the preceding comparison data.
At $\lambda=0$,
\begin{align}
&\nabla\psi_r=F|\nabla F|\nu-\nabla(F^2/2-\psi_r). \nonumber
\end{align}
Lemma~\ref{six:lem:flux-comparison}, equation~\eqref{six:eq:flux-replacement}, and
\eqref{interval:eq:primitive-gradient} give
\begin{align}
&\left|I_r(0)+\int_{\{r<F<2r\}}(F/r)^{-3}|\nabla F|^2\left(\frac{\Sc}{2}-\Ric(\nu,\nu)\right)dV\right|\le C\delta_r r^4. \nonumber
\end{align}
By Lemma~\ref{unified:lem:distance-limit}, equation~\eqref{unified:eq:distance-limit},
\begin{align}
&\lim_{r\to\infty}r^{-4}[-I_r(0)]=2L. \nonumber
\end{align}
For fixed $r$ and $\lambda\in\{0,1\}$, all coefficients are bounded on the compact annulus $\overline{U_r}$, and
\begin{align}
&\left|\left\langle\nabla\chi_\varepsilon,V_3(\nabla^2\phi_\lambda)\nabla\phi_\lambda\right\rangle\right|\le\frac Cr\sup_{\overline{\Omega_r}}\bigl(|V_3(\nabla^2\phi_\lambda)||\nabla\phi_\lambda|\bigr)\mathbf1_{\Omega_r}\in L^1(U_r). \nonumber
\end{align}
Dominated convergence and Lemma~\ref{six:lem:flux-comparison}, equation~\eqref{six:eq:flux-derivative}, imply
\begin{align}
&\lim_{\varepsilon\downarrow0}I_{r,\varepsilon}(\lambda)=I_r(\lambda),\qquad |I_{r,\varepsilon}(1)-I_{r,\varepsilon}(0)|\le\int_0^1|I'_{r,\varepsilon}(\lambda)|\,d\lambda\le C\delta_r r^4, \nonumber\\
&|I_r(1)-I_r(0)|\le C\delta_r r^4,\qquad \lim_{r\to\infty}r^{-4}[-I_r(1)]=2L. \label{six:eq:endpoint-comparison}
\end{align}
The limit $\varepsilon\downarrow0$ is taken at fixed $r$.
At $\lambda=1$,
\begin{align}
&\nabla(F^2/2)=F|\nabla F|\nu,\qquad \nabla^2(F^2/2)|_{T\Sigma_t}=F|\nabla F|S. \nonumber
\end{align}
Lemma~\ref{unified:lem:normal-contraction}, equation~\eqref{unified:eq:normal-contraction}, gives
\begin{align}
&\frac16V_3\bigl(\nabla^2(F^2/2)\bigr)(\nu,\nu)=F^3|\nabla F|^3Q_6(S). \nonumber
\end{align}
Substitution in $I_r(1)$ and coarea, with $dV=|\nabla F|^{-1}dA\,dt$, yield
\begin{align}
&r^{-4}[-I_r(1)]=\frac1r\int_r^{2r}\int_{\Sigma_t}|\nabla F|^4Q_6(S)\,dA\,dt, \label{six:eq:weighted-endpoint}\\
&\lim_{r\to\infty}\frac1r\int_r^{2r}\int_{\Sigma_t}|\nabla F|^4Q_6(S)\,dA\,dt=2L. \nonumber
\end{align}
The same contraction and Lemma~\ref{six:lem:flux-comparison}, equation~\eqref{six:eq:flux-mass}, give
\begin{align}
&F^3|\nabla F|^3|Q_6(S)|\le\left|\frac16V_3\bigl(\nabla^2(F^2/2)\bigr)\right|, \nonumber\\
&\frac1r\int_r^{2r}\int_{\Sigma_t}|\nabla F|^4|Q_6(S)|\,dA\,dt=\frac1r\int_{\{r<F<2r\}}|\nabla F|^5|Q_6(S)|\,dV \nonumber\\
&\hspace{12pt}\le Cr^{-4}\int_{\Omega_r}\left|\frac16V_3\bigl(\nabla^2(F^2/2)\bigr)\right|\,dV\le C. \label{six:eq:absolute-boundary}
\end{align}
Here $F\ge r$ and $1/2\le|\nabla F|\le2$ for all sufficiently large $r$.
Put
\begin{align}
&\epsilon_{r,3}=\sup_{\{r\le F\le2r\}}\bigl||\nabla F|^{-4}-1\bigr|\le128\sup_{\{r\le F\le2r\}}\bigl||\nabla F|-1\bigr|,\qquad \lim_{r\to\infty}\epsilon_{r,3}=0. \nonumber
\end{align}
Since $|1-|\nabla F|^4|=|\nabla F|^4\bigl||\nabla F|^{-4}-1\bigr|$, equation~\eqref{six:eq:absolute-boundary} gives
\begin{align}
&\frac1r\left|\int_r^{2r}\int_{\Sigma_t}(1-|\nabla F|^4)Q_6(S)\,dA\,dt\right|\le C\epsilon_{r,3}. \nonumber
\end{align}
Together with \eqref{six:eq:weighted-endpoint}, this proves \eqref{six:eq:unweighted-boundary-limit}.
\end{proof}

\section{Gauss-Bonnet-Chern and the sharp six-dimensional bound}\label{intro:sec:six}

Write $|\mathbb S^{n-1}(1)|=n\omega_n =2\pi\omega_{n-2}$. For a symmetric
endomorphism $S$, its positive part $S_+$ has eigenvalues
$\max\{\kappa_i,0\}$ in an eigenframe of $S$.

\begin{lemma}\label{uni:lem:determinant}
If $S$ is the outward shape
endomorphism of $\partial D$, define $S_+$ by replacing each
eigenvalue $\kappa_i$ with $\max\{\kappa_i,0\}$. For a real number $x$ put $x_-=\max\{-x,0\}$.
Then
\begin{align}
&\int_{\partial D}\det S_+\,dA\ge n\omega_n\AVR(M^n,g), \nonumber\\
&\int_{\partial D}\det S\,dA \ge n\omega_n\AVR(M^n,g)-\int_{\partial D}(\det S)_-\,dA. \label{uni:eq:determinant-volume}
\end{align}
This statement holds in every dimension under consideration and
does not require convexity of $D$.
\end{lemma}
\begin{proof}
For $x\in\partial D$, put $\gamma_x(s)=\exp_x(s\nu_x)$ and
\begin{align}
&\tau(x)=\sup\{s\ge0:d(D,\gamma_x(s))=s\}. \nonumber
\end{align}
Minimizing segments minimize on every prefix. Hence, for $a>0$,
\begin{align}
&\{x\in\partial D:\tau(x)\ge a\}=\{x\in\partial D:d(D,\gamma_x(a))=a\}\quad\text{is closed}, \nonumber\\
&\tau:\partial D\longrightarrow[0,\infty]\quad\text{is measurable},\qquad \int_{\partial D}\int_0^R\mathbf1_{\{s=\tau(x)<\infty\}}\,ds\,dA=0. \nonumber
\end{align}
The normal exponential map is smooth and Lipschitz on each compact parameter band. Thus
\begin{align}
&\Vol\{\gamma_x(\tau(x)):x\in\partial D,\ \tau(x)<\infty\}=0, \nonumber\\
&M\setminus D\subset\{\gamma_x(s):x\in\partial D,\ 0<s<\tau(x)\}\cup\{\gamma_x(\tau(x)):\tau(x)<\infty\}. \nonumber
\end{align}
The second variation excludes focal points for $0\le s<\tau(x)$. In a parallel frame,
\begin{align}
&J''+\mathcal R_\nu J=0,\qquad J(0)=\Id,\qquad J'(0)=S, \nonumber\\
&\mathcal R_\nu\ge0,\qquad \det J(s)>0,\qquad A=J'J^{-1}=A^T,\qquad A'+A^2+\mathcal R_\nu=0, \nonumber\\
&B=S_+(\Id+sS_+)^{-1},\qquad B'+B^2=0,\qquad A(0)=S\le S_+=B(0), \nonumber\\
&H=A-B,\qquad C=(A+B)/2,\qquad H'=-CH-HC-\mathcal R_\nu,\qquad V'=CV,\qquad V(0)=\Id, \nonumber\\
&\det V(s)=\exp\left(\int_0^s\tr C(t)\,dt\right)>0,\qquad (V^THV)'=-V^T\mathcal R_\nu V\le0, \nonumber\\
&V^THV\le H(0)\le0,\qquad A\le B, \nonumber\\
&\log\det J(s)=\int_0^s\tr A(t)\,dt\le\int_0^s\tr B(t)\,dt=\log\det(\Id+sS_+). \nonumber
\end{align}
For $D_r=\{y:d(D,y)<r\}$, the area formula gives
\begin{align}
&\Vol(D_r)\le\Vol(D)+ \int_{\partial D}\int_0^r\det(\Id+sS_+)\,ds\,dA. \nonumber
\end{align}
Fix $p\in D$ and $A_D=\max_{x\in D}d(p,x)$. Then
\begin{align}
&B_p(r)\subset D_r\subset B_p(r+A_D),\qquad \lim_{r\to\infty}r^{-n}\Vol(D_r)=\omega_n\AVR(M^n,g), \nonumber\\
&\det(\Id+sS_+)=\sum_{j=0}^{n-1}s^j\sigma_j(S_+), \nonumber\\
&\omega_n\AVR(M^n,g)\le\lim_{r\to\infty}r^{-n}\left(\Vol(D)+\sum_{j=0}^{n-1}\frac{r^{j+1}}{j+1}\int_{\partial D}\sigma_j(S_+)\,dA\right) \nonumber\\
&{}=\frac1n\int_{\partial D}\det S_+\,dA. \nonumber
\end{align}
Finally,
\begin{align}
&\min_i\kappa_i<0\ \Longrightarrow\ \det S_+=0\le\det S+(\det S)_-,\qquad \min_i\kappa_i\ge0\ \Longrightarrow\ \det S_+=\det S, \nonumber\\
&\int_{\partial D}\det S\,dA\ge\int_{\partial D}\det S_+\,dA-\int_{\partial D}(\det S)_-\,dA \nonumber\\
&{}\ge n\omega_n\AVR(M^n,g)-\int_{\partial D}(\det S)_-\,dA. \nonumber
\end{align}
\end{proof}

\begin{lemma}
\label{general:lem:CGB-boundary}
Let $m\ge1$ and let $(D^{2m},g)$ be a compact smooth
Riemannian manifold with smooth boundary.  Use the connection
and curvature-form conventions of
\eqref{pole:eq:connection-forms}.  Use the forms
$\Phi_j$ in \eqref{pole:eq:Phi} for every integer
$0\le j\le m-1$, and use $\Psi_m$ from
\eqref{pole:eq:Psi}.  Put
\begin{align}
&c_{m,j}:=\frac{1}{(2m-2j-1)!!\,2^{m+j}j!}, \qquad 0\le j\le m-1,\qquad \Pi_{2m-1}:=\frac1{\pi^m}\sum_{j=0}^{m-1}c_{m,j}\Phi_j, \nonumber\\
&\mathcal E_{2m}:=\frac1{2^{2m}\pi^m m!}\Psi_m. \nonumber
\end{align}
Here
\begin{align}
&(2r-1)!!:=1\cdot3\cdots(2r-1),\qquad 1!!:=1. \nonumber
\end{align}
Then
\begin{align}
&d\Phi_j=\Psi_j-\frac{2m-2j-1}{2(j+1)}\Psi_{j+1}, \qquad 0\le j\le m-1,\qquad d\Pi_{2m-1}=-\mathcal E_{2m}. \nonumber
\end{align}
If $\nu$ denotes the outward unit normal section of the unit
tangent bundle over $\partial D$, then
\begin{align}
&\chi(D) = \int_D\mathcal E_{2m} + \int_{\partial D}\nu^*\Pi_{2m-1}, \nonumber\\
&\pi^m\chi(D) = \frac1{2^{2m}m!}\int_D\Psi_m + \sum_{j=0}^{m-1} \frac1{(2m-2j-1)!!\,2^{m+j}j!} \int_{\partial D}\nu^*\Phi_j. \label{general:eq:CGB-boundary}
\end{align}
No convexity assumption on $\partial D$ and no product-metric
assumption in a neighborhood of $\partial D$ is required.
For a nonorientable $D$, the same scalar identity follows by
applying the oriented formula to the orientation double cover and
dividing by two.
\end{lemma}

\begin{proof}
Chern's transgression formula \cite[pp.~675--679, especially formulas~(4), (6)--(11), and~(19)]{ChernCurvatura1945}, with convention~\eqref{pole:eq:connection-forms}, gives
\begin{align}
&d\Phi_j = \Psi_j-\frac{2m-2j-1}{2(j+1)}\Psi_{j+1}. \nonumber
\end{align}
The coefficients satisfy
\begin{align}
&c_{m,j+1} = \frac{2m-2j-1}{2(j+1)}c_{m,j}, \qquad 0\le j\le m-2,\qquad \frac{c_{m,m-1}}{2m} = \frac1{2^{2m}m!}. \nonumber
\end{align}
Since $\Psi_0=0$, telescoping gives
\begin{align}
&d\left(\sum_{j=0}^{m-1}c_{m,j}\Phi_j\right) = -\frac1{2^{2m}m!}\Psi_m. \nonumber
\end{align}
Thus $d\Pi_{2m-1}=-\mathcal E_{2m}$. Chern's smooth-boundary formula \cite[p.~678, formula~(19), and p.~679]{ChernCurvatura1945} gives
\begin{align}
&\chi(D)=\int_D\mathcal E_{2m}+\int_{\partial D}\nu^*\Pi_{2m-1}. \nonumber
\end{align}
For the same boundary correction without a product metric, see \cite{GilkeyBoundary1975}. On the orientation double cover $\widetilde D\to D$,
\begin{align}
&\chi(\widetilde D)=2\chi(D),\qquad \int_{\widetilde D}\mathcal E_{2m}=2\int_D\mathcal E_{2m},\qquad \int_{\partial\widetilde D}\widetilde\nu^*\Pi_{2m-1}=2\int_{\partial D}\nu^*\Pi_{2m-1}. \nonumber
\end{align}
\end{proof}

Use $P_3,E_3,W_2$ from Notations~\ref{notation-P3-E3}
and~\ref{notation-W2-S}, equations~\eqref{audit:eq:P3-E3}
and~\eqref{audit:eq:W2-definition}.
In dimension six put
\begin{align}
&Q(S):=Q_6(S), \label{six:eq:P3-W2}
\end{align}
Then $E_3=0$, and $P_3=1$ at sectional curvature one.
No sign of $P_3$ is assumed. The equation for $F$ is not used.

\begin{lemma}
\label{six:lem:CGB}
Every compact smooth six-dimensional Riemannian manifold $D$
with smooth boundary has outward unit normal $\nu$ and shape operator
$S(X)=\nabla_X\nu$. Use $P_3$ and $W_2$ from
Notations~\ref{notation-P3-E3} and~\ref{notation-W2-S},
equations~\eqref{audit:eq:P3-E3} and~\eqref{audit:eq:W2-definition},
and $Q$ from equation~\eqref{six:eq:P3-W2}. Then
\begin{align}
&\int_{\partial D}\left(\det S+\frac14Q(S)\right)dA +\frac38\int_{\partial D}W_2(S)\,dA +\frac{15}{8}\int_D P_3\,dV =\pi^3\chi(D). \label{six:eq:CGB}
\end{align}
No convexity assumption and no product-metric assumption near
$\partial D$ is imposed.
\end{lemma}

\begin{proof}
Apply Lemma~\ref{general:lem:CGB-boundary}, equation~\eqref{general:eq:CGB-boundary}, with $m=3$; see \cite[pp.~675--679, especially formulas~(4), (6)--(11), and (19)]{ChernCurvatura1945}. In convention~\eqref{pole:eq:connection-forms},
\begin{align}
&c_{3,0}=\frac1{120},\qquad c_{3,1}=\frac1{48},\qquad c_{3,2}=\frac1{64},\qquad \Pi_5=\frac1{\pi^3}\left(\frac{\Phi_0}{120}+\frac{\Phi_1}{48}+\frac{\Phi_2}{64}\right), \nonumber\\
&\nu^*\Phi_0=120\det S\,dA,\qquad \nu^*\Phi_1=12Q(S)\,dA,\qquad \nu^*\Phi_2=24W_2(S)\,dA, \nonumber\\
&\pi^3\nu^*\Pi_5=\left(\det S+\frac14Q(S)+\frac38W_2(S)\right)dA. \nonumber
\end{align}
By Notation~\ref{notation-P3-E3}, equation~\eqref{audit:eq:P3-E3},
\begin{align}
&\Psi_3=720P_3\,dV,\qquad \mathcal E_6=\frac{\Psi_3}{2^6\pi^3\,3!}=\frac{15}{8\pi^3}P_3\,dV. \nonumber
\end{align}
Chern's boundary formula \cite[p.~678, formula~(19), and p.~679]{ChernCurvatura1945} yields
\begin{align}
&\chi(D)=\int_D\mathcal E_6+\int_{\partial D}\nu^*\Pi_5=\frac{15}{8\pi^3}\int_D P_3\,dV+\frac1{\pi^3}\int_{\partial D}\left(\det S+\frac14Q(S)+\frac38W_2(S)\right)dA. \nonumber
\end{align}
\end{proof}

\begin{lemma}\label{lem-general-ineq-for-curvature-2-3-order-terms}
Let $(M^n,g)$ be a smooth, complete, connected, noncompact
Riemannian manifold and with $\sec_g\ge0$ and $n\geq 6$. Fix $0<\varepsilon<1/8$ and
let $\zeta_\varepsilon$ be the mollifier in
equation~\eqref{interval:eq:mollifier}. For each $r>0$, define
\begin{align}
&\theta_\varepsilon(s)=\int_{1-\varepsilon}^{2+\varepsilon} \zeta_\varepsilon(s-\tau)\,d\tau,\qquad b_{r,\varepsilon}(x)=r^{6-n}\int_{F(x)/r}^{\infty} \theta_\varepsilon(s)\,ds,\qquad \Omega_r=\{r/2<F<5r/2\}, \nonumber\\
&\delta_r=\operatorname*{ess\,sup}_{\Omega_r}|\nabla F-\nabla \mathfrak{b}| +\sup_{\Omega_r}||\nabla F|^2-1|. \nonumber
\end{align}
The cutoff satisfies
\begin{align}
&0\le\mathbf1_{[1,2]}\le\theta_\varepsilon\le1,\qquad \supp\theta_\varepsilon\subset[1-2\varepsilon,2+2\varepsilon], \nonumber\\
&\int_{\mathbb R}\theta_\varepsilon(s)\,ds=1+2\varepsilon,\qquad \|\theta_\varepsilon'\|_\infty\le C/\varepsilon. \label{interval:eq:upper-cutoff}
\end{align}
Then $\delta_r\to0$ as $r\to\infty$, and there are
$r_0<\infty$ and $C_\varepsilon<\infty$, with $C_\varepsilon$
independent of $r$, such that, for every $r\ge r_0$,
\begin{align}
&r^{5-n}\int_0^\infty\theta_\varepsilon(t/r) \int_{\Sigma_t}W_2(S)\,dA\,dt +5\int_M b_{r,\varepsilon} \bigl(P_3-E_3(\nabla \mathfrak{b},\nabla \mathfrak{b})\bigr)\,dV \ge-C_\varepsilon\delta_r. \label{audit:eq:general-combined-lower}
\end{align}
\end{lemma}

\begin{proof}
Fix $n\ge6$ and $0<\varepsilon<1/8$. Since $\supp\zeta_\varepsilon\subset[-\varepsilon,\varepsilon]$ and $\int\zeta_\varepsilon=1$,
\begin{align}
&s\in[1,2]\quad\Longrightarrow\quad[s-\varepsilon,s+\varepsilon]\subset[1-\varepsilon,2+\varepsilon]\quad\Longrightarrow\quad\theta_\varepsilon(s)=1, \nonumber\\
&0\le\theta_\varepsilon\le1,\qquad\supp\theta_\varepsilon\subset[1-2\varepsilon,2+2\varepsilon]\Subset(1/2,5/2),\qquad\int_{\mathbb R}\theta_\varepsilon=1+2\varepsilon, \nonumber\\
&\theta_\varepsilon'(s)=\zeta_\varepsilon(s-1+\varepsilon)-\zeta_\varepsilon(s-2-\varepsilon),\qquad\|\theta_\varepsilon'\|_\infty\le2\|\zeta_\varepsilon\|_\infty\le C/\varepsilon. \nonumber
\end{align}
Put $\chi_{r,\varepsilon}=-b_{r,\varepsilon}$. Corollary~\ref{one:cor:gradient-u}, equation~\eqref{one:eq:gradient-u}, and Lemma~\ref{one:lem:fixed-F}, equation~\eqref{eq:gradient-limits}, give $\lim_{r\to\infty}\delta_r=0$. Hold $\varepsilon$ fixed as $r\to\infty$.

For $A=\Hess F+h\Id\ge0$, positivity in
Lemma~\ref{pt:lem:algebra}, equation~\eqref{qc:eq:J-positive}, and
Lemma~\ref{lem-E3-property},
equation~\eqref{audit:eq:E3-contractions}, give
\begin{align}
&|\mathcal J_{\Hess F}| \le\tr\mathcal J_A+(n-5)h\tr E_2=(n-5)\bigl(\tr(E_2\circ\Hess F)+2h\tr E_2\bigr). \nonumber
\end{align}
Choose
\begin{align}
&0\le\xi_0\in C_c^\infty((1/4,3)),\qquad \xi_0\le1,\qquad\xi_0|_{[1/2,5/2]}=1,\qquad\xi_r=\xi_0(F/r),\qquad |\nabla\xi_r|\le C/r. \nonumber
\end{align}
By $\diver E_2=0$, $\tr E_2=(n-4)P_2$, $h\le Cr^{-3}$, and Lemma~\ref{pt:lem:mass}, equation~\eqref{pt:eq:mass-bound},
\begin{align}
&\int_{\Omega_r}|\mathcal J_{\Hess F}|\,dV \le C_n\left(-\int E_2(\nabla\xi_r,\nabla F)\,dV +2\int\xi_rh\tr E_2\,dV\right) \nonumber\\
&{}\le C(r^{-1}+r^{-3})\int_{\supp\xi_r}P_2\,dV \le Cr^{n-5}. \label{one:eq:J-annulus}
\end{align}
Moreover,
\begin{align}
&0\le b_{r,\varepsilon}\in C_c^\infty(D_{5r/2}),\qquad b_{r,\varepsilon}|_{\{F\le(1-2\varepsilon)r\}}=r^{6-n}(1+2\varepsilon). \nonumber
\end{align}
Differentiation gives
\begin{align}
&\nabla\chi_{r,\varepsilon} =r^{5-n}\theta_\varepsilon(F/r)\nabla F,\qquad \Hess\chi_{r,\varepsilon} =r^{5-n}\theta_\varepsilon(F/r)\Hess F +r^{4-n}\theta_\varepsilon'(F/r)dF\otimes dF. \nonumber
\end{align}
Moreover $\mathcal J_{dF\otimes dF}\ge0$ and
\begin{align}
&\tr\mathcal J_{dF\otimes dF} =(n-5)E_2(\nabla F,\nabla F)\le C_nP_2. \nonumber
\end{align}
Together with \eqref{one:eq:J-annulus}, this proves
\begin{align}
&\int_M|\mathcal J_{\Hess\chi_{r,\varepsilon}}|\,dV \le Cr^{5-n}\int_{\Omega_r}|\mathcal J_{\Hess F}|\,dV +C\varepsilon^{-1}r^{4-n}\int_{\Omega_r}P_2\,dV\le C(1+\varepsilon^{-1}). \nonumber
\end{align}
Since $\sup(|\nabla F|+|\nabla \mathfrak{b}|)\le C$,
\begin{align}
&\left|\int\mathcal J_{\Hess\chi_{r,\varepsilon}}(\nabla F,\nabla F)\,dV -\int\mathcal J_{\Hess\chi_{r,\varepsilon}}(\nabla \mathfrak{b},\nabla \mathfrak{b})\,dV\right| \le C_\varepsilon\delta_r. \nonumber
\end{align}
Apply Lemma~\ref{one:lem:weak-convex},
equation~\eqref{audit:eq:general-weak-convex}, to
$b_{r,\varepsilon}$.
Since $\chi_{r,\varepsilon}=-b_{r,\varepsilon}$, it follows that
\begin{align}
&\int\mathcal J_{\Hess\chi_{r,\varepsilon}}(\nabla F,\nabla F)\,dV +5\int b_{r,\varepsilon} \bigl(P_3-E_3(\nabla \mathfrak{b},\nabla \mathfrak{b})\bigr)\,dV \ge-C_\varepsilon\delta_r. \nonumber
\end{align}
The Hessian-transfer identity in
Lemma~\ref{audit:lem:cubic-divergence},
equation~\eqref{one:eq:Hessian-transfer}, and coarea give
\begin{align}
&\int\mathcal J_{\Hess\chi_{r,\varepsilon}}(\nabla F,\nabla F)\,dV =r^{5-n}\int\theta_\varepsilon(F/r)|\nabla F|^3W_2\,dV \nonumber\\
&{}=r^{5-n}\int_0^\infty\theta_\varepsilon(t/r) \int_{\Sigma_t}|\nabla F|^2W_2\,dA\,dt. \nonumber
\end{align}
The $\theta_\varepsilon'$ term vanishes since $\mathcal J_{dF\otimes dF}(\nabla F,\nabla F)=0$. Also $\mathcal J_{\Hess F}(\nu,\nu)=|\nabla F|W_2$, so
\begin{align}
&r^{5-n}\int_0^\infty\theta_\varepsilon(t/r) \int_{\Sigma_t}|W_2|\,dA\,dt \le r^{5-n}\int_{\Omega_r}|\mathcal J_{\Hess F}|\,dV\le C. \nonumber
\end{align}
Consequently,
\begin{align}
&r^{5-n}\left|\int_0^\infty\theta_\varepsilon(t/r)\int_{\Sigma_t}(|\nabla F|^2-1)W_2\,dA\,dt\right|\le C\sup_{\Omega_r}\bigl||\nabla F|^2-1\bigr|\le C\delta_r, \nonumber\\
&r^{5-n}\int_0^\infty\theta_\varepsilon(t/r)\int_{\Sigma_t}W_2\,dA\,dt+5\int_Mb_{r,\varepsilon}\bigl(P_3-E_3(\nabla \mathfrak{b},\nabla \mathfrak{b})\bigr)\,dV\ge-C_\varepsilon\delta_r. \nonumber
\end{align}
\end{proof}

\begin{lemma}\label{audit:lem:combined-boundary}
Let $(M^6,g)$ be a smooth, complete, connected, noncompact
Riemannian manifold and with $\sec_g\ge0$. For every regular exterior level set
 define $Q_6(S)$ and
$\mathcal A_6(F,t)$ by equation~\eqref{uni:eq:common-boundary}. Then
\begin{align}
&\limsup_{r\to\infty}\frac1r\int_r^{2r} \int_{\Sigma_t}\left(\det S+\frac14Q_6(S)\right)dA\,dt \le\pi^3\chi(M). \label{six:eq:averaged-boundary}
\end{align}
\end{lemma}

\begin{proof}
For $n=6$, $E_3=0$ and $r^{6-n}=1$. Lemma~\ref{lem-general-ineq-for-curvature-2-3-order-terms}, equation~\eqref{audit:eq:general-combined-lower}, gives
\begin{align}
&\frac1r\int_0^\infty\theta_\varepsilon(t/r) \int_{\Sigma_t}W_2\,dA\,dt +5\int_M b_{r,\varepsilon}P_3\,dV\ge-C_\varepsilon\delta_r. \nonumber
\end{align}
For fixed $r$,
\begin{align}
&\frac1r\int_0^\infty\theta_\varepsilon(t/r)\int_{D_t}|P_3|\,dV\,dt\le(1+2\varepsilon)\Vol(D_{5r/2})\sup_{D_{5r/2}}|P_3|<\infty. \nonumber
\end{align}
Thus Fubini gives
\begin{align}
&\frac1r\int_0^\infty\theta_\varepsilon(t/r) \int_{D_t}P_3\,dV\,dt =\int_M b_{r,\varepsilon}P_3\,dV. \nonumber
\end{align}
By Lemma~\ref{six:lem:CGB}, equation~\eqref{six:eq:CGB}, and $\chi(D_t)=\chi(M)$ on $\supp\theta_\varepsilon(t/r)$,
\begin{align}
&\frac1r\int_0^\infty\theta_\varepsilon(t/r)\mathcal A_6(F,t)\,dt=\pi^3\chi(M)(1+2\varepsilon) \nonumber\\
&{}-\frac38\left(\frac1r\int_0^\infty\theta_\varepsilon(t/r)\int_{\Sigma_t}W_2\,dA\,dt+5\int_Mb_{r,\varepsilon}P_3\,dV\right). \nonumber
\end{align}
Hence
\begin{align}
&\frac1r\int_0^\infty\theta_\varepsilon(t/r) \mathcal A_6(F,t)\,dt \le\pi^3\chi(M)(1+2\varepsilon)+C_\varepsilon\delta_r. \label{one:eq:fixed-topological-bound}
\end{align}

By \eqref{uni:eq:common-boundary} and Lemma~\ref{six:lem:compact}, equation~\eqref{six:eq:negative-determinants}, applied at radii $r/2$ and $2r$,
\begin{align}
&\mathcal A_6(F,t)_-\le\int_{\Sigma_t}\bigl((\det S)_-+(Q_6(S))_-/4\bigr)\,dA, \nonumber\\
&\lim_{r\to\infty}\frac1r\left(\int_{r/2}^{r}\mathcal A_6(F,t)_-\,dt+\int_{2r}^{4r}\mathcal A_6(F,t)_-\,dt\right)=0. \nonumber
\end{align}
On regular compact level bands, $\mathcal A_6(F,\cdot)\in L^1_{\mathrm{loc}}$. For large $r$, Lemma~\ref{lem-general-ineq-for-curvature-2-3-order-terms}, equation~\eqref{interval:eq:upper-cutoff}, gives
\begin{align}
&0\le\theta_\varepsilon(t/r)-\mathbf1_{[r,2r]}(t)\le1,\qquad\supp\bigl(\theta_\varepsilon(\cdot/r)-\mathbf1_{[r,2r]}\bigr)\subset[r/2,r]\cup[2r,4r], \nonumber\\
&\frac1r\int_0^\infty\bigl(\theta_\varepsilon(t/r)-\mathbf1_{[r,2r]}(t)\bigr)\mathcal A_6(F,t)\,dt \nonumber\\
&{}\ge-\frac1r\left(\int_{r/2}^{r}\mathcal A_6(F,t)_-\,dt+\int_{2r}^{4r}\mathcal A_6(F,t)_-\,dt\right). \nonumber
\end{align}
Therefore
\begin{align}
&\frac1r\int_r^{2r}\mathcal A_6(F,t)\,dt\le\frac1r\int_0^\infty\theta_\varepsilon(t/r)\mathcal A_6(F,t)\,dt \nonumber\\
&\qquad+\frac1r\left(\int_{r/2}^{r}\mathcal A_6(F,t)_-\,dt+\int_{2r}^{4r}\mathcal A_6(F,t)_-\,dt\right). \nonumber
\end{align}
For each fixed $\varepsilon$, take $r\to\infty$ in \eqref{one:eq:fixed-topological-bound}:
\begin{align}
&\limsup_{r\to\infty}\frac1r\int_r^{2r}\mathcal A_6(F,t)\,dt \le\pi^3\chi(M)(1+2\varepsilon). \nonumber
\end{align}
Then take $\varepsilon\downarrow0$:
\begin{align}
&\limsup_{r\to\infty}\frac1r\int_r^{2r}\mathcal A_6(F,t)\,dt\le\lim_{\varepsilon\downarrow0}\pi^3\chi(M)(1+2\varepsilon)=\pi^3\chi(M). \nonumber
\end{align}
\end{proof}

\begin{proposition}\label{six:prop:Euler-bound}
Let $(M^6,g)$ be a smooth, complete, connected, noncompact
Riemannian manifold with $\sec_g\ge0$. Then
\begin{align}
&L(M,g)\le2\pi^3\bigl(\chi(M)- \AVR(M^6,g)\bigr). \label{six:eq:Euler-bound}
\end{align}
\end{proposition}

\begin{proof}
Write $L=L(M,g)$. Lemmas~\ref{uni:lem:determinant} and~\ref{six:lem:compact}, equations~\eqref{uni:eq:determinant-volume} and~\eqref{six:eq:negative-determinants}, give
\begin{align}
&\lim_{r\to\infty}\frac1r\int_r^{2r}\int_{\Sigma_t}(\det S)_-\,dA\,dt=0, \nonumber\\
&\liminf_{r\to\infty}\frac1r\int_r^{2r}\int_{\Sigma_t}\det S\,dA\,dt\ge\pi^3\AVR(M^6,g). \label{lower bound of detS}
\end{align}
Proposition~\ref{six:prop:boundary-limit}, equation~\eqref{six:eq:unweighted-boundary-limit}, gives
\begin{align}
&\lim_{r\to\infty}\frac1r\int_r^{2r}\int_{\Sigma_t}Q_6(S)\,dA\,dt=2L. \label{Qs-limit}
\end{align}
Combining \eqref{lower bound of detS} and \eqref{Qs-limit} with Lemma~\ref{audit:lem:combined-boundary}, equation~\eqref{six:eq:averaged-boundary}, yields
\begin{align}
&\pi^3\AVR(M^6,g)+\frac L2\le\liminf_{r\to\infty}\frac1r\int_r^{2r}\int_{\Sigma_t}\left(\det S+\frac14Q_6(S)\right)dA\,dt \nonumber\\
&\hspace{24pt}\le\limsup_{r\to\infty}\frac1r\int_r^{2r}\int_{\Sigma_t}\left(\det S+\frac14Q_6(S)\right)dA\,dt\le\pi^3\chi(M).
\end{align}
This proves \eqref{six:eq:Euler-bound}.
\end{proof}

\begin{theorem}
\label{uni:cor:Euler-bound}\label{low:prop:numerical-bound}\label{low:thm:numerical-bound}
Let $(M^n,g)$ be smooth, complete, connected, noncompact and without
boundary, with $\sec_g\ge0$ and $3\le n\le6$. Then
\begin{align}
&L\le4\pi\omega_{n-2}\bigl(\chi(M)-\AVR(M^n,g)\bigr), \label{uni:eq:Euler-bound}\\
&L\le8\pi\omega_{n-2}(1-\AVR(M^n,g))\le8\pi\omega_{n-2}. \label{low:eq:numerical-bound}
\end{align}
The last constant is attained by
$(\mathbb S^2,h)\times\R^{n-2}$ for every smooth $h$ with $K_h\ge0$.
\end{theorem}

\begin{proof}
By Lemma~\ref{low:lem:Euclidean-products}, \eqref{ext:eq:Euclidean-product}, the Euler characteristic, volume ratio and scalar limit of $M\times\mathbb R^{6-n}$ are
\begin{align}
&\chi(M),\qquad \AVR(M^n,g),\qquad \frac{\omega_4}{\omega_{n-2}}L. \nonumber
\end{align}
Proposition~\ref{six:prop:Euler-bound}, \eqref{six:eq:Euler-bound}, gives
\begin{align}
&\frac{\omega_4}{\omega_{n-2}}L\le2\pi^3(\chi(M)-\AVR(M^n,g))=4\pi\omega_4(\chi(M)-\AVR(M^n,g)), \nonumber\\
&L\le4\pi\omega_{n-2}(\chi(M)-\AVR(M^n,g)). \nonumber
\end{align}
For $L=0$, $0\le \AVR(M^n,g)\le1$ gives \eqref{low:eq:numerical-bound}. For $L>0$, Lemma~\ref{low:lem:Euler-topology}, \eqref{low:eq:Euler-volume-comparison}, gives
\begin{align}
&\chi(M)-\AVR(M^n,g)\le2(1-\AVR(M^n,g)), \nonumber\\
&L\le8\pi\omega_{n-2}(1-\AVR(M^n,g))\le8\pi\omega_{n-2}. \nonumber
\end{align}
Lemma~\ref{low:lem:Euclidean-products}, \eqref{ext:eq:Euclidean-product}, and Gauss--Bonnet \cite[Section~4-5]{doCarmo}, give
\begin{align}
&L((\mathbb S^2,h)\times\mathbb R^{n-2})=\omega_{n-2}\int_{\mathbb S^2}\Sc_h\,dA_h=8\pi\omega_{n-2}. \nonumber
\end{align}
\end{proof}

\begin{theorem}\label{lowdim:thm:sharp}
Let $(M^n,g)$ be a smooth, complete, connected, noncompact
Riemannian manifold without boundary, with $3\le n\le6$ and
$\sec_g\ge0$. 
Then
\begin{align}
&L(M,g)=8\pi\omega_{n-2}(1-\AVR(M^n,g)) \nonumber\\
&\quad\Longleftrightarrow\quad (M,g)\cong\R^n\ \text{or}\ (M,g)\cong(\mathbb S^2,h)\times\R^{n-2},\qquad K_h\ge0. \label{low:eq:volume-equality}
\end{align}
Here $\cong$ denotes a global Riemannian isometry, and each Euclidean
factor carries its Euclidean metric. 
\end{theorem}

\begin{proof}
Suppose first that $\dim\mathcal S=0$. The soul theorem, gives $\chi(M)=1$. Hence Theorem~\ref{low:thm:numerical-bound}, equation~\eqref{uni:eq:Euler-bound}, yields
\begin{align}
&L\le4\pi\omega_{n-2}(1-\AVR(M^n,g))\le4\pi\omega_{n-2}<8\pi\omega_{n-2}. \nonumber
\end{align}
If $\AVR(M^n,g)<1$, then
\begin{align}
&L\le4\pi\omega_{n-2}(1-\AVR(M^n,g))<8\pi\omega_{n-2}(1-\AVR(M^n,g)). \nonumber
\end{align}
If $\AVR(M^n,g)=1$, Bishop--Gromov volume comparison and its equality case \cite{CE} give
\begin{align}
&1\ge\frac{\Vol_gB_p(r)}{\omega_nr^n}\ge \AVR(M^n,g)=1\quad(r>0),\qquad (M,g)\cong\mathbb R^n,\qquad L=0. \nonumber
\end{align}
This proves the point-soul alternative in \eqref{low:eq:volume-equality}.

Suppose now that $\dim\mathcal S>0$. Lemma~\ref{soul:lem:hessian} gives $\AVR(M^n,g)=0$, so
\begin{align}
&L=8\pi\omega_{n-2}(1-\AVR(M^n,g))\quad\Longleftrightarrow\quad L=8\pi\omega_{n-2}. \nonumber
\end{align}
Theorem~\ref{ext:thm:positive-soul}, equation~\eqref{ext:eq:positive-soul-equality}, classifies both equalities as
\begin{align}
&(M,g)\cong(\mathbb S^2,h)\times\mathbb R^{n-2},\qquad K_h\ge0. \nonumber
\end{align}
Conversely, Euclidean space and these products have
\begin{align}
&(v_{\mathbb R^n},L(\mathbb R^n))=(1,0), \nonumber\\
&(v_{(\mathbb S^2,h)\times\mathbb R^{n-2}},L((\mathbb S^2,h)\times\mathbb R^{n-2}))=(0,8\pi\omega_{n-2}), \nonumber
\end{align}
by Lemma~\ref{low:lem:Euclidean-products}, equation~\eqref{ext:eq:Euclidean-product}, and Theorem~\ref{ext:thm:positive-soul}. This proves both equivalences. 
\end{proof}

\section{Compact sections and the curvature normalization}\label{cmp:sec:compact}

%\subsection{Objects, external inputs, and dependence of the proof}
In this section $m\ge3$ denotes the dimension of a \emph{compact}
manifold; the corresponding noncompact cone has dimension $m+1$.
Let $(N^m,h)$ be smooth, connected, compact and without boundary. Put
\begin{align}
&J(h)=\int_N\Sc_h\,dV_h, V(h)=\Vol_h(N), \nonumber\\
&\Area(\mathbb S^m,g_{\mathbb S^m})=(m+1)\omega_{m+1}, c_m=(m-1)(m-2). \label{cmp:eq:data}
\end{align}
The round metric $g_{\mathbb S^m}$ has sectional curvature one.
The curvature operator $\mathscr R_h$ is the self-adjoint endomorphism
of $\Lambda^2TN$ defined by
\begin{align}
&\langle\mathscr R_h(X\wedge Y),Z\wedge W\rangle =\langle\Rm_h(X,Y)W,Z\rangle. \nonumber
\end{align}
Write $G_h$ for the curvature tensor whose operator is the identity,
so $(G_h)^{ab}_{cd}=\delta^a_c\delta^b_d-\delta^a_d\delta^b_c$
in an orthonormal frame. A curvature tensor $A$ is called
operator-nonnegative here precisely when its operator is positive
semidefinite on \emph{all} of $\Lambda^2TN$.

For $k_0>0$, let $\mathcal C_m^{\mathrm{op}}(k_0)$ be the class of all
pairs $(N^m,h)$ with the smoothness, connectedness, compactness and
empty-boundary conditions above and with
$\mathscr R_h\ge k_0\Id$. Define $\mathcal C_m^{\mathrm{sec}}(k_0)$
by the same conditions, replacing the operator inequality by
$K_h(\Pi)\ge k_0$ for every tangent two-plane $\Pi$ at every point.
All suprema over compact manifolds in this section refer to these
classes; no orientability or simple-connectedness condition is imposed.

For a curvature tensor $A$ on an $m$-dimensional Euclidean space,
put $P_0(A)=1$ and, for $1\le j\le\lfloor m/2\rfloor$, define
\begin{align}
&P_j(A)=\frac{1}{2^j(2j)!} \delta^{i_1\cdots i_{2j}}_{l_1\cdots l_{2j}} \prod_{s=1}^j A^{l_{2s-1}l_{2s}}_{i_{2s-1}i_{2s}}. \label{cmp:eq:Pk-definition}
\end{align}
All indices range from $1$ to $m$. In particular
$P_1(A)=\Sc(A)/2$, and $P_2,P_3$ agree with their earlier definitions
in this manuscript.

An expanding gradient Ricci soliton is a triple $(X,G,f)$ with a
complete smooth metric $G$, a smooth real function $f$, and
$\Hess_Gf=\Ric_G+G/2$. The asymptotic coordinate estimates used for
these triples are stated explicitly in Lemma~\ref{cmp:lem:expander},
equation~\eqref{cmp:eq:AC-estimates}.

The following four cited theorems are used with the stated hypotheses.
\begin{externaltheorem}\label{ext:thm:closed-CGB}
For an oriented closed Riemannian $2k$-manifold $Y$, the Euler
integral formula, with $P_k$ defined in
\eqref{cmp:eq:Pk-definition}, is
\begin{align}
&\int_Y P_k(\Rm_Y)\,dV_Y=\frac{(2k+1)\omega_{2k+1}}2\chi(Y). \label{cmp:eq:closed-Euler}
\end{align}
\end{externaltheorem}
\noindent\emph{Source.} This is the normalization of \cite[equations~(10) and~(18), pp.~677--678]{ChernCurvatura1945}; see also \cite[pp.~747--752]{Chern1944}; its value on the unit
$2k$-sphere fixes the coefficient. No sign of the Euler integrand is
part of this imported statement.

\begin{externaltheorem}\label{ext:thm:expanders}
For smooth simply connected compact $(N^m,h)$ with
$\mathscr R_h>\Id$, Deruelle's existence theorem
\cite[Theorem~1.3, p.~2]{Deruelle2016} supplies a complete expanding gradient
Ricci soliton $(X^{m+1},G,f)$ with positive curvature operator and smooth
asymptotic cone $dr^2+r^2h$. For the compactness statement, fix an integer $n\ge3$, $V_0>0$, and
finite constants $\Lambda_k\ge0$ for every integer $k\ge0$. Consider
complete expanders $(X^n,G,f,p)$ normalized by
\begin{align}
&\Hess_G f=\Ric_G+\frac12G,\qquad \int_X e^{-f}\,dV_G=(4\pi)^{n/2},\qquad \{x\in X:\nabla f(x)=0\}=\{p\}. \nonumber
\end{align}
Assume their curvature operators are nonnegative and
\begin{align}
&\lim_{R\to\infty}R^{-n}\Vol_G B_p(R)\ge V_0,\qquad \limsup_{d_G(p,x)\to\infty}d_G(p,x)^{2+k}|\nabla^k\Rm_G(x)|_G\le\Lambda_k\qquad(k\ge0). \nonumber
\end{align}
Then every sequence has a subsequence converging to an expander in the
same class, smoothly in the pointed sense and in the smooth conical
topology of \cite[Theorem~4.9, p.~51]{Deruelle2016}; in particular the
asymptotic cones converge in that topology. The normalization of the
volume ratio here is the one in that source, without division by
$\omega_n$. Its potential normalization is specified on p.~53.
Remark~4.11 (p.~52) relates the displayed curvature-derivative bounds to
uniform bounds on the covariant curvature derivatives of the links.

The non-strict endpoint used below is deduced in
Lemma~\ref{cmp:lem:expander}, not inferred from the strict theorem alone.
\end{externaltheorem}
\noindent\emph{Source.} All theorem numbers and pages refer to the 55-page author version of
\cite{Deruelle2016} specified in the bibliography.
The compactness assumptions are the bounds in Theorem~4.9, p.~51;
no compactness assertion for arbitrary expanders is being invoked.

\begin{externaltheorem}\label{ext:thm:expander-dichotomy}
A nonflat complete expanding gradient Ricci soliton with
nonnegative curvature operator and a smooth compact asymptotic link
is either positive on all two-forms or is K\"ahler and positive on
real $(1,1)$-forms. This consequence of the curvature-operator strong
maximum principle and holonomy classification is the dichotomy in
\cite[Section~4.3.3, pp.~52--53, proof of Theorems~1.3--1.4]{Deruelle2016}.
Here all curvature derivatives are bounded on compact sets, and the
curvature is globally bounded, with
$\lim_{r\to\infty}\sup_{d_G(p,x)\ge r}|\Rm_G(x)|=0$.
The image-bundle form of Hamilton's strong maximum principle is stated
in \cite[p.~1]{Kotschwar2014}, and the holonomy classification used in that argument is due to
\cite{Berger1955}. We use the dichotomy only for the expanders in Theorem~\ref{ext:thm:expanders}.
\end{externaltheorem}
\noindent\emph{Source.} The page numbers for \cite{Deruelle2016} refer to the 55-page author version; the cited statement in \cite{Kotschwar2014} is on p.~1 of arXiv version~1.

\begin{externaltheorem}\label{ext:thm:Bishop-rigidity}
For a smooth, connected, compact $m$-manifold $(Y,k)$ without
boundary, $m\ge3$, the inequality $\Ric_k\ge(m-1)k$ implies
$\Vol_k(Y)\le(m+1)\omega_{m+1}$. Equality holds if and only if $(Y,k)$ is
isometric to the unit round sphere. The volume inequality is Bishop's
theorem, stated in \cite[Theorem~1.1, p.~1]{BrayGuiLiuZhang2019}; the precise
maximal-volume rigidity is recorded in
\cite[Introduction, p.~2, paragraph following Theorem~0.1]{ChenRongXu2019}.
For bounds retaining the covering degree, we apply this result to the
universal cover after proving its compactness and finite degree.
\end{externaltheorem}
\noindent\emph{Source.} Pages refer to arXiv:1903.12317v1 and arXiv:1604.06986v3, respectively, as specified in the bibliography.

For clarity, the K\"ahler alternative in Theorem~\ref{ext:thm:expander-dichotomy} means that there is a
parallel orthogonal complex structure $\mathcal I$ and the curvature
operator is positive definite on the real two-forms satisfying
$\omega(\mathcal I X,\mathcal I Y)=\omega(X,Y)$.
In both alternatives its image contains a two-form whose top exterior
power is nonzero when $\dim X$ is even.

All constants in the asymptotic estimates below depend on the fixed
link, the chosen expander, the derivative order and fixed cutoffs.
They are independent of the parameter $T$ used to exhaust and rescale
that expander. The constants in the final inequalities are explicit
functions of $m$ and, when present, $k_0$; the actual finite covering
degree $d_N$ is displayed separately. The proof has the following order:
\begin{itemize}[leftmargin=*]
\item Theorem~\ref{ext:thm:closed-CGB},
equation~\eqref{cmp:eq:closed-Euler}, and
Lemma~\ref{cmp:lem:Euler-algebra},
equation~\eqref{cmp:eq:positive-wedges}, yield the even-dimensional
estimate.
\item Theorem~\ref{ext:thm:expanders} and
Lemma~\ref{cmp:lem:conformal-cap},
equations~\eqref{cmp:eq:cap-volume}--\eqref{cmp:eq:cap-scalar}, yield
the odd-dimensional estimate.
\item Theorem~\ref{ext:thm:expander-dichotomy},
Proposition~\ref{cmp:prop:even-identity},
equation~\eqref{cmp:eq:even-deficit}, and
Lemma~\ref{cmp:lem:conformal-cap},
equation~\eqref{cmp:eq:core-lower-bound}, yield odd-dimensional round
rigidity.
\item Theorem~\ref{cmp:thm:all-operator},
equation~\eqref{cmp:eq:compact-bound}, and
Theorem~\ref{ext:thm:Bishop-rigidity} yield the maximum of $J$ and
rigidity.
\end{itemize}
None of these arguments assumes the unrestricted sharp bound in
noncompact dimension $m+1$.

\Needspace{6\baselineskip}
%\subsection{Scaling}
\begin{lemma}\label{cmp:lem:scaling}
Let $m\ge3$ and let $(N^m,h)$ be smooth, connected, compact, and
without boundary; define $J,V$ as in \eqref{cmp:eq:data}. Then, for
every $a>0$,
\begin{align}
&J(a^2h)=a^{m-2}J(h),\qquad V(a^2h)=a^mV(h). \label{cmp:eq:scaling}
\end{align}
\end{lemma}
\begin{proof}
Under $h\mapsto a^2h$,
\begin{align}
&\Sc_{a^2h}=a^{-2}\Sc_h,\qquad dV_{a^2h}=a^m dV_h, \nonumber\\
&J(a^2h)=a^{m-2}J(h),\qquad V(a^2h)=a^mV(h). \nonumber
\end{align}
\end{proof}

%\subsection{Euler contractions and a compact even-dimensional identity}
\begin{lemma}
\label{cmp:lem:Euler-algebra}
Let $A$ be an algebraic curvature tensor on an $m$-dimensional
Euclidean space. If $A$ is operator-nonnegative, let
$\lambda_\alpha\ge0$ and $\omega_\alpha$ be its eigenvalues and an
orthonormal basis of eigen-two-forms. Then, for every integer
$1\le j\le\lfloor m/2\rfloor$,
\begin{align}
&P_j(A)=\frac{2^j}{(2j)!} \sum_{\alpha_1,\ldots,\alpha_j} \lambda_{\alpha_1}\cdots\lambda_{\alpha_j} |\omega_{\alpha_1}\wedge\cdots\wedge\omega_{\alpha_j}|^2\ge0. \label{cmp:eq:positive-wedges}
\end{align}
For each such $j$, the polarization of $P_j$ is nonnegative on
operator-nonnegative arguments. Consequently $P_j(A+B)\ge P_j(A)$
when $A,B\ge0$.
Independently of this sign assumption, let $G$ denote the
identity-curvature tensor on this Euclidean space.
If $m=2k$, the following polynomial identity holds for every algebraic
curvature tensor $A$, with no operator or sectional-curvature sign:
\begin{align}
&P_k(G+A) =1+\frac{\Sc(A)}{2(m-1)} +\sum_{j=2}^k b_{k,j}P_j(A), b_{k,j}=\frac{\binom{k}{j}}{\binom{2k}{2j}}>0. \label{cmp:eq:Euler-expansion}
\end{align}
For operator-nonnegative $A$ and $k\ge2$, all the terms with $j\ge2$ vanish if and only if
$P_2(A)=0$; this is equivalent to
\begin{align}
&\omega\wedge\eta=0 \quad\text{for every }\omega,\eta\in\operatorname{im}A. \label{cmp:eq:wedge-equality}
\end{align}
\end{lemma}
\begin{proof}
Substitute the spectral decomposition in \eqref{cmp:eq:Pk-definition} and alternate the upper and lower indices:
\begin{align}
&A^{ab}_{cd}=\sum_\alpha\lambda_\alpha(\omega_\alpha)^{ab}(\omega_\alpha)_{cd}, \nonumber\\
&P_j(A)=\frac{2^j}{(2j)!}\sum_{\alpha_1,\ldots,\alpha_j}\lambda_{\alpha_1}\cdots\lambda_{\alpha_j}|\omega_{\alpha_1}\wedge\cdots\wedge\omega_{\alpha_j}|^2\ge0. \nonumber
\end{align}
Separate spectral decompositions in the $j$ arguments give the same sum with nonnegative products of their eigenvalues; thus the polarization is nonnegative. Consequently
\begin{align}
&A,B\ge0\ \Longrightarrow\ P_j(A+B)\ge P_j(A). \nonumber
\end{align}
For $m=2k$, contraction of $k-j$ identity factors gives $2^{k-j}(m-2j)!$ times the remaining alternating delta. Therefore
\begin{align}
&\frac{\binom{k}{j}}{2^k m!}\,2^{k-j}(m-2j)!\,2^j(2j)!=\frac{\binom{k}{j}}{\binom{2k}{2j}}=b_{k,j}, \nonumber\\
&P_k(G+A)=\sum_{j=0}^k\frac{\binom{k}{j}}{\binom{2k}{2j}}P_j(A)=1+\frac{P_1(A)}{m-1}+\sum_{j=2}^kb_{k,j}P_j(A), \nonumber\\
&P_1(A)=\frac{\Sc(A)}2,\qquad P_j(zG)=\binom m{2j}z^j. \nonumber
\end{align}
This proves \eqref{cmp:eq:Euler-expansion}. For $A\ge0$ and $k\ge2$, \eqref{cmp:eq:positive-wedges} gives
\begin{align}
&P_2(A)=0\ \Longleftrightarrow\ \omega_\alpha\wedge\omega_\beta=0\quad\text{whenever }\lambda_\alpha\lambda_\beta>0, \nonumber\\
&\operatorname{im}A=\operatorname{span}\{\omega_\alpha:\lambda_\alpha>0\}, \nonumber\\
&P_2(A)=0\ \Longleftrightarrow\ \omega\wedge\eta=0\quad(\omega,\eta\in\operatorname{im}A) \nonumber\\
&\ \Longleftrightarrow\ P_j(A)=0\quad(2\le j\le k). \nonumber
\end{align}
The middle equivalence is bilinearity; the last follows by factoring every exterior product through its first two factors.
\end{proof}

\begin{proposition}
\label{cmp:prop:even-identity}
Let $m=2k\ge4$ and let $(N^m,h)$ be smooth, connected, compact, and
without boundary. Use the notation in \eqref{cmp:eq:data} and put
$A_h=\Rm_h-G_h$. Without a curvature assumption,
\begin{align}
&J(h)-c_mV(h) =(m-1)(m+1)\omega_{m+1}\chi(N) -2(m-1)\sum_{j=2}^k b_{k,j}\int_NP_j(A_h)\,dV_h. \label{cmp:eq:even-deficit}
\end{align}
If $\mathscr R_h\ge\Id$, the last sum is nonnegative. Its vanishing
is equivalent to $P_2(A_h)\equiv0$.
\end{proposition}
\begin{proof}
For oriented $N$, Lemma~\ref{cmp:lem:Euler-algebra}, \eqref{cmp:eq:Euler-expansion}, and Theorem~\ref{ext:thm:closed-CGB}, \eqref{cmp:eq:closed-Euler}, give
\begin{align}
&\Sc(A_h)=\Sc_h-m(m-1), \nonumber\\
&\frac{(m+1)\omega_{m+1}}2\chi(N)=V(h)+\frac{J(h)-m(m-1)V(h)}{2(m-1)}+\sum_{j=2}^kb_{k,j}\int_NP_j(A_h)\,dV_h, \nonumber\\
&J(h)-c_mV(h)=(m-1)(m+1)\omega_{m+1}\chi(N)-2(m-1)\sum_{j=2}^kb_{k,j}\int_NP_j(A_h)\,dV_h. \nonumber
\end{align}
For the orientation double cover $\pi:\widetilde N\to N$,
\begin{align}
&\chi(\widetilde N)=2\chi(N),\qquad \int_{\widetilde N}\pi^*(u\,dV_h)=2\int_Nu\,dV_h, \nonumber
\end{align}
and division by two gives the same identity. If $\mathscr R_h\ge\Id$, Lemma~\ref{cmp:lem:Euler-algebra}, \eqref{cmp:eq:positive-wedges} and \eqref{cmp:eq:wedge-equality}, and continuity imply
\begin{align}
&P_j(A_h)\ge0,\qquad \sum_{j=2}^kb_{k,j}\int_NP_j(A_h)\,dV_h=0\ \Longleftrightarrow\ P_2(A_h)\equiv0. \nonumber
\end{align}
\end{proof}

%\subsection{A smooth realization and compact metrics in one higher dimension}
\begin{lemma}
\label{cmp:lem:expander}
Let $m\ge3$ and let $(N^m,h)$ be smooth, connected, compact, without
boundary, and simply connected. If $\mathscr R_h\ge\Id$, there is a
complete smooth expanding gradient Ricci soliton $(X^{m+1},G,f)$ with
\begin{align}
&\Hess_Gf=\Ric_G+\tfrac12G,\qquad \mathscr R_G\ge0,\qquad f-|\nabla f|_G^2=\Sc_G. \label{cmp:eq:expander}
\end{align}
The function $f$ is proper, has exactly one critical point, and every
sufficiently large sublevel $\{f\le T\}$ is a smooth closed disk.
There exist $r_0>0$ and coordinates
$(r,y)\in(r_0,\infty)\times N$ on its end with $f=r^2/4$ such that,
for every integer $j\ge0$, there is $C_j<\infty$ with
\begin{align}
&\sup_{\{r=s\}}|\nabla_{G_C}^{\,j}(G-G_C)|_{G_C} \le C_j s^{-2-j}\quad(s\ge r_0), \qquad G_C=dr^2+r^2h. \label{cmp:eq:AC-estimates}
\end{align}
In particular $N$ is diffeomorphic to $\mathbb S^m$.
\end{lemma}
\begin{proof}
For $i\ge2$, set $h_i=(1-i^{-1})^2h$. Then
\begin{align}
&\mathscr R_{h_i}=(1-i^{-1})^{-2}\mathscr R_h>\Id,\qquad \lim_{i\to\infty}h_i=h\quad\text{in }C^\infty, \nonumber\\
&\inf_{i\ge2}V(h_i)>0,\qquad \sup_{i\ge2}\sup_N|\nabla_{h_i}^{\,j}\Rm_{h_i}|_{h_i}<\infty\quad(j\ge0). \nonumber
\end{align}
Theorem~\ref{ext:thm:expanders} supplies expanders with these links. The displayed bounds give its uniform asymptotic-volume and curvature-derivative bounds; its conical compactness yields a subsequential limit with
\begin{align}
&\Hess_Gf=\Ric_G+G/2,\qquad \mathscr R_G\ge0,\qquad \text{asymptotic cone }dr^2+r^2h. \nonumber
\end{align}
The estimates \eqref{cmp:eq:AC-estimates} are those following Definition~1.2 of \cite{Deruelle2016}. In the Ricci-flat cone case,
\begin{align}
&\Ric_h=(m-1)h,\qquad \mathscr R_h-\Id\ge0,\qquad 2\tr(\mathscr R_h-\Id)=\Sc_h-m(m-1)=0, \nonumber\\
&\mathscr R_h=\Id,\qquad (N,h)\cong(\mathbb S^m,g_{\mathbb S^m}),\qquad (X,G)=(\R^{m+1},g_{\R^{m+1}}), \nonumber
\end{align}
and the conical errors are zero.
Divergence, trace, and contracted Bianchi give
\begin{align}
&\Delta f=\Sc_G+\frac{m+1}{2},\qquad \nabla\Sc_G=-2\Ric_G(\nabla f,\cdot), \nonumber\\
&d(f-|\nabla f|^2-\Sc_G)=df-2\Hess f(\nabla f,\cdot)-d\Sc_G=0. \nonumber
\end{align}
Shift $f$ by a constant so that $f-|\nabla f|^2=\Sc_G$. The induced change $r\mapsto\sqrt{r^2+\text{constant}}$ preserves \eqref{cmp:eq:AC-estimates}.
For a unit-speed minimizing geodesic $\gamma$ from a fixed $o$,
\begin{align}
&(f\circ\gamma)''\ge\tfrac12,\qquad f(x)\ge f(o)-|\nabla f(o)|d_G(o,x)+\tfrac14d_G(o,x)^2, \nonumber\\
&f^{-1}(( -\infty,T])\ \text{compact},\qquad \{\nabla f=0\}=\{p\},\qquad \Hess f(p)\ge G(p)/2. \nonumber
\end{align}
Properness gives the minimum $p$; strict convexity gives uniqueness. For the flow $\Phi_s$ of $\nabla f/|\nabla f|^2$ on regular compact level intervals,
\begin{align}
&\frac{d}{ds}f(\Phi_s(x))=1,\qquad \Phi_{b-a}:\{f=a\}\longrightarrow\{f=b\}\ \text{a diffeomorphism}\quad(\min f<a<b). \nonumber
\end{align}
Morse coordinates at the nondegenerate minimum therefore give
\begin{align}
&\{f\le T\}\cong\overline B^{m+1}\quad(T>\min f),\qquad N\cong\partial\{f\le T\}\cong\mathbb S^m\quad(T\text{ large}). \nonumber
\end{align}
\end{proof}

\begin{lemma}
\label{cmp:lem:conformal-cap}
Assume the hypotheses of Lemma~\ref{cmp:lem:expander}, let
$(X^{m+1},G,f)$ be the expander supplied there with asymptotic link
$(N^m,h)$, and put $d=m+1\ge4$. For all sufficiently large $T$ define
on $\{f\le T\}$
\begin{align}
&\phi_T=\frac{T+f}{\sqrt T},\qquad \widehat G_T=\phi_T^{-2}G. \nonumber
\end{align}
Then $\mathscr R_{\widehat G_T}\ge\Id$. There are smooth metrics
$g_T$ on $\mathbb S^d$ satisfying $\mathscr R_{g_T}\ge\Id$ such that,
with $B_j=\int_0^\pi\sin^j\theta\,d\theta$,
\begin{align}
&\lim_{T\to\infty}\Vol(\mathbb S^d,g_T)=B_mV(h), \label{cmp:eq:cap-volume}\\
&\lim_{T\to\infty}J(g_T) =m(m+1)B_mV(h) +B_{m-2}\bigl[J(h)-m(m-1)V(h)\bigr]. \label{cmp:eq:cap-scalar}
\end{align}
Moreover, there is $C<\infty$, independent of $T$ and $\varepsilon$,
such that for every $0<\varepsilon<1/10$,
\begin{align}
&\limsup_{T\to\infty}\Vol_{\widehat G_T}\{f\le\varepsilon T\} \le C\varepsilon^{d/2},\qquad \limsup_{T\to\infty}\int_{\{f\le\varepsilon T\}} \Sc_{\widehat G_T}\,dV_{\widehat G_T} \nonumber\\
&{}\le C\bigl(\varepsilon^{d/2}+\varepsilon^{(d-2)/2}\bigr). \label{cmp:eq:tip-integrability}
\end{align}
The metrics $g_T$ may be chosen by doubling $\{f\le T\}$. There are
numbers $\gamma_T$ with $0<\gamma_T\le1$ and
$\lim_{T\to\infty}\gamma_T=1$ such that, for every fixed compact
$K\subset X$ and all sufficiently large $T$, the restriction of
$g_T$ to either embedded copy of $K$ equals
$\gamma_T\phi_T^{-2}G$. If $d=2k$, on each such copy of $K$ one has
\begin{align}
&P_k(\Rm_{g_T}-G_{g_T})\,dV_{g_T} \ge P_k(\Rm_G)\,dV_G. \label{cmp:eq:core-lower-bound}
\end{align}
\end{lemma}
\begin{proof}
For a symmetric covariant tensor $B$, define $B\mathbin{\odot}G$ on $\Lambda^2TX$ by
\begin{align}
&\langle(B\mathbin{\odot}G)(X\wedge Y),Z\wedge W\rangle=B(X,Z)G(Y,W)+B(Y,W)G(X,Z) \nonumber\\
&{}-B(X,W)G(Y,Z)-B(Y,Z)G(X,W). \nonumber
\end{align}
In a $G$-orthonormal eigenbasis of $B$,
\begin{align}
&(B\mathbin{\odot}G)(e_i\wedge e_j)=(B_{ii}+B_{jj})e_i\wedge e_j,\qquad B\ge0\ \Longrightarrow\ B\mathbin{\odot}G\ge0,\qquad G\mathbin{\odot}G=2\Id. \nonumber
\end{align}
The conformal connection and its curvature are
\begin{align}
&\widehat\nabla_XY=\nabla_XY-\phi^{-1}\bigl[d\phi(X)Y+d\phi(Y)X-G(X,Y)\nabla\phi\bigr], \nonumber\\
&\mathscr R_{\phi^{-2}G}=\phi^2\mathscr R_G+\phi(\Hess\phi\mathbin{\odot}G)-|\nabla\phi|^2\Id, \nonumber
\end{align}
with operators identified by rescaled orthonormal frames. Lemma~\ref{cmp:lem:expander}, \eqref{cmp:eq:expander}, gives
\begin{align}
&f=|\nabla f|^2+\Sc_G\ge0,\qquad \Hess\phi_T=T^{-1/2}(\Ric_G+G/2), \nonumber\\
&\mathscr R_{\widehat G_T}-\Id=\phi_T^2\mathscr R_G+\frac{T+f}{T}(\Ric_G\mathbin{\odot}G)+\frac{\Sc_G}{T}\Id\ge0. \label{cmp:eq:conformal-curvature}
\end{align}
On the end,
\begin{align}
&r=2\sqrt T\tan(\theta/2),\qquad 0<\theta\le\pi/2,\qquad f/T=\tan^2(\theta/2), \nonumber\\
&\phi_T^{-2}G_C=d\theta^2+\sin^2\theta\,h=:g_s. \label{cmp:eq:suspension-metric}
\end{align}
On every fixed compact cylinder in $(0,\pi/2]\times N$, Lemma~\ref{cmp:lem:expander}, \eqref{cmp:eq:AC-estimates}, gives
\begin{align}
&\|\widehat G_T-g_s\|_{C^j}\le C_j/T\qquad(j\ge0), \nonumber
\end{align}
where the norms use a fixed finite coordinate cover and $g_s$. Choose a smooth cutoff $\zeta$ with $0\le\zeta\le1$, $\zeta=0$ on $\theta\le\pi/3$, and $\zeta=1$ near $\pi/2$. On the collar use $(1-\zeta)\widehat G_T+\zeta g_s$. It is positive definite for large $T$ and
\begin{align}
&\mathscr R_{g_s}|_{\partial_\theta\wedge TN}=\Id,\qquad \mathscr R_{g_s}|_{\Lambda^2TN}=\Id+\sin^{-2}\theta(\mathscr R_h-\Id)\ge\Id, \nonumber\\
&\|(1-\zeta)\widehat G_T+\zeta g_s-g_s\|_{C^2}\le C/T,\qquad \mathscr R_{(1-\zeta)\widehat G_T+\zeta g_s}\ge(1-C/T)\Id. \nonumber
\end{align}
Outside the collar use \eqref{cmp:eq:conformal-curvature}. Double across $\theta=\pi/2$; smoothness follows from
\begin{align}
&\sin^2(\pi/2+s)=\sin^2(\pi/2-s),\qquad \{f\le T\}\cong\overline B^d,\qquad \operatorname{Double}\{f\le T\}\cong\mathbb S^d. \nonumber
\end{align}
Multiply the doubled metric by $\gamma_T=1-C/T>0$. Then
\begin{align}
&\mathscr R_{g_T}\ge\gamma_T^{-1}(1-C/T)\Id=\Id,\qquad \lim_{T\to\infty}\gamma_T=1. \nonumber
\end{align}
The scalar trace of \eqref{cmp:eq:conformal-curvature} is
\begin{align}
&\Sc_{\widehat G_T}=d(d-1)+\left[\phi_T^2+2(d-1)\frac{T+f}{T}+\frac{d(d-1)}T\right]\Sc_G. \label{cmp:eq:conformal-scalar}
\end{align}
For a fixed $p$, convexity, Lemma~\ref{cmp:lem:expander}, \eqref{cmp:eq:expander}, and the end estimates imply
\begin{align}
&c(1+d_G(p,x)^2)\le1+f(x)\le C(1+d_G(p,x)^2),\qquad 0\le\Sc_G(x)\le\frac C{1+d_G(p,x)^2}. \nonumber
\end{align}
Bishop--Gromov volume comparison \cite{CE}, and dyadic integration give, for $u\ge1$,
\begin{align}
&\Vol_G\{f\le u\}\le Cu^{d/2},\qquad \int_{\{f\le u\}}\Sc_G\,dV_G\le C(1+u^{(d-2)/2}). \nonumber
\end{align}
By \eqref{cmp:eq:conformal-scalar}, for $0<\varepsilon<1/10$ and $\varepsilon T\ge1$,
\begin{align}
&\sqrt T\le\phi_T\le(1+\varepsilon)\sqrt T\quad\text{on }\{f\le\varepsilon T\},\qquad dV_{\widehat G_T}=\phi_T^{-d}dV_G, \nonumber\\
&\Vol_{\widehat G_T}\{f\le\varepsilon T\}\le T^{-d/2}C(\varepsilon T)^{d/2}=C\varepsilon^{d/2}, \nonumber\\
&\int_{\{f\le\varepsilon T\}}\Sc_{\widehat G_T}\,dV_{\widehat G_T}\le CT^{-d/2}(\varepsilon T)^{d/2}+CT^{1-d/2}\bigl(1+(\varepsilon T)^{(d-2)/2}\bigr), \nonumber\\
&\limsup_{T\to\infty}\Vol_{\widehat G_T}\{f\le\varepsilon T\}\le C\varepsilon^{d/2}, \nonumber\\
&\limsup_{T\to\infty}\int_{\{f\le\varepsilon T\}}\Sc_{\widehat G_T}\,dV_{\widehat G_T}\le C\bigl(\varepsilon^{d/2}+\varepsilon^{(d-2)/2}\bigr). \nonumber
\end{align}
This proves \eqref{cmp:eq:tip-integrability}. These sets miss the collar; the same bounds hold on both copies after multiplication by $\gamma_T$. On their complement the convergence is smooth. For $g_s$ in \eqref{cmp:eq:suspension-metric},
\begin{align}
&\Sc_{g_s}=m(m+1)+\sin^{-2}\theta(\Sc_h-m(m-1)),\qquad dV_{g_s}=\sin^m\theta\,d\theta\,dV_h, \nonumber\\
&\int_{(0,\pi)\times N}dV_{g_s}=B_mV(h), \nonumber\\
&\int_{(0,\pi)\times N}\Sc_{g_s}\,dV_{g_s}=m(m+1)B_mV(h)+B_{m-2}\bigl[J(h)-m(m-1)V(h)\bigr]. \nonumber
\end{align}
First take $T\to\infty$ on the complement, then $\varepsilon\downarrow0$; the displayed bounds give
\begin{align}
&\lim_{T\to\infty}\Vol(\mathbb S^d,g_T)=\int_{(0,\pi)\times N}dV_{g_s},\qquad \lim_{T\to\infty}J(g_T)=\int_{(0,\pi)\times N}\Sc_{g_s}\,dV_{g_s}. \nonumber
\end{align}
Thus \eqref{cmp:eq:cap-volume}--\eqref{cmp:eq:cap-scalar} hold.
For $K\subset X$ compact and $T$ sufficiently large, the interpolation misses $K$. By \eqref{cmp:eq:conformal-curvature} and Lemma~\ref{cmp:lem:Euler-algebra}, \eqref{cmp:eq:positive-wedges},
\begin{align}
&g_T|_K=\gamma_T\phi_T^{-2}G,\qquad \mathscr R_{g_T}-\Id\ge\gamma_T^{-1}\phi_T^2\mathscr R_G\ge0, \nonumber\\
&P_k(\Rm_{g_T}-G_{g_T})\ge\gamma_T^{-k}\phi_T^{2k}P_k(\Rm_G),\qquad dV_{g_T}=\gamma_T^k\phi_T^{-2k}dV_G, \nonumber\\
&P_k(\Rm_{g_T}-G_{g_T})\,dV_{g_T}\ge P_k(\Rm_G)\,dV_G. \nonumber
\end{align}
This is \eqref{cmp:eq:core-lower-bound}.
\end{proof}

\begin{lemma}
\label{cmp:lem:cover-volume}
Let $m\ge3$ and let $(N^m,h)$ be smooth, connected, compact and
without boundary. Suppose $\Ric_h\ge(m-1)h$.
Then $d_N:=|\pi_1(N)|<\infty$ and
\begin{align}
&V(h)\le\frac{(m+1)\omega_{m+1}}{d_N}. \label{cmp:eq:cover-volume}
\end{align}
Equality holds if and only if the universal cover with its lifted
metric is the unit round sphere.
\end{lemma}
\begin{proof}
Let $\pi:(\widetilde N,\widetilde h)\to(N,h)$ be the universal cover, with $\widetilde h=\pi^*h$. Then $\widetilde h$ is complete and $\Ric_{\widetilde h}\ge(m-1)\widetilde h$. For a unit-speed minimizing $\gamma:[0,\ell]\to\widetilde N$, $\ell>0$, choose parallel orthonormal normal fields $E_\alpha$ and set $W_\alpha(s)=\sin(\pi s/\ell)E_\alpha(s)$. Theorem~\ref{ext:thm:comparison} gives
\begin{align}
&0\le\sum_{\alpha=1}^{m-1}\int_0^\ell\left(|\nabla_{\dot\gamma}W_\alpha|^2-\langle\Rm_{\widetilde h}(W_\alpha,\dot\gamma)\dot\gamma,W_\alpha\rangle\right)\,ds \nonumber\\
&{}=\frac{(m-1)\pi^2}{2\ell}-\int_0^\ell\sin^2(\pi s/\ell)\Ric_{\widetilde h}(\dot\gamma,\dot\gamma)\,ds\le\frac{m-1}{2\ell}(\pi^2-\ell^2), \nonumber\\
&\ell\le\pi,\qquad \operatorname{diam}_{\widetilde h}\widetilde N\le\pi. \nonumber
\end{align}
Hopf--Rinow gives compactness. Each fiber is closed and discrete, so
\begin{align}
&d_N=|\pi_1(N)|=\#\pi^{-1}(x)<\infty,\qquad \Vol_{\widetilde h}(\widetilde N)=d_NV(h). \nonumber
\end{align}
Theorem~\ref{ext:thm:Bishop-rigidity} gives
\begin{align}
&d_NV(h)\le(m+1)\omega_{m+1},\qquad d_NV(h)=(m+1)\omega_{m+1}\ \Longleftrightarrow\ (\widetilde N,\widetilde h)\cong(\mathbb S^m,g_{\mathbb S^m}). \nonumber
\end{align}
\end{proof}

%\subsection{The compact bound in every dimension}
The numerical inequality in Theorem~\ref{cmp:thm:all-operator}, equation~\eqref{cmp:eq:compact-bound}, and the total scalar-integral maximum in Corollary~\ref{cmp:cor:total-supremum} are also proved by Ge, Li and Li \cite[Theorem~1.1, p.~2; Section~3.3, p.~9]{GeLiLi2026}, including the finite-cover normalization. We retain the derivation here to record the even-dimensional remainder and the equality criteria for the volume-corrected bound.

\begin{theorem}
\label{cmp:thm:all-operator}
Let $m\ge3$ and let $(N^m,h)$ be smooth, connected, compact and without
boundary. Suppose
\begin{align}
&\mathscr R_h\ge\Id\quad\text{on }\Lambda^2TN. \nonumber
\end{align}
Then $d_N:=|\pi_1(N)|$ is finite and
\begin{align}
&\boxed{\displaystyle J(h)\le\frac{2(m-1)(m+1)\omega_{m+1}}{d_N}+c_mV(h).} \label{cmp:eq:compact-bound}
\end{align}
For odd $m$, equality in \eqref{cmp:eq:compact-bound} holds if and only
if $h$ has constant sectional curvature one. For even $m\ge4$, equality
holds if and only if $P_2(\Rm_h-G_h)\equiv0$.
In every dimension the constants in \eqref{cmp:eq:compact-bound} are
attained by unit-curvature spherical space forms; in particular
\begin{align}
&\sup_{\mathscr R_h\ge\Id}\bigl(J(h)-c_mV(h)\bigr) =2(m-1)(m+1)\omega_{m+1}. \label{cmp:eq:sharp-sup}
\end{align}
For odd $m$, equality in the last universal bound occurs only on the
unit round sphere.
\end{theorem}
\begin{proof}
Suppose first $\pi_1(N)=0$. Lemma~\ref{cmp:lem:expander} gives $N\cong\mathbb S^m$. For $m=2k$, Proposition~\ref{cmp:prop:even-identity}, \eqref{cmp:eq:even-deficit}, gives
\begin{align}
&\chi(N)=2,\qquad J(h)-c_mV(h)=2(m-1)(m+1)\omega_{m+1}-2(m-1)\sum_{j=2}^kb_{k,j}\int_NP_j(\Rm_h-G_h)\,dV_h, \nonumber\\
&J(h)-c_mV(h)\le2(m-1)(m+1)\omega_{m+1},\qquad J(h)-c_mV(h)=2(m-1)(m+1)\omega_{m+1}\ \Longleftrightarrow\ P_2(\Rm_h-G_h)\equiv0. \nonumber
\end{align}
For odd $m$, put $d=m+1=2k$ and apply the even-dimensional bound to $g_T$ from Lemma~\ref{cmp:lem:conformal-cap}, \eqref{cmp:eq:cap-volume}--\eqref{cmp:eq:cap-scalar}. Integration by parts gives
\begin{align}
&B_m=\frac{m-1}{m}B_{m-2},\qquad (m+2)\omega_{m+2}=(m+1)\omega_{m+1}B_m. \nonumber
\end{align}
Thus
\begin{align}
&\lim_{T\to\infty}\bigl[J(g_T)-m(m-1)V(g_T)\bigr]=B_{m-2}\bigl[J(h)-c_mV(h)\bigr]\le2m(m+2)\omega_{m+2}, \label{cmp:eq:odd-transfer}\\
&J(h)-c_mV(h)\le\frac{2m(m+2)\omega_{m+2}}{B_{m-2}}=2(m-1)(m+1)\omega_{m+1}. \nonumber
\end{align}
Suppose equality holds and $m$ is odd. By \eqref{cmp:eq:odd-transfer},
\begin{align}
&\Delta_T:=2m(m+2)\omega_{m+2}-J(g_T)+m(m-1)V(g_T) \nonumber\\
&{}=2m\sum_{j=2}^kb_{k,j}\int_{\mathbb S^d}P_j(\Rm_{g_T}-G_{g_T})\,dV_{g_T}\ge0,\qquad \lim_{T\to\infty}\Delta_T=0. \nonumber
\end{align}
Since $b_{k,k}=1$, Lemma~\ref{cmp:lem:conformal-cap}, \eqref{cmp:eq:core-lower-bound}, yields
\begin{align}
&0\le\int_KP_k(\Rm_G)\,dV_G\le\frac{\Delta_T}{2m}\qquad(K\subset X\text{ compact},\ T\text{ sufficiently large}), \nonumber\\
&P_k(\Rm_G)\equiv0. \nonumber
\end{align}
If $G$ is nonflat, Theorem~\ref{ext:thm:expander-dichotomy} gives a form $\omega\in\operatorname{im}\mathscr R_G$ with $\omega^{\wedge k}\ne0$: take $\sum_{i=1}^ke^{2i-1}\wedge e^{2i}$ in the positive-operator case, or the K\"ahler form in the K\"ahler case. In a spectral eigenbasis,
\begin{align}
&\omega=\sum_{\lambda_\alpha>0}a_\alpha\omega_\alpha,\qquad \omega^{\wedge k}\ne0 \nonumber\\
&{}\Longrightarrow\quad\exists\alpha_1,\ldots,\alpha_k:\quad\lambda_{\alpha_1}\cdots\lambda_{\alpha_k}>0,\quad\omega_{\alpha_1}\wedge\cdots\wedge\omega_{\alpha_k}\ne0 \nonumber\\
&{}\Longrightarrow\quad P_k(\Rm_G)>0\qquad\text{by Lemma~\ref{cmp:lem:Euler-algebra}, \eqref{cmp:eq:positive-wedges}}, \nonumber
\end{align}
a contradiction. Therefore
\begin{align}
&\Rm_G=0,\qquad (X,G)\cong(\R^{m+1},g_{\R^{m+1}}),\qquad \Hess f=G/2,\qquad f-|\nabla f|^2=0, \nonumber\\
&f(x)=\tfrac14|x-p|^2,\qquad (N,h)\cong(\mathbb S^m,g_{\mathbb S^m}). \nonumber
\end{align}
Here completeness and the disk exhaustion give the simply connected Euclidean realization. The round metric attains equality by substitution.
For general $N$, Lemma~\ref{cmp:lem:cover-volume} gives a compact universal cover $\pi:(\widetilde N,\widetilde h)\to(N,h)$ with
\begin{align}
&\Ric_h\ge(m-1)h,\qquad d_N<\infty,\qquad \mathscr R_{\widetilde h}\ge\Id, \nonumber\\
&V(\widetilde h)=d_NV(h),\qquad J(\widetilde h)=d_NJ(h), \nonumber\\
&d_N\bigl[J(h)-c_mV(h)\bigr]\le2(m-1)(m+1)\omega_{m+1}. \nonumber
\end{align}
The equality criteria pass through the local isometry; for odd $m$, equality is equivalent to $(N,h)$ being a unit-curvature spherical quotient. Finally,
\begin{align}
&J(h)-c_mV(h)=2(m-1)(m+1)\omega_{m+1}\ \Longrightarrow\ d_N=1, \nonumber\\
&J(g_{\mathbb S^m})-c_mV(g_{\mathbb S^m})=2(m-1)(m+1)\omega_{m+1}. \nonumber
\end{align}
This proves \eqref{cmp:eq:sharp-sup} and its odd-dimensional equality assertion.
\end{proof}

%\subsection{The total scalar integral and its maximizers}
Combining Theorem~\ref{cmp:thm:all-operator},
equation~\eqref{cmp:eq:compact-bound}, with
Lemma~\ref{cmp:lem:cover-volume}, equation~\eqref{cmp:eq:cover-volume},
determines the sharp bound for $J$.

\begin{corollary}
\label{cmp:cor:total-supremum}
For every $(N^m,h)\in\mathcal C_m^{\mathrm{op}}(1)$,
\begin{align}
&\boxed{\displaystyle J(h)\le\frac{m(m-1)(m+1)\omega_{m+1}}{d_N}.} \label{cmp:eq:total-degree-bound}
\end{align}
Equality holds if and only if $(N,h)$ is isometric to
$\mathbb S^m(1)/\Gamma$ with its quotient metric, where $\Gamma$
is a finite group of isometries of order $d_N$ acting freely:
$\gamma(x)=x$ implies $\gamma=\Id$. Consequently
\begin{align}
&\boxed{\displaystyle \sup_{(N^m,h)\in\mathcal C_m^{\mathrm{op}}(1)}J(h) =m(m-1)(m+1)\omega_{m+1},\qquad m\ge3.} \label{cmp:eq:total-supremum}
\end{align}
The unit round sphere is the unique maximizer up to isometry,
in both even and odd dimensions.
\end{corollary}
\begin{proof}
For every orthonormal pair,
\begin{align}
&K_h(e_i,e_j)=\langle\mathscr R_h(e_i\wedge e_j),e_i\wedge e_j\rangle\ge1,\qquad \Ric_h(e_i,e_i)=\sum_{j\ne i}K_h(e_i,e_j)\ge m-1. \nonumber
\end{align}
Theorem~\ref{cmp:thm:all-operator}, \eqref{cmp:eq:compact-bound}, and Lemma~\ref{cmp:lem:cover-volume}, \eqref{cmp:eq:cover-volume}, give
\begin{align}
&J(h)\le\frac{2(m-1)(m+1)\omega_{m+1}}{d_N}+c_mV(h)\le\frac{\bigl(2(m-1)+c_m\bigr)(m+1)\omega_{m+1}}{d_N}=\frac{m(m-1)(m+1)\omega_{m+1}}{d_N}. \nonumber
\end{align}
Since $c_m=(m-1)(m-2)>0$, equality implies
\begin{align}
&V(h)=\frac{(m+1)\omega_{m+1}}{d_N},\qquad (\widetilde N,\widetilde h)\cong(\mathbb S^m,g_{\mathbb S^m})\qquad\text{by Lemma~\ref{cmp:lem:cover-volume}}, \nonumber\\
&(N,h)\cong\mathbb S^m(1)/\Gamma,\qquad |\Gamma|=d_N,\qquad \gamma(x)=x\ \Longrightarrow\ \gamma=\Id. \nonumber
\end{align}
Here $\Gamma$ is the deck group, acting isometrically. Conversely, each such quotient satisfies
\begin{align}
&\Sc_h=m(m-1),\qquad V(h)=(m+1)\omega_{m+1}/d_N,\qquad J(h)=\frac{m(m-1)(m+1)\omega_{m+1}}{d_N}. \nonumber
\end{align}
Thus \eqref{cmp:eq:total-degree-bound} is attained precisely as stated. Finally,
\begin{align}
&d_N\ge1,\qquad J(g_{\mathbb S^m})=m(m-1)(m+1)\omega_{m+1}, \nonumber\\
&J(h)=m(m-1)(m+1)\omega_{m+1}\ \Longrightarrow\ d_N=1\ \Longrightarrow\ (N,h)\cong(\mathbb S^m,g_{\mathbb S^m}), \nonumber
\end{align}
which proves \eqref{cmp:eq:total-supremum} and uniqueness.
\end{proof}

%\subsection{Arbitrary curvature scale and explicit maxima}
\begin{corollary}
\label{cmp:cor:dimensions}\label{cmp:cor:scaled-total}
Let $m\ge3$, $k_0>0$, and
$(N^m,h)\in\mathcal C_m^{\mathrm{op}}(k_0)$. Then
\begin{align}
&J(h)\le \frac{2(m-1)(m+1)\omega_{m+1}}{d_N\,k_0^{(m-2)/2}}+c_mk_0V(h), \label{cmp:eq:scaled-bound}\\
&J(h)\le\frac{m(m-1)(m+1)\omega_{m+1}} {d_N\,k_0^{(m-2)/2}}, \label{cmp:eq:scaled-total-degree}\\
&\sup_{(N^m,h)\in\mathcal C_m^{\mathrm{op}}(k_0)}J(h) =m(m-1)(m+1)\omega_{m+1}\,k_0^{-(m-2)/2}. \label{cmp:eq:scaled-total-supremum}
\end{align}
Equality in \eqref{cmp:eq:scaled-total-degree} holds precisely for
quotients of the round sphere of sectional curvature $k_0$ by finite
free isometric actions. The round sphere is the unique universal
maximizer in \eqref{cmp:eq:scaled-total-supremum}.
At $k_0=1$, \eqref{cmp:eq:scaled-bound} gives
\begin{align}
&m=3:\quad J(h)\le\frac{8\pi^2}{d_N}+2V(h),\qquad m=4:\quad J(h)\le\frac{16\pi^2}{d_N}+6V(h), \nonumber\\
&m=5:\quad J(h)\le\frac{8\pi^3}{d_N}+12V(h),\qquad m=6:\quad J(h)\le\frac{32\pi^3}{3d_N}+20V(h). \label{cmp:eq:dimension-table}
\end{align}
In dimension three, $\mathscr R_h\ge k_0\Id$ is equivalent to
$\sec_h\ge k_0$, and
\begin{align}
&\sup_{(N^3,h)\in\mathcal C_3^{\mathrm{sec}}(k_0)}J(h) =12\pi^2k_0^{-1/2}. \label{cmp:eq:sectional-three-total}
\end{align}
Equality in \eqref{cmp:eq:scaled-bound} for $m=3$ is precisely
constant sectional curvature $k_0$; with $d_N$ replaced by one,
it is precisely the round sphere of curvature $k_0$.
\end{corollary}
\begin{proof}
Set $\widehat h=k_0h$. In the respective orthonormal frames,
\begin{align}
&\mathscr R_{\widehat h}=k_0^{-1}\mathscr R_h\ge\Id,\qquad J(\widehat h)=k_0^{(m-2)/2}J(h),\qquad V(\widehat h)=k_0^{m/2}V(h). \nonumber
\end{align}
The covering degree is unchanged. Theorem~\ref{cmp:thm:all-operator}, \eqref{cmp:eq:compact-bound}, and Corollary~\ref{cmp:cor:total-supremum}, \eqref{cmp:eq:total-degree-bound}--\eqref{cmp:eq:total-supremum}, give
\begin{align}
&k_0^{(m-2)/2}J(h)\le\frac{2(m-1)(m+1)\omega_{m+1}}{d_N}+c_mk_0^{m/2}V(h), \nonumber\\
&k_0^{(m-2)/2}J(h)\le\frac{m(m-1)(m+1)\omega_{m+1}}{d_N},\qquad \sec_{\widehat h}\equiv1\ \Longleftrightarrow\ \sec_h\equiv k_0. \nonumber
\end{align}
Dividing by $k_0^{(m-2)/2}>0$ gives the bounds and transfers all equality statements. The numerical values are
\begin{align}
&4\omega_4=2\pi^2,\qquad 5\omega_5=8\pi^2/3,\qquad 6\omega_6=\pi^3,\qquad 7\omega_7=16\pi^3/15, \nonumber
\end{align}
which give \eqref{cmp:eq:dimension-table}. In dimension three choose an oriented orthonormal coframe with $*\omega=|\omega|e^3$ for each nonzero two-form. Then
\begin{align}
&\omega=|\omega|e^1\wedge e^2,\qquad \mathscr R_h\ge k_0\Id\ \Longleftrightarrow\ \sec_h\ge k_0, \nonumber\\
&\sup_{\mathcal C_3^{\mathrm{sec}}(k_0)}J=3\cdot2\cdot4\omega_4k_0^{-1/2}=12\pi^2k_0^{-1/2}. \nonumber
\end{align}
Thus \eqref{cmp:eq:sectional-three-total} and its equality statements follow.
\end{proof}

At curvature scale one, the maxima in Corollary~\ref{cmp:cor:total-supremum},
equation~\eqref{cmp:eq:total-supremum}, are
\begingroup
\renewcommand{\arraystretch}{1.35}
\begin{align}
&\begin{array}{c|rrrrrr} m&3&4&5&6&7&8\\ \hline \displaystyle\sup_{\mathcal C_m^{\mathrm{op}}(1)}J &12\pi^2&32\pi^2&20\pi^3&32\pi^3&14\pi^4& \dfrac{256}{15}\pi^4 \end{array} \nonumber
\end{align}
\endgroup
Indeed $2\omega_2=2\pi$, $3\omega_3=4\pi$, and
$(m+1)\omega_{m+1}=2\pi\omega_{m-1}$ for $m\ge3$; substitution into
$m(m-1)(m+1)\omega_{m+1}$ gives every entry.

%\subsection{Sectional curvature and the quadratic contraction}
\label{cmp:sec:sectional-data}
The hypotheses $\mathscr R_h\ge\Id$ and $\sec_h\ge1$ are different
when $m\ge4$. With $A_h=\Rm_h-G_h$, the second hypothesis means
\begin{align}
&\langle(\mathscr R_h-\Id)(X\wedge Y),X\wedge Y\rangle\ge0 \quad(X,Y\in T_xN,\ x\in N), \nonumber
\end{align}
whereas the first imposes the inequality on every element of
$\Lambda^2T_xN$. For example,
$(e_1\wedge e_2+e_3\wedge e_4)\wedge (e_1\wedge e_2+e_3\wedge e_4)\ne0$ in an orthonormal four-frame,
so this two-form is not decomposable. Theorem~\ref{cmp:thm:all-operator} therefore does not supply the
sectional-curvature estimate in dimensions $m\ge4$.

The following calculations use Theorem~\ref{ext:thm:closed-CGB},
equation~\eqref{cmp:eq:closed-Euler},
Theorem~\ref{ext:thm:Bishop-rigidity},
Lemma~\ref{six:lem:P2-positive}, equation~\eqref{six:eq:P2-positive},
and the displayed metrics.
The constants are numerical; the curvature-scale parameter is $k_0>0$.

%\subsection{The direct four-dimensional sectional-curvature estimate}
\label{cmp:sec:sectional-four}
\begin{proposition}
\label{cmp:prop:sectional-four}
Let $(N^4,h)$ be smooth, connected, compact and without boundary.
Without a curvature assumption,
\begin{align}
&J(h) =8\pi^2\chi(N)+6V(h) -6\int_NP_2(A_h)\,dV_h. \label{cmp:eq:sectional-four}
\end{align}
If $\sec_h\ge1$, then
\begin{align}
&J(h)\le8\pi^2\chi(N)+6V(h). \label{cmp:eq:four-sectional-Euler-bound}
\end{align}
Equality in \eqref{cmp:eq:four-sectional-Euler-bound} holds precisely
when $P_2(A_h)\equiv0$. More generally, $\sec_h\ge k_0>0$ gives
\begin{align}
&J(h)\le\frac{8\pi^2}{k_0}\chi(N)+6k_0V(h), \nonumber
\end{align}
with equality precisely when $P_2(\Rm_h-k_0G_h)\equiv0$.
\end{proposition}
\begin{proof}
Proposition~\ref{cmp:prop:even-identity}, \eqref{cmp:eq:even-deficit}, in dimension four reads
\begin{align}
&P_2(G_h+A_h)=1+\frac{\Sc(A_h)}6+P_2(A_h)=\frac{\Sc_h}6-1+P_2(A_h), \nonumber\\
&\int_NP_2(\Rm_h)\,dV_h=\frac{4\pi^2}{3}\chi(N)=\frac{J(h)}6-V(h)+\int_NP_2(A_h)\,dV_h, \nonumber\\
&J(h)=8\pi^2\chi(N)+6V(h)-6\int_NP_2(A_h)\,dV_h. \nonumber
\end{align}
For nonorientable $N$, all integrals and $\chi$ on the orientation double cover are twice those on $N$, so division by two gives the same identity. Lemma~\ref{six:lem:P2-positive}, \eqref{eq:P2-four-plane-sum}, and continuity yield
\begin{align}
&\sec_h\ge1\ \Longrightarrow\ P_2(A_h)\ge0,\qquad \int_NP_2(A_h)\,dV_h=0\ \Longleftrightarrow\ P_2(A_h)\equiv0. \nonumber
\end{align}
For $\sec_h\ge k_0>0$, apply this to $k_0h$. Lemma~\ref{cmp:lem:scaling}, \eqref{cmp:eq:scaling}, gives
\begin{align}
&J(k_0h)=k_0J(h),\qquad V(k_0h)=k_0^2V(h), \nonumber\\
&\Rm_{k_0h}-G_{k_0h}=k_0^{-1}(\Rm_h-k_0G_h)\quad\text{in the respective orthonormal frames}, \nonumber\\
&J(h)\le\frac{8\pi^2}{k_0}\chi(N)+6k_0V(h),\qquad P_2(\Rm_{k_0h}-G_{k_0h})=k_0^{-2}P_2(\Rm_h-k_0G_h). \nonumber
\end{align}
The last identity gives the scaled equality criterion.
\end{proof}

For the whole class $\mathcal C_4^{\mathrm{sec}}(1)$ the exact
remaining distinction is
\begin{align}
&\frac{16\pi^2}{d_N}+6V(h)-J(h) =6\int_NP_2(A_h)\,dV_h -8\pi^2\left(\chi(N)-\frac2{d_N}\right), \nonumber\\
&J(h)\le\frac{16\pi^2}{d_N}+6V(h) \quad\Longleftrightarrow\quad \int_NP_2(A_h)\,dV_h\ge \frac{4\pi^2}{3}\left(\chi(N)-\frac2{d_N}\right). \label{cmp:eq:four-topological-threshold}
\end{align}
The lower bound in \eqref{cmp:eq:four-topological-threshold} remains
unproved for arbitrary topology. Similarly,
\begin{align}
&\frac{32\pi^2}{d_N}-J(h) ={}6\left(\frac{8\pi^2}{3d_N}-V(h)\right) +6\int_NP_2(A_h)\,dV_h +8\pi^2\left(\frac2{d_N}-\chi(N)\right). \nonumber
\end{align}
The first two terms are nonnegative under $\sec_h\ge1$. No sign
for their sum with the last term is deduced here.

Under the operator hypothesis the required topology follows already
from Lemma~\ref{cmp:lem:expander} on the universal cover, which is a
sphere and has Euler characteristic two. An independent cited
justification is the theorem of B\"ohm--Wilking: on a closed smooth
manifold of dimension at least three, a positive curvature operator
is evolved by normalized Ricci flow to a constant-positive-curvature
metric \cite[Main Theorem]{BohmWilking2008}. It is the topology of
the underlying manifold, not constant curvature of the original
metric, that is used here. Euler characteristic multiplies by the
finite covering degree, giving $d_N\chi(N)=2$.

%\subsection{Five-dimensional sections and a six-dimensional identity}
\label{cmp:sec:sectional-five}
For $(N^5,h)\in\mathcal C_5^{\mathrm{sec}}(1)$ the two assertions
under consideration are
\begin{align}
&J(h)\le\frac{8\pi^3}{d_N}+12V(h), \label{cmp:eq:five-sectional-target}\\
&J(h)\le\frac{20\pi^3}{d_N}. \label{cmp:eq:five-sectional-total-target}
\end{align}
Both hold under the operator hypothesis. For the sectional class,
the first remains unproved and would imply the second by
Lemma~\ref{cmp:lem:cover-volume}. Passing to the finite universal
cover multiplies $J$ and $V$ by $d_N$; hence it suffices to prove the
first with $d_N=1$. The unit sphere, for which $\Sc=20$,
$V=\pi^3$ and $J=20\pi^3$, attains both candidate constants.

By Poincar\'e duality over $\mathbb F_2$, every closed five-manifold
has $\chi=0$: its Betti numbers satisfy $b_i=b_{5-i}$ and cancel
in the alternating sum. The sign $P_2(A_h)\ge0$ therefore does not
produce a five-dimensional counterpart of
Proposition~\ref{cmp:prop:sectional-four},
equation~\eqref{cmp:eq:sectional-four}. The relevant higher-dimensional
identity is the following.

\begin{proposition}
\label{cmp:prop:sectional-six}
Let $(Y^6,k)$ be smooth, connected, compact and without boundary,
and put $A_k=\Rm_k-G_k$. With no curvature hypothesis,
\begin{align}
&J(k) =\frac{16\pi^3}{3}\chi(Y)+20V(k) -2\int_YP_2(A_k)\,dV_k-10\int_YP_3(A_k)\,dV_k. \nonumber
\end{align}
In particular, if $\chi(Y)=2$, then
\begin{align}
&J(k)\le\frac{32\pi^3}{3}+20V(k) \nonumber
\end{align}
is equivalent to
\begin{align}
&\int_Y\left(\frac15P_2(A_k)+P_3(A_k)\right)dV_k\ge0. \label{cmp:eq:six-combined-sectional}
\end{align}
The hypothesis $\sec_k\ge1$ guarantees $P_2(A_k)\ge0$, but no sign
of $P_3(A_k)$ is asserted under that hypothesis.
\end{proposition}
\begin{proof}
Lemma~\ref{cmp:lem:Euler-algebra}, \eqref{cmp:eq:Euler-expansion}, with $k=3$, and Theorem~\ref{ext:thm:closed-CGB}, \eqref{cmp:eq:closed-Euler}, give
\begin{align}
&P_3(G+A)=1+\frac{\Sc(A)}{10}+\frac15P_2(A)+P_3(A),\qquad \Sc(A_k)=\Sc_k-30, \nonumber\\
&\int_YP_3(\Rm_k)\,dV_k=\frac{7\omega_7}{2}\chi(Y)=\frac{8\pi^3}{15}\chi(Y), \nonumber\\
&\frac{8\pi^3}{15}\chi(Y)=\frac{J(k)}{10}-2V(k)+\frac15\int_YP_2(A_k)\,dV_k+\int_YP_3(A_k)\,dV_k, \nonumber\\
&J(k)=\frac{16\pi^3}{3}\chi(Y)+20V(k)-2\int_YP_2(A_k)\,dV_k-10\int_YP_3(A_k)\,dV_k. \nonumber
\end{align}
The orientation double cover gives the same identity when $Y$ is nonorientable. For $\chi(Y)=2$,
\begin{align}
&\frac{32\pi^3}{3}+20V(k)-J(k)=10\int_Y\left(\frac15P_2(A_k)+P_3(A_k)\right)dV_k. \nonumber
\end{align}
Lemma~\ref{six:lem:P2-positive} and \eqref{eq:P2-four-plane-sum} give $\sec_k\ge1\Longrightarrow P_2(A_k)\ge0$.
\end{proof}

The lack of a pointwise cubic sign is a precise algebraic statement.
Geroch constructs a six-dimensional algebraic curvature tensor $B$
with positive sectional curvatures and $P_3(B)<0$
\cite[pp.~267--270]{Geroch1976}. For all sufficiently large
$\lambda>0$,
\begin{align}
&\frac15P_2(\lambda B)+P_3(\lambda B) =\frac{\lambda^2}{5}P_2(B)+\lambda^3P_3(B)<0, \nonumber
\end{align}
although $G+\lambda B$ has sectional curvature greater than one.
This is not a closed-manifold example and does not settle the
integral in Proposition~\ref{cmp:prop:sectional-six},
equation~\eqref{cmp:eq:six-combined-sectional}.

%\subsection{Scope of the compact bounds}
\label{cmp:sec:sectional-realization}
The conical derivative bounds in Lemma~\ref{cmp:lem:expander},
equation~\eqref{cmp:eq:AC-estimates}, necessarily imply
\begin{align}
&\mathscr R_h-\Id\ge0 \nonumber
\end{align}
whenever the realizing metric has nonnegative curvature operator:
the rescaled tangential curvature block converges uniformly to
$\mathscr R_h-\Id$, and nonnegativity passes to this limit.
The unit round sphere gives
\begin{align}
&\sup_{(N^m,h)\in\mathcal C_m^{\mathrm{sec}}(1)}J(h) \ge m(m-1)(m+1)\omega_{m+1}. \nonumber
\end{align}
The reverse inequality for $m\ge4$ remains unproved here; the operator realization does not prove it.

\end{document}